\documentclass[11pt,reqno]{article}

\usepackage[margin=1in]{geometry}
\usepackage{amsthm}
\usepackage{amsmath}
\usepackage{amsfonts}
\usepackage{amssymb}
\usepackage{prettyref}
\usepackage[colorlinks,linkcolor=red,citecolor=blue,urlcolor=black]{hyperref}
\usepackage[numbers]{natbib}
\usepackage{graphicx}
\usepackage{framed}
\usepackage{enumerate}
\usepackage{bm}
\usepackage{mathrsfs}
\usepackage{xcolor}
\usepackage{multirow}
\usepackage{pbox}
\usepackage{caption}
\usepackage{subcaption}
\usepackage{placeins}

\newcommand{\N}{\mathbb{N}}
\newcommand{\R}{\mathbb{R}}

\newrefformat{eq}{(\ref{#1})}
\newrefformat{ineq}{(\ref{#1})}
\newrefformat{thm}{Theorem~\ref{#1}}
\newrefformat{th}{Theorem~\ref{#1}}
\newrefformat{chap}{Chapter~\ref{#1}}
\newrefformat{sec}{Section~\ref{#1}}
\newrefformat{seca}{Section~\ref{#1}}
\newrefformat{algo}{Algorithm~\ref{#1}}
\newrefformat{assump}{Assumption~\ref{#1}}
\newrefformat{fig}{Fig.~\ref{#1}}
\newrefformat{tab}{Table~\ref{#1}}
\newrefformat{rmk}{Remark~\ref{#1}}
\newrefformat{clm}{Claim~\ref{#1}}
\newrefformat{def}{Definition~\ref{#1}}
\newrefformat{cor}{Corollary~\ref{#1}}
\newrefformat{lmm}{Lemma~\ref{#1}}
\newrefformat{lem}{Lemma~\ref{#1}}
\newrefformat{lemma}{Lemma~\ref{#1}}
\newrefformat{prop}{Proposition~\ref{#1}}
\newrefformat{pr}{Proposition~\ref{#1}}
\newrefformat{app}{Appendix~\ref{#1}}
\newrefformat{apx}{Appendix~\ref{#1}}
\newrefformat{ex}{Example~\ref{#1}}
\newrefformat{exer}{Exercise~\ref{#1}}
\newrefformat{soln}{Solution~\ref{#1}}

\newcommand{\Expect}{\mathbb{E}}

\newcommand{\Prob}{\mathbb{P}}

\newcommand{\calN}{\mathcal{N}}
\newcommand{\calP}{\mathcal{P}}

\newcommand{\pnorm}[2]{\lVert #1\rVert_{#2}}

\newcommand{\biggpnorm}[2]{\bigg\lVert#1\bigg\rVert_{#2}}

\newcommand{\iprod}[2]{\left\langle#1,#2\right\rangle}

\renewcommand{\epsilon}{\varepsilon}

\renewcommand{\d}[1]{\mathrm{d}#1}

\renewcommand{\hat}{\widehat}
\renewcommand{\tilde}{\widetilde}

\DeclareMathOperator{\E}{\mathbb{E}}

\DeclareMathOperator{\tr}{tr}
\DeclareMathOperator{\var}{Var}
\DeclareMathOperator{\cov}{Cov}
\DeclareMathOperator{\rank}{rank}
\DeclareMathOperator{\grad}{grad}
\DeclareMathOperator{\tangent}{\text{Tan}}
\DeclareMathOperator{\op}{op}

\DeclareMathOperator{\ent}{\text{ent}}

\DeclareMathOperator{\kl}{\text{KL}}

\DeclareMathOperator{\diag}{\textrm{diag}}

\let\liminf\relax
\DeclareMathOperator*\liminf{lim\,inf}
\let\limsup\relax
\DeclareMathOperator*\limsup{lim\,sup}

\DeclareMathOperator*{\argmin}{arg\,min}

\def\E{\mathbb{E}}
\def\R{\mathbb{R}}
\def\r{\mathbf{r}}
\def\a{\mathbf{a}}
\def\b{\mathbf{b}}
\def\y{\mathbf{y}}
\def\x{\mathbf{x}}
\def\z{\mathbf{z}}
\def\v{\mathbf{v}}
\def\s{\mathbf{s}}
\def\1{\mathbf{1}}
\def\u{\mathbf{u}}

\def\e{\mathbf{e}}
\def\P{\mathbb{P}}
\def\X{\mathbf{X}}

\def\B{\mathbf{B}}

\def\g{\mathbf{g}}
\def\V{\mathbf{V}}
\def\W{\mathbf{W}}
\def\boldeta{{\boldsymbol\eta}}
\def\balpha{{\boldsymbol\alpha}}
\def\btheta{{\boldsymbol\theta}}

\def\beps{{\boldsymbol\varepsilon}}
\def\bvarphi{{\boldsymbol\varphi}}
\def\blambda{{\boldsymbol\lambda}}
\def\bSigma{{\boldsymbol\Sigma}}
\def\bGamma{{\boldsymbol\Gamma}}
\def\N{\mathcal{N}}

\def\cP{\mathcal{P}}
\def\cS{\mathcal{S}}
\def\cT{\mathcal{T}}
\def\Id{\operatorname{Id}}
\def\d{\mathrm{d}}

\def\DKL{D_{\mathrm{KL}}}
\def\dTV{d_{\mathrm{TV}}}
\def\grad{\operatorname{grad}}
\def\Tr{\operatorname{Tr}}

\def\op{{\operatorname{op}}}
\def\eps{\varepsilon}
\def\Var{\operatorname{Var}}

\def\FR{{\operatorname{FR}}}
\def\cK{\mathcal{K}}

\newcommand{\revise}[1]{{\color{blue} #1}}

\newcommand{\beq}{\begin{equation}}
\newcommand{\eeq}{\end{equation}}
\newcommand{\beqa}{\begin{equation} \begin{aligned}}
\newcommand{\eeqa}{\end{aligned} \end{equation}}
\newcommand{\beqas}{\begin{equation*} \begin{aligned}}
\newcommand{\eeqas}{\end{aligned} \end{equation*}}

\newcommand{\bit}{\begin{itemize}}
	\newcommand{\eit}{\end{itemize}}
\newcommand{\bmat}{\begin{bmatrix}}
	\newcommand{\emat}{\end{bmatrix}}

\newtheorem{theorem}{Theorem}[section]
\newtheorem{lemma}[theorem]{Lemma}

\newtheorem{proposition}[theorem]{Proposition}
\newtheorem{corollary}[theorem]{Corollary}
\theoremstyle{definition}
\newtheorem{remark}[theorem]{Remark}

\newtheorem{assumption}[theorem]{Assumption}

\newcommand*\samethanks[1][\value{footnote}]{\footnotemark[#1]}

\begin{document}

\title{Gradient flows for empirical Bayes in high-dimensional linear models}
\author{Zhou Fan\thanks{Department of Statistics and Data Science, Yale
University}, Leying Guan\thanks{Department of Biostatistics, Yale University \\ 
\indent\indent \texttt{zhou.fan@yale.edu}, \texttt{leying.guan@yale.edu},
\texttt{yandi.shen@yale.edu}, \texttt{yihong.wu@yale.edu}\\
This work is supported in part by NSF DMS-2142476, NSF
DMS-2310836, NSF CCF-1900507.},
Yandi Shen\samethanks[1], Yihong Wu\samethanks[1]}
\date{}

\maketitle

\begin{abstract}
Empirical Bayes provides a powerful framework for learning and adapting to latent structure in data. In sequence models where an independent observation is associated to each latent parameter, theory and methods around empirical Bayes are well-developed. However, in models where latent parameters and observed data interact through more complex designs, many statistical and algorithmic questions remain unanswered.

In this work, we study a canonical setting of empirical Bayes estimation for the distribution of regression coefficients in a high-dimensional Bayesian or random effects linear model.  Computationally, we propose a new system of gradient flow equations for computing a nonparametric maximum likelihood estimator (NPMLE), which jointly optimizes over the prior and posterior distributions of the regression coefficients in a Gibbs variational representation of the marginal log-likelihood. A diffusion-based implementation yields an adaptive Langevin dynamics algorithm in which the prior evolves continuously to optimize a sequence model log-likelihood defined by the coordinates of the Langevin sample. Theoretically, we show polynomial-time convergence of the proposed gradient flow to a near-NPMLE from any initialization within a convex sub-level set of the marginal log-likelihood, by developing a high-temperature log-Sobolev inequality for the posterior law. We establish the statistical consistency of any near-NPMLE under deterministic conditions for the regression design as $n,p \to \infty$.
\end{abstract}

\section{Introduction}
\label{sec:intro}

Introduced by Robbins \cite{Robbins51, Robbins56} in the 1950s, empirical Bayes (EB)
is a powerful framework for large-scale inference that
learns and adapts to latent structure in data. In the most commonly studied
setting, one or several independent samples are
observed corresponding to each latent parameter, for example in the
Gaussian sequence model
\begin{align}\label{def:gau_seq}
y_i = \theta_i + \epsilon_i \qquad \text{ for } i=1,\ldots,n.
\end{align}
Treating $\{\theta_i\}_{i=1}^n$ as i.i.d.\ draws from a prior
distribution $g_\ast$, empirical Bayes methods
proceed by estimating $g_\ast$ from the observations $\{y_i\}_{i=1}^n$
and using the estimated prior $g_\ast$ to perform downstream Bayesian
inference. Such methods have found fruitful applications in numerous
areas of statistics and the applied sciences, and we refer to the review
articles of \cite{casella1985introduction,zhang2003compound,efron2021empirical}
and the monographs of \cite{maritz1989empirical,carlin1996bayes,efron2010large}
for systematic treatments of this broad subject.

In many applications, the data $\y \in \R^n$ may instead have
a more complex joint dependence on the latent parameters $\btheta \in \R^p$.
In this paper, we consider a canonical such setting of a Bayesian or random
effects linear model
\begin{align}\label{def:linear_model_intro}
\y=\X\btheta+\beps, \qquad \theta_j \overset{iid}{\sim} g_\ast
\text{ for } j=1,\ldots,p,
\qquad \eps_i \overset{iid}{\sim} \N(0,\sigma^2) \text{ for } i=1,\ldots,n.
\end{align}
Assuming that the sample size $n$ and dimension $p$ are both large, we study
the analogous question of estimating the prior distribution $g_*$ of the
regression coefficients from the data $(\X,\y)$.
One application in which this problem arises is in estimating
effect size distributions and genetic architectures of complex traits in
statistical genetics, and we refer to 
\cite{zhang2018estimation,oconnor2021distribution,zhou2021fast,spence2022flexible,morgante2023flexible}
for examples of methods developed in this context.

In practice, a common paradigm for estimating $g_*$ is
maximum marginal likelihood: Letting $\cP$ be a parametric or nonparametric
family of prior distributions on $\R$, one aims to compute
\begin{align}\label{def:npmle}
\hat{g}=\argmin_{g \in \cP} \bar{F}_n(g),
\quad \bar{F}_n(g):=-\frac{1}{p}\log \int
\frac{1}{(2\pi\sigma^2)^{n/2}}
\exp\left(-\frac{\|\y-\X\btheta\|^2}{2\sigma^2}\right)\prod_{j=1}^p
\d g(\theta_j)
\end{align}
where $\bar F_n(g)$ is the (rescaled) negative log-likelihood of $\y$,
marginalizing over the latent regression coefficients $\btheta \in \R^p$.
Such an approach is well-studied in the Gaussian sequence model
\eqref{def:gau_seq}, corresponding to the identity design $\X=\Id$
in the linear model. In this
setting (deconvolution), many works have established consistency and rates of
estimation for $\widehat g$ and for associated empirical Bayes estimates of
$\btheta$ (see Section \ref{sec:related_work} for pointers to this literature),
and nonparametric
optimization of $\bar F_n$ is also amenable to convex programming techniques.
In contrast, in the linear model
for general regression designs $\X$, even basic questions such as
conditions for consistency of $\widehat g$ in parametric contexts
as $n,p \to \infty$ are
unanswered, and (non-convex) optimization of $\bar F_n$ is a substantially
more challenging task. Various approximate optimization methods based on
mean-field variational inference (VI)
\cite{carbonetto2012scalable,spence2022flexible,kim2022flexible} and Monte
Carlo Expectation-Maximization (MCEM) or hierarchical Bayesian sampling
\cite{zhou2013polygenic,ge2019polygenic} have been proposed, but a complete
theoretical understanding of such methods is still lacking.

In this work, focusing on nonparametric estimation of $g_*$ over the prior
class of distributions with a fixed and bounded interval of support $[-M,M]$,
we make three contributions to the theory and computation of a nonparametric
maximum likelihood estimator (NPMLE):

\begin{enumerate}[1.]
\item In an asymptotic setting where $n,p \to \infty$, under a
deterministic assumption on the design matrix $\X$, we establish nonparametric
consistency for any near-NPMLE, i.e.\ any approximate optimizer $\widehat g$
of the marginal log-likelihood satisfying $\bar F_n(\widehat g) \leq
\bar{F}_n(g_*)+o(1)$.

Our assumption for $\X$ is reminiscent of inference procedures such as the
debiased lasso \cite{zhang2014confidence,van2014asymptotically} for individual
regression coefficients, and requires the existence of test
vectors $\{\z_j\}_{j=1}^p$ that can ``approximately orthogonalize'' a sufficient
number of variables $\{\x_j\}_{j=1}^p$ against the other variables
of the design. This extends a recent result of \cite{mukherjee2023mean} showing
consistency of $\widehat g$ under the stronger assumption that $n \geq p$ and
$\X$ is well-conditioned, and encompasses also settings where $n<p$
and $\X$ is a random design, having i.i.d.\ subgaussian rows with
well-conditioned population covariance.

\item Inspired by and building upon a line of work on optimizing functionals
of probability measures using gradient flows
\citep{jordan1998variational,chizat2018global,mei2018mean,lambert2022variational,yan2023learning},
we propose a system of bivariate gradient flow equations for computing a
near-NPMLE. Briefly, our approach consists of reparametrizing\footnote{
This reparameterization is introduced for multiple reasons, such as ensuring 
$\bvarphi$ has a smooth density on $\R^p$ which lends itself to Langevin-type
sampling algorithms for its posterior. See \prettyref{sec:phireparam} for a detailed discussion.} the linear
model (\ref{def:linear_model_intro}) using an auxiliary latent
variable $\bvarphi \in \R^p$ as
\begin{equation}\label{eq:reparamintro}
\y=\X\bvarphi+\tilde \beps,
\end{equation}
where $\bvarphi$ has i.i.d.\ coordinates with a Gaussian convolution
prior $\N_{\tau}*g_*$, and then optimizing
the Gibbs variational representation of $\bar{F}_n(g)$ in this
parametrization by $\bvarphi$,
\begin{equation}
\bar{F}_n(g)=\min_{q \in \cP_*(\R^p)} F_n(q,g),
\quad F_n(q,g)={-}\frac{1}{p}\int \Big[\log P_g(\y,\bvarphi)
-\log q(\bvarphi)\Big]\d q(\bvarphi).
\label{eq:varrep}
\end{equation}
Here $P_g(\y,\bvarphi)$ is the joint density of $(\y,\bvarphi)$
under the convolution prior $\N_\tau*g$ for $\bvarphi$,
the minimization in \eqref{eq:varrep}
is over all probability density functions $q(\bvarphi)$ on
$\R^p$, and its minimum attained at the posterior density
$q(\bvarphi)=P_g(\bvarphi \mid \y)$.

We propose to optimize $\bar F_n(g)$ by optimizing the Gibbs free energy
$F_n(q,g)$ jointly over $(q,g)$, via the simultaneous gradient flow equations
\begin{align}\label{eq:joint_flow_intro}
\begin{cases}
\frac{\d}{\d t} q_t(\bvarphi) &= {-}p \cdot \grad^{W_2}_q F_{n}(q_t,g_t)[\bvarphi]\\
\frac{\d}{\d t} g_t(\theta) &= {-}\alpha \cdot
\grad^{\FR}_g F_{n}(q_t,g_t)[\theta].
\end{cases}
\end{align}
In \eqref{eq:joint_flow_intro}, $\grad^{W_2}_q$ denotes the $q$-gradient of $F_n(q,g)$ under the Wasserstein-2 geometry over probability densities on $\R^p$,
and $\grad^{\FR}_g$ denotes the $g$-gradient of $F_n(q,g)$
under the Fisher-Rao geometry over densities on $[-M,M]$;
we refer to Sections \ref{sec:Gibbsrepr} and \ref{sec:phireparam}
for details. The tuning
parameter $\alpha>0$ is a relative learning rate of $g$ versus $q$.

As recognized in the pioneering work of \cite{jordan1998variational},
the time-evolution of $q_t$ in
(\ref{eq:joint_flow_intro}) represents the law of $\bvarphi_t \in \R^p$
evolving according to a Langevin diffusion. 
Consequently, the time-evolution of the prior density $g_t$ in
(\ref{eq:joint_flow_intro}) may be approximated by simulating a trajectory
$\{\bvarphi_t\}_{t \geq 0}$ of this diffusion. In our setting, this
diffusion has a time-evolving drift parameter, corresponding to a time-evolving
stationary distribution that depends on $g_t$. This leads to an adaptive
Langevin dynamics algorithm --- which we will refer to as EBflow ---
where the prior density $\{g_t\}_{t \geq 0}$ adapts to the sampled
Langevin chain $\{\bvarphi_t\}_{t \geq 0}$.

\item We carry out a theoretical analysis of the gradient flow
system \eqref{eq:joint_flow_intro} and its individual components,
establishing the following basic properties:
\begin{itemize}
\item Under mild assumptions for the initial conditions $(q_0,g_0)$, we show existence, uniqueness, and smoothness of the solution
$\{(q_t,g_t)\}_{t \geq 0}$ to (\ref{eq:joint_flow_intro}).
\item For fixed $q_t=q$, we will see that the ``univariate'' evolution of $g_t$ 
in (\ref{eq:joint_flow_intro}) is a Fisher-Rao gradient flow for optimizing
the marginal log-likelihood in the Gaussian sequence model (\ref{def:gau_seq}). 
We prove global convergence of this univariate flow
$\{g_t\}_{t \geq 0}$ to the sequence model NPMLE as $t \to \infty$, at the rate
$O(1/t)$.

\item For fixed $g_t=g$, the ``univariate'' evolution of $q_t$ in
(\ref{eq:joint_flow_intro}) describes a Langevin diffusion with fixed
stationary distribution given by the posterior law $P_g(\bvarphi \mid \y)$.
Adapting an insight of \cite{bauerschmidt2019simple} for log-Sobolev
inequalities (LSIs) in
high-temperature spin systems, we show for sufficiently large
noise variance $\sigma^2>0$ that $P_g(\bvarphi \mid \y)$ satisfies a
dimension-free LSI. This implies global convergence of
$\{q_t\}_{t \geq 0}$ to $P_g(\bvarphi \mid \y)$ as $t \to \infty$, at a
dimension-free exponential rate in KL-divergence.

\item For the joint evolution of $\{(q_t,g_t)\}_{t \geq 0}$ in
(\ref{eq:joint_flow_intro}), under this condition of a dimension-free LSI
for the posterior laws $\{P_{g_t}(\bvarphi \mid \y)\}$ established
above, we show local convergence of $\{g_t\}_{t
\geq 0}$ to a near-NPMLE within a time horizon that is linear in $(n,p)$,
provided that $\bar F_n$ is convex on the sub-level set
$\{g:\bar F_n(g) \leq F_n(q_0,g_0)\}$ defined by
the initial condition $(q_0,g_0)$.
\end{itemize}
We note that these results for the univariate evolutions of $g_t$ and $q_t$ may
be of independent interest in other applications involving NPMLE in the
Gaussian sequence model and posterior sampling in the linear model,
respectively.

While the convergence guarantees in this work are stated for the
continuous-time gradient flow (\ref{eq:joint_flow_intro}), we believe
they yield useful insights for an analysis also of the corresponding EBflow
algorithm. We discuss this and other directions for future work
in Section \ref{sec:conclusion}.
\end{enumerate}

The remainder of the paper is organized as follows:
The gradient flow and resulting EBflow method are
motivated and discussed further in Section \ref{sec:method}. Theoretical
results are presented and discussed in Section \ref{sec:theory}.
Numerical simulations and discussions of parameter tuning are
provided in Section \ref{sec:simulation}. Directions for future work are
discussed in Section \ref{sec:conclusion}.

\subsection{Related work}\label{sec:related_work}
\phantom{}
{\bf Empirical Bayes in sequence models.} The NPMLE was proposed
in \cite{robbins1950generalization, kiefer1956consistency} for nonparametric
estimation of an unknown prior distribution in sequence models. The monograph 
\cite{lindsay1995mixture} provides a systematic treatment of structural and
geometric properties, and we refer to \cite{polyanskiy2020self} for more recent
results. Early consistency results on the NPMLE include \cite{jewell1982mixtures, lindsay198geometry2, lambert1984asymptotic, pfanzagl1988consistency}; see also \cite{chen2017consistency} for a recent survey. 
We refer to \cite{genovese2000rates, ghosal2001entropies,
zhang2009generalized, jiang2009general, saha2020nonparametric,
polyanskiy2021sharp} for work on related problems of marginal
density estimation and EB-compound estimation. Various algorithms have been
proposed and studied for computing the NPMLE in sequence models,
including EM-type methods in \cite{dempster1977maximum, laird1978nonparametric,
cover1984algorithm, vardi1993image, jiang2009general}, general constrained
convex optimization techniques (e.g.\ the interior-point method) in \cite{liu2007partially,
koenker2014convex} and Frank-Wolf algorithm (in this literature known as the vertex direction methods) in \cite{lindsay1983ageometry1,lesperance1992algorithm,
bohning1995review, wang2007fast, groeneboom2008support}. We also mention the
recent work of \cite{zhang2022efficient} which proposed an EM-type algorithm in
general dimensions without a fixed support grid.\\

\noindent {\bf Empirical Bayes in the linear model.} 
EB estimation in linear models is related to that in the Gaussian
sequence model with correlated errors, which was 
discussed in \cite{efron2010correlated,schwartzman2010comment,sun2018solving}
in a context of multiple testing and false discovery rate control. For designs
$\X$ with $n \geq p$ and full column rank, the least squares
estimator $\hat\btheta_{\text{OLS}}=(\X^\top \X)^{-1} \X^{\top}
\y=\btheta+\beps'$ is a sufficient statistic and
follows such a correlated-errors model where
$\beps'\sim\mathcal{N}(0,\sigma^2(\X^\top\X)^{-1})$, establishing an equivalence
between these two models. These models
are no longer equivalent when $n<p$, as
$\y$ depends on $\btheta \in \R^p$ only via its projection onto the row span
of $\X$.

Some early and representative works studying EB in linear
models include \cite{nebebe1986bayes,george2000calibration,yuan2005efficient}
for parametric priors. A large body of literature in statistical
genetics has developed EB methods in both parametric and nonparametric contexts;
a few representative approaches and references include EM optimization of a
composite likelihood \cite{zhang2018estimation},
MCEM or hierarchical Bayes sampling algorithms
\citep{zhou2013polygenic,lloyd2019improved,ge2019polygenic,zhou2021fast}, mean-field variational approximations
\citep{spence2022flexible,morgante2023flexible}, and moment-matching procedures
in the Fourier domain \citep{oconnor2021distribution}. Nonparametric EB via a
discretized computation of the NPMLE was also suggested
in \cite{wang2021nonparametric}.

Theoretical results on EB in the
linear model seem more scarce. The work \cite{jiang2021nonparametric} studied a
mixture of (low-dimensional) regression models that is more akin to the sequence
model (\ref{def:gau_seq}). Recent results of Mukherjee et
al.\ \cite{mukherjee2023mean} 
analyzed the NPMLE and variational approximations in high-dimensional
regression models, showing in particular that the NPMLE is
consistent as $n,p \to \infty$ when $n \geq p$ and $\X/\sigma$ has all
singular values of constant order, and establishing conditions under which
consistency holds also for the optimizer of a mean-field variational
approximation to the marginal log-likelihood.\\

\noindent \textbf{Gradient flows.} 
Gradient flows of probability measures play a central role in our work, and we
refer to \cite{santambrogio2017euclidean} for an introduction and
\cite{ambrosio2008gradient, villani2009optimal} for fuller and
systematic treatments. The connection between gradient-based optimization of
entropy in the Wasserstein-2 geometry and Langevin diffusions was
recognized in \cite{jordan1998variational,otto2001geometry}, and has 
enabled a fruitful line of work developing new analyses of Langevin dynamics as
well as new Bayesian sampling algorithms
\citep{pereyra2016proximal,wibisono2018sampling,durmus2019analysis,ahn2021efficient,ma2021there}.
Wasserstein gradient flows have also been applied to understand
training dynamics of shallow neural networks \citep{nitanda2017stochastic,
mei2018mean, sirignano2020mean, rotskoff2018neural, chizat2018global} and to
develop new methods for variational inference \citep{lambert2022variational,
yao2022mean}.

Ideas of the Fisher-Rao geometry and gradient flow date back to classical work of \cite{hellinger1909, hotelling1930spaces,
rao1945information, kakutani1948equivalence}. More recently, together with the
Wasserstein geometry, it has played an important role in the study of unbalanced
optimal transport and the associated Wasserstein-Fisher-Rao flow
\citep{CHIZAT20183090, chizat2018interpolating, kondratyev2016new,
liero2018optimal, liero2016optimal,yan2023learning}. For sampling, recent work
of \cite{lu2019accelerating, lu2023birth, domingo2023explicit} showed that dynamics
resulting from the combined Wasserstein-Fisher-Rao gradient flow can accelerate
the convergence of Langevin dynamics, especially when the target is not
log-concave.\\

\noindent {\bf Monte Carlo procedures for maximum marginal likelihood estimation.}
Our diffusion-based implementation of the gradient flow 
leads to a continuous-time Monte Carlo EM
method \citep{wei1990monte,booth1999maximizing,levine2001implementations}, where
an intractable E-step is simulated using an Unadjusted Langevin Algorithm
\citep{roberts1996exponential,dalalyan2017theoretical,durmus2017nonasymptotic}.
Classical theoretical treatments of MCEM include
\cite{chan1995monte,fort2003convergence,neath2013convergence}, which assume
an asymptotically growing number of MCMC iterations 
between M-step parameter updates. Related SEM and SAEM methods were
studied in \cite{diebolt1993asymptotic,delyon1999convergence}, using instead a
small number of samples of the latent variable to inform an incremental E-step
update prior to each M-step;
the analyses of these works assume sampling from the exact
conditional distribution of the latent variable
rather than via a MCMC scheme. In contrast, our analyses of the gradient
flow equations may be understood as studying MCEM convergence in the
opposite regime where the prior is continuously updated using an
out-of-equilibrium Langevin sample $\bvarphi_t$,
whose law $q_t(\bvarphi)$ is always
different from the target posterior distribution $P_{g_t}(\bvarphi \mid \y)$ of
the latent variable $\bvarphi$ under the current prior $g_t$.

Closest to our work is the variational view of EM put forth in
\cite{neal1998view}, and the recent proposals of
\cite{kuntz2023particle,akyildiz2023interacting} which build upon this
viewpoint to develop bivariate gradient-flow methods for maximum marginal
likelihood estimation in general latent variable models. The main differences
with our setting are: (1) \cite{kuntz2023particle,akyildiz2023interacting}
address estimation of $g \equiv g_\balpha \in \cP$ within a parametric family
parametrized by a prior parameter $\balpha$, and suggest estimation or
sampling of $\balpha$ via the Gibbs free energy associated to
the latent variable $\btheta$. In the linear model, we develop additional
ideas around reparametrization by a smoothed latent variable $\bvarphi$,
to address the lack of regularity arising
in estimation over a nonparametric family. (2)
\cite{kuntz2023particle,akyildiz2023interacting} suggest ``particle''
implementations using a large number $N$ of parallel Langevin sampling chains
$\{\{\btheta_t^n\}_{t \geq 0}\}_{n=1}^N$. In our context where
$\btheta \in \R^p$ is itself high-dimensional with
product-law prior, we propose a procedure using a single Langevin chain.
(3) The theoretical analyses of these works assume that
the full-data log-likelihood ${-}\log P(\btheta,\y \mid \balpha)$
is jointly convex in the latent variable $\btheta$
and prior parameter $\balpha$, ensuring that the
Gibbs free energy is also Wasserstein-2 convex in $q$ for each fixed
$\balpha$ (see \cite{kuntz2023particle}).
This convexity in $q$ does not hold in our setting of a linear model 
with general regression designs $\X$, requiring the development of
different ideas including a new dimension-free LSI to analyze
the gradient flow dynamics.\\

\noindent {\bf Recent literature.} Following the initial arXiv posting of our
work, several additional papers have carried out related analyses:
\cite{caprio2025error} provides an analysis of global convergence of the
method in \cite{kuntz2023particle} under a LSI-type assumption for
$F_n(q,g)$ jointly in $(q,g)$ (which precludes settings where
$\bar F_n(g)=\min_q F_n(q,g)$ may have multiple local optima). 
\cite{marion2024implicit}
discusses a similar method in a general bilevel optimization framework with
parametric prior $g \equiv g_\balpha$, and proves local convergence in the
sense of vanishing averaged gradient $\frac{1}{T} \int_0^T \|\nabla_\balpha \bar
F_n(g_\balpha)\|^2 \d t \to 0$ under a uniform LSI assumption similar to what we
establish for the posterior law.
\cite{lim2024momentum} extends \cite{kuntz2023particle} to a
momentum-accelerated method based on the underdamped Langevin algorithm.
\cite{montanari2025provably} employs an idea similar to our reparametrization by $\bvarphi$ and
proof of posterior LSI, to develop a hierarchical algorithm for
posterior sampling of $\btheta$ in the linear model. \cite{lee2026parametric} studies
parametric EB in the linear model using mean-field variational inference,
establishing conditions under which the variational EB estimate
is consistent with $1/\sqrt{p}$-rate of convergence. Finally, 
\cite{fan2025dynamicalI,fan2025dynamicalII} (with overlapping authors as our
current manuscript) study a parametric analogue of EBflow for the linear
model, and carry out precise asymptotic
analyses of the dynamics and possible fixed points of a single-chain Langevin
implementation of this method when $n \asymp p$ and $\X$ is a random design
with i.i.d.\ coordinates.

\subsection{Notational conventions}

$\cP(M)$, $\cP(\R)$, $\cP([-M,M]^p)$, $\cP(\R^p)$ are the spaces of probability
distributions on $[-M,M]$, $\R$, $[-M,M]^p$, $\R^p$ with Borel sigma-field,
respectively, and $\cP_*(M),\cP_*(\R)$ etc.\ are the spaces of probability
densities on these domains with respect to Lebesgue measure.

For two distributions $P$ and $Q$, the total variation distance is
$\dTV(P,Q) = \frac{1}{2}\int|\d P-\d Q|$. The Wasserstein-$\ell$ distance for $\ell \geq 1$ is given by $W_\ell^\ell(P,Q) = \inf_{\Gamma \in \text{couplings of (P,Q)}}
(\E_{(X,Y) \sim \Gamma} \|X-Y\|_\ell^\ell)^{1/\ell}$.
When $P$ is absolutely continuous to $Q$, the
Kullback-Leibler divergence is $\DKL(P\|Q) = \int \log \frac{\d P}{\d Q}
\d P$.

We use $\N_\sigma$ as a shorthand for the Gaussian law $\N(0,\sigma^2)$, and
$\N_\sigma(x)$ for its density function at $x \in \R$. We use $P*Q$ for
the convolution of two measures, so
$[\N_\sigma*g](x)$ is the density of the convolution of $\N_\sigma$ with $g$.

Throughout this work, we treat the design $\X$ as deterministic, and
write posterior laws as $P_g(\btheta \mid \y)$ with conditioning on $\X$
being implicit.
We write $P_g(\y),P_g(\y,\btheta),P_g(\y,\bvarphi),P_g(\btheta \mid
\y),P_g(\bvarphi \mid \y)$ etc.\ for the
marginal, joint, and conditional densities of $\y,\btheta,\bvarphi$
in the linear model \eqref{def:linear_model_intro} or \eqref{eq:reparamintro}
with priors $\theta_j \overset{iid}{\sim} g$
and $\varphi_j \overset{iid}{\sim} \N_\tau * g$.

$\nabla f$, $\Delta f$, and $\nabla \cdot f$ denote the gradient, Laplacian, and
divergence of $f:\R^d \to \R$. A function $f:\Omega \to \R$
belongs to $C^r(\Omega)$ (resp.\ $C^\infty(\Omega)$)
if it is $r$-times (resp.\ infinitely-times) continuously differentiable on
$\Omega$. We use $\pnorm{\mathbf{A}}{\op}$ and 
$\|\mathbf{A}\|_{\mathrm{F}}$ for the $\ell_2$-operator norm and Frobenius norm
of a matrix $\mathbf{A}$, and
$\lambda_{\min}(\mathbf{A})$ (resp.\ $\lambda_{\max}(\mathbf{A})$) for its
smallest (resp.\ largest) eigenvalue when $\mathbf{A}$ is symmetric. 

\section{Gradient flow and adaptive Langevin dynamics method}\label{sec:method}

We adopt a similar setting as in \cite{mukherjee2023mean} and focus on
nonparametric estimation of the regression coefficient prior
\[g_\ast \in \cP(M)\]
in the linear model \eqref{def:linear_model_intro},
where $\cP(M)$ denotes the family of distributions supported on 
a bounded interval $[-M,M] \subset \R$.
This nonparametric class of priors is well-studied in the empirical Bayes literature for the sequence model (see, e.g.,~\cite{singh1979empirical,zhang2003compound,greenshtein2009asymptotic,jiang2009general,saha2020nonparametric,polyanskiy2021sharp})

We will assume, for simplicity, that this support constraint $M$ and noise variance
$\sigma^2$ in \eqref{def:linear_model_intro} are both known. In practice, one
may estimate $\sigma^2$ and the second moment of $g_*$ using a
method-of-moments procedure for variance components 
(see e.g.\ \cite{rao1972estimation,Dicker2014Variance,pazokitoroudi2020efficient}), and deduce
from the latter a reasonable bound for $M$.

To motivate our proposed method for estimating $g_*$, we introduce its
ideas in five stages: In Section \ref{sec:FisherRaoseq}, we review the Fisher-Rao
gradient flow for computing the NPMLE in the Gaussian sequence model.
In Section \ref{sec:Gibbsrepr}, we discuss optimization of the marginal
log-likelihood $\bar F_n(g)$ in the linear model via a bivariate gradient
flow on the Gibbs free energy.
In Section \ref{sec:phireparam}, we discuss the variable
reparametrization by $\bvarphi$. In Section
\ref{sec:discretize}, we discuss the implementation of the gradient flow
equations via an adaptive Langevin dynamics algorithm EBflow. Finally,
Section \ref{sec:posteriorinference} discusses how to implement
downstream tasks of posterior inference.

\subsection{Fisher-Rao gradient flow for the Gaussian sequence model}
\label{sec:FisherRaoseq}

Consider first the sequence model (\ref{def:gau_seq}) with Gaussian errors of
variance $\sigma^2$, corresponding to the
identity design matrix $\X=\Id$ with $p=n$. The NPMLE for priors supported on $[-M,M]$ is
\begin{align}\label{def:npmle_id}
\hat{g}=\argmin_{g \in \cP(M)}\bar{F}_n(g), \qquad
\bar{F}_n(g) = -\frac{1}{n}\sum_{i=1}^n \log [\N_\sigma \ast g](y_i)
\end{align}
where $\N_\sigma*g$ denotes the convolution of
$\N_\sigma \equiv \N(0,\sigma^2)$ with $g$.

As mentioned in \prettyref{sec:related_work}, many algorithms are available for optimizing $\bar F_n(g)$. Among these,
several versions of gradient flows were studied in \cite{yan2023learning}.
We review here the (arguably) simplest such flow discussed
in \cite{yan2023learning},
operating in the Fisher-Rao geometry for probability densities $g \in \cP_*(M)$.
This is the geometry induced by the Fisher-Rao or Hellinger distance on
$\cP_*(M)$,
\[d(g,h)=\left(4\int_{-M}^M \Big(\sqrt{g(\theta)}-\sqrt{h(\theta)}\Big)^2
\d\theta\right)^{1/2}.\]
The associated Fisher-Rao gradient of a functional $\bar F_n(g)$
takes a form
\[\grad_g^\text{FR} \bar F_n(g)=g \cdot \delta \bar F_n[g],\]
resulting in the gradient flow equation
\begin{equation}\label{eq:FisherRaoflowgeneral}
\frac{\d}{\d t} g_t(\theta)
={-}g_t(\theta) \cdot \delta \bar F_n[g_t](\theta) \quad
\text{ for each } \theta \in [-M,M].
\end{equation}
A time discretization would give
\begin{equation}\label{eq:discreteFR}
g_{t+\Delta t}(\theta)
=\Big(1-\Delta t \cdot \delta \bar F_n[g_t](\theta)\Big)g_t(\theta)
\end{equation}
which may be interpreted as updating the weight that $g_t(\cdot)$
assigns to each value $\theta \in [-M,M]$ by the multiplicative factor
$1-\Delta t \cdot \delta \bar F_n[g_t](\theta)$.

In the above, $\delta \bar F_n[g]:[-M,M] \rightarrow \R$ is the functional
derivative or first variation of $\bar F_n$ in $g$, defined for each strictly
positive density $g$ on $[-M,M]$ as the unique
(up to an additive constant) function for which
\begin{align}\label{def:first_variation}
\frac{\d}{\d \eps}\Big|_{\eps=0} \bar F_n(g + \eps\chi) = \int_{-M}^M \delta
\bar F_n[g](\theta)\chi(\theta)\d \theta
\end{align}
for every bounded perturbation $\chi:[-M,M] \rightarrow \R$ of $g$
satisfying $\int_{-M}^M \chi(\theta)\d \theta = 0$.
When $\bar F_n$ is the (negative) marginal log-likelihood in
(\ref{def:npmle_id}), its first variation may be computed to be
\begin{equation}\label{eq:deltaFnid}
\delta \bar F_n[g](\theta)=-\frac{1}{n}\sum_{i=1}^n
\frac{\N_\sigma(y_i-\theta)}{\N_\sigma*g(y_i)}+1,
\end{equation}
where this choice of constant $+1$ ensures the conservation-of-mass condition
$\int_{-M}^M g(\theta) \cdot \delta \bar F_n[g](\theta) \d\theta=0$
for the gradient flow \eqref{eq:FisherRaoflowgeneral}.
The form of this Fisher-Rao gradient flow is then
\begin{equation}\label{eq:fisherraoseq}
\frac{\d}{\d t}\,g_t(\theta)={-}g_t(\theta) \cdot \delta \bar F_n[g_t](\theta)
=g_t(\theta)\left[\frac{1}{n}\sum_{i=1}^n
\frac{\N_\sigma(y_i-\theta)}{\N_\sigma*g_t(y_i)}-1\right].
\end{equation}

As noted in \cite{yan2023learning}, choosing a step size of
$\Delta t=1$ in \eqref{eq:discreteFR} results in 
the widely studied Expectation-Maximization (EM) algorithm
\begin{align}\label{def:em_intro}
g_{t+1}(\theta)=\frac{1}{n}\sum_{i=1}^n
\frac{g_t(\theta)\N_\sigma(y_i-\theta)}{\N_\sigma*g_t(y_i)}
\end{align}
for computing the NPMLE in the Gaussian sequence model.
Thus the Fisher-Rao flow may be understood as an infinitesimal version of EM.

We will use this Fisher-Rao flow as one component of our bivariate gradient
flow \eqref{eq:joint_flow_intro} primarily for its simplicity:
\begin{enumerate}[1.]
\item It preserves naturally the constraint of a bounded support $[-M,M]$ for
$g$, in contrast to a Wasserstein gradient flow which (without additional
modifications) would operate on densities that are positive over
the whole real line.
\item It preserves naturally the positivity condition $g(\theta) \geq 0$, in
contrast to a more naive gradient flow $\frac{\d}{\d t} g_t(\theta)={-}\delta
\bar F_n[g_t](\theta)$ in the linear geometry which (without additional
modifications) would not preserve positivity.
\end{enumerate}
We note that this choice of Fisher-Rao geometry for $g$ in
\eqref{eq:joint_flow_intro} will not be essential for the implementability
of the final EBflow method,
and the method may be modified to use other choices for this geometry.
\cite{yan2023learning} provides evidence that estimation in the sequence model
may be more accurate using a hybrid Wasserstein-Fisher-Rao geometry, especially
when the number of support points used to numerically approximate $g(\theta)$ is
small, and it may be interesting to explore similar proposals in our setting.

\subsection{Gradient flow on the Gibbs free energy in the linear model}
\label{sec:Gibbsrepr}

Consider now the linear model with i.i.d.\ prior $g_\ast$ and
noise variance $\sigma^2>0$, which we restate here for clarity:
\begin{equation}\label{eq:linear_model}
\y=\X\btheta+\beps, \qquad \theta_j \overset{iid}{\sim} g_\ast
\text{ for } j=1,\ldots,p,
\qquad \eps_i \overset{iid}{\sim} \N(0,\sigma^2) \text{ for } i=1,\ldots,n,
\end{equation}
with $\btheta$ and $\beps$ independent.
For general regression designs $\X$, computing the first variation
$\delta \bar F_n(g)$ involves an intractable posterior integral
(see \eqref{eq:first_variation_equiv} of the Supplementary Appendices
for its form), rendering direct gradient-based
optimization of $\bar F_n(g)={-}\frac{1}{p}\log P_g(\y)$ challenging.

We propose to address this challenge by performing a gradient flow on a Gibbs
variational representation of $\bar F_n(g)$.
With respect to
the parametrization \eqref{eq:linear_model} by $\btheta
\in \R^p$, for any prior $g \in \cP_*(M)$ admitting a Lebesgue density,
this representation takes the form
\begin{equation}\label{eq:Gibbstheta}
\bar F_n(g)=\min_{q \in \cP_*([-M,M]^p)} \widetilde F_n(q,g),
\quad \widetilde F_n(q,g)={-}\frac{1}{p}
\int\Big[\log P_g(\y,\btheta)-\log q(\btheta)\Big]q(\btheta)\d\btheta
\end{equation}
where $\widetilde F_n(q,g)$ is the (rescaled) \emph{Gibbs free energy}.
This representation follows by noting that
\[\bar F_n(g)-\widetilde F_n(q,g)
=\frac{1}{p}\,\DKL(q(\btheta) \| P_g(\btheta \mid \y)) \geq 0.\]
The minimization in \eqref{eq:Gibbstheta} is over all densities $q(\btheta)$ on
$[-M,M]^p$, with minimum attained at the posterior density $\tilde
q(\btheta)=P_g(\btheta \mid \y)$. Optimizing $\bar F_n(g)$
is thus equivalent to optimizing $\tilde F_n(q,g)$ jointly over $(q,g)$.

This variational representation of $\bar F_n(g)$
is also the basis of variational inference
(VI), which replaces $\bar F_n(g)$ by an upper bound $\min_{q \in
\mathcal{Q}} \tilde F_n(q,g)$ restricting
to an approximating
sub-class of densities $\mathcal{Q}$ over $[-M,M]^p$, e.g.\ product densities 
in mean-field~VI. Strong conditions on the regression design
$\X$ would be needed to ensure that $P_g(\btheta \mid \y)$
is close to this approximating class $\mathcal{Q}$, so that
the prior $g$ minimizing this upper bound is close to the NPMLE. (We refer to
\cite{mukherjee2022variational,mukherjee2023mean,lee2026parametric} for a
detailed study of such conditions.)

Motivated by applications where such
conditions for $\X$ do not hold, we may consider instead
optimization of $\tilde F_n(q,g)$ over all densities $q(\btheta)$,
using a bivariate gradient flow
\begin{equation}\label{eq:gradientflowtheta}
\begin{cases}
\frac{\d}{\d t} q_t(\btheta) &= {-}p \cdot \grad^{W_2}_q \tilde
F_{n}(q_t,g_t)[\btheta]\\
\frac{\d}{\d t} g_t(\theta) &= {-}\alpha \cdot
\grad^{\FR}_g \tilde F_{n}(q_t,g_t)[\theta].
\end{cases}
\end{equation}
Here $\grad_g^\FR$ is the $g$-gradient of $\tilde F_n(q,g)$ in the
Fisher-Rao geometry discussed in Section~\ref{sec:FisherRaoseq}, 
$\grad_q^{W_2}$ is the $q$-gradient of $\tilde F_n(q,g)$ in the
Wasserstein-2 geometry, and $\alpha>0$ is a relative learning rate
which we will take independent of $n,p$.

For expositional purposes, let us derive explicit forms of the
gradient flow equations \eqref{eq:gradientflowtheta},
momentarily assuming that $g_t$ and $q_t$ are positive and smooth densities on
all of $\R$ and $\R^p$ with unbounded supports.
Applying $\log P_g(\y,\btheta)={-}\frac{n}{2}\log 2\pi \sigma^2
{-}\frac{1}{2\sigma^2}\|\y-\X\btheta\|^2+\sum_{j=1}^p \log g(\theta_j)$
in the definition \eqref{eq:Gibbstheta} of $\tilde F_n$,
the first variations of $\tilde F_n$ may be computed to be
\begin{equation}\label{eq:firstvariationsthetaparam}
\begin{aligned}
\delta_q \tilde F_{n}[q,g](\btheta)&=\frac{1}{p}
\bigg[\frac{1}{2\sigma^2}\|\y-\X\btheta\|^2
-\sum_{j=1}^p \log g(\theta_j)+\log q(\btheta)\bigg],\\
\delta_g \tilde F_{n}[q,g](\theta) &= {-}\frac{\bar q(\theta)}{g(\theta)}+1,
\end{aligned}
\end{equation}
where
\[\bar q(x)=\frac{1}{p}\sum_{j=1}^p \int_{\R^{p-1}}
q(\theta_1,\ldots,\theta_{j-1},x,\theta_{j+1},\ldots,\theta_p)
\prod_{i:i \neq j} \d \theta_i\]
is the empirical average of the marginal densities of 
coordinates $\theta_1,\ldots,\theta_p$ under $q(\cdot)$, and the constant $+1$
in \eqref{eq:firstvariationsthetaparam}
is again chosen to ensure the conservation-of-mass property
$\int g(\theta) \cdot \delta_g \tilde F_{n}(q,g)[\theta] \d \theta=0$.
The preceding Wasserstein-2 and Fisher-Rao gradients are of the form
\[\grad^{W_2}_q \tilde F_{n}(q,g)
={-}\nabla_\btheta \cdot (q\,\nabla_\btheta \delta_q \tilde F_n[q,g]),
\quad \grad_g^\FR \tilde F_n(q,g)=g \cdot \delta_g \tilde F_n[q,g],\]
which, upon substitution into \eqref{eq:gradientflowtheta},
leads to the gradient flow equations
\begin{align}
\frac{\d}{\d t}q_t(\btheta)&=\nabla_\btheta \cdot
\left[q_t(\btheta) \left(\frac{1}{\sigma^2}\X^\top(\X\btheta-\y)
-\left(\frac{g_t'(\theta_j)}{g_t(\theta_j)}\right)_{j=1}^p
\right)\right]+\Delta_\btheta q_t(\btheta),\label{eq:qflowtheta}\\
\frac{\d}{\d t}g_t(\theta)&=\alpha\big(\bar q_t(\theta)-g_t(\theta)\big)
\label{eq:gflowtheta}
\end{align}
where $\nabla_\btheta \cdot f(\btheta)$ and $\Delta_\btheta f(\btheta)$
denote the divergence and Laplacian with respect to $\btheta \in \R^p$.

Equations \eqref{eq:qflowtheta} and \eqref{eq:gflowtheta} admit the
following interpretation:
Similar to the identification in \cite{jordan1998variational}, the
evolution (\ref{eq:qflowtheta}) may be identified as
a Fokker-Planck equation for the density
$q_t(\btheta_t)$ of a sample $\btheta_t \in \R^p$
evolving according to a Langevin diffusion
\begin{equation}\label{def:Langevintheta}
\d \btheta_t = \left({-}\frac{1}{\sigma^2}\X^\top (\X\btheta_t-\y)
+\left(\frac{g_t'(\theta_{t,j})}{g_t(\theta_{t,j})}\right)_{j=1}^p
\right)\d t + \sqrt{2}\,\d \B_t
\end{equation}
where $\{\B_t\}_{t \geq 0}$ is a standard Brownian motion on $\R^p$.
This diffusion has a time-dependent drift depending on the current prior $g_t$.
The density $\bar q_t$ in \eqref{eq:gflowtheta} is the average marginal
density of individual coordinates of $\btheta_t$ under this diffusion. For
each fixed density $q$ on $\R^p$, it may be checked that the explicit
minimizer of $g \mapsto \widetilde F_n(q,g)$ is precisely $g=\bar q$,
and \eqref{eq:gflowtheta} may be understood as a process that evolves
$g_t$ towards $\bar q_t$ to enforce this condition. The scalings of the
gradient equations in \eqref{eq:gradientflowtheta} by $p$ and by $\alpha=O(1)$
lead to the prior $g_t(\cdot)$ and individual
coordinates of $\btheta_t$ simultaneously evolving at a comparable rate.

This discussion motivates the specific choice of the Wasserstein-2 geometry for
the $q$-gradient in \eqref{eq:gradientflowtheta}, as it enables simulation
of the $g$-flow in \eqref{eq:gflowtheta} using a sample
$\{\btheta_t\}_{t \geq 0}$ of this Langevin diffusion to approximate $\bar q_t$.
A similar method for maximum likelihood estimation in general latent
variable models with parametric priors was proposed in
\cite{kuntz2023particle}, building on the variational view of
EM in \cite{neal1998view}.

\subsection{Variable reparametrization}
\label{sec:phireparam}

In the Gaussian sequence model \eqref{def:gau_seq} corresponding to $\X=\Id$,
it is known that the NPMLE $\hat g=\argmin_g \bar F_n(g)$ is a discrete
distribution almost surely \cite{lindsay1995mixture},
and thus the corresponding
posterior law $P_{\hat g}(\btheta \mid \y)$ is also discretely supported.

Anticipating that a similar property holds for general regression designs
$\X$ in the linear model,\footnote{To our knowledge, for general regression designs,
it is not yet shown in the literature whether the NPMLE $\argmin_g \bar F_n(g)$
is unique, or whether it is given by a discrete measure. For simplicity, we
write ``a NPMLE $\widehat g=\argmin_g \bar F_n(g)$'' in our informal
discussions; our theoretical results will pertain to near-optimizers $\widehat
g$ of $\bar F_n(g)$ without assuming uniqueness or structural properties of
$\widehat g$.} we
expect the density $\{g_t\}_{t \geq 0}$ in the preceding gradient flow
equations \eqref{eq:gradientflowtheta} to converge to a discrete
distribution as $t \to \infty$, leading to a degeneracy of the drift
coefficient of the Langevin diffusion \eqref{def:Langevintheta} 
and inadequacy of Langevin dynamics as a posterior sampler for large times.

To address this issue, we consider a variation of the
idea in Section \ref{sec:Gibbsrepr} using a reparametrization of the linear
model via a smoothed latent variable $\bvarphi$: Fix any parameter $\tau^2>0$ satisfying
$\tau^2<\sigma^2/\|\X\|_\op^2$, and decompose the Gaussian noise
in \eqref{eq:linear_model}
as $\beps=\tilde \beps+\X\z$ where $\z \sim \N(0,\tau^2\Id)$ and 
$\tilde \beps \sim \N(0,\,\bSigma)$ with 
\begin{equation}\label{eq:psdconstraint}
\bSigma=\sigma^2\Id-\tau^2\X\X^\top\succ 0.
\end{equation}
Setting $\bvarphi=\btheta+\z$,
the marginal law of $\y$ is equivalently expressed as
\begin{align}\label{def:smooth_model}
\y=\X\bvarphi+\tilde \beps, \qquad \varphi_j \overset{iid}{\sim}
\N_\tau*g_\ast, \qquad \tilde \beps \sim \N(0,\,\bSigma)
\end{align}
where $\bvarphi$ and $\tilde \beps$ are independent, and
$\N_\tau*g_*$ denotes the convolution of $\N(0,\tau^2)$ and $g_*$.
Corresponding to this new latent variable $\bvarphi$, we may write an
alternative Gibbs variational representation of $\bar F_n(g)$ as
\[\bar{F}_n(g)=\min_{q \in \cP_*(\R^p)} F_n(q,g),
\quad F_n(q,g)={-}\frac{1}{p}\int \Big[\log P_g(\y,\bvarphi)-\log
q(\bvarphi)\Big]q(\bvarphi)\d\bvarphi\]
with minimization over all probability densities $q(\bvarphi)$ on $\R^p$.
The minimum is now attained at the posterior density $q(\bvarphi)=P_g(\bvarphi
\mid \y)$ of $\bvarphi$ rather than $\btheta$.
Analogous to \eqref{eq:firstvariationsthetaparam}, the first variations of
$F_n(q,g)$ may be computed to be (see Appendix \ref{sec:Fn_grad})
\begin{equation}\label{eq:firstvariations}
\begin{aligned}
\delta_q F_{n}[q,g](\bvarphi)&=\frac{1}{p}
\bigg[\frac{1}{2}(\y-\X\bvarphi)^\top \bSigma^{-1} (\y-\X\bvarphi)
-\sum_{j=1}^p \log[\N_\tau*g](\varphi_j)
+\log q(\bvarphi)\bigg],\\
\delta_g F_{n}[q,g](\theta) &= -\left[\N_\tau*\frac{\bar q}{\N_\tau*g}\right](\theta)+1,
\end{aligned}
\end{equation}
where
\begin{align}\label{def:density_avg}
\bar{q}(x)=\frac{1}{p}\sum_{j=1}^p \int_{\R^{p-1}}
q(\varphi_1,\ldots,\varphi_{j-1},x,\varphi_{j+1},\ldots,\varphi_p)
\prod_{i:i \neq j} \d \varphi_i
\end{align}
is again the average marginal density of individual coordinates under
the joint density
$q(\bvarphi)$. The gradient flow equations (\ref{eq:joint_flow_intro}) for
minimizing $F_n(q,g)$ then take the explicit forms
\begin{align}
\frac{\d}{\d t}q_t(\bvarphi)&=\nabla_\bvarphi \cdot
\left[q_t(\bvarphi) \left(\X^\top \bSigma^{-1}(\X\bvarphi-\y)
-\left(\frac{[\N_\tau*g_t]'(\varphi_j)}{[\N_\tau*g_t](\varphi_j)}\right)_{j=1}^p
\right)\right]+\Delta_\bvarphi q_t(\bvarphi),\label{eq:qflow}\\
\frac{\d}{\d t}g_t(\theta)&=\alpha\,g_t(\theta)
\left(\left[\N_\tau*\frac{\bar q_t}{\N_\tau*g_t}\right](\theta)-1\right),
\label{eq:gflow}
\end{align}
where $\nabla_\bvarphi \cdot f$ and $\Delta_\bvarphi f$ are the divergence and
Laplacian with respect to $\bvarphi \in \R^p$.

The evolution (\ref{eq:qflow}) is now a
Fokker-Planck equation representing the density
$q_t(\bvarphi_t)$ of the Langevin diffusion
\begin{equation}\label{def:Langevin}
\d \bvarphi_t=\left({-}\X^\top \bSigma^{-1}(\X\bvarphi_t-\y)
+\left(\frac{[\N_\tau*g_t]'(\varphi_{t,j})}{[\N_\tau*g_t]
(\varphi_{t,j})}\right)_{j=1}^p \right)\d t + \sqrt{2}\,\d \B_t
\end{equation} 
in the smoothed regression parameter $\bvarphi$ rather than $\btheta$.
We emphasize that the Fisher-Rao gradient in (\ref{eq:gflow})
is computed with respect to the unsmoothed prior $g$ for $\btheta$,
not the smoothed prior $\N_\tau*g$ for $\bvarphi$,
as performing a gradient flow in the latter would not directly yield an
estimate for $g_\ast$ and, more importantly, would require a method of
constraining or projecting the density $\N_\tau*g$ to the space of Gaussian
convolution measures along the flow.

For any fixed prior $g_t$, \eqref{def:Langevin} is a Langevin diffusion with
stationary distribution given by the posterior law
$P_{g_t}(\bvarphi \mid \y)$ of $\bvarphi$. In contrast to the procedure
discussed in Section \ref{sec:Gibbsrepr} based on the parametrization by
$\btheta$, here the prior density $\N_\tau * g_t$ remains smooth even if
$\{g_t\}_{t \geq 0}$ converges to a discrete measure, and
$P_{g_t}(\bvarphi \mid \y)$ admits a positive and bounded density on all of $\R^p$.
This regularity of $\N_\tau * g_t$ and $P_{g_t}(\bvarphi \mid \y)$ 
is essential to all of our theoretical analyses in Section~\ref{sec:theory}.
We remark that this variable reparametrization idea also underlies our proof of
consistency of the NPMLE, as well as the measure decomposition in our proof of
the posterior LSI, and it is inspired by the arguments of Bauerschmidt and
Bodineau in \cite{bauerschmidt2019simple}.

\subsection{Adaptive Langevin dynamics implementation}\label{sec:discretize}

Simulation of the gradient flow equation \eqref{eq:gflow} for $\{g_t\}_{t
\geq 0}$ requires knowledge of $\bar q_t$, which typically remains intractable
to exactly compute. However, $\bar q_t$ may be estimated from a sample
$\{\bvarphi_t\}_{t \geq 0}$ of the Langevin diffusion \eqref{def:Langevin}.
An important remark is that
while estimating the high-dimensional density $q_t(\bvarphi)$ on $\R^p$ may be
challenging, estimating its average marginal distribution
$\bar q_t(\varphi)$ on $\R$ is conceivably easier. Indeed, we propose the simple
empirical estimate
\begin{equation}\label{eq:barqestimate}
\bar q_t \approx \frac{1}{p}\sum_{j=1}^p \delta_{\varphi_{t,j}}
\end{equation}
based on the coordinates of a
single sample $\{\bvarphi_t\}_{t \geq 0}$ of this Langevin diffusion.
Substituting this into \eqref{eq:gflow} yields an empirical estimate
for this gradient equation,
\begin{equation}\label{eq:gflowest}
\frac{\d}{\d t} g_t(\theta) \approx \alpha\,g_t(\theta) \left[\frac{1}{p}
\sum_{j=1}^p \frac{\N_\tau(\varphi_{t,j}-\theta)}{\N_\tau * g_t(\varphi_{t,j})}
-1\right].
\end{equation}
Comparing with \eqref{eq:fisherraoseq}, we see that 
\eqref{eq:gflowest} admits a conceptually appealing interpretation as a
Fisher-Rao gradient flow for computing the NPMLE based on a
Gaussian sequence model for the coordinates of the Langevin iterate
$\bvarphi_t$,
\begin{equation}\label{eq:varphiseqmodel}
\bvarphi_t=\btheta+\z,
\qquad \theta_j \overset{iid}{\sim} g_*,
\qquad z_j \overset{iid}{\sim} \N(0,\tau^2).
\end{equation}
Our experiments support a conjecture that for large $n,p$ and design
matrices $\X$ with limited long-range dependence across variables, this
estimate \eqref{eq:gflowest} using a single Langevin trajectory
$\{\bvarphi_t\}_{t \geq 0}$ concentrates around the true gradient
\eqref{eq:gflow}. (Following the initial posting of our work, one setting
of this conjecture has been confirmed in \cite{fan2025dynamicalI} for a
parametric analogue of our method and i.i.d.\ design matrices $\X$.)

To numerically implement the resulting algorithm,
we approximate the prior $g(\theta)$ via a discretization in space as
\begin{equation}\label{eq:gdiscretization}
g\approx \sum_{k=1}^K w_k \delta_{b_k}
\end{equation}
using an equally spaced discrete support $b_1,\ldots,b_K \in [-M,M]$.
Then
\begin{align*}
\frac{1}{p}\sum_{j=1}^p
\frac{\N_\tau(\varphi_{t,j}-\theta)}{\N_\tau * g_t(\varphi_{t,j})}
& \approx \frac{1}{p}\sum_{j=1}^p \frac{\N_\tau(\varphi_{t,j}-\theta)}
{\sum_{k=1}^K w_k\,\N_\tau(\varphi_{t,j}-b_k)},\\
\frac{[\N_\tau*g]'(\varphi)}{[\N_\tau*g](\varphi)}
&\approx \frac{\sum_{k=1}^K {-}\frac{1}{\tau^2}(\varphi-b_k)
w_k\,\N_\tau(\varphi-b_k)}{\sum_{k=1}^K w_k\,\N_\tau(\varphi-b_k)}.
\end{align*}
Applying these approximations together with a discretization of
(\ref{eq:gflow}) and (\ref{def:Langevin}) in time using a simple
forward Euler scheme leads to our proposed EBflow algorithm
\begin{align}
\bvarphi_{t+1}&=\bvarphi_t-\eta^\varphi_t
\Bigg[\X^\top\bSigma^{-1}(\X\bvarphi_t-\y)+
\frac{1}{\tau^2}\bigg(\frac{\sum_{k=1}^K (\varphi_{t,j}-b_k)
w_{t,k}\,\N_\tau(\varphi_{t,j}-b_k)}{\sum_{k=1}^K
w_{t,k}\,\N_\tau(\varphi_{t,j}-b_k)}\bigg)_{j=1}^p\Bigg]\notag\\
&\hspace{2in}+\sqrt{2\eta^\varphi_t}\,\b_t,\label{eq:phiupdate}\\
w_{t+1,k}&=w_{t,k}+\eta^w_t w_{t,k}
\Bigg[\frac{1}{p}\sum_{j=1}^p \frac{\N_\tau(\varphi_{t+1,j}-b_k)}
{\sum_{i=1}^K w_{t,i}\,\N_\tau(\varphi_{t+1,j}-b_i)}-1\Bigg]
\text{ for } k=1,\ldots,K.\label{eq:wupdate}
\end{align}
Here, $\b_t \sim \N(0,\Id)$ is a standard Gaussian vector in $\R^p$ generated
independently in each iteration, $w_t:=(w_{t,1},\ldots,w_{t,K})$
are the probability weights representing $g_t \approx
\sum_{k=1}^K w_{t,k}\delta_{b_k}$, and $\eta_t^\varphi,\eta_t^w>0$
are step sizes for the updates of $\bvarphi_t$ and $w_t$ respectively,
with
\[\eta_t^w/\eta_t^\varphi=\alpha\]
representing the relative learning rate
in (\ref{eq:joint_flow_intro}) and (\ref{eq:gflow}).
(In our experiments, we will allow
$\eta_t^w$ and $\eta_t^\varphi$ to vary with the iteration $t$, fixing
their ratio $\eta_t^w/\eta_t^\varphi=\alpha$.)

We make two additional remarks for interpretation of this EBflow method:
\begin{enumerate}[1.]
\item The prior update (\ref{eq:wupdate}) has the
equivalent form 
\[w_{t+1,k}=(1-\eta^w_t)w_{t,k}
+\eta_t^w \cdot \frac{1}{p}\sum_{j=1}^p \frac{w_{t,k}\,
\N_\tau(\varphi_{t+1,j}-b_k)}{\sum_{i=1}^K
w_{t,i}\,\N_\tau(\varphi_{t+1,j}-b_i)}.\]
Thus, as long as $\eta_t^w \leq 1$, we have $w_{t,k} \geq 0$ always, an
appealing property resulting from the positivity of the Fisher-Rao gradient
flow.
\item The two steps (\ref{eq:phiupdate}--\ref{eq:wupdate}) are
``incremental'' versions of the E-step and M-step in an EM
algorithm with latent variable $\bvarphi$. Indeed, an exact EM algorithm in this
parametrization would minimize $F_n(q,g)$ via the following alternating updates
of $q$ and $g$:
\begin{equation}\label{eq:EM}
\begin{aligned}
&\text{(E-step)}
&q_t&=P_{g_t}(\bvarphi \mid \y) \qquad \text{(the posterior law of $\bvarphi$ under the prior
$\N_\tau*g_t$)},\\
&\text{(M-step)}
&g_{t+1}&=\argmin_{g \in \cP(M)} \int {-}\log[\N_\tau*g](\varphi)\bar{q}_t(\varphi)\d\varphi.
\end{aligned}
\end{equation}
Both steps have computational difficulties in practice: The E-step requires
computation of the posterior law $P_{g_t}(\bvarphi \mid \y)$. The update (\ref{eq:phiupdate}) replaces this
E-step by a single step of an Unadjusted Langevin Algorithm (ULA)
\cite{roberts1996exponential}, and then approximates
$\bar q_t$ needed in the M-step by the empirical distribution of coordinates
of the sampled $\bvarphi_t=(\varphi_{t,j})_{j=1}^p$
c.f.~\prettyref{eq:barqestimate}. The M-step solves
a NPMLE problem based on a Gaussian sequence model for $\bvarphi_t$
c.f.~\eqref{eq:varphiseqmodel}, which is
expensive to carry out in every iteration.
The update (\ref{eq:wupdate}) replaces this by a single step
of Fisher-Rao gradient descent for this NPMLE optimization.
\end{enumerate}

\subsection{Posterior inference}\label{sec:posteriorinference}

For downstream tasks of posterior inference, we perform Monte Carlo
approximation using samples $\bvarphi_t \in \R^p$
of (\ref{eq:phiupdate}--\ref{eq:wupdate}) for a set of final iterations $t \in
\{T+1,\ldots,T+T'\}$ after fixing the final estimated weights
$w_T=(w_{T,1},\ldots,w_{T,K})$.
Typically one is interested in computing expectations under the posterior law of
the original parameter $\btheta$, rather than under the reparametrized variable
$\bvarphi$. Such expectations may be estimated using samples
$\{\bvarphi_t\}$ via the relation
\begin{align*}
\E_g[f(\btheta) \mid \y]&=\int_{\R^p} \int_{[-M,M]^p}
f(\btheta) P_g(\btheta \mid \bvarphi) P_g(\bvarphi \mid \y)\d\btheta \d\bvarphi\\
&\approx \frac{1}{T'}\sum_{t=T+1}^{T+T'}
\int_{[-M,M]^p} f(\btheta)P_g(\btheta \mid \bvarphi_t)\d\btheta.
\end{align*}
Here, coordinates of $\btheta$ are independent conditional on $\bvarphi$,
with $P_g(\btheta \mid \bvarphi)=\prod_{j=1}^p P(\theta_j \mid \varphi_j)
\propto \prod_{j=1}^p g(\theta_j)\N_\tau(\varphi_j-\theta_j)$.
This enables Monte Carlo approximation of $\E_g[f(\btheta) \mid \y]$ 
via explicit computation of the integrals
$\int_{[-M,M]^p} f(\btheta)P_g(\btheta \mid \bvarphi_t)$ for
functions $f$ that depend only on a small number of coordinates of $\btheta$.

As a concrete example, the empirical Bayes
posterior mean $\E_{\widehat g}[\btheta \mid \y] \in \R^p$
corresponding to the final estimated prior
$\widehat g=g_T \approx \sum_{k=1}^K w_{T,k}\delta_{b_k}$ may be estimated as
\begin{align}
\E_{\widehat g}[\btheta \mid \y] &\approx \frac{1}{T'}\sum_{t=T+1}^{T+T'}
\int_{[-M,M]^p} \btheta P_{\widehat g}(\btheta \mid \bvarphi_t)\d\btheta\nonumber\\
&=\frac{1}{T'}\sum_{t=T+1}^{T+T'} \left(\frac{\int_{-M}^M \theta \cdot \widehat g(\theta)\,\N_\tau(\theta-\varphi_{t,j}) \d\theta}
{\int_{-M}^M \widehat g(\theta)
\N_\tau(\theta-\varphi_{t,j})\d\theta}\right)_{j=1}^p\nonumber\\
&\approx \frac{1}{T'}\sum_{t=T+1}^{T+T'}
\left(\frac{\sum_{k=1}^K b_kw_{T,k}\,\N_\tau(b_k-\varphi_{t,j})}{\sum_{k=1}^K
w_{T,k}\,\N_\tau(b_k-\varphi_{t,j})}\right)_{j=1}^p.\label{eq:posteriormean}
\end{align}
This estimate of $\E_{\widehat g}[\btheta \mid \y]$ may then be used for
downstream tasks, such as prediction of the response $y_\text{new}$ for a new
sample $\x_\text{new} \in \R^p$.

\section{Theoretical guarantees}\label{sec:theory}

In this section, we show statistical consistency for estimators of $g_* \in \cP(M)$
that approximately
maximize the marginal log-likelihood, and we
analyze the solution of the gradient flow (\ref{eq:qflow}--\ref{eq:gflow}).

Unless otherwise stated, we continue to treat the design $\X$ as
deterministic, i.e.\ in settings of random design, our theorems hold
conditional on the realization of $\X$.

\subsection{Consistency of near-NPMLEs} \label{sec:consistency}

We first discuss consistency of near-NPMLEs for estimating $g_* \in \cP(M)$.
We impose the following assumption on the design matrix $\X$. 

\begin{assumption}\label{assump:design}
Let $\x_1,\ldots,\x_p$ denote the columns of $\X \in \R^{n \times p}$.
There are constants $C_\ast,c_\ast,\beta>0$ such that
$\sigma^2 \in (c_*\|\X\|_\op^2,\frac{C_*}{p}\|\X\|_{\mathrm{F}}^2)$,
and the following two additional conditions hold for $\X$
as $n,p \to \infty$:
\begin{enumerate}[(a)]
\item Let $\1 \in \R^p$ be the all-1's vector. Then $\frac{1}{p}\|\X\1\|_2^2
\geq c_\ast \|\X\|_\op^2$.
\item For a subset $\cS \subseteq \{1,\ldots,p\}$ of size $|\cS| \geq
c_*p$ and for each $j \in \cS$, there exists a vector $\z_j \in \R^n$
(deterministic conditional on $\X$)
such that $\|\X\|_\op\|\z_j\|_2 \in [c_\ast,C_\ast]$,
$|\z_j^\top \X\1| \leq C_\ast$,
\begin{align}\label{ineq:test_vector}
\max_{j \in \cS} |\z_j^\top \x_j-1| \to 0,
\qquad \max_{j \in \cS} \sum_{k: k\neq j} |\z_j^\top \x_k|^{2+\beta} \to 0,
\end{align}
and for any $\{a_j:j \in \cS\}$ satisfying $|a_j| \leq 1$,
\begin{align}\label{eq:zjcorr}
\|\X\|_\op^2 \cdot \left\|\sum_{j \in \cS} a_j\z_j\right\|_2^2 \leq C_\ast p.
\end{align}
\end{enumerate}
\end{assumption}

Under Assumption \ref{assump:design}, the following results establish
consistency of near optimizers of the marginal log-likelihood
$\bar{F}_n$ in \eqref{def:npmle}, together with
consistency in normalized KL-divergence of the corresponding posteriors of
smoothed regression coefficients.
Its proof is presented in Appendix \ref{sec:proof_consistency}.

\begin{theorem}\label{thm:consistency}
Let $g_* \in \cP(M)$, and suppose Assumption \ref{assump:design} holds as $n,p
\to \infty$. Let $\hat{g}_n \in \cP(M)$ be any sequence of priors
such that, with probability approaching 1,
\begin{equation}\label{eq:nearNPMLE}
 \bar{F}_n(g_\ast) - \bar{F}_n(\hat{g}_n) = \frac{1}{p}\Big(\log
P_{\hat{g}_n}(\y) - \log P_{g_*}(\y)\Big) \geq -\delta_n
\end{equation}
for some $\delta_n \downarrow 0$ as $n,p \to \infty$. Then $\hat{g}_n \to
g_*$ weakly in probability.
\end{theorem}

\begin{corollary}\label{cor:consistencyposterior}
In the setting of Theorem \ref{thm:consistency}, consider the
parametrization \eqref{def:smooth_model} of the linear model by the variable
$\bvarphi$ for any
fixed constant $\tau^2 \in (0,\sigma^2/\|\X\|_\op^2)$. Then also
\[\frac{1}{p}\,\DKL\big(P_{\widehat g_n}(\bvarphi \mid \y)\|
P_{g_*}(\bvarphi \mid \y)\big) \to 0\]
in probability.
\end{corollary}

Let us give a discussion of the various aspects of
Assumption \ref{assump:design}:
\begin{enumerate}
\item For the Gaussian sequence model where $\X=\Id$ and $n=p$,
classical consistency results for the NPMLE (e.g.\ \cite{kiefer1956consistency})
pertain to fixed noise variance $\sigma^2>0$ and $n \to \infty$. The condition
$\sigma^2 \in (c_*\|\X\|_\op^2,\frac{C_*}{p}\|\X\|_{\mathrm{F}}^2)$ may be
understood as an
extension of this fixed noise variance condition to more general
regression designs, requiring that both the
maximum and average eigenvalues of $\X^\top \X \in \R^{p \times p}$
are of the same order as $\sigma^2$. (This requires the rank of $\X$ to be of
the same order as the dimension $p$.)

We note that an upper bound for $\sigma^2$ is needed, as it represents a
requirement of sufficiently strong signal-to-noise ratio for the model. 
In fact, the assumed upper bound $\sigma^2 = O(\frac{1}{p}\|\X\|_{\mathrm{F}}^2)$ is nearly optimal in the sense that 
$\sigma^2 = O(\frac{1}{p^{1+o(1)}}\|\X\|_{\mathrm{F}}^2)$ is necessary for
consistent nonparametric estimation of the prior even in the sequence model.
We require the lower bound for $\sigma^2$ in our arguments to ensure sufficient
regularity and continuity in $g$ of the marginal density $P_g(\y)$. It may be
an interesting question to understand whether such a lower bound is needed for
the NPMLE to succeed, which we do not address in this work. On the other hand,
we believe this lower bound assumption for $\sigma^2$ is mild, as one may add
artificial noise to $\y$ to satisfy this lower bound if it does not hold.
\item Conditions (a) and (b) of Assumption \ref{assump:design} on the design
$\X$ have two distinct
roles in our analysis. Assumption \ref{assump:design}(a) ensures correct
estimation of the mean of $g_\ast$. Note that if $\X\1=0$, then the mean of
$g_\ast$ is not identifiable, as changing this mean would not affect the
distribution of $(\X,\y)$. Assumption \ref{assump:design}(a) is a quantitative
strengthening of the requirement $\X\1 \neq 0$ for identifiability of this mean.

If we assume that the prior $g_*$ is centered such that $g_*$ has mean 0,
and we restrict also to mean-zero estimators $\widehat g$ satisfying
\eqref{eq:nearNPMLE}, then 
Assumption \ref{assump:design}(a) is not needed, and
our proof of Theorem \ref{thm:consistency} establishes the consistency of
$\widehat g$ under Assumption \ref{assump:design}(b) alone.

\item Assumption \ref{assump:design}(b) is the main non-trivial condition of the
theorem, and ensures correct estimation of properties of $g_*$ beyonds its first
two moments. We note that for many dense regression designs $\X$ (in contrast
to the sequence model $\X=\Id$), the marginal distribution of any $O(1)$-sized
subset of coordinates of $\y$ would be approximately normal for large $p$ (by
the Central Limit Theorem), and this
normal distribution depends only on the first two moments of $g_*$.
Thus consistent estimation of $g_*$ beyond its first two moments must
combine information from many coordinates of $\y$ to ``invert'' this averaging
of coordinates of $\btheta$ by the regression design.

Informally, Assumption \ref{assump:design}(b) 
posits the existence of test vectors $\{\z_j\}_{j \in \cS}$ such that
each vector $\z_j$ can approximately orthogonalize $\x_j$ against the other
variables $\{\x_k\}_{k \neq j}$ of the design.
Under this condition, $\z_j^\top \y$ would 
isolate the $j$-th coordinate of $\btheta$ and resemble an
observation $\theta_j+\eps_j'$ from a Gaussian sequence model
(\ref{def:gau_seq}) with inflated noise variance $\Var[\eps_j']>\sigma^2$. These
errors $\eps_j'$ are correlated for different coordinates $j \in \cS$, and the
condition \eqref{eq:zjcorr} ensures that the correlations are sufficiently
weak to guarantee concentration of
$\sum_{j \in \cS} f_j(\z_j^\top \y)$ for Lipschitz test functions $f_j:\R
\to \R$, which is needed in the proof.

Assumption \ref{assump:design}(b) is reminiscent of the
construction of approximate orthogonalization vectors in
the debiased Lasso \citep{zhang2014confidence}. An important difference is
that, here, $\{\z_j\}_{j\in \cS}$ is used only as a proof device for analyzing
the NPMLE, and plays no role in the estimation procedure. In settings of random
design, they may depend on the population covariance of the design
even when this covariance cannot be consistently estimated from the data.
\end{enumerate}

A related consistency result was shown in the recent work of
\cite{mukherjee2023mean} studying variational inference, in settings where
$p \leq n$ and
$c\sigma^2 \leq \lambda_{\min}(\X^\top \X) \leq \lambda_{\max}(\X^\top
\X) \leq C\sigma^2$. Assumption~\ref{assump:design} is readily verified under
such conditions, where the test vectors in Assumption \ref{assump:design}(b)
can be constructed as $\z_j=\X(\X^\top \X)^{-1}\mathbf{e}_j$ with $\mathbf{e}_j$
denoting the $j$-th standard basis vector in $\R^p$, so that $\{\z_j^\top \y\}_{j=1}^p$
are the ordinary least squares estimators of $\{\theta_j\}_{j=1}^p$.
Assumption~\ref{assump:design} can be seen as a
relaxation of these conditions in \cite{mukherjee2023mean}, allowing also for
undersampled designs where $n<p$ and the row span of $\X$ is not all of
$\R^p$. The following result shows, for example,
that Assumption \ref{assump:design} is satisfied
for random subgaussian designs with well-conditioned
population covariance. We present its proof in Appendix \ref{sec:design}.

\begin{proposition}\label{prop:condition_gaussian}
Suppose $\sigma^2>0$ is a fixed constant, and
$n/p \geq \gamma$ for a constant $\gamma>0$ (where possibly $\gamma<1$).
Suppose $\X\in\R^{n\times p}$ has i.i.d.\ rows
$\{\frac{1}{\sqrt{n}}x^{(i)}\}_{i=1}^n$ where, for some constants $C,c>0$,
each $x^{(i)}$ has mean 0, covariance $\bSigma_X \in \R^{p \times p}$ satisfying
\begin{align*}
c \leq \lambda_{\min}(\bSigma_X) \leq \lambda_{\max}(\bSigma_X) \leq C,
\end{align*}
and $x^{(i)}$ is $C$-subgaussian.\footnote{$x \in \R^p$ is
$C$-subgaussian if $\|a^\top x\|_{\psi_2} \leq C$ for all unit vectors $a \in
\R^p$, where $\|X\|_{\psi_2}=\inf\{K>0:\E e^{X^2/K^2} \leq 2\}$ is the 
subgaussian norm.} Then with probability approaching 1
as $n,p \to \infty$, Assumption~\ref{assump:design} is
satisfied with $\cS=\{1,\ldots,p\}$, any $\beta>0$, some $C_\ast,c_\ast>0$
depending only on $\sigma^2,\gamma,C,c$, and 
\begin{align*}
\z_j=\Pi\X\bSigma_X^{-1}\e_j
\end{align*}
where $\Pi \in \R^{n \times n}$ is the projection matrix orthogonal to $\X\1$ and $\e_j \in \R^p$ is the $j$-th standard basis vector.
\end{proposition}

\begin{remark}
We note that Assumption~\ref{assump:design} is a sufficient --- but not
necessary --- condition
for consistency of $\widehat g$. For example, consider a design
$\widetilde \X=[\X,\X] \in \R^{n \times 2p}$ consisting of two copies of a
matrix $\X \in \R^{n \times p}$, where $\X$
satisfies Assumption~\ref{assump:design}.
Observation of a linear model $\widetilde \y=\widetilde \X \widetilde
\btheta+\beps$ with $\widetilde \theta_j \overset{iid}{\sim} g_*$ is equivalent
to observation of a linear model
$\y=\X\btheta+\beps$ with $\theta_j=\widetilde \theta_j+\widetilde \theta_{j+p}
\overset{iid}{\sim} g_* * g_*$ for each $j=1,\ldots,p$. Our results then imply
that any near-NPMLE $\widehat g$ of $g_*$ based on $(\widetilde \X,\widetilde \y)$ must satisfy $\widehat g * \widehat g \to g_* * g_*$ weakly in probability,
as $n,p \to \infty$. This implies the weak consistency $\widehat g \to g_*$.
However, $\widetilde \X$ does not satisfy Assumption~\ref{assump:design}, as any
vector $\z_j$ for which $\z_j^\top \widetilde \x_j \to 1$ must also satisfy
$\z_j^\top \widetilde \x_{j+p} \to 1$, and vice versa. It may be an interesting
question for future work to identify a simple condition on $\X$ that is both
necessary and sufficient for consistent estimation of $g_*$.
\end{remark}

\subsection{Convergence of gradient flows}\label{sec:algo}

We now analyze convergence of the gradient flows
(\ref{eq:qflow}--\ref{eq:gflow}) introduced in Section \ref{sec:phireparam}.

In Section \ref{subsec:g_flow}, we
study convergence of the Fisher-Rao flow $\{g_t\}$ in \prettyref{eq:gflow} where
$q_t$ is replaced by a fixed distribution $q \in \cP(\R^p)$,
corresponding to a sequence model NPMLE problem.
In Section \ref{subsec:q_flow}, we study the Langevin diffusion corresponding to
the Wasserstein-2 flow  $\{q_t\}$ in \prettyref{eq:qflow} where $g_t$ is
replaced by a fixed distribution $g \in \cP(M)$, showing a
uniform log-Sobolev inequality for the posterior law $P_g(\bvarphi \mid \y)$
under large noise variances $\sigma^2>0$.
Lastly in Section \ref{subsec:joint_flow}, we combine the previous two
ingredients to study the joint gradient flow of both $\{q_t,g_t\}$.

\subsubsection{Univariate $g$-flow}\label{subsec:g_flow}

First consider the univariate Fisher-Rao gradient flow
\begin{align}\label{def:g_flow}
\frac{\d}{\d t}\,g_t(\theta) = g_t(\theta)\left[\N_\tau*\frac{\bar
q}{\N_\tau*g_t}(\theta)-1\right]
:=g_t(\theta)\left[\int \frac{\N_\tau(\theta-\varphi)}{\N_\tau*g_t(\varphi)}
\d\bar q(\varphi)-1\right]
\end{align}
for $\theta \in [-M,M]$ and
any fixed probability distribution $\bar q \in \cP(\R)$. This corresponds to the
evolution of $g_t$ alone in (\ref{eq:gflow}). Via a calculation
analogous to (\ref{eq:fisherraoseq}), it may be seen that this is
a Fisher-Rao gradient flow for minimizing the objective
\[\bar F(g \mid \bar q):=\int {-}\log[\N_\tau*g](\varphi)\d\bar{q}(\varphi).\]
Here $\bar q$ may be discrete, and we do not require
$\bar q$ to have density with respect to Lebesgue measure. In particular,
if $\bar q=\frac{1}{n}\sum_{i=1}^n \delta_{y_i}$ is
the empirical distribution of
$n$ samples, then this objective is exactly the marginal log-likelihood
$\bar F_n(g)$ in the Gaussian sequence model (\ref{def:gau_seq})
with noise variance $\tau^2$, and (\ref{def:g_flow}) is exactly the
gradient flow (\ref{eq:fisherraoseq}) for computing the sequence model NPMLE.

Corollary \ref{cor:exist_gflow} in
Appendix \ref{subsec:exist_gflow} shows that, for any
$\bar q \in \cP(\R)$ with finite moment-generating function,
if (\ref{def:g_flow}) is initialized at $g_0 \in \cP_*(M)$ having strictly
positive density on all of $[-M,M]$, then it has a
unique solution $\{g_t\}_{t \geq 0}$ in the space of positive densities.
The following (deterministic) result establishes convergence of this solution
as $t \to \infty$.

\begin{theorem}\label{thm:g_flow}
Let $\bar q \in \cP(\R)$ with
$\int e^{\lambda\varphi}\bar q(\varphi)\d\varphi<\infty$ for all $\lambda \in \R$.
Suppose that
(\ref{def:g_flow}) is initialized from a density $g_0 \in \cP_*(M)$
that is strictly positive on $[-M,M]$.
Then the solution $\{g_t\}_{t \geq 0}$ to (\ref{def:g_flow}) satisfies, for any $h \in \cP_*(M)$,
\begin{align}\label{ineq:subopt_gap_g}
\bar{F}(g_{t} \mid \bar q) - \bar{F}(h \mid \bar q)
\leq \frac{\DKL(h \| g_0)}{t}.
\end{align}
\end{theorem}

If $\bar q=\N_\tau*g_\ast$ is a Gaussian mixture for some true prior with 
density $g_\ast \in \cP_*(M)$,
then the minimizer of $\bar F(g \mid \bar q)$ is exactly $g_\ast$,
and applying Theorem \ref{thm:g_flow} with $h=g_\ast$ bounds
directly the suboptimality gap of $g_t$. If instead $\bar q$ is the empirical
distribution of a sample $(y_1,\ldots,y_n)$
and $\hat g$ is the minimizer of $\bar F(g \mid \bar q)=\bar F_n(g)$
given by the NPMLE in the sequence model (\ref{def:gau_seq}),
then we may apply
Theorem \ref{thm:g_flow} with $h$ being a smoothed approximation of $\hat g$,
to show
\[\bar{F}_n(g_{t}) - \bar{F}_n(\hat g)
\leq \frac{\DKL(h \| g_0)}{t} + \bar{F}_n(h)-\bar{F}_n(\hat g).\]
In particular, taking $t \to \infty$ first, followed by $h \to \hat g$
weakly, and noting that $\bar F_n(g)$ is weakly continuous in $g$,
we have $\lim_{t \to \infty} \bar{F}_n(g_{t})=\bar{F}_n(\hat g)$,
so $g_t$ is a near-NPMLE for large $t$.

The proof of Theorem \ref{thm:g_flow} can be found in
Appendix \ref{subsec:proof_g_flow}.
Note that the objective function $\bar F$ is not geodesically convex in the Fisher-Rao geometry (see \prettyref{app:nonconvex}) so standard results such as \cite[Theorem 4.0.4]{ambrosio2008gradient} are inapplicable. Instead, motivated by viewing (\ref{def:g_flow}) as a mirror flow \citep{nemirovsky1983problem}, the analysis exploits the convexity of $\bar F(g) \equiv \bar F(g \mid \bar q)$ in the standard linear geometry for any fixed $\bar q$, namely
\begin{equation}\label{eq:barFconvex}
\bar F(\lambda g_0+(1-\lambda)g_1)
\leq \lambda \bar F(g_0)+(1-\lambda) \bar F(g_1),
\end{equation}
to bound $\bar F(g_t)-\bar F(h)$ by the integrated first variation $\int_{-M}^M
{-}\delta \bar F[g_t](\theta)h(\theta)\d\theta$. We note that our
arguments for showing Theorem \ref{thm:g_flow} require $\{g_t\}_{t \geq 0}$
to have a bounded support $[-M,M]$, to ensure the boundedness of the first variation
\[\N_\tau * \frac{\bar q}{\N_\tau * g_t}(\theta)-1
=\int_\R \frac{e^{-\frac{1}{2\tau^2}(\varphi-\theta)^2}}
{\int_{-M}^M e^{-\frac{1}{2\tau^2}(\varphi-\theta')^2}g_t(\theta')\d \theta'}\,
\bar q(\varphi)\d \varphi-1\]
uniformly over $\theta$ belonging to this support
and over densities $g_t \in \cP_*(M)$ along the gradient flow.
This property will also be important for our results to follow.

Some related results in the literature include
\citep{lu2019accelerating, lu2023birth, domingo2023explicit} which established
exponential convergence of the Fisher-Rao gradient flow for optimizing a
different objective $g \mapsto \DKL(g \| \pi)$ with $\pi$ denoting a target
measure, and \cite[Theorem 4]{yan2023learning} which showed a conditional convergence
result for a discrete-time version of (\ref{def:g_flow}).
Complementary to this result, Theorem \ref{thm:g_flow} provides a
quantitative and unconditional convergence guarantee for the continuous-time
flow.

\subsubsection{Univariate $q$-flow}\label{subsec:q_flow}
Next, define the negative log-posterior-density for prior $g(\cdot)$ (up to a normalizing constant),
\begin{align}\label{Ug_smooth}
U_g(\bvarphi)=\frac{1}{2}(\y-\X\bvarphi)^\top \bSigma^{-1}(\y-\X\bvarphi) -
\sum_{j=1}^p \log \N_\tau \ast g(\varphi_j).
\end{align}
We consider the Langevin diffusion
\begin{equation}\label{def:Langevin_fixedg}
\d \bvarphi_t = {-}\nabla U_g(\bvarphi_t)
+ \sqrt{2}\,\d \B_t
\end{equation} 
with a fixed prior $g \in \cP(M)$,
whose law $q_t$ satisfies the Fokker-Planck evolution
\begin{equation}\label{def:q_flow}
\frac{\d}{\d t}q_t(\bvarphi)=\nabla \cdot
\left[q_t(\bvarphi)\nabla U_g(\bvarphi)\right]
+\Delta q_t(\bvarphi).
\end{equation}
This corresponds to the evolution of $q_t$ alone in (\ref{eq:qflow}), fixing
$g_t=g$.

We prove a uniform log-Sobolev inequality (LSI) for the stationary
distribution $P_g(\bvarphi \mid \y)$
of the diffusion (\ref{def:Langevin_fixedg}),
for any prior $g \in \cP(M)$ and any
noise level $\sigma^2>M^2\|\X\|_\op^2$. As this type of result may be of
independent interest in other contexts, we also provide a LSI
for the posterior law $P_g(\btheta \mid \y)$
in the original $\btheta$ parametrization when $g \in \cP_*(M)$
has a density on $[-M,M]$ bounded from above and below. 
We will write as shorthand
\begin{align}
\nu[g](\bvarphi)&:=P_g(\bvarphi \mid \y)=\frac{\exp(-U_g(\bvarphi))}{\int
\exp(-U_g(\bvarphi'))\d\bvarphi'} & \text{ for } g \in
\cP(M),\label{posterior_smooth}\\
\mu[g](\btheta)&:=P_g(\btheta \mid \y) & \text{ for } g \in
\cP_*(M)\label{posterior_nonsmooth}
\end{align}
for the posterior densities of $\bvarphi$ and $\btheta$.

Following \cite[Eq.~(2.18)]{ledoux2006concentration}, we say that a distribution
$\rho$ on an open domain $\Omega \subseteq \R^p$
satisfies a log-Sobolev inequality (LSI) on $\Omega$ with log-Sobolev
constant $C>0$ if, for any
bounded, Lipschitz, and continuously differentiable function
$f:\Omega \rightarrow \R$,
\begin{align}\label{def:LSI}
\ent_\rho(f^2):=\E_\rho f^2\log f^2 - \E_\rho f^2\cdot \log \E_\rho f^2 \leq 2C \cdot \E_\rho \pnorm{\nabla f}{}^2.
\end{align}
The following (deterministic) result,
whose proof is given in Section \ref{subsec:proof_q_flow},
establishes a LSI for the posterior distributions in both variable
parametrizations.

\begin{theorem}\label{thm:q_lsi}
Let $g \in \cP(M)$ and fix a constant $\delta>0$.
Suppose that $\sigma^2/\pnorm{\X}{\op}^2>M^2+\delta$. Then:
\begin{enumerate}[(a)]
\item If $\tau^2 \in (0,\sigma^2/\|\X\|_\op^2-\delta)$,
then $\nu[g](\bvarphi)$ satisfies a LSI on $\R^p$
with log-Sobolev constant $C=C(M,\tau,\delta)>0$ (not depending on $n,p$).
\item Suppose that $g \in \cP_*(M)$ and $\kappa(g) \equiv
\frac{\sup_{\theta\in[-M,M]} g(\theta)}{\inf_{\theta\in [-M,M]} g(\theta)} <
\infty$. Then $\mu[g](\btheta)$ satisfies a LSI on $(-M,M)^p$ with
log-Sobolev constant $C \cdot \kappa(g)^2$
for some $C=C(M,\delta)>0$ (not depending on $n,p$).
\end{enumerate}
\end{theorem}
A direct and well-known consequence of Theorem \ref{thm:q_lsi}(a)
(see \cite[p.~288]{villani2009optimal})
is that the Langevin diffusion (\ref{def:Langevin_fixedg}) converges
exponentially fast in KL-divergence to $\nu[g]$.
\begin{corollary}\label{cor:q_flow}
Under the assumption of Theorem \ref{thm:q_lsi}(a), let $q_0 \in \cP(\R^p)$,
and let $\{q_t\}_{t \geq 0}$ be the marginal laws of
$\{\bvarphi_t\}_{t \geq 0}$ in (\ref{def:Langevin_fixedg}). Then for any $t \geq 0$,
\begin{align*}
\DKL(q_t \| \nu[g]) \leq \exp(-2t/C) \cdot \DKL(q_0 \| \nu[g]),
\end{align*}
where $C$ is the log-Sobolev constant of Theorem \ref{thm:q_lsi}(a). 
\end{corollary}

The posterior laws $\nu[g](\bvarphi)=P_g(\bvarphi \mid \y)$ and
$\mu[g](\btheta)=P_g(\btheta \mid \y)$ are, in general,
not log-concave and also not close to product measures, and a dimension-free
LSI cannot be obtained by a naive application of Holley-Stroock perturbation
in $\R^p$. Our proof of Theorem
\ref{thm:q_lsi} uses instead the idea of \cite{bauerschmidt2019simple} to
express these posterior laws as log-concave mixtures of product distributions.
In our setting of a Bayesian linear model,
this idea has an intuitive interpretation related to the
reparametrization of $\btheta$ by $\bvarphi$:
For the posterior of $\btheta$, we may consider the generative process
$\btheta \mapsto \bvarphi' \mapsto \y$ 
where $\bvarphi'=\btheta+\z$ is given as in (\ref{def:smooth_model}) for some
${\tau'}^2 > 0$ to be specified, and write
\begin{align*}
P_g(\btheta \mid \y)
=\int P_{g}(\btheta \mid \bvarphi') P_g(\bvarphi' \mid \y)\d\bvarphi'.
\end{align*}
Then a LSI for $P_g(\btheta \mid \y)$ follows from those for
$P_{g}(\btheta \mid \bvarphi')$ (uniform over $\bvarphi'$) and $P_g(\bvarphi' \mid \y)$ (uniform over $\y$), which are established as follows:
\begin{itemize}
\item Since $P_{g}(\btheta \mid \bvarphi')$ is a product measure, by
tensorization it suffices to establish a uniform LSI over its univariate
marginals. We do this using a Holley-Stroock perturbation argument on $\R$
and the characterization of LSI for one-dimensional
measures \citep{bobkov1999exponential}  (c.f.\ Lemma \ref{lem:exp_tilt_lsi}).
\item The smoothed prior $\mathcal{N}_{\tau'} \ast g$, and hence also
the posterior measure $P_g(\bvarphi' \mid \y)$, are strongly log-concave when
${\tau'}^2>M^2$, and a LSI follows from the Bakry-Emery criterion \citep{bakry1985diffusions}. 
\end{itemize}
The LSI for $P_g(\bvarphi \mid \y)$ follows from a similar argument, where we
consider instead the generative process $\btheta \mapsto \bvarphi \mapsto
\bvarphi' \mapsto \y$ to represent
\begin{align*}
P_{g}(\bvarphi \mid \y)=\int P_{g}(\bvarphi \mid \bvarphi') P_g(\bvarphi' \mid
\y)\d\bvarphi'.
\end{align*}

We remark that if $\tau^2>M^2$ in our starting reparametrization by
$\bvarphi$ for the joint gradient flow,
then we may simply take $\bvarphi'=\bvarphi$, and the LSI follows
directly from log-concavity of $P_g(\bvarphi \mid \y)$. However, choosing a large value of $\tau^2$
leads to a difficult sequence-model NPMLE problem to be solved by the
evolution of $\{g_t\}$ in the joint flow, which we observe in practice can
lead to unstable behavior of the EBflow algorithm and slow convergence.
Theorem \ref{thm:q_lsi} shows that as long as the noise variance
satisfies $\sigma^2>M^2\|\X\|_\op^2$, a LSI holds for $P_g(\bvarphi \mid \y)$
under any choice of reparametrization variance for $\bvarphi$ in the allowable
range $\tau^2 \in (0,\sigma^2/\|\X\|_\op^2)$ , even if the resulting
posterior law $P_g(\bvarphi \mid \y)$ is far from a log-concave measure.

Importantly, this requirement $\sigma^2>M^2\|\X\|_\op^2$ is compatible
with the condition for $\sigma^2$ in Assumption \ref{assump:design} and Theorem
\ref{thm:consistency} ensuring consistency of near-NPMLEs. Noting that
$\Var[y_i]=\|\x^{(i)}\|_2^2 \cdot \Var_{g_*}[\theta]+\sigma^2$, and considering
a typical setting where the rows of $\X$ and prior variance scale as
$\|\x^{(i)}\|_2 \asymp \|\X\|_\op$ and $\Var_{g_*}[\theta] \asymp M^2$,
the scaling $\sigma^2 \asymp M^2\|\X\|_\op^2$ may be interpreted as a regime
where $\beps$ explains a constant fraction of the variance of $\y$, which we
believe is also a relevant regime for applications. We note that the precise
requirement $\sigma^2>M^2\|\X\|_\op^2$ is likely loose by a constant factor
for any given prior and a random design, and should be interpreted as a
sufficient condition that guarantees dimension-free mixing of Langevin
dynamics under worst-case priors and designs.

The mixing of Langevin diffusions for posterior sampling in Bayesian regression
models has been studied before, for example
in \cite{dalalyan2012sparse} using a Meyn-Tweedie approach
and in \cite{dalalyan2017theoretical} for log-concave priors. We mention also
the recent results of \cite{nickl2022polynomial} that show rapid
mixing of Langevin diffusions in settings where the posterior concentrates
in a locally log-concave region of the parameter space. However, 
these results do not imply polynomial-time mixing
in many settings of interest in our application,
such as that of Proposition \ref{prop:condition_gaussian} where $p$ scales
linearly with $n$ and the posterior $P_g(\bvarphi \mid \y)$ is not 
log-concave and does not contract around the posterior mode.

\subsubsection{Joint flow}\label{subsec:joint_flow}
Equipped with the understanding of its two components, we now study the joint
flow (\ref{eq:qflow}--\ref{eq:gflow}), recorded here for the reader's convenience:
\begin{equation}\label{eq:jointflowabbr}
\begin{aligned}
\frac{\d}{\d t}q_t(\bvarphi)&=\nabla \cdot
\left[q_t(\bvarphi) \nabla U_{g_t}(\bvarphi)\right]+\Delta q_t(\bvarphi)\\
\frac{\d}{\d t}g_t(\theta)&=\alpha \cdot g_t(\theta)
\left(\left[\N_\tau*\frac{\bar q_t}{\N_\tau*g_t}\right](\theta)-1\right). 
\end{aligned}
\end{equation}

The following theorem, proved in Appendix \ref{subsec:joint_flow_existence},
verifies the existence and uniqueness of a solution
to this joint gradient flow in $\cP_*(\R^p) \times \cP_*(M)$.

\begin{theorem}\label{thm:flow_solution}
Fix any $T>0$. Suppose the initialization $(q_0,g_0) \in \cP_*(\R^p) \times
\cP_*(M)$ satisfies:
\begin{enumerate}[(i)]
\item The density $q_0$ is continuous and upper bounded,
and for any $\blambda \in \R^p$ we have $\E_{\bvarphi \sim q_0}
e^{\blambda^\top \bvarphi}<\infty$.
\item The density $g_0$ is strictly positive on all of $[-M,M]$.
\end{enumerate}
Then there exists a unique solution
$\{(q_t,g_t)\}_{t\in[0,T]}$ to (\ref{eq:qflow}--\ref{eq:gflow})
such that $(q_t,g_t) \in \cP_*(\R^p) \times \cP_*(M)$ for all $t \in [0,T]$.
This solution $\{(q_t,g_t)\}_{t \in [0,T]}$ satisfies:
\begin{enumerate}
\item $t \mapsto g_t(\theta)$ belongs to $C^\infty([0,T])$ for each fixed
$\theta \in [-M,M]$. Furthermore, $g_t(\theta)\geq g_0(\theta) e^{-t}$ for each
$\theta \in [-M,M]$ and $t \in [0,T]$
\item $\{q_t\}_{t \in [0,T]}$ are the densities of the marginal laws of
$\{\bvarphi_t\}_{t \in [0,T]}$ solving the SDE (\ref{def:Langevin})
with initialization $\bvarphi_0 \sim q_0$. For any $\blambda \in \R^p$,
\begin{align}\label{q_mgf_bound}
\sup_{t\in [0,T]} \E_{\bvarphi \sim q_t} e^{\blambda^\top \bvarphi}<\infty.
\end{align}
Furthermore, $(t,\bvarphi) \mapsto q_t(\bvarphi)$ belongs to
$C^\infty([0,T] \times \R^p)$.
\end{enumerate}
\end{theorem}

The following result is the main theoretical guarantee
of this work for the joint gradient flow (\ref{eq:qflow}--\ref{eq:gflow}). Under
the assumption of a uniform LSI for the posterior laws
$\{P_g(\bvarphi \mid \y):g \in \cP(M)\}$ as shown in Theorem \ref{thm:q_lsi},
it establishes that:
\begin{itemize}
\item If the marginal log-likelihood objective
$\bar F_n$ is convex in a sub-level set defined by the initial objective value,
then $\{g_t\}$ reaches a near-NPMLE in polynomial time.

\item In the absence of such a local convexity guarantee,
$\{g_t\}$ reaches a density that approximately
satisfies an averaged first-order optimality condition for
$\bar F_n(g)$, defined in terms of the first variation
\begin{equation}\label{eq:deltaFn}
\delta \bar F_n[g](\theta):={-}\N_\tau*\frac{\bar\nu[g]}{\N_\tau*g}(\theta)+1
\end{equation}
where $\bar\nu[g]$ is the average marginal density of
$\nu[g]=P_g(\bvarphi \mid \y)$ as defined by (\ref{def:density_avg}).
(See Appendix \ref{sec:Fn_grad} for a derivation of this form.)
\end{itemize}

\begin{theorem}\label{thm:algo}
Suppose that $\nu[g](\bvarphi) \equiv
P_g(\bvarphi \mid \y)$ satisfies a LSI on $\R^p$ with some log-Sobolev constant $C=C(M,\tau)>0$ (not depending on $n,p$),
simultaneously over all priors $g \in \cP(M)$.
Let $(q_0,g_0) \in \cP_*(\R^p) \times \cP_*(M)$ satisfy the conditions of Theorem \ref{thm:flow_solution}, and let $\{(q_t,g_t)\}_{t \geq 0}$ be the
solution to (\ref{eq:qflow}--\ref{eq:gflow}). 
Denote
\[F_{n,\ast}=\min_g \bar F_n(g)=\min_{(q,g)} F_n(q,g),
\qquad \mathcal{K}_n:=\{g: \bar{F}_n(g) \leq F_n(q_0,g_0)\}.\]
Then there exist $(M,\tau,\alpha)$-dependent constants
$T_0,c_0,\eps_0>0$
such that for any density $h \in \cP_*(M)$, any $\eps \in (0,\eps_0)$,
and some time
\begin{align}\label{def:joint_flow_time}
t \leq T(h,\eps):=T_0\Big(\eps^{-1}\DKL(h \|
g_0)+p\eps^{-3.01}(F_n(q_0,g_0)-F_{n,*})\Big),
\end{align}
it holds that
\begin{align}
\int_{-M}^M -\delta \bar{F}_n[g_t](\theta) h(\theta) \d\theta &\leq \eps
\label{ineq:subgrad}
\end{align}
and
\begin{align}
\DKL(q_t \|\nu[g_t]) \leq \eps^2.
\label{ineq:langevinconverge}
\end{align}

If $h \in \mathcal{K}_n$ and $\bar{F}_n$ is convex (in the usual sense of
(\ref{eq:barFconvex})) over $\mathcal{K}_n$, then also
\begin{align}\label{ineq:subopt_gap_joint}
\bar{F}_n(g_t) - \bar{F}_n(h) \leq \eps.
\end{align}
\end{theorem}

Together with Theorem \ref{thm:consistency} and Corollary
\ref{cor:consistencyposterior}, we obtain the following
consequence regarding estimation of $g_*$ by the above prior $g_t$ in the
gradient flow, and closeness of their associated posterior laws.

\begin{corollary}\label{cor:algo}
Suppose that Assumption \ref{assump:design} holds and $g_* \in \cP_*(M)$.
In the setting of Theorem \ref{thm:algo}, fix a sequence $\eps_n$ such that
$\eps_n \to 0$ as $n,p \to \infty$, suppose that
$g_* \in \cK_n$ and $\bar F_n$ is convex on $\cK_n$ for all large $n,p$, and
let $t_n$ be the time for which \eqref{ineq:langevinconverge}
and \eqref{ineq:subopt_gap_joint} hold with $h=g_*$ and $\eps=\eps_n$.

Then as $n,p \to \infty$,
$g_{t_n} \to g_*$ weakly in probability, and its corresponding
posterior law $\nu[g_{t_n}]=P_{g_{t_n}}(\bvarphi \mid \y)$ satisfies
\begin{equation}\label{eq:qnuW2converge}
\frac{1}{p}\,\DKL(\nu[g_{t_n}]\|\nu[g_*]) \to 0
\end{equation}
in probability.
Furthermore, for any bounded function $f:[-M,M] \to \R$, denoting by
$f(\btheta)$ its entrywise application to $\btheta \in \R^p$, the
posterior expectations of $f(\btheta)$ satisfy
\begin{equation}\label{eq:posteriormeanconverge}
\frac{1}{p}\big\|\E_{g_{t_n}}[f(\btheta) \mid \y]-\E_{g_*}[f(\btheta) \mid
\y]\big\|_2^2 \to 0.
\end{equation}
\end{corollary}

The proofs of Theorem \ref{thm:algo} and Corollary \ref{cor:algo} are given in 
Appendix \ref{subsec:proof_main_algo}. 
Let us make several remarks regarding the conditions and
interpretation of this theorem.

\begin{enumerate}
\item The required uniform LSI condition is provided, for example, by Theorem
\ref{thm:q_lsi} in high-noise settings where $\sigma^2>M^2\|\X\|_\op^2$.
(In theory, if $\sigma^2>M^2\|\X\|_\op^2$ does not hold,
one may add extra Gaussian noise to
$\y$ to satisfy the condition. In practice, we observe that when $\sigma^2$ is
very small, adding a little bit of noise may indeed
improve the mixing of EBflow and estimation accuracy for $g_*$.)
\item The guarantee (\ref{ineq:subgrad}) may be understood as an
averaged version of a first-order optimality condition for $g_t$ to locally minimize $\bar F_n(g)$.
Indeed, any local minimizer $\tilde g$ of $\bar F_n(g)$ must satisfy
$\frac{\d}{\d t} \bar{F}_n\big((1-t)\tilde{g} + t\delta_\theta\big)|_{t=0+} \geq 0$ for
all $\theta\in [-M,M]$, which implies that
\begin{align}\label{ineq:npmle_necessary}
-\delta \bar{F}_n(\tilde{g})(\theta) \leq 0 \quad \text{for all }
\theta \in [-M,M]. 
\end{align}
The statement (\ref{ineq:subgrad}) averages this
inequality with respect to a fixed (arbitrary) reference density $h$.

\item If $g_\ast$ admits a density on $[-M,M]$, then
choosing $h=g_\ast$ in (\ref{ineq:subopt_gap_joint}) gives a certification that
$g_t$ is a near-NPMLE, in the sense of Theorem \ref{thm:consistency}.
(If $g_\ast$ does not admit a density, then the theorem may be applied
with $h$ being a smoothed approximation of $g_\ast$
similar to the discussion after Theorem \ref{thm:g_flow}.)

For general regression designs, the marginal log-likelihood
$\bar F_n(g)$ is non-convex in $g$ and may have multiple local optima,
so Theorem \ref{thm:algo} only
ensures (\ref{ineq:subopt_gap_joint}) from an initialization $(q_0,g_0)$
for which the sub-level set $\cK_n=\{g:\bar F_n(g) \leq F_n(q_0,g_0)\}$ is convex.
If $q_0=P_{g_0}(\bvarphi \mid \y)$ is the posterior law corresponding
to $g_0$ (obtained e.g.\ from a burn-in run of a Langevin diffusion
with the fixed prior $g_0$), then $F_n(q_0,g_0)=\bar F_n(g_0)$, so this
requires convexity of $\{g:\bar F_n(g) \leq \bar F_n(g_0)\}$.

We emphasize that $\cK_n$ is a sub-level set for the univariate prior $g$, and the
theorem does \emph{not} require the much stronger joint convexity of
$F_n(q,g)$ on the sub-level set $\{(q,g):F_n(q,g) \leq F_n(q_0,g_0)\}$.
In certain asymptotic settings as $n,p \to \infty$, we expect that
$\bar F_n(g)$ and its functional derivatives may converge to those of a
deterministic functional $\bar F(g)$, and that
$\bar F_n(g)$ for all large $n,p$ would be convex in a fixed
and dimension-free sub-level set $\mathcal{K}$ of this limit $\bar F(g)$;
we leave a more detailed investigation of this to future work.
\item Under Assumption \ref{assump:design}, for any choice
$\tau^2<\sigma^2/\|\X\|_\op^2-\delta$, and for any initial conditions
$(q_0,g_0)$ satisfying the mild requirements $\DKL(g_*\|g_0)=O(1)$, 
$\DKL(q_0\|(\N_\tau*g_0)^{\otimes p})=O(p)$,
and $\E_{q_0}[\|\bvarphi\|_2^2]=O(p)$, the time horizon $T(g_*,\eps)$ in
\eqref{def:joint_flow_time} is bounded as
\begin{equation}\label{eq:Tbound}
T(g_*,\eps)=O(\eps^{-3.01}(n+p)).
\end{equation}
We refer to Appendix \ref{appendix:initialgap} for a verification of this
claim. Thus under these conditions, Theorem \ref{thm:algo} establishes
closeness of $g_t$ and $g_*$ on a time horizon that is linear in $(n,p)$.
\end{enumerate}

As mentioned in the introduction,
convergence of MCEM methods to fixed points of the EM algorithm has been
classically studied in
\cite{chan1995monte,fort2003convergence,neath2013convergence}, in settings of
asymptotically growing numbers of EM iterations between discrete M-step updates.
Theorem \ref{thm:algo} may be understood as a convergence result for a
continuous-time version of MCEM, in a regime where the prior $g_t$ is
continuously updated before the law $q_t$ of the posterior sample equilibrates
to the stationary law $\nu[g_t]$.

\section{Experiments}\label{sec:simulation}

To allow the flexibility of fitting a smooth prior, we consider optimizing the marginal likelihood $\bar F_n(g)$ in (\ref{def:npmle}) with an additional smoothing-spline penalty (as suggested classically in \cite{goodd1971nonparametric}, and also studied for
deconvolution estimation in sequence models in \cite{madrid2018deconvolution}):
\begin{align}\label{eq:smoothobjective}
\bar F_{n,\lambda}(g)&= \bar F_n(g) +\frac{\lambda}{2}\int_{-M}^M g''(\theta)^2 d\theta =  -\frac{1}{p}\log P_g(\y)
+\frac{\lambda}{2}\int_{-M}^M g''(\theta)^2 d\theta.
\end{align}
By tuning the magnitude of $\lambda$, we may trade off optimization of the data
log-likelihood $\bar F_n(g)$ with the smoothness of the fitted prior $\hat g$. We discuss the implementation details of the EBflow algorithm for (\ref{eq:smoothobjective}) in Appendix~\ref{subsec:ebflow_smoothness}.

 We compare EBflow described in Section \ref{sec:discretize} with
several versions of Monte-Carlo expectation maximization
(MCEM), and with an approach based on mean-field variational inference using
coordinate ascent updates (CAVI).
We implement all methods to estimate the weights of a
discrete prior $\sum_{k=1}^K w_k \delta_{b_k}$ with $K=61$ fixed and equally
spaced support points $b_1,\ldots,b_K \in [-3,3]$. 

For MCEM, the most direct comparison is with an EM implementation
using the same latent parameter $\bvarphi$ as in EBflow, sampling iterates
$\{\bvarphi_t\}_{t \geq 0}$ using an Unadjusted Langevin Algorithm (ULA) to
compute a Monte Carlo approximation for the E-step. We perform $T_\text{iter}$
ULA iterations (\ref{eq:phiupdate}) with step size $\eta^\varphi$ between
M-steps, and update the prior in each M-step by computing the NPMLE (with fixed
support and estimated weights) in the sequence
model (\ref{eq:varphiseqmodel}) based on the coordinates
$(\varphi_{t,j})_{j=1}^p$ of the past $T_\text{iter}$ iterations. This update is
a convex program, which we solve using CVXR with the CSC solver. We refer to
this method as Langevin-MCEM.

We also compare to a Gibbs-MCEM method and to CAVI, both operating
in the original latent variable parametrization by $\btheta$.
For Gibbs-MCEM, we perform $T_\text{iter}$ Gibbs sampling iterations
between updates of the prior,
each iteration resampling $\{\theta_j\}_{j \in [p]}$ in sequence.
For CAVI, we optimize the Gibbs free energy
over priors $g(\theta)$ and product-law posteriors $q(\btheta)=\prod_{j=1}^p
q_j(\theta_j)$ by updating $\{q_j\}_{j \in [p]}$ in sequence, followed by
the prior $g$.

As a short burn-in period, we perform the first 200 iterations of EBflow,
Langevin-MCEM, and Gibbs-MCEM without prior updates,
under a uniform prior initialization $w=(\frac{1}{K},\ldots,\frac{1}{K})$.
Further details of the methods are discussed in Appendix~\ref{sec:other_algo}.

\subsection{Tuning parameter choices}\label{sec:tuning}
In this section, we first discuss some considerations on choosing
tuning parameters of the various methods,
including the step size schedules $\eta_t^\varphi,\eta_t^w$, smoothing-spline penalty $\lambda$, and variable reparametrization
variance $\tau^2$ for EBflow. More details of algorithm implementation can be found in Appendix~\ref{sec:other_algo}.

\phantom{}\\
\noindent \textbf{Variable reparametrization variance and smoothness
regularization.} We first discuss the choice of $\tau^2$ in EBflow and
$\lambda$ in the spline-penalized objective \eqref{eq:smoothobjective}.

It is understood that the approximation accuracy of ULA for its continuous-time
Langevin diffusion improves with the smoothness of the log-density of its
stationary distribution \citep{vempala2019rapid,balasubramanian2022towards},
and a smoother log-density will also allow for larger step sizes and faster
convergence.
In the parametrization by $\bvarphi$, the negative Hessian of the
log-posterior-density is (by Tweedie's formula)
\begin{equation}\label{eq:hessianposterior}
-\nabla^2 \log P(\bvarphi \mid \y)
=\X^\top \bSigma^{-1}\X+\frac{1}{\tau^2}\Id
-\frac{1}{\tau^4}\diag\Big(\Var[\theta_j \mid \varphi_j]\Big)_{j=1}^p
\end{equation}
where $\bSigma=(\sigma^2\Id-\tau^2\X\X^\top)^{-1}$. 
Ignoring the last $\bvarphi$-dependent term and treating the largest eigenvalue
\begin{equation}\label{eq:lambdamax}
\lambda_{\max}:=\lambda_{\max}(\X^\top \bSigma^{-1}\X+\tau^{-2}\Id)
\end{equation}
of the first two terms as a proxy for
the smoothness of this log-posterior, this motivates a choice
\begin{equation}\label{eq:tausqsetting}
\tau^2=0.5\,\sigma^2/\|\X\|_\op^2
\end{equation}
which is the value minimizing $\lambda_{\max}$ in \eqref{eq:lambdamax}.
Table \ref{tab:tausq} of the Supplementary Appendices
provides a numerical comparison of results of EBflow for different choices of
\[\tau^2=c_\tau \cdot \sigma^2/\|\X\|_\op^2\]
with $c_\tau \in \{0.8,0.5,0.2\}$ (where we recall from \eqref{eq:psdconstraint}
that we must have $c_\tau \in (0,1)$), under various
priors and regression designs to be discussed in Section \ref{sec:comparison}.
We observe that results are fairly robust to this choice of $c_\tau$,
with the choice $c_\tau=0.5$ in \eqref{eq:tausqsetting} yielding lowest errors
in several settings. We will henceforth fix this choice \eqref{eq:tausqsetting} 
in the subsequent experiments.

For purposes of smooth density estimation, the choice of regularization
parameter $\lambda$ in the smoothing spline penalty \eqref{eq:smoothobjective}
represents a standard bias-variance tradeoff, where the best choice
depends on the smoothness of the true prior $g_*$. Many heuristics based on
information criteria and cross-validation, as well as sensitivity analysis along
the full regularization path corresponding to all $\lambda>0$, are
discussed in \cite{madrid2018deconvolution} and apply
equally in our setting. (We do not develop a fast algorithm for
computing this full regularization path as in \cite{madrid2018deconvolution},
and leave this as an interesting question for future work.)

Table \ref{tab:lambda} of the Supplementary
Appendices provides a numerical comparison of results of EBflow for the
choices $\lambda \in \{0.01,0.03,0.001,\text{1e-4},\text{1e-5},0\}$
and the same priors and designs as in Table \ref{tab:tausq}. 
Denoting by $\sum_{k=1}^K
w_{*,k} \delta_{b_k}$ a discretization of the true prior density $g_*(\theta)$
using weights $w_{*,k}=g_*(b_k)\Delta:=g_*(b_k)(b_{k+1}-b_k)$,
Tables \ref{tab:tausq} and \ref{tab:lambda} report the (discretized)
total-variation error
\begin{equation}\label{eq:TV}
\dTV(\widehat g,g_\ast)=\frac{1}{2}\sum_{k=1}^K |\widehat
w_k-w_{\ast,k}|
\end{equation}
of the final EBflow density estimate $\widehat g=\sum_{k=1}^K \widehat w_k
\delta_{b_k}$, together with the attained value of the negative
marginal log-likelihood $\bar F_n(\widehat g)$ for the identity design
$\X=\Id$, or the standardized prediction mean-squared-error
\begin{equation}\label{eq:predictionMSE}
\text{MSE}=\frac{\E\|\x_\text{new}^\top(\btheta_\ast-\E_{\widehat g}[\btheta
\mid \y])\|^2}{\E\|\x_\text{new}^\top\btheta_\ast\|^2}
\end{equation}
over fresh samples $\x_{\text{new}}$ for the remaining random designs.
The posterior mean $\E_{\widehat g}[\btheta \mid \y]$ is computed as described
in Section \ref{sec:posteriorinference} using $T'=50{,}000$ additional iterates of $\bvarphi_t$, and
the expectations over $\x_{\text{new}}$ are approximated via
$n_\text{test}=1000$ independent test samples.
All TV errors, log-likelihoods, and prediction MSEs are 
averaged over 10 random trials, using the same design $\X$ and
independently sampled regression coefficients and noise $(\btheta,\beps)$ per trial.

As expected, the negative log-likelihood $\bar F_n(\widehat g)$
is smallest for all priors in the setting of no regularization
$\lambda=0$. We observe that the
prediction MSE is also nearly minimized for most priors and designs at
$\lambda=0$, and is approximately constant across a range of
sufficiently small values of $\lambda$. The smallest errors in TV distance for
the estimated prior densities are generally attained at the smoothest
estimates within these ranges of $\lambda$ that yield near-optimal prediction
MSE.\\

\noindent \textbf{Step sizes.} 
Theoretical results of \cite{vempala2019rapid,balasubramanian2022towards}
on ULA suggest a step size inversely proportional to the smoothness of the
target density. As such, 
we will rescale the step size $\eta_t^\varphi$ in EBflow and Langevin-MCEM by
$1/\lambda_{\max}$ in \eqref{eq:lambdamax},
i.e.\ all further reference to $\eta_t^\varphi$ will be
in the following reparametrized form of (\ref{eq:phiupdate}):
\begin{align}
\bvarphi_{t+1}&=\bvarphi_t-\frac{\eta^\varphi_t}{\lambda_{\max}}
\Bigg[\X^\top\bSigma^{-1}(\X\bvarphi_t-\y)+
\frac{1}{\tau^2}\bigg(\frac{\sum_{k=1}^K w_{t,k}(\varphi_{t,j}-b_k)
\N_\tau(\varphi-b_k)}{\sum_{k=1}^K
w_{t,k}\,\N_\tau(\varphi_{t,j}-b_k)}\bigg)_{j=1}^p\Bigg]\nonumber\\
&\hspace{1in}+\sqrt{\frac{2\eta^\varphi_t}{\lambda_{\max}}}\,\b_t.
\label{eq:phiupdatereparam}
\end{align}

A larger step size
$\eta_t^{\varphi}$ in EBflow leads to faster mixing, but can result in a higher
discretization bias for ULA. The step size $\eta_t^w$ determines an exponential
weighting of past iterates when updating the prior, where large
$\eta_t^w$ can lead to high Monte Carlo variance in the estimated prior,
while small $\eta^w_t$ can lead to slow convergence.

To illustrate these tradeoffs,
Figure \ref{fig:tuning}(a) of the Supplementary Appendices
displays the TV error \eqref{eq:TV}
of the density estimate $g_t=\sum_{k=1}^k w_t
\delta_{b_t}$ across 10{,}000 iterations of EBflow, for 10 trials with
independently sampled $(\btheta,\beps)$ under a
Gaussian prior $g_*=\N(0,1)$ and i.i.d.\ Gaussian design $x_{ij}
\overset{iid}{\sim} \N(0,\frac{1}{n})$.
In this example, fixing larger step sizes
$\eta^{\varphi}_{t}=1.0$ and $\eta^w_t=0.01$ results in better performance,
with smaller $(\eta_t^w,\eta_t^\varphi)$ leading to slower convergence, and
larger $\eta_t^w/\eta_t^\varphi$ leading to higher variance in the
estimated prior. There is a small amount of ULA bias incurred by the large step
size choice $\eta^{\varphi}_{t}=1.0$, and
Figure \ref{fig:tuning}(b) displays the TV error under an
alternative log-linear step size decay over 10,000 iterations,
\begin{equation}\label{eq:loglinearstep}
\eta_t^\varphi=a \cdot b^t \text{ with } \eta_t^\varphi=1.0 \text{ at } t=1
\text{ and } \eta_t^\varphi=0.1 \text{ at } t=10000,
\qquad \eta_t^w=0.01\eta_t^\varphi.
\end{equation}
The reduction of ULA bias for large $t$
yields a reduction of the final TV error (averaged across 10 trials) to 0.040,
from 0.070 for the fixed choice
$(\eta_t^\varphi,\eta_w^\varphi)=(1.0,0.01)$. We recommend using this log-linear
decay (\ref{eq:loglinearstep}) as the default option for problems of this size,
with the possibility of increasing $\eta_t^w$ relative to $\eta_t^\varphi$ for
problems with higher dimension $p$.

Analogous trade-offs apply for the other MCEM approaches when determining the
number of iterations $T_\text{iter}$ between prior updates.
Figure \ref{fig:tuning}(c--d) displays TV error for Langevin-MCEM with
$\eta_t^\varphi \in \{1.0,0.1\}$ and $T_\text{iter} \in \{100,1000\}$, and for
Gibbs-MCEM with $T_\text{iter} \in \{10,100,1000\}$, both across the
same 10 random simulations of $(\btheta,\beps)$. The best estimates are
attained at $T_\text{iter}=100$ for both methods, and we will fix this choice
in the subsequent method comparisons.\\

\subsection{Method comparisons}\label{sec:comparison}
In this section, we now compare the various methods across a range of
priors and designs.

\phantom{}\\
\noindent\textbf{Priors and designs.} We generate data $\y=\X\btheta+\beps$
where $\{\theta_j\}_{j=1}^p \overset{iid}{\sim} g_\ast$, and the prior
distribution $g_\ast$ is one of four smooth densities:
\begin{itemize}
\item (gaussian) $g_\ast=\N(0,1)$
\item (skew) $g_\ast=\frac{1}{3}\N(-2,0.33^2)+\frac{1}{3}\N(-1.5,0.5^2)
+\frac{1}{3}\N(0,1)$
\item (scale-mixture)
$g_\ast=\frac{1}{3}\N(0,1)+\frac{1}{3}\N(0,0.3^2)+\frac{1}{3}\N(0,0.1^2)$
\item (bimodal) $g_\ast=\frac{1}{2}\N(-1.5,0.5^2)+\frac{1}{2}\N(1.5,0.5^2)$
\end{itemize}
We generate $\X$ to be one of the following designs:
\begin{itemize}
\item (identity) $\X=\Id$ and $n=p$.
\item (iid) $\X$ has i.i.d.\ Gaussian entries $x_{ij} \overset{iid}{\sim}
\N(0,\frac{1}{n})$.
\item (block02corr0.9) $\X$ has independent multivariate Gaussian rows with marginal
distribution of coordinates $x_{ij} \sim \N(0,\frac{1}{n})$, consisting of
$\frac{p}{2}$ independent pairs of variables with correlation 0.9.
\item (block10corr0.5) $\X$ has independent multivariate Gaussian rows with marginal
distribution of coordinates $x_{ij} \sim \N(0,\frac{1}{n})$, consisting of
$\frac{p}{10}$ independent blocks of variables with pairwise correlation 0.5.
\item (genotype) $\X$ represents a design matrix that may arise in a
single-locus association analysis in statistical genetics studies.
We simulate independent (human) European genotypes
in a 5 megabase locus of Chromosome 6 using the coalescent simulator
stdpopsim with an out-of-Africa migration model and estimated recombination
rates \cite{adrion2020community}:
\begin{center}
\texttt{stdpopsim HomSap -c chr6 --left 33000000 --right 38000000 -o
genotypes.ts -d OutOfAfrica\_2T12 EUR:3000}
\end{center}
We then filter down to a subset of $p=1000$ common variants in
this locus, restricting to minor allele frequency $>0.01$,
and standardize columns of $\X$ to mean 0 and variance $\frac{1}{n}$. A plot of
pairwise variable correlations (i.e.\ the linkage disequilibrium matrix) in this
design $\X$ is depicted in Figure \ref{fig:LD} of the Supplementary Appendices.
\end{itemize}
We fix the dimension $p=1000$ for all designs. For the latter four random
designs, we vary the training sample sizes $n \in \{500,2000\}$
corresponding to $n/p=0.5$ and $n/p=2.0$.
The noise variance $\sigma^2$ is set
so that $\beps$ explains 50\% of the variance of $\y$
for the Gaussian prior, and 20\% of the variance of $\y$
for the (more difficult) skew, scale-mixture, and bimodal priors.
We perform 10 independent trials for each method,
using the same sampled design $\X$ and independently sampled $(\btheta,\beps)$
per trial.\\

\noindent\textbf{Experimental setup.} EBflow is run with the choice of
$\tau^2$ in (\ref{eq:tausqsetting}) and
log-linear step size decay (\ref{eq:loglinearstep}) over 10{,}000 iterations.
We tested also a preconditioned variant 
of EBflow with the $\bvarphi$ update in
(\ref{eq:phiupdatereparam}) replaced by
\begin{align}\label{eq:phiupdateprecond}
\bvarphi_{t+1}&=\bvarphi_t-\eta^\varphi_t \mathbf{Q}^{-1}
\Bigg[\X^\top\bSigma^{-1}(\X\bvarphi_t-\y)+
\frac{1}{\tau^2}\bigg(\frac{\sum_{k=1}^K w_{t,k}(\varphi_{t,j}-b_k)
\N_\tau(\varphi-b_k)}{\sum_{k=1}^K
w_{t,k}\,\N_\tau(\varphi_{t,j}-b_k)}\bigg)_{j=1}^p\Bigg]\notag\\
&\hspace{1in}+\sqrt{2\eta^\varphi_t}\mathbf{Q}^{-1/2}\,\g_t,
\end{align}
with preconditioner
$\mathbf{Q}=\X^\top \bSigma^{-1}\X+\tau^{-2}\Id$ as a proxy for the
Hessian (\ref{eq:hessianposterior}). Related forms of preconditioning have been
suggested, for example, for posterior sampling in logistic regression models in
\cite{dalalyan2017theoretical} to accelerate mixing of Langevin dynamics with
ill-conditioned designs. One may check that this preconditioning corresponds
to the gradient flow (\ref{eq:joint_flow_intro}) with an alternative
choice of inner-product on $\R^p$ underlying the Wasserstein-2 geometry.
Langevin-MCEM and Gibbs-MCEM are run fixing $T_\text{iter}=100$ Monte Carlo iterations
between M-step updates, also over 10{,}000 total Monte Carlo iterations.
CAVI is run for 1000 iterations, as it converges much faster than the MCEM methods.

All methods are applied to minimize the smoothing-spline regularized
log-likelihood objective \eqref{eq:smoothobjective}, where
the penalty is fixed at $\lambda=0.003$ for the (smoothest) Gaussian prior,
$\lambda=0.001$ for the bimodal prior, and $\lambda=\text{1e-5}$ for the skew
and (least smooth) scale-mixture priors. These choices correspond roughly to
the largest settings of $\lambda$ in Table \ref{tab:lambda} within the penalty
ranges that give similar prediction MSE as the unregularized estimate 
($\lambda=0$). To simplify the experiments, we use these fixed choices of
$\lambda$ and apply all methods to optimize the same regularized objective with
this fixed penalty; in practice, this choice of $\lambda$ may be approximately
determined using cross-validation.

The posterior mean $\E_{\widehat g}[\btheta \mid \y]$ corresponding to the
final estimated prior $\widehat g$ is computed in EBflow and Langevin-MCEM
using $T'=50{,}000$ additional iterates of $\bvarphi_t$ as described in Section
\ref{sec:posteriorinference}, using $T'=50{,}000$ additional iterates of
$\btheta_t$ in Gibbs-MCEM, and using the average of the posterior
marginals estimated in CAVI.\\

\noindent\textbf{Results.} Figure \ref{fig:block02corr0.9} displays the
estimated priors in the example of the block02corr0.9
design with $n=2000$ and $p=1000$. Estimated priors in
the other 4 regression designs are depicted in Figures
\ref{fig:identity}, \ref{fig:iid}, \ref{fig:block10corr0.5}, and
\ref{fig:genotype} of the Supplementary Appendices.

Tables \ref{tab:gaussian}--\ref{tab:bimodal} of the Supplementary Appendices
summarize estimation accuracy and runtime across all tested methods, priors,
and designs. We report the TV error \eqref{eq:TV}, attained
negative log-likelihood $\bar F_n(g)$ for identity design $\X=\Id$,
and prediction MSE \eqref{eq:predictionMSE} for the remaining random designs,
all averaged over 10 trials. We report also a measure of the compute time,
calculated as the median across 10 trials of the time in seconds needed for each
method to attain TV error $<0.3$, under our R-based implementations of these
methods. We report a compute time of Inf if fewer than half of the trials
attained TV error $<0.3$.

In Tables \ref{tab:gaussian}--\ref{tab:bimodal},
we report also an ``oracle'' benchmark $\bar F_n(g_*)$ 
for the negative log-likelihood in the identity design $\X=\Id$, and 
\[\text{MSE}=\frac{\E\|\x_\text{new}^\top(\btheta_\ast-\E_{g_*}[\btheta \mid
\y])\|^2}{\E\|\x_\text{new}^\top\btheta_\ast\|^2}\]
for the prediction MSE in the remaining four random designs,
all computed with knowledge of the true prior $g_*$.
The oracle value of $\bar F_n(g_*)$ for $\X=\Id$ and prediction MSE for standard
Gaussian prior $g_*=\N(0,\Id)$ are computed analytically, and the prediction
MSEs for the remaining priors are computed via MCMC approximation of
$\E_{g_*}[\btheta \mid \y]$ using the Langevin-MCEM method (without prior
updates) and a fixed step size $\eta^\varphi=0.1$ corresponding to the final
step size of EBflow in \eqref{eq:loglinearstep}.\\

\revise{
\begin{figure}
\centering
\includegraphics[width=1\textwidth]{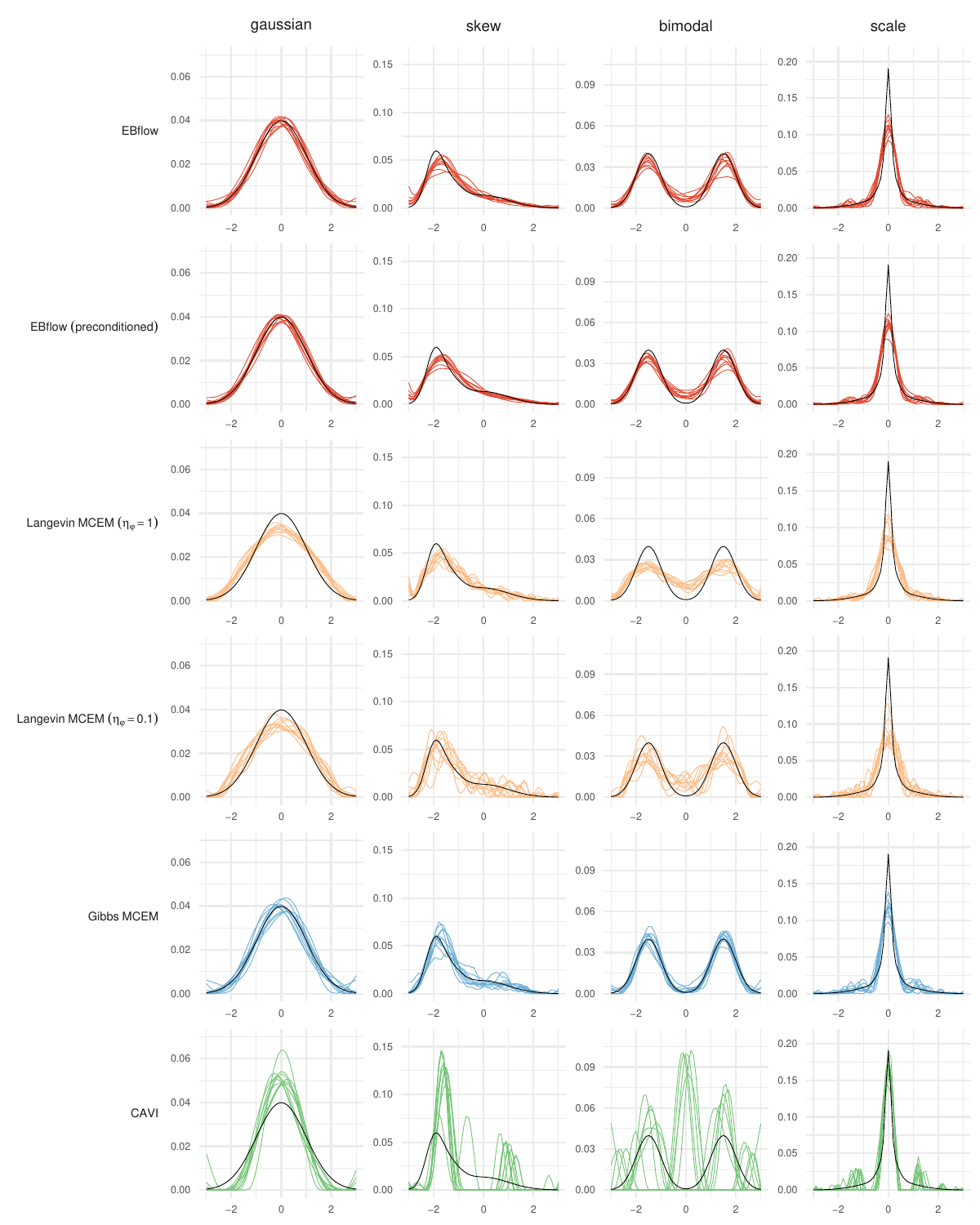}%
\caption{Estimated prior densities across 10 independent simulations of
$(\btheta,\beps)$. The regression design is block02corr0.9
($p/2$ independent pairs of Gaussian variables with correlation 0.9 within each
pair), with $n=2000$ and $p=1000$. The true priors left to right are: gaussian,
skew, scale-mixture, bimodal.}\label{fig:block02corr0.9}
\end{figure}
}

We provide here a high-level summary of these comparisons:
\begin{itemize}
\item For identity design $\X=\Id$ (i.e.\ the Gaussian sequence
model \eqref{def:gau_seq}),
all methods attain a negative log-likelihood value $\bar F_n(\widehat g)$
close to or smaller than the ``oracle'' value $\bar F_n(g_*)$, and they return
reasonably accurate estimates of the true prior.
Density estimates of EBflow seem somewhat more accurate and less
variable across trials than the other methods.
\item For the remaining random (non-identity) designs $\X$, estimates of EBflow
and the MCEM procedures are generally more accurate than CAVI,
in terms of both TV error and prediction MSE. The difference with CAVI is
particularly large in some settings of block-correlated and genotype designs,
but also apparent for i.i.d.\ Gaussian designs (where a TAP correction
\cite{celentano2023mean} would be needed for CAVI to achieve consistent
estimation of $g_*$ in the high-dimensional settings tested here).
\item Across EBflow and the MCEM procedures, prediction MSEs are generally
similar and close to the oracle benchmark that has access to
$g_*$.\footnote{We note that the prediction MSE depends both on the
difference between posterior means $\E_{g_*}[\btheta \mid \y]$ and $\E_{\widehat
g}[\btheta \mid \y]$,
and on the entrywise variance of the MCMC estimate of
$\E_{\widehat g}[\btheta \mid \y]$. The latter depends in turn on the
autocorrelation of the Markov Chain, so prediction MSE is
an aggregate measure of the estimation error of $\widehat g$ and mixing rate of
the MCMC sampler. There are several
settings, for example under the bimodal prior,
where the implemented methods achieve smaller
prediction MSE than the ``oracle'' benchmark: We believe this is caused by the
phenomenon that the
tested MCMC procedures --- including the ULA scheme used to compute the oracle
--- mix more slowly for the posterior of the true prior $g_*$ than
for the estimated prior $\widehat g$, and this difference has a larger effect on
the prediction MSE than the difference between
$\E_{g_*}[\btheta \mid \y]$ and $\E_{\widehat g}[\btheta \mid \y]$.}
EBflow generally attains lower TV error than these fixed tunings of
Langevin-MCEM, and the lowest TV error across many designs for the gaussian 
prior. Gibbs-MCEM attains the lowest TV error across many designs for the
bimodal prior, and EBflow is competitive with Gibbs-MCEM in TV error
for the skew and scale-mixture priors.
\item In our implementations, EBflow generally has faster runtime than both
Langevin-MCEM and Gibbs-MCEM. We note that long runtimes for Langevin-MCEM arise
from the full computation of the NPMLE using CVXR/CSC in each M-step update.
\item Our tested preconditioning of EBflow improves estimation
accuracy in the genotype design, and has a lesser effect in the other 
designs.
\end{itemize}

\section{Discussion}\label{sec:conclusion}
We have studied empirical Bayes estimation of an unknown prior distribution for
the regression coefficients of a
linear model, introducing a system of bivariate gradient flow
equations for optimizing the marginal log-likelihood via its Gibbs variational
representation. This gradient flow system may be simulated using a Langevin
diffusion, where the prior is continuously updated based on an estimate of the
average marginal distribution of coordinates of the Langevin sample.
Although the regression model and Langevin sample are both high-dimensional, the
one-dimensional nature of the estimand $g_\ast$ enables its estimation using a
fully nonparametric approach. Similar ideas can be explored in
models beyond linear regression, where a high-dimensional latent
parameter is comprised of independent draws from an unknown
univariate or low-dimensional prior law.

We believe there are many interesting directions for future work, and we discuss
here two of these directions:
\begin{enumerate}
\item (Global landscape analysis) For general regression designs,
the optimization landscape of the negative log-likelihood $\bar F_n(g)$ may have
multiple local minimizers, and we find empirically that the prior density estimated by EBflow
is sometimes sensitive to the choice of initialization. In such settings, we
have only provided a theoretical guarantee of convergence of the gradient flow
to a near-NPMLE from a ``local'' initialization within a convex sub-level
set of $\bar F_n(g)$.
It is an open question to obtain a better understanding of the global landscape
geometry of $\bar F_n(g)$ and global convergence properties of the proposed
gradient flow in the presence of multiple local minimizers.

\item (Time discretization and propagation-of-chaos) In this work, we have
studied the dynamics and convergence of the idealized continuous-time gradient
flow, but have not yet taken the steps to translate this to
convergence guarantees for the EBflow algorithm.
We believe that a careful analysis of the EBflow algorithm may lead to improved
understanding of issues related to parameter tuning, such as the choice of
variance $\tau^2$ in our reparametrization by the smoothed regression variable
$\bvarphi$.

One particular question of interest is to understand conditions under which the
true law $\bar q_t$ may be consistently estimated by the empirical average of
coordinates in (\ref{eq:barqestimate}) over a time horizon that is
long enough for convergence of the gradient flow, using
only one realization $\{\bvarphi_t\}_{t \in [0,T]}$ of the Langevin diffusion
rather than an asymptotically growing number of independent parallel chains. We
note that this type of ``single-chain propagation-of-chaos'' result is only
possible in settings of growing dimensions $p \to \infty$, and we believe that
its study may lead to improved understanding of functional inequalities and
concentration of measure for the law $\{q_t\}_{t \in [0,T]}$ as it evolves
along the gradient flow.
\end{enumerate}

Following the initial posting of our manuscript, the works
\cite{fan2025dynamicalI,fan2025dynamicalII} have taken a step to address
several of these questions for a parametric analogue of EBflow,
in a setting of design matrices $\X$ with i.i.d.\
coordinates. It was shown in \cite{fan2025dynamicalI} that the above
single-chain propagation-of-chaos holds over dimension-free time horizons
$[0,T]$, and exact asymptotic characterizations of both the landscape
of $\bar F_n(g)$ and global dynamics of the adaptive Langevin procedure
(rather than the gradient flow) were given in
\cite{fan2025dynamicalII} under an assumption of a uniform posterior LSI similar
to this work. These
questions remain open for the nonparametric EBflow procedure described here,
and for design matrices $\X$ beyond the i.i.d.\ setting.

Our numerical simulations suggest that EBflow is competitive in many settings
with alternative approaches such as MCEM with coordinate-wise Gibbs sampling,
in terms of both estimation accuracy and runtime. More importantly, we hope
that this EBflow approach can make available a broader toolbox of ideas related
to Langevin-based algorithms for empirical Bayes estimation,
including ideas around adaptive preconditioning \citep{girolami2011riemann}
and underdamped/Hamiltonian dynamics \citep{neal2011mcmc,cheng2018underdamped}
to handle more ill-conditioned designs, and stochastic gradient
approximations \citep{welling2011bayesian,ma2015complete} to scale to
larger-$(n,p)$ problems. In extensions to more complex models where
latent parameters are organized in multi-layer architectures,
such gradient-based ideas may also have advantages over coordinate-wise Gibbs
sampling, due to the lower complexity of computing gradients using
back-propagation procedures. We believe that such extensions to more complex
models and sampling algorithms would be interesting to explore also in future
work.

\appendix

\section{Proof of consistency}\label{sec:proof_consistency}

In this appendix, we prove Theorem \ref{thm:consistency} and Proposition
\ref{prop:condition_gaussian}.

\subsection{TV lower bound}

For any $\eta>0$ and $g\in \cP(M)$, we write as shorthand
\[g_\eta=\N_\eta*g=\N(0,\eta^2)*g.\]
Let
\begin{align}\label{def:smooth_W1}
d_\eta(g,g')=W_1\big(g_\eta,g_\eta'\big)
\end{align}
be the smoothed Wasserstein-1 distance on $\cP(M)$.

All constants in the lemmas and proofs may depend on the support constraint $M$,
on $C_\ast, c_\ast,\beta$ in Assumption \ref{assump:design},
and on $\sigma^2$ and the lower bound $\gamma$ for $n/p$ in the
setting of Proposition \ref{prop:condition_gaussian}.

\begin{lemma}\label{lemma:TVlower}
Under Assumption \ref{assump:design},
there exists a sufficiently large constant $v>0$, constants $c,c'>0$, and
a sequence $\delta_n \to 0$ such that for all large $n,p$ and all
$g,g' \in \cP(M)$,
\begin{align*}
\dTV(P_g(\y),P_{g'}(\y)) \geq 1-2e^{-cp(d_v(g,g')-\delta_n)_+^2}.
\end{align*}
\end{lemma}

\begin{proof}
Note that $\dTV(P_g(\y),P_{g'}(\y))$ and all conditions of
Assumption \ref{assump:design} are invariant under a simultaneous
rescaling $(\X,\y,\sigma,\z_j) \mapsto (\alpha
\X,\alpha\y,\alpha\sigma,\alpha^{-1}\z_j)$ for any
$\alpha>0$. Thus, applying such a rescaling, we may assume without loss
of generality that 
\begin{equation}\label{eq:Xnorm}
\|\X\|_\op=1, \qquad c_*<\sigma^2<C_*,
\end{equation}
where the bounds for $\sigma^2$ follow from Assumption \ref{assump:design}
and $p^{-1}\|\X\|_{\mathrm{F}}^2 \leq \|\X\|_\op^2$.

Fixing any constant $\eta>0$, the smoothed distribution $g_{\eta}$ satisfies a
LSI with uniformly bounded LSI constant for all
$g \in \cP(M)$ \cite[Theorem 1.1]{bardet2018functional}. Note that convolving $g$
with $\N_\eta$ corresponds to adding an independent $\N(0,\eta^2\,\X\X^\top)$ 
noise vector to $\y$. Hence by the data processing inequality,
\begin{align*}
\dTV\big(P_g(\y),P_{g'}(\y)\big)
\geq \dTV\big(P_{g_{\eta}}(\y),P_{g_{\eta}'}(\y)\big),
\end{align*}
and it suffices to lower bound the TV distance on the right side for priors
$g_{\eta}$ and $g_{\eta}'$.

Noting that $g_{\eta}$ has variance at most $M^2+\eta^2$ for every $g \in
\cP(M)$, we pick a constant $v>0$ satisfying
\begin{equation}\label{eq:vdef}
v^2 \geq \max_{j\in \cS} \left(\eta^2+v_j \cdot (M^2+\eta^2)+w_j\right),
\qquad \text{ where } v_j=\sum_{k:k \neq j} (\z_j^\top \x_k)^2,
\qquad w_j=\sigma^2 \|\z_j\|_2^2,
\end{equation}
and $\{\z_j\}_{j\in\cS}$ are given by Assumption \ref{assump:design}. We will
check below that $v_j,w_j$ are bounded above by a constant uniformly over
$j \in \cS$, so that we may indeed pick $v$ here as a constant independent of
$n,p$. Denote $\Delta=d_v(g,g')=W_1(g_v,g_v')$.
To establish the desired TV lower bound, it suffices to exhibit an explicit test statistic $T(\y)$ with rejection threshold
$t=(\E_{g_{\eta}}[T(\y)]+\E_{g_{\eta}'}[T(\y)])/2$ such that
\begin{align}\label{eq:testingerrors}
\max(\P_{g_{\eta}}[T(\y)<t],\P_{g_{\eta}'}[T(\y) \geq t]) \leq e^{-cp(\Delta-\delta_n)_+^2}.
\end{align}
Indeed, this immediately implies the conclusion of the lemma,
\[\dTV(P_{g_\eta}(\y),P_{g_\eta'}(\y)) \geq
\Big|\P_{g_{\eta}}[T(\y) \geq t]-\P_{g_{\eta}'}[T(\y) \geq t]\Big| \geq
1-2e^{-cp(\Delta-\delta_n)_+^2}.\]

Let $\kappa_1,\kappa_2$ and $\kappa_1',\kappa_2'$ denote the means and
variances of $g_{\eta},g_{\eta}'$ respectively. To construct $T(\y)$ satisfying
(\ref{eq:testingerrors}), we consider three cases. The constants
$C,c>0$ below will change from instance to instance.\\

{\bf Case 1:} $|\kappa_1-\kappa_1'|>c_1\Delta$ for a constant $c_1>0$ to be
specified. Suppose
without loss of generality $\kappa_1>\kappa_1'$, and consider the test
statistic
\begin{align*}
T(\y)=\1^\top \X^\top \y=\1^\top \X^\top \X\btheta+\1^\top \X^\top \beps.
\end{align*}
Under Assumption \ref{assump:design}(a) and the normalization
$\|\X\|_\op=1$ in \eqref{eq:Xnorm}, we have
\begin{equation}\label{eq:meansep1}
\E_{g_{\eta}}[T(\y)]-\E_{g_{\eta}'}[T(\y)]
=\1^\top \X^\top \X\1 \cdot (\kappa_1-\kappa_1')
\geq c\Delta p.
\end{equation}
Computing the gradient of $T(\y)$ in $(\btheta,\beps)$, we observe that
$T(\y)$ is $L$-Lipschitz in $(\btheta,\beps)$ with respect to
the $\ell_2$-norm, where
$L^2=\1^\top(\X^\top \X)^2 \1+\1^\top \X^\top \X\1 \leq Cp$
by \eqref{eq:Xnorm}.
Then by the LSI for $(\btheta,\beps)$ under both $g_{\eta}$ and $g_{\eta}'$,
for any $s>0$, $T(\y)$ has the subgaussian concentration arounds its mean
\begin{align*}
\max\Big(\P_{g_{\eta}}[T(\y)-\E_{g_{\eta}} T(\y) \leq -s\Delta
p],\;\P_{g_{\eta}'}[T(\y)-\E_{g_{\eta}'} T(\y) \geq s\Delta p]\Big) \leq
e^{-cs^2\Delta^2 p}.
\end{align*}
Choosing a constant $s>0$ small enough and
combining with (\ref{eq:meansep1}), this shows that (\ref{eq:testingerrors}) holds with $\delta_n=0$.\\

{\bf Case 2:} $|\kappa_2-\kappa_2'| \geq c_2\max(\Delta,\Delta^2)$
for a constant $c_2>0$ to be specified. Suppose without loss of generality
$\kappa_2>\kappa_2'$, let $\Pi$ denote the projection in $\R^n$
orthogonal to $\X\1$,
and consider the test statistic
\begin{align*}
T(\y)=\|\X^\top \Pi\y\|^2=\|\X^\top\Pi(\X\btheta+\beps)\|^2
=\begin{pmatrix} \btheta \\ \beps \end{pmatrix}^\top
\underbrace{\begin{pmatrix} \X^\top \Pi\X\X^\top \Pi\X & \X^\top \Pi\X\X^\top \Pi \\
\Pi\X\X^\top \Pi \X & \Pi\X\X^\top\Pi\end{pmatrix}}_{:=M(\X,\Pi)}
\begin{pmatrix} \btheta \\
\beps \end{pmatrix}
\end{align*}
Since $\|\X\|_{\mathrm{F}}^2 \geq p\sigma^2/C_*$ by Assumption
\ref{assump:design}, this and the conditions of \eqref{eq:Xnorm} imply that
the matrix $\X^\top \Pi\X$ has at
least $cp$ eigenvalues greater than $c$, for some sufficiently small constant
$c>0$. Then, irrespective of the means $\kappa_1,\kappa_1'$, we have
\begin{align*}
\E_{G_{\eta}}[T(\y)]-\E_{G_{\eta}'}[T(\y)]=\Tr \X^\top \Pi\X\X^\top\Pi\X \cdot
(\kappa_2-\kappa_2') \geq c\max(\Delta,\Delta^2)p.
\end{align*}
We have also $\|M(\X,\Pi)\|_\op \leq C$ and $\|M(\X,\Pi)\|_{\mathrm{F}}^2
\leq Cp$ for the above matrix $M(\X,\Pi)$, by \eqref{eq:Xnorm}
and the fact that each block of $M(\X,\Pi)$ has rank at most $p$.
Then by the Hanson-Wright inequality \cite[Theorem 1.1]{rudelson2013hanson},
viewing $T(\y)$ as a quadratic form in $(\btheta,\beps)$ where $\btheta$ is subgaussian under both $g_{\eta}$ and $g_{\eta}'$, for any $s>0$,
\begin{align*}
\P_{g_{\eta}}[T(\y)-\E_{g_\eta} T(\y) &\leq -s\max(\Delta,\Delta^2)p] \leq
e^{-c\min(s,s^2)\Delta^2p}\\
\P_{g_{\eta}'}[T(\y)-\E_{g_\eta'} T(\y) &\geq s\max(\Delta,\Delta^2)p] \leq
e^{-c\min(s,s^2)\Delta^2p}.
\end{align*}
Then (\ref{eq:testingerrors}) holds also in this case with $\delta_n=0$.\\

{\bf Case 3:} $|\kappa_1-\kappa_1'|<c_1\Delta$ and $|\kappa_2-\kappa_2'|<c_2\max(\Delta,\Delta^2)$. Let $\cS \subseteq \{1,\ldots,p\}$ and $\{\z_j:j \in \cS\}$ be as in Assumption \ref{assump:design}. For each $j \in \cS$, consider
\begin{align*}
s_j=\z_j^\top \y=\underbrace{\z_j^\top \x_j \cdot \theta_j}_{\equiv a_j}
+\underbrace{\sum_{k:k \neq j} \z_j^\top \x_k \cdot \theta_k}_{\equiv b_j}
+\underbrace{\z_j^\top \beps}_{\equiv e_j},
\end{align*}
where these three terms $a_j,b_j,e_j$ are independent. Let
$\mathcal{L}(f(\btheta))$ denote the law of $f(\btheta)$ when $\theta_j
\overset{iid}{\sim} g_{\eta}$. Then
by Assumption \ref{assump:design}(b),
\begin{align*}
W_1(\mathcal{L}(a_j),g_\eta) = W_1(\mathcal{L}(a_j),\mathcal{L}(\theta_j)) \leq |\z_j^\top \x_j-1| \cdot \E_{g_{\eta}}[|\theta_j|] \to 0.
\end{align*}
For $b_j$, we have
\begin{align*}
\E_{g_{\eta}}[b_j]=m_j\kappa_1, \qquad \Var_{g_\eta}[b_j]=v_j\kappa_2, \qquad
\text{where } m_j=\sum_{k:k \neq j} \z_j^\top \x_k, \qquad v_j=\sum_{k:k \neq j} (\z_j^\top \x_k)^2.
\end{align*}
The Wasserstein-$(2+\beta)$ CLT of \cite[Eq.\ (4.1)]{rio2009upper} shows that for some $c = c(\beta) > 0$, 
\begin{align*}
W_{2+\beta}\Big(\mathcal{L}(b_j),\N(m_j\kappa_1,v_j\kappa_2)\Big)^{2+\beta} \leq
c \sum_{k:k \neq j} |\z_j^\top \x_k|^{2+\beta} \cdot
\E_{g_\eta}[|\theta_k|^{2+\beta}].
\end{align*}
Assumption \ref{assump:design}(b) provides the Lyapunov condition ensuring that the
right side converges to 0 as $n,p \to \infty$, uniformly over $j \in \cS$.
Then also
\begin{align*}
W_1\Big(\mathcal{L}(b_j),\N(m_j\kappa_1,v_j\kappa_2)\Big) \leq W_{2+\beta}\Big(\mathcal{L}(b_j),\N(m_j\kappa_1,v_j\kappa_2)\Big) \to 0.
\end{align*}
Finally, $e_j \sim \N(0,w_j)$ with variance $w_j=\sigma^2\|\z_j\|_2^2$. Putting this together,
\begin{equation}\label{eq:Lsapprox}
W_1\Big(\mathcal{L}(s_j),\,g_{\eta}*\N(m_j\kappa_1,v_j\kappa_2+w_j)\Big) \leq \delta_n/4
\end{equation}
for some $\delta_n \to 0$. Similarly, writing $\mathcal{L}'(\cdot)$ for the law under $g_{\eta}'$,
\begin{equation}\label{eq:LprimeW1bound}
W_1\Big(\mathcal{L}'(s_j),\,g_{\eta}'*\N(m_j\kappa_1',v_j\kappa_2'+w_j)\Big) \leq \delta_n/4.
\end{equation}

We use $|\kappa_1-\kappa_1'|<c_1\Delta$ and
$|\kappa_2-\kappa_2'|<c_2\max(\Delta,\Delta^2)$ for some sufficiently small
$c_1,c_2>0$ to compare
these normal distributions. Note that for any two univariate normal laws,
by the coupling $(\mu+\sigma Z,\mu'+\sigma'Z)$ for $Z \sim \N(0,1)$,
\begin{align*}
W_1(\N(\mu,\sigma^2),\N(\mu',{\sigma'}^2))
&\leq |\mu-\mu'|+|\sigma-\sigma'|\cdot \E[|Z|]\\
&\leq |\mu-\mu'|+C\min\Big(|\sigma^2-{\sigma'}^2|,
|\sigma^2-{\sigma'}^2|^{1/2}\Big),
\end{align*}
the second inequality holding when $\sigma^2,{\sigma'}^2$ are lower bounded by a
constant. By the definitions of $m_j,v_j$, under Assumption
\ref{assump:design}(b) and \eqref{eq:Xnorm}, we have $|m_j|=|\z_j^\top \X \1-\z_j^\top \x_j| \leq C$, $v_j \leq \z_j^\top \X \X^\top \z_j \leq C$, and $w_j=\sigma^2 \|\z_j\|_2^2 \in [c,C]$ for some constants $C,c>0$. Then for a sufficiently small choice of constants $c_1,c_2>0$, we have 
\begin{align*}
W_1(\N(m_j\kappa_1,v_j\kappa_2+w_j),\,\N(m_j\kappa_1',v_j\kappa_2'+w_j))<\Delta/2.
\end{align*}
Combining with (\ref{eq:LprimeW1bound}), this implies
\begin{equation}\label{eq:Lspapprox}
W_1\Big(\mathcal{L}'(s_j),\,g_{\eta}'*\N(m_j\kappa_1,v_j\kappa_2+w_j)\Big) \leq \delta_n/4+\Delta/2.
\end{equation}

Then by (\ref{eq:Lsapprox}), (\ref{eq:Lspapprox}), and the triangle inequality and translation invariance for $W_1$,
\begin{align*}
W_1(\mathcal{L}(s_j),\mathcal{L}'(s_j)) &\geq W_1\Big(g_{\eta}*\N(m_j\kappa_1,v_j\kappa_2+w_j),\,g_{\eta}'*\N(m_j\kappa_1,v_j\kappa_2+w_j)\Big) -\delta_n/2-\Delta/2\\
&=W_1\Big(g_{\sqrt{\eta^2+v_j\kappa_2+w_j}},\,g_{\sqrt{\eta^2+v_j\kappa_2+w_j}}'\Big) -\delta_n/2-\Delta/2.
\end{align*}
Recalling the choice of $v$ in (\ref{eq:vdef}) and applying $\kappa_2 \leq
M^2+\eta^2$, we have $\eta^2 + v_j\kappa_2+w_j \leq v^2$. Then by monotonicity
of the smoothed Wasserstein distance in
the smoothing parameter \cite[Theorem~3]{goldfeld2020gaussian},
$W_1(g_{\sqrt{\eta^2+v_j\kappa_2+w_j}},\,g_{\sqrt{\eta^2+v_j\kappa_2+w_j}}') \geq W_1(g_v,g_v')=\Delta$, so
\begin{align*}
W_1(\mathcal{L}(s_j),\mathcal{L}'(s_j)) \geq \Delta/2-\delta_n/2.
\end{align*}

Finally, this implies by the Kantorovich duality for $W_1$ that there exists for each $j \in \cS$ a smooth 1-Lipschitz test function $f_j:\R \to \R$ such that
$\E_{g_{\eta}}[f_j(s_j)]-\E_{g_{\eta}'}[f_j(s_j)] \geq (\Delta-\delta_n)_+/3$.
We consider the test statistic
\begin{align*}
T(\y)=\sum_{j \in \cS} f_j(s_j)=\sum_{j \in \cS} f_j(\z_j^\top (\X\btheta+\beps)),
\end{align*}
which then satisfies
$\E_{g_{\eta}}[T(\y)]-\E_{g_{\eta}'}[T(\y)] \geq cp(\Delta-\delta_n)_+$
because $|\cS| \geq c_* p$.
Computing the gradient in $(\btheta,\beps)$, we see that
$T(\y)$ is $L$-Lipschitz in $(\btheta,\beps)$ for
\begin{align*}
L^2=\biggpnorm{\X^\top \sum_{j \in \cS} \z_j \cdot f_j'(s_j)}{}^2 + \biggpnorm{\sum_{j \in \cS} \z_j \cdot f_j'(s_j)}{}^2.
\end{align*}
Applying $|f_j'(s_j)| \leq 1$ for all $j \in \cS$, $\|\X\|_{\op}=1$,
and the final condition of Assumption \ref{assump:design}(b),
this gives $L^2 \leq Cp$.
Then (\ref{eq:testingerrors}) holds by the subgaussian concentration of $T(\y)$
implied by the LSI for $(\btheta,\beps)$, similar to Case 1.
\end{proof}

A direct consequence of Lemma \ref{lemma:TVlower} is the following pointwise comparison of the likelihood objective.
\begin{corollary}\label{cor:pointwiseprob}
Let $v>0$ and the sequence $\delta_n \downarrow 0$ be as in Lemma
\ref{lemma:TVlower}. Fix any constant $\delta > 0$. Then there exist constants $c = c(\delta)$ and $\tau = \tau(\delta)$ such that for all large $n,p$ and any $g \in \cP(M)$ with $d_v(g,g_*) \geq \delta$,
\begin{align}\label{eq:logPconsistency}
\P_{g_*}\left[\frac{1}{p}\Big(\log P_g(\y)-\log P_{g_*}(\y)\Big) \geq
-\tau\right] \leq e^{-cp}.
\end{align}
\end{corollary}
\begin{proof}
Writing as shorthand $P_g=P_g(\y)$, Lemma \ref{lemma:TVlower} shows that for some constant $c >0$, all large $n,p$, and any
$g \in \cP(M)$ with $d_v(g,g_*) \geq \delta$, we have
\begin{align*}
e^{-cp\delta^2} &\geq 1-\dTV(P_{g_*},P_g)
=\int \min(P_{g_*},P_g)\d\y\\
&\geq \int e^{-p\tau}P_{g_*}\1\{P_g \geq e^{-p\tau} P_{g_*}\}\d\y
=e^{-p\tau}\P_{g_*}\left[p^{-1}\left(\log P_g(\y)-\log P_{g_*}(\y)\right)\geq
{-}\tau\right]
\end{align*}
for any $\tau > 0$. The corollary follows upon choosing $\tau = c\delta^2/2$ and adjusting constants.
\end{proof}

\subsection{Covering net}

We now define the following pseudo-metric
$\tilde{d}_{\eta,B}$ over $\cP(M)$: Let $G_\eta(\cdot),G_\eta'(\cdot)$
denote the CDFs of $g_\eta=\N_\eta*g$ and $g_\eta'=\N_\eta*g'$. (Here $G_\eta'$
is not to be confused with a derivative of $G_\eta$.) We set 
\begin{align*}
\tilde{d}_{\eta,B}(g,g')=\max_{x:|x| \leq B} |G_\eta(x)-G_\eta'(x)|.
\end{align*}

\begin{lemma}\label{lemma:continuity}
Suppose Assumption \ref{assump:design} holds, and
fix any constants $C_0,\iota>0$. Let $\Pi_X$ be the projection onto the column
span of $\X$. Then there exist constants
$\eta,B,L>0$ such that, on the event where $\|\Pi_X\y\|_2^2 \leq
C_0p\,\|\X\|_\op^2$, for any $g,g' \in \cP(M)$ we have
\begin{align*}
\frac{1}{p}\Big|\log P_g(\y)-\log P_{g'}(\y)\Big| \leq L\,\tilde{d}_{\eta,B}(g,g')+\iota.
\end{align*}
\end{lemma}
\begin{proof}
We note that $\log P_g(\y)-\log P_{g'}(\y)$
and the event $\|\Pi_X\y\|_2^2 \leq C_0p\,\|\X\|_\op^2$ are invariant
under the simultaneous rescaling $(\X,\y,\sigma) \mapsto (\alpha
\X,\alpha\y,\alpha\sigma)$ for any $\alpha>0$, so we may again assume
without loss of generality the conditions \eqref{eq:Xnorm}.

Let $[\V_X \mid \W_X] \in \R^{n \times n}$ be an orthogonal matrix where
$\Pi_X=\V_X\V_X^\top$ and the columns of $\V_X$ form an orthonormal basis for
the column span of $\X$. (If the columns of $\X$ span all of $\R^n$,
then $\W_X$ is empty, and we may take $\V_X=\Id$.)
Write $\tilde \y=\V_X^\top \y$ and $\tilde \X=\V_X^\top \X$.
Then $\tilde\y=\tilde\X\btheta+\V_X^\top\beps$ is independent of
$\W_X^\top\y=\W_X^\top\beps$, and
the law of the latter does not depend on the prior $g$. Thus
\[\log P_g(\y)=\log P_g(\tilde\y)+\text{constant}\]
for a constant independent of $g$,
where $P_g(\tilde\y)$ denotes the marginal density of $\tilde\y$ under $g$.
Observe that under the normalization \eqref{eq:Xnorm},
\begin{equation}\label{eq:Xtildenorm}
\|\tilde\X\|_{\op}=\|\X\|_{\op}=1.
\end{equation}
Then, choosing $\eta=\sigma^2/2$, we may reparametrize the law of
$\tilde\y$ equivalently as
\[\tilde\y=\tilde\X\bvarphi+\tilde\beps, \qquad \varphi_j \overset{iid}{\sim}
g_\eta, \qquad \tilde \beps \sim \N(0,\bSigma),\]
where $\bSigma=\sigma^2 \Id-\eta\,\tilde\X\tilde\X^\top$ satisfies
$(\sigma^2/2)\Id \preceq \bSigma \preceq \sigma^2 \Id$.
The marginal log-likelihood $\log P_g(\y)$ then has the equivalent form
\begin{equation}\label{eq:logpG}
\log P_g(\y)=\log \int 
\exp\left(-\frac{1}{2}(\tilde\y-\tilde \X\bvarphi)^\top \bSigma^{-1}
(\tilde\y-\tilde\X\bvarphi)\right)\prod_{j=1}^p g_\eta(\varphi_j)\d\bvarphi
+\text{constant}
\end{equation}
for a constant independent of $g$.

We now define a change-of-measure from $G_\eta$ to $G_\eta'$ by introducing
$F_\eta(x)={G_\eta'}^{-1}(G_\eta(x))$ where ${G_\eta'}^{-1}(\cdot)$ is the
functional inverse of the CDF $G_\eta'(\cdot)$. (This is the univariate optimal
transport map from $G_\eta$ to $G_\eta'$.) Let us write $F_\eta(\bvarphi)$ for
the coordinate-wise application of $F_\eta$ to $\bvarphi$. Then the marginal
log-likelihood $\log P_{g'}(\y)$ may also be written as
\begin{equation}\label{eq:logpGprime}
\log P_{g'}(\y)=\log
\int \exp\left(-\frac{1}{2}(\tilde\y-\tilde\X
F_\eta(\bvarphi))^\top \bSigma^{-1} (\tilde\y-\tilde\X F_\eta(\bvarphi))\right)
\prod_{j=1}^p g_\eta(\varphi_j)\d\bvarphi+\text{constant}.
\end{equation}
Taking the difference of (\ref{eq:logpG}) and (\ref{eq:logpGprime}),
\begin{align*}
\log P_{g'}(\y)-\log P_g(\y)
=\log \left\langle e^{-\frac{1}{2}(\tilde\y-\tilde\X F_\eta(\bvarphi))^\top \bSigma^{-1}
(\tilde\y-\tilde\X F_\eta(\bvarphi))+\frac{1}{2}(\tilde\y-\tilde\X\bvarphi)^\top
\bSigma^{-1}(\tilde\y-\tilde\X\bvarphi)}
\right\rangle,
\end{align*}
where we introduce the shorthand
\[\langle f(\bvarphi) \rangle=\frac{\int f(\bvarphi)\,
e^{-H(\bvarphi)}\d G_\eta(\bvarphi)}{\int e^{-H(\bvarphi)}\d G_\eta(\bvarphi)},
\quad H(\bvarphi)= \frac{1}{2}(\tilde\y-\tilde\X\bvarphi)^\top
\bSigma^{-1}(\tilde\y-\tilde\X\bvarphi), \quad
\d G_\eta(\bvarphi)=\prod_{j=1}^p g_\eta(\varphi_j)\d\bvarphi\]
for the expectation with respect to the posterior law of $\bvarphi$ under
$g_\eta$ given $(\tilde\X,\tilde\y)$.

Let $C,c>0$ denote constants changing from instance to instance, which may depend on $\eta$. By Jensen's inequality,
\begin{align*}
\log P_{g'}(\y)-\log P_g(\y) &\geq \frac{1}{2}\left\langle -(\tilde\y-\tilde\X
F_\eta(\bvarphi))^\top \bSigma^{-1}
(\tilde\y-\tilde\X F_\eta(\bvarphi))+(\tilde\y-\tilde\X\bvarphi)^\top
\bSigma^{-1}(\tilde\y-\tilde\X\bvarphi)
\right\rangle.
\end{align*}
Setting $\u=\tilde\y-\tilde\X\bvarphi$ and $\v=\tilde\y-\tilde\X F_\eta(\bvarphi)$, we apply
\begin{align*}
\langle {-}\v^\top \bSigma^{-1}\v+\u^\top \bSigma^{-1}\u \rangle
&=\big\langle {-}(\u-\v)^\top \bSigma^{-1}(\u-\v)+
2(\u-\v)^\top \bSigma^{-1}\u \big\rangle\\
&\geq -\langle (\u-\v)^\top \bSigma^{-1}(\u-\v) \rangle
-2\langle (\u-\v)^\top \bSigma^{-1}(\u-\v) \rangle^{1/2}
\langle \u^\top \bSigma^{-1} \u \rangle^{1/2}.
\end{align*}
Together with the conditions $\bSigma^{-1} \preceq (2/\sigma^2)\Id$,
\eqref{eq:Xnorm}, and \eqref{eq:Xtildenorm}, this gives
\begin{equation}\label{eq:logpGdifflower}
\log P_{g'}(\y)-\log P_g(\y) \geq -C\left(\big\langle
\|F_\eta(\bvarphi)-\bvarphi\|_2^2 \big\rangle +\big\langle
\|F_\eta(\bvarphi)-\bvarphi\|_2^2 \big\rangle^{1/2}
\big\langle H(\bvarphi)\big\rangle^{1/2}\right).
\end{equation}

To lower bound (\ref{eq:logpGdifflower}), we may first bound $\langle H(\bvarphi)
\rangle$ as follows: Introduce the partition function
$Z=\int e^{-H(\bvarphi)}\d G_\eta(\bvarphi)$.
Applying $H(\bvarphi) \geq 0$, $e^{-H(\bvarphi)}/Z \leq 1$ on the event
$e^{-H(\bvarphi)} \leq Z$, and $H(\bvarphi)<-\log Z$ on the complementary event $e^{-H(\bvarphi)}>Z$, we have
\[\langle H(\bvarphi)\rangle
=\int H(\bvarphi)\frac{e^{-H(\bvarphi)}}{Z}\d G_\eta(\bvarphi)\\
\leq \int H(\bvarphi)\d G_\eta(\bvarphi)+\max(0,-\log Z)
\leq 2\int H(\bvarphi)\d G_\eta(\bvarphi),\]
the last step applying the Jensen's inequality lower bound $\log Z \geq -\int
H(\bvarphi)\d G_\eta(\bvarphi)$. Thus we have bounded the mean of $H(\bvarphi)$
under the posterior by its mean under the prior.
On the event $\|\tilde\y\|_2^2=\|\Pi_X\y\|_2^2 \leq C_0p\|\X\|_\op^2$,
applying
$\bSigma^{-1} \preceq (2/\sigma^2)\Id$, \eqref{eq:Xnorm},
\eqref{eq:Xtildenorm},
and that $g_\eta$ is subgaussian (since $g$ has bounded support),
we have $\int H(\bvarphi)\d G_\eta(\bvarphi) \leq Cp$. This implies
\begin{equation}\label{eq:EH}
\langle H(\bvarphi) \rangle \leq Cp.
\end{equation}

To bound $\langle \|F_\eta(\bvarphi)-\bvarphi\|_2^2 \rangle$, let us fix a small constant $t>0$ to be determined, pick $B=B(t)>0$ sufficiently large, and define
$N_B(\bvarphi)=|\{j:|\varphi_j|>B\}|$. Applying again $\log Z \geq -\int
H(\bvarphi)\d G_\eta(\bvarphi) \geq -Cp$ and $H(\bvarphi) \geq 0$, we have
\begin{equation}\label{eq:posteriorprioravg}
\langle \1\{N_B(\bvarphi)>tp\} \rangle =\int
\1\{N_B(\bvarphi)>tp\}\frac{e^{-H(\bvarphi)}}{Z} \d G_\eta(\bvarphi) \leq
e^{Cp} \int \1\{N_B(\bvarphi)>tp\} \d G_\eta(\bvarphi).
\end{equation}
Under the product prior law $\prod_{j=1}^p g_\eta(\varphi_j)$, this quantity
$N_B(\bvarphi)$ is a binomial random variable whose success probability is at
most $q = 2e^{-cB^2}$ for a constant $c:=c(M,\eta)$, by subgaussianity of $g_\eta$. For
any fixed constant $t > 0$, by choosing $B = B(t) > 0$ large enough, we have $q \leq t/2$, so by the Chernoff bound for the binomial distribution \cite[pp. 24]{boucheron2013concentration}
\begin{align*}
\int \1\{N_B(\bvarphi)>tp\} \d G_\eta(\bvarphi) \leq \Prob_{g_\eta}\Big(\text{Bin}(p,q) - pq \geq \frac{t}{2}p\Big) \leq e^{-p h_q(q + \frac{t}{2})},
\end{align*}
where $h_q(a) = \kl(\text{Ber}(a) || \text{Ber}(q))$. Since $h_q(q+t/2)
\rightarrow \infty$ as $q \downarrow 0$, we may then choose $B = B(t)$
large enough such that $\int \1\{N_B(\bvarphi)>tp\} \d G_\eta(\bvarphi) \leq
e^{-2Cp}$. Consequently,
$\langle \1\{N_B(\bvarphi)>tp\} \rangle \leq e^{Cp} \cdot e^{-2Cp} \leq
e^{-Cp}$, and hence
\begin{equation}\label{eq:NBbound}
\langle N_B(\bvarphi) \rangle \leq
tp+p \cdot \langle \1\{N_B(\bvarphi)>tp\} \rangle \leq 2tp.
\end{equation}
Now note that for all $g' \in \cP$, the smoothed distributions $g_\eta'$ have densities lower-bounded by a $\eta$-dependent constant on any fixed compact interval, and hence the inverse CDFs ${G_\eta'}^{-1}(\cdot)$ are uniformly Lipschitz-continuous over any fixed compact sub-interval of $(0,1)$. Furthermore, for the above choice of $B=B(t)>0$, the interval $[G_\eta(-B),G_\eta(B)]$ is contained in a common compact sub-interval of $(0,1)$ across all $g \in \cP(M)$.
Then there is a constant $C_1=C_1(t)>0$ such that for all $g,g' \in \cP$, ${G_\eta'}^{-1}$ is $C_1$-Lipschitz over the interval $[G_\eta(-B),G_\eta(B)]$. This shows that if $x\in[-B,B]$, 
\begin{equation}\label{eq:Fbound1}
|F_\eta(x)-x|=|{G_\eta'}^{-1}(G_\eta(x))-{G_\eta'}^{-1}(G_\eta'(x))| \leq C_1|G_\eta(x)-G_\eta'(x)| \leq C_1 \tilde{d}_{\eta,B}(g,g').
\end{equation}
For any $g,g' \in \cP(M)$,
note that if $\theta \sim g$, $\theta' \sim g'$, and $\eps \sim \N(0,\eta^2)$ are independent, then also
$\P[\theta'+\eps \leq x-2M] \leq \P[\theta+\eps \leq x] \leq \P[\theta'+\eps \leq x+2M]$
for every $x \in \R$ because $|\theta-\theta'| \leq 2M$. Then $G_\eta'(x-2M) \leq G_\eta(x) \leq G_\eta'(x+2M)$, implying that
\begin{equation}\label{eq:Fbound2}
|F_\eta(x)-x|=|{G_\eta'}^{-1}(G_\eta(x))-x| \leq 2M
\end{equation}
for all $x \in \R$. Combining these bounds (\ref{eq:Fbound1}) and
(\ref{eq:Fbound2}) and recalling $N_B(\bvarphi)=|\{j:|\varphi_j|>B\}|$ which satisfies (\ref{eq:NBbound}),
\begin{equation}\label{eq:Etransport}
\Big\langle \|F_\eta(\bvarphi)-\bvarphi\|_2^2 \Big \rangle
\leq \Big\langle C_1^2\tilde{d}_{\eta,B}(g,g')^2 \cdot p+(2M)^2 \cdot N_B(\bvarphi)
\Big \rangle \leq (C_1^2\tilde{d}_{\eta,B}(g,g')^2+8M^2t)p.
\end{equation}

We may apply (\ref{eq:EH}) and (\ref{eq:Etransport}) back to (\ref{eq:logpGdifflower}), now choosing the constant $t>0$ sufficiently small based on the value of $\iota$ given in the lemma, and then $B = B(t)$, $C_1=C_1(t)$, and $L=L(t)$ sufficiently large, to obtain
\begin{align*}
\log P_{g'}(\y)-\log P_g(\y) \geq -p\Big(L \cdot \max(\tilde{d}_{\eta,B}(g,g')^2,\tilde{d}_{\eta,B}(g,g')) +\iota\Big).
\end{align*}
This is simplified by observing that $\tilde{d}_{\eta,B}(G,G') \leq 1$ by definition. Then, applying the same lower bound with $g,g'$ interchanged, we obtain as desired $|\log P_{g'}(\y)-\log P_g(\y)|
\leq p(L \cdot \tilde{d}_{\eta,B}(g,g')+\iota)$.
\end{proof}

\subsection{Completing the proof}
\begin{corollary}\label{cor:uniformprob}
Suppose Assumption \ref{assump:design} holds.
Fix any $\delta>0$, and let $v>0$ be the constant of Lemma
\ref{lemma:TVlower}. Then there exist constants $\tau>0$ and $C,c>0$ such that
\[\P_{g_*}\left[\sup_{g \in \cP(M):d_v(g,g_*) \geq \delta}
\frac{1}{p}\Big(\log P_g(\y)-\log P_{g_*}(\y)\Big)
\geq -\tau \right] \leq Ce^{-cp}.\]
\end{corollary}
\begin{proof}
By Corollary \ref{cor:pointwiseprob}, there exists a constant $\tau>0$
such that \eqref{eq:logPconsistency} holds with $\tau/3$
in place of $\tau$. Fix this $\tau$, and
in Lemma \ref{lemma:continuity} set $\iota=\tau/3$ and set $C_0>0$ large
enough such that $\P_{g_*}[\|\Pi_X\y\|_2^2>C_0p\,\|\X\|_\op^2] \leq
e^{-cp}$. (This is
possible since $\btheta \in \R^p$ has subgaussian
coordinates under $g_*$, and $\|\Pi_X\beps\|_2^2 \sim \sigma^2 \cdot \chi_k^2$
where $\sigma^2 \leq C_*\|\X\|_\op^2$ and
$k=\rank(\X) \leq p$.) Let $\eta,B,L$ be as in Lemma \ref{lemma:continuity}.
Take a
$(\tau/3L)$-covering net of $\{g \in \cP(M):d_v(g,g_*) \geq \delta\}$ in the
pseudo-metric $\tilde{d}_{\eta,B}(\cdot,\cdot)$, which by a standard packing
argument has finite cardinality $C(\eta,B,L,\tau)$ because all CDFs
$G_\eta$ for $g \in \cP(M)$ have a constant lower bound for their derivatives
over $[-B,B]$. The result then follows from applying
Lemma \ref{lemma:continuity}, combined with the
conclusion \eqref{eq:logPconsistency} of Corollary \ref{cor:pointwiseprob}
(with $\tau/3$ in place of $\tau$) applied to each point in this net.
\end{proof}

\begin{proof}[Proof of Theorem \ref{thm:consistency}]
The above corollary implies that for any fixed
constant $\delta>0$, the distribution
$\widehat{g}$ in Theorem \ref{thm:consistency} satisfies
$d_v(\widehat{g},g_*)<\delta$ with probability approaching 1 as $n,p \to
\infty$. Since the smoothed Wasserstein distance $d_v(\cdot,\cdot)$ metrizes
weak convergence, Theorem \ref{thm:consistency} follows.
\end{proof}

To show Corollary \ref{cor:consistencyposterior}, we record here a
simple bound on density ratios of the smoothed priors
$\{\N_\tau * g:g \in \cP(M)\}$.

\begin{proposition}\label{prop:Nratiobound}
For any probability distributions $g,h$ on $[-M,M]$,
\[\frac{\N_\tau*g(\varphi)}{\N_\tau*h(\varphi)} \leq
e^{(M^2/2\tau^2)}e^{(2M/\tau^2)|\varphi|}\]
\end{proposition}
\begin{proof}
Observe that for any $\theta,\theta' \in [-M,M]$,
\[\frac{\N_\tau(\theta-\varphi)}{\N_\tau(\theta'-\varphi)}
=\frac{e^{-\frac{1}{2\tau^2}(\theta-\varphi)^2}}
{e^{-\frac{1}{2\tau^2}(\theta'-\varphi)^2}}
\in \left[e^{-\frac{M^2}{2\tau^2}-\frac{2M}{\tau^2}|\varphi|},\;
e^{\frac{M^2}{2\tau^2}+\frac{2M}{\tau^2}|\varphi|}\right].\]
Then also
\[\frac{\N_\tau(\theta-\varphi)}{\N_\tau*h(\varphi)}
=\left(\int_{-M}^M \frac{\N_\tau(\theta'-\varphi)}{\N_\tau(\theta-\varphi)}
\d h(\theta')\right)^{-1} \leq
e^{\frac{M^2}{2\tau^2}+\frac{2M}{\tau^2}|\varphi|},\]
and the result follows from further integrating this with respect to $\d g(\theta)$.
\end{proof}

\begin{proof}[Proof of Corollary \ref{cor:consistencyposterior}]
Observe that
\[\frac{1}{p}\,\DKL(P_{\hat g_n}(\bvarphi \mid \y)\|P_{g_*}(\bvarphi \mid \y))
=\frac{1}{p}\int P_{\hat g_n}(\bvarphi \mid \y)
\log \frac{P_{\hat g_n}(\bvarphi \mid \y)}{P_{g_*}(\bvarphi \mid \y)}\d\bvarphi.\]
Writing $P_g(\bvarphi \mid \y)=P(\y \mid \bvarphi)P_g(\bvarphi)/P_g(\y)$
where $P_g(\bvarphi)=\prod_{j=1}^p (\N_\tau*g)(\varphi_j)$ and
$P(\y \mid \bvarphi)$ does not depend on the prior, this gives
\begin{align*}
\frac{1}{p}\,\DKL(P_{\hat g_n}(\bvarphi \mid \y)\|P_{g_*}(\bvarphi \mid \y))
&=\frac{1}{p}\Big(\log P_{g_*}(\y)-\log P_{\hat g_n}(\y)\Big)
+\int P_{\hat g_n}(\bvarphi \mid \y)
\left(\frac{1}{p}\sum_{j=1}^p \log \frac{\N_\tau*\hat g_n(\varphi_j)}
{\N_\tau*g_*(\varphi_j)}\right)\d\bvarphi\\
&=\Big(\bar F_n(\hat g_n)-\bar F_n(g_*)\Big)
+\int P_{\hat g_n}(\bvarphi \mid \y)
\left(\frac{1}{p}\sum_{j=1}^p \log \frac{\N_\tau*\hat g_n(\varphi_j)}
{\N_\tau*g_*(\varphi_j)}\right)\d\bvarphi.
\end{align*}
By assumption, the first term satisfies
$\limsup_{n \to \infty} \bar F_n(\hat g_n)-\bar F_n(g_*) \leq 0$ in probability.

It remains to bound the second term.
Write $\langle f(\bvarphi)\rangle$ for the posterior average under
$P_{\hat g_n}(\bvarphi \mid \y)$, so this second term is
\[\int P_{\hat g_n}(\bvarphi \mid \y)
\left(\frac{1}{p}\sum_{j=1}^p \log \frac{\N_\tau*\hat g_n(\varphi_j)}
{\N_\tau*g_*(\varphi_j)}\right)\d\bvarphi
=\left\langle \frac{1}{p}\sum_{j=1}^p \log \frac{\N_\tau*\hat g_n(\varphi_j)}
{\N_\tau*g_*(\varphi_j)} \right\rangle.\]
Let $\eta=\tau/\sqrt{2}$, so that
$\N_\tau*g(\varphi)=\N_\eta*\N_\eta*g(\varphi)$, and write
\begin{equation}\label{eq:phipsidecomp}
\N_\tau*g(\varphi)=
\int \frac{1}{\sqrt{2\pi\eta^2}}
e^{-\frac{(\varphi-\psi)^2}{2\eta^2}}\N_\eta*g(\psi)\d \psi
\end{equation}
Note that $W_1(\N_\eta*\hat g_n,\,\N_\eta*g_*) \to 0$ in
probability as $n,p \to \infty$, by the weak convergence $\hat g_n \to g_*$
established in Theorem \ref{thm:consistency}. Furthermore,
the functions $\psi \mapsto e^{-\frac{(\varphi-\psi)^2}{2\eta^2}}$
in \eqref{eq:phipsidecomp}
are uniformly Lipschitz over $\varphi \in \R$. Then this $W_1$-convergence
and \eqref{eq:phipsidecomp} imply
\[\sup_{\varphi \in \R} \big|\N_\tau*\hat g_n(\varphi)-\N_\tau*g_*(\varphi)\big|
\to 0\]
in probability. Let $B>0$ be any large constant.
Then $\N_\tau * g_*(\varphi)$ is bounded
away from 0 over $\varphi \in [-B,B]$, so this implies
\[\sup_{\varphi \in [-B,B]} \big|\log \N_\tau*\hat g_n(\varphi)-\log
\N_\tau*g_*(\varphi)\big| \to 0\]
in probability. Thus
\begin{equation}\label{eq:goodindexbound}
\left\langle
\frac{1}{p}\sum_{j=1}^p \log \frac{\N_\tau*\hat g_n(\varphi_j)}
{\N_\tau*g_*(\varphi_j)}\1\{\varphi_j \in [-B,B]\}\right\rangle \to 0
\end{equation}
in probability.

For the coordinates $\varphi_j \notin [-B,B]$, observe by
Proposition \ref{prop:Nratiobound} that
\begin{equation}\label{eq:badindexbound}
\frac{1}{p}\sum_{j=1}^p \log \frac{\N_\tau*\hat g_n(\varphi_j)}
{\N_\tau*g_*(\varphi_j)}\1\{\varphi_j \notin [-B,B]\}
\leq \frac{1}{p}\sum_{j=1}^p f(\varphi_j)
\end{equation}
where we set
\[f(\varphi)=C_0(1+|\varphi|)\1\{\varphi \notin [-B,B]\},
\qquad C_0=\max\left(\frac{M^2}{2\tau^2},\frac{2M}{\tau^2}\right).\]
By the same argument as in \eqref{eq:posteriorprioravg}, there exists a constant
$C>0$ such that for any $t>0$,
\begin{equation}\label{eq:loglikposteriorestimate}
\left\langle \1\left\{\frac{1}{p}\sum_{j=1}^p f(\varphi_j)>t\right\}
\right\rangle
\leq e^{Cp}\int \1\left\{\frac{1}{p}\sum_{j=1}^p f(\varphi_j)>t\right\}
\prod_{j=1}^p \N_\tau * \hat g_n(\varphi_j)\d\varphi_j,
\end{equation}
i.e.\ the posterior probability of this event is at most $e^{Cp}$ times the
prior probability. By a Chernoff
bound, for any $\lambda>0$,
\[\int \1\left\{\frac{1}{p}\sum_{j=1}^p f(\varphi_j)>t\right\}
\prod_{j=1}^p \N_\tau * \hat g_n (\varphi_j)\d\varphi_j
\leq e^{-\lambda tp} \left(\int e^{\lambda f(\varphi)}
\N_\tau * \hat g_n (\varphi)\d\varphi\right)^p\]
Noting that $\N_\tau*\hat g_n$ is stochastically dominated above and below by
$\N(M,\tau^2)$ and $\N(-M,\tau^2)$, we may bound
\begin{align*}
\int e^{\lambda f(\varphi)}
\N_\tau * \hat g_n (\varphi)\d\varphi
&\leq 1+\int_{\R \setminus [-B,B]}
e^{C_0\lambda(1+|\varphi|)}\N_\tau * \hat g_n (\varphi)\d\varphi\\
&\leq 1+2\int_B^\infty e^{C_0\lambda(1+\varphi)} \frac{1}{\sqrt{2\pi\tau^2}}
e^{-\frac{(\varphi-M)^2}{2\tau^2}}\d \varphi\\
&=1+2e^{C_0\lambda+C_0\lambda M+\frac{1}{2}{C_0}^2\lambda^2\tau^2}
\int_B^\infty \frac{1}{\sqrt{2\pi\tau^2}}
e^{-\frac{(\varphi-M-C_0\lambda\tau^2)^2}{2\tau^2}}\d \varphi
\end{align*}
Choosing $\lambda=cB$ for a small enough constant $c>0$ and applying a standard
Gaussian tail bound shows
that this is at most $1+e^{-c'B^2}$ for all $B \geq B_0$ and some $c'>0$
(where $c,c',B_0$ depend on $(M,\tau^2)$). Thus, applying this in the Chernoff
bound,
\[\int \1\left\{\frac{1}{p}\sum_{j=1}^p f(\varphi_j)>t\right\}
\prod_{j=1}^n \N_\tau * \hat g_n (\varphi_j)\d\varphi_j
\leq e^{-cBtp}(1+e^{-c'B^2})^p \leq e^{-cBtp+p}.\]
Applying this back to \eqref{eq:loglikposteriorestimate}, for any $t>0$
and $B>B_0$,
\[\left\langle \1\left\{\frac{1}{p}\sum_{j=1}^p f(\varphi_j)>t\right\}
\right\rangle \leq e^{(C+1)p} \cdot e^{-cBtp}.\]
Then applying $\int_{t_0}^\infty \P[X>t]\d t
=\E[(X-t_0)\1\{X \geq t_0\}]$ and
integrating this bound from $t_0=C'/B$ to $\infty$ for a large enough
constant $C'>0$, this shows
\[\lim_{n,p \to \infty}
\left\langle \left(\frac{1}{p}\sum_{j=1}^p f(\varphi_j)-\frac{C'}{B}\right)
\1\left\{\frac{1}{p}\sum_{j=1}^p
f(\varphi_j)>\frac{C'}{B}\right\}\right\rangle \leq 0.\]
Hence, by this bound and \eqref{eq:badindexbound},
\[\limsup_{n,p \to \infty} \left\langle
\frac{1}{p}\sum_{j=1}^p \log \frac{\N_\tau*\hat g_n(\varphi_j)}
{\N_\tau*g_*(\varphi_j)}\1\{\varphi_j \notin [-B,B]\}\right\rangle
\leq \limsup_{n,p \to \infty} \left\langle \frac{1}{p}\sum_{j=1}^p f(\varphi_j)
\right\rangle \leq \frac{2C'}{B}.\]
Combining this with \eqref{eq:goodindexbound} and noting that $B>B_0$ is
arbitrary, this shows
\[\limsup_{n,p \to \infty}
\left\langle \frac{1}{p}\sum_{j=1}^p \log \frac{\N_\tau*\hat g_n(\varphi_j)}
{\N_\tau*g_*(\varphi_j)}\right\rangle \leq 0\]
in probability, concluding the proof.
\end{proof}

\subsection{Assumption \ref{assump:design} for random designs}\label{sec:design}

\begin{proof}[Proof of Proposition \ref{prop:condition_gaussian}]
Let $\Pi \in \R^{n \times n}$ be the projection orthogonal to $\X\1 \in
\R^n$.  For any $j\in[p]$, let
\begin{align*}
\z_j=\Pi\X\bSigma_X^{-1}\e_j=\b_j-\r_j,
\qquad
\b_j = \X\bSigma_X^{-1}\e_j, \qquad \r_j =
\frac{\X\bm{1}\bm{1}^\top\X^\top}{\pnorm{\X\bm{1}}{}^2}
\X\bSigma_X^{-1}\e_j.
\end{align*}
Fixing constants $C_0,c_0>0$, define the event
\begin{align*}
E=\Big\{&\pnorm{\X}{\op} \leq C_0,\;\Tr \X^\top\X  \geq c_0 p,\;
\|\X\bm{1}\|^2 \geq c_0p,\;
c_0 \leq \min_{j\in[p]}\pnorm{\b_j}{}^2 \leq \max_{j\in[p]}\pnorm{\b_j}{}^2 \leq C_0,\\
& \max_{j\in[p]}|\b_j^\top \x_j- 1|,\;\max_{j\neq k}|\b_j^\top \x_k|
\leq C_0\sqrt{\frac{\log p}{n}},\;\;
\max_{j \in [p]} \|\r_j\| \leq C_0\Big(\frac{1}{\sqrt{p}}+
\sqrt{\frac{\log p}{n}}\Big)\Big\}
\end{align*}
On this event $E$,
from these conditions and the decomposition $\z_j=\b_j-\r_j$, for a constant
$C_\beta>0$ we have
\[\max_{j \in [p]} \sum_{k:k \neq j} |\z_j^\top\x_k|^{2+\beta}
\leq C_\beta p\Big(\sqrt{\frac{\log p}{n}}+\frac{1}{\sqrt{p}}\Big)^{2+\beta}.\]
This approaches 0 as $n,p \to \infty$ for $n/p \geq \gamma$ and
any fixed $\beta,\gamma>0$. If $|a_j| \leq 1$ for all $j \in [p]$, then
we have also
\begin{align*}
\left\|\sum_{j\in \cS} a_j\z_j\right\|^2
&=\a^\top \bSigma_X^{-1}\X^\top \Pi\X \bSigma_X^{-1} \a
\leq \pnorm{\a}{}^2 \cdot \biggpnorm{\bSigma_X^{-1}\X^\top \Pi\X
\bSigma_X^{-1}}{\op} \leq Cp
\end{align*}
when $\pnorm{\X}{\op} \leq C_0$. Then all conditions of
Assumption \ref{assump:design} hold with $\cS=\{1,\ldots,p\}$
on this event $E$, for
sufficiently large and small constants $C_0>0$ and $c_0>0$.

Next we show that for some sufficiently large $C_0>0$ and small $c_0>0$,
the event $E$ holds with probability at least $1-p^{-10}$ for
all large $n,p$. In the following, $C,C',c,c'>0$ denote constants changing from
instance to instance. We have $\|\sqrt{n}\X\bSigma_X^{-1/2}\|_\op \leq
C(\sqrt{n}+\sqrt{p})$ with probability $1-e^{-cp}$, by
\cite[Theorem 4.6.1]{vershynin2018high}. Then, since $n/p \geq \gamma$
and $\|\bSigma_X\|_\op \leq C$, this implies $\|\X\|_\op \leq C'$.
For the remaining statements, observe
that for any (deterministic) unit vectors $\u,\v \in \R^p$, each of the products
$\X\u,\X\v \in \R^n$ has i.i.d.\ $(C/\sqrt{n})$-subgaussian entries.
Then $\u^\top\X^\top \X\v$ is a sum of
$n$ i.i.d.\ $(C'/n)$-sub-exponential random variables, and Bernstein's
inequality \cite[Theorem 2.8.1]{vershynin2018high} implies
\[\P\Big[\Big|\u^\top\X^\top \X\v-\u^\top\bSigma_X\v\Big| \geq t\Big]
=\P\Big[\Big|\u^\top\X^\top \X\v-\E \u^\top\X^\top \X\v\Big| \geq t\Big]
\leq 2e^{-cn\min(t^2,t)}.\]
Applying this with $\u=\v=\e_j$ and taking a union bound over $j \in [p]$ shows
$\Tr \X^\top\X \geq cp$.
Applying this with $\u=\v=\1/\sqrt{p}$
shows $\|\X\1\|^2 \geq cp$. Applying this with
$\u=\v=\bSigma_X^{-1}\e_j/\|\bSigma_X^{-1}\e_j\|$
shows $c \leq \|\b_j\|^2 \leq C$.
These statements all hold with probability at least $1-e^{-cn} \geq 1-p^{-20}$.

Furthermore, applying the above with
$\u=\bSigma_X^{-1}\e_j/\|\bSigma_X^{-1}\e_j\|$
and $\v=\e_j$ or $\v=\e_k$ shows $|\b_j^\top \x_j-1| \leq
C_0\sqrt{\log p/n}$ and $|\b_j^\top \x_k| \leq C_0\sqrt{\log p/n}$
for any $j \neq k \in [p]$, with probability $1-e^{-cn \cdot C_0^2(\log p)/n}
\geq 1-p^{-20}$ for sufficiently large $C_0>0$. Applying $\|\X\1\|^2 \geq
cp$ already shown in the above, we have $\|\r_j\| \leq (cp)^{-1/2}|\1^\top \X^\top
\X\bSigma_X^{-1}\e_j|$. Then applying the above with $\u=\1/\sqrt{p}$
and $\v=\bSigma_X^{-1}\e_j/\|\bSigma_X^{-1}\e_j\|$ shows
$\|\r_j\| \leq C/\sqrt{p}+C_0\sqrt{\log p/n}$ also with probability
$1-e^{-cn \cdot C_0^2(\log p)/n} \geq 1-p^{-20}$.
Taking a union bound shows that for some constants $C_0,c_0>0$, the event $E$
holds with probability at least $1-p^{-10}$, as claimed.
\end{proof}

\section{Analysis of algorithm dynamics}\label{sec:proof_algo}

\subsection{Fisher-Rao flow for the sequence model} \label{subsec:proof_g_flow}

\begin{proof}[Proof of Theorem \ref{thm:g_flow}]
Assume that $\DKL(h\|g_0)<\infty$ for otherwise there is nothing to prove.
By Corollary \ref{cor:exist_gflow}, $\inf_{t \in [0,T]} \inf_{\theta \in [-M,M]} g_t(\theta)/g_0(\theta)\geq e^{-t}$.
As such, 
$\DKL(h\|g_t)=\DKL(h\|g_0)+\int h \log \frac{g_0}{g_t} \leq \DKL(h\|g_0) + t$
which is finite for all $t$. Together with the weak lower semicontinuity of the KL divergence, this implies 
\begin{equation}
\lim_{t\to0}\DKL(h\|g_t)=\DKL(h\|g_0).
\label{eq:DKLcontinuous}
\end{equation}

Next, write for simplicity $\bar F(g)=\bar F(g \mid \bar q)$, and define its first
variation
\[\delta \bar F[g](\theta)=-\N_\tau*\frac{\bar q}{\N_\tau*g}(\theta)+1.\]
Applying Proposition \ref{prop:Nratiobound},
$\N_\tau*\frac{\bar q}{\N_\tau*g}(\theta)
\leq \int Ce^{C|\varphi|}\d\bar q(\varphi)<\infty$,
so $\delta \bar F[g]$ is uniformly bounded over $\theta \in [-M,M]$.
Let $h$ be any density on $[-M,M]$.
Applying the definition of the flow (\ref{def:g_flow}) and using this
boundedness to differentiate under the integral
by the dominated convergence theorem,
\begin{equation}\label{eq:dominatedconvergence}
\int_{-M}^M {-}\delta \bar{F}[g_t](\theta)
h(\theta)\d\theta = \int_{-M}^M \frac{\d}{\d t}\left(\log
\frac{g_t(\theta)}{h(\theta)}\right)
h(\theta)\d\theta=-\frac{\d}{\d t}\DKL(h\| g_t).
\end{equation}

By convexity of $x \mapsto {-}\log x$, we have
\begin{align*}
\bar{F}(g)-\bar{F}(h)
&=\int \Big({-}\log[\N_\tau*g](\varphi)+\log[\N_\tau*h](\varphi)\Big)
\d\bar q(\varphi)\\
&\leq \int \frac{\d}{\d \eps}\Big[\log[\N_\tau*(g+\eps(h-g))]
(\varphi)\Big]_{\eps=0}\d\bar q(\varphi)\\
&=\int \frac{[\N_\tau*(h-g)](\varphi)}{[\N_\tau*g](\varphi)}
\d\bar q(\varphi)\\
&=\int \frac{\N_\tau(\theta-\varphi)h(\theta)}
{\N_\tau*g(\varphi)}\d\bar q(\varphi)-1
=\int_{-M}^M {-}\delta \bar{F}[g](\theta)h(\theta)\d\theta.
\end{align*}
Combining the above two displays
and integrating from time $t$ to time $T$ yields
\[\int_t^T \big[\bar{F}(g_s) - \bar{F}(h)\big] \d s
\leq \int_t^T {-}\frac{\d}{\d t}\DKL(h\|g_t)\Big|_{t=s}\d s
=\DKL(h\|g_t)-\DKL(h\|g_T).\]
So taking the limit $t \to
0$ above gives
\begin{equation}\label{eq:FatouKL}
\int_0^T \big[\bar{F}(g_s) - \bar{F}(h)\big] \d s \leq \DKL(h\|g_0)
-\DKL(h\|g_T) \leq \DKL(h\|g_T).
\end{equation}
Now observe that
\[\left|\frac{\d}{\d t}\int_{-M}^M \N_\tau(\varphi-\theta)
g_t(\theta)\d\theta\right|
=\left|\int_{-M}^M \N_\tau(\varphi-\theta)g_t(\theta)\delta \bar
F[g_t](\theta)\d\theta\right|
\leq C\max_{\theta \in [-M,M]}\N_\tau(\varphi-\theta)\]
for a constant $C>0$,
using boundedness of $|\delta \bar F[g]|$ as shown above to
differentiate under the integral. Then
\[\left|\frac{\d}{\d t}\log [\N_\tau*g_t](\varphi)\right|
=\left|\frac{\frac{\d}{\d t}\int_{-M}^M \N_\tau(\varphi-\theta)
g_t(\theta)\d\theta}{\N_\tau*g_t(\varphi)}\right|
\leq Ce^{C|\varphi|}\]
by Proposition \ref{prop:Nratiobound}, so again
differentiating under the integral by the dominated convergence theorem
and applying $\int g(\theta)\delta \bar F[g](\theta)\d\theta=0$,
\begin{align}
\frac{\d}{\d t} \bar F(g_t)
&=\int {-}\frac{\d}{\d t}\log[\N_\tau*g_t](\varphi)\d\bar q(\varphi)\nonumber\\
&=\int {-}\frac{\N_\tau(\varphi-\theta)g_t(\theta)\delta \bar F[g_t](\theta)}
{\N_\tau*g_t(\varphi)}\d\bar q(\varphi)
={-}\int \Big(\delta \bar F[g_t](\theta)\Big)^2g_t(\theta)\d\theta \leq 0,
\label{eq:Fbarnonincreasing}
\end{align}
i.e.\ $\bar F$ is non-increasing along the gradient flow path. Thus
$T(\bar F(g_T)-\bar F(h)) \leq
\int_0^T [\bar{F}(g_s) - \bar{F}(h)] \d s \leq \DKL(h \| g_0)$
and rearranging yields the theorem.
\end{proof}

\subsection{LSI for Langevin dynamics in high noise}\label{subsec:proof_q_flow}

\begin{proof}[Proof of Theorem \ref{thm:q_lsi}(a)]
Recall that $\bvarphi=\btheta+\z$ where $\z \sim \mathcal{N}(0,\tau^2\Id)$
and $\tau^2+\delta<\sigma^2/\|\X\|_\op^2$.
Using also $M^2+\delta<\sigma^2/\pnorm{\X}{\op}^2$, we may pick
some constant ${\tau'}^2 \in (\min(M^2,\tau^2),\sigma^2/\pnorm{\X}{\op}^2)$ 
depending only on $(M,\tau,\delta)$. Then, defining
$\eta^2={\tau'}^2-\tau^2>0$ and $\bvarphi'=\bvarphi+\z'=\btheta+\z+\z'$ for
$\z' \sim \mathcal{N}(0, \eta^2\Id)$ independent of $\z$, we have the equivalent
representation $\y=\X\bvarphi'+\beps'$ with $\beps' \sim \N(0, \bSigma')$
independent of $\bvarphi'$ such that
$\bSigma'=\sigma^2\Id-{\tau'}^2\X\X^\top$ is positive definite.

Write as shorthand $\nu(\bvarphi)=P_g(\bvarphi \mid \y)$,
$\nu_{\bvarphi'}(\bvarphi)=P_g(\bvarphi \mid \bvarphi')$,
and $\mu(\bvarphi') = P_g(\bvarphi' \mid \y)$. Following the decomposition of
\cite{bauerschmidt2019simple},
for any bounded, Lipschitz, and
continuously differentiable function $f:\R^p \rightarrow \R$, by the
tower property
\begin{align}\label{eq:ent_decomposition}
\ent_{\nu}(f^2)&=\E_\nu f^2\log f^2-\E_\nu f^2 \cdot \log \E_\nu f^2\nonumber\\
&=\E_{\bvarphi' \sim \mu}\Big[\E_{\nu_{\bvarphi'}} f^2\log f^2
-\E_{\nu_{\bvarphi'}} f^2 \cdot \log \E_{\nu_{\bvarphi'}} f^2\Big]\nonumber\\
&\hspace{1in}+\E_{\bvarphi' \sim \mu}\Big[\E_{\nu_{\bvarphi'}} f^2 \cdot \log
\E_{\nu_{\bvarphi'}} f^2\Big]
-\E_{\bvarphi' \sim \mu}[\E_{\nu_{\bvarphi'}} f^2] \cdot \log
\E_{\bvarphi' \sim \mu}[\E_{\nu_{\bvarphi'}} f^2]\nonumber\\
&=\E_{\mu} \ent_{\nu_{\bvarphi'}}(f^2) + \ent_{\mu} \E_{\nu_{\bvarphi'}}(f^2).
\end{align}

For the first term, note that $\nu_{\bvarphi'}(\bvarphi)$ is a product measure
with marginals
\begin{align}\label{eq:marginal_posterior}
\nu_{\bvarphi',j}(\varphi_j) \propto P_g(\varphi_j \mid \varphi_j') &\propto
e^{-\frac{(\varphi_j'-\varphi_j)^2}{2\eta^2}} \cdot \int
e^{-\frac{(\varphi_j-\theta)^2}{2\tau^2}}\d g(\theta) \propto \int
e^{-\frac{(\varphi_j - c(\theta))^2}{2\gamma^2}} \d\bar{g}(\theta),
\end{align}
where
\begin{align}\label{def:gamma_c}
\gamma^2 = \frac{\eta^2\tau^2}{\eta^2 + \tau^2}, \quad c(\theta) =
\frac{\tau^2\varphi_j' + \eta^2 \theta}{\eta^2 + \tau^2}, \quad \d\bar{g}(\theta)
 = \frac{e^{-\frac{(\theta - \varphi_j')^2}{2(\eta^2 + \tau^2)}}\d g(\theta)}{\int
e^{-\frac{(u - \varphi_j')^2}{2(\eta^2 + \tau^2)}}\d g(u)}.
\end{align}
In other words, $\nu_{\bvarphi',j}$ is the convolution between
$\calN(0,\gamma^2)$ and the pushforward of $\bar g$ under $c(\cdot)$, which is supported on an interval of length at most $2M$. By \cite[Theorem 1]{zimmermann2016elementary}, 
$\nu_{\bvarphi',j}$ satisfies a one-dimensional LSI with constant $C=C(\tau,\tau',M)$:
\begin{align}\label{ineq:nu_varphi_lsi}
\ent_{\nu_{\bvarphi',j}}(f^2) \leq C\cdot \E_{\nu_{\bvarphi',j}} |\partial_j f(\bvarphi)|^2
\end{align}
where we write
$\ent_{\nu_{\bvarphi',j}}$ and $\E_{\nu_{\bvarphi',j}}$ for the partial
entropy and expectation over the coordinate $\varphi_j$ of $\bvarphi \sim
\nu_{\bvarphi'}$. 
Finally, by tensorization of LSI \cite[Theorem 9.9(i)]{villani2009optimal}, the product measure $\nu_{\bvarphi'}$ satisfies a LSI with the same $C$, namely,
\begin{align}\label{ineq:entropy_tensor}
\ent_{\nu_{\bvarphi'}}(f^2) \leq C \cdot \E_{\nu_{\bvarphi'}} \pnorm{\nabla f(\bvarphi)}{}^2.
\end{align}
Taking further expectation under $\mu$ yields
\begin{align}\label{ineq:ent_1}
\E_{\mu} \ent_{\nu_{\bvarphi'}}(f^2) \leq C\cdot \E_{\nu} \pnorm{\nabla f(\bvarphi)}{}^2.
\end{align}

To bound the second term in (\ref{eq:ent_decomposition}), note that
$\mu(\bvarphi') \propto
\exp\big(-\frac{1}{2}(\y-\X\bvarphi')^\top{\bSigma'}^{-1}(\y-\X\varphi') -
\sum_{j=1}^p V(\varphi_j')\big)$, where
\begin{align*}
V(\varphi')={-}\log\int
e^{-\frac{1}{2{\tau'}^2}(\varphi'-\theta)^2}\d g(\theta). 
\end{align*}
We have
\begin{align*}
V''(\varphi')=\frac{1}{{\tau'}^2} - \partial_{\varphi'}^2 \log \int
e^{\frac{\varphi'\theta}{{\tau'}^2} - \frac{\theta^2}{2{\tau'}^2}}\d
g(\theta) = \frac{1}{{\tau'}^2} - \frac{\var_g(\theta \mid
\varphi')}{{\tau'}^4} \geq \frac{1}{{\tau'}^2} -
\frac{M^2}{{\tau'}^4}=:c(\tau',M)>0,
\end{align*}
using the fact that $P(\theta \mid \varphi')$ is supported on $[-M,M]$ and
${\tau'}^2>M^2$. Hence the density of $\mu(\bvarphi')$ is strictly log-concave,
so by the Bakry-Emery criterion \cite[Theorem 9.9(iii)]{villani2009optimal}, $\mu(\bvarphi')$
satisfies a LSI with constant $c(\tau',M)^{-1}$. Let $G(\bvarphi') \equiv [\E_{\nu_{\bvarphi'}}(f^2)]^{1/2}$. Then
\begin{align}
\ent_{\mu} \E_{\nu_{\bvarphi'}}(f^2) = \ent_{\mu}(G^2(\bvarphi')) \leq
c(\tau',M)^{-1} \cdot \E_{\mu} \pnorm{\nabla G(\bvarphi')}{}^2. 
\label{eq:LSI-mu}
\end{align}
We have
\begin{align}
\partial_{\varphi_i'} G(\bm{\varphi}') = \frac{\partial_{\varphi_i'}
G^2(\bvarphi')}{2G(\bvarphi')} =
\frac{1}{2\eta^2}\frac{\cov_{\nu_{\bvarphi'}}(f^2,\varphi_i)}{[\E_{\nu_{\bvarphi'}}(f^2)]^{1/2}}
=
\frac{1}{2\eta^2}\frac{\E_{\nu_{\bvarphi',-i}}\cov_{\nu_{\bvarphi',i}}(f^2,\varphi_i)}{[\E_{\nu_{\bvarphi'}}(f^2)]^{1/2}},
\label{eq:Ggrad}
\end{align}
where 
the second equality follows from $\partial_{\varphi_i'} G^2(\bvarphi') = \E_{\nu_{\bvarphi'}}[f^2 \partial_{\varphi_i'} \log P_g(\varphi_i|\varphi_i')]$ and 
$\partial_{\varphi_i'} \log P_g(\varphi_i|\varphi_i') = \frac{1}{\eta^2}(\varphi_i-\E[\varphi_i|\varphi'_i])$ from \prettyref{eq:marginal_posterior}, and
the last equality uses the fact that $\nu_{\bvarphi'}$ is a product
measure. 
Shorthand $\cov_{\nu_{\bvarphi',i}}(f^2,\varphi_i)$ as $\cov_i(f^2,
\varphi_i\mid \varphi_i')$. Note that (\ref{eq:marginal_posterior}) yields the
representation of the conditional law
$(\varphi_i \mid \varphi_i') \stackrel{d}{=} c(\theta) + z$
where $c(\theta)$ is given in (\ref{def:gamma_c}), and
$\theta \sim \bar{g}(\cdot)$ and $z \sim \N(0,\gamma^2)$ are independent
(conditional on $\varphi_i'$). Consequently,
\begin{align*}
\cov_i(f^2(\bvarphi), \varphi_i \mid \varphi_i') &= \cov_i(f^2(\bvarphi_{-i},
c(\theta) + z), c(\theta) + z \mid \varphi_i')\\
&= \cov_i(f^2(\bvarphi_{-i}, c(\theta) + z), c(\theta) \mid \varphi_i')
+ \cov_i(f^2(\bvarphi_{-i}, c(\theta) + z), z \mid \varphi_i')\\
&\equiv (I) + (II). 
\end{align*}
To bound $(I)$, let $\theta_1,\theta_2 \sim \bar g(\cdot)$ and $z_1,z_2 \sim
\N(0,\gamma^2)$ be independent. Then
\begin{align*}
2 \cdot (I) &= \E\Big[(f^2(\bvarphi_{-i}, c(\theta_1) + z_1) - f^2(\bvarphi_{-i},
c(\theta_2) + z_2))(c(\theta_1) - c(\theta_2))\;\Big|\; \varphi_i'\Big]\\
&\leq \E\Big[\Big(f(\bvarphi_{-i}, c(\theta_1) + z_1) - f(\bvarphi_{-i},
c(\theta_2) + z_2)\Big)^2\;\Big|\;\varphi_i'\Big]^{1/2}\\
&\qquad \cdot \E\Big[\Big(f(\bvarphi_{-i}, c(\theta_1) + z_1) + f(\bvarphi_{-i},
c(\theta_2) + z_2)\Big)^2(c(\theta_1) - c(\theta_2))^2\;\Big|\;
\varphi_i'\Big]^{1/2}\\
&= (2\var_{\nu_{\bvarphi',i}}[f])^{1/2} \cdot \frac{\eta^2}{{\tau'}^2} \cdot
\E\Big[\Big(f(\bvarphi_{-i}, c(\theta_1) + z_1) + f(\bvarphi_{-i}, c(\theta_2) +
z_2)\Big)^2(\theta_1 - \theta_2)^2\;\Big|\;\varphi_i'\Big]^{1/2}\\
&\leq C(M,\tau,\tau') \cdot \big[\E_{\nu_{\bvarphi',i}}(\partial_i
f)^2\big]^{1/2}\cdot \big[\E_{\nu_{\bvarphi',i}}(f^2)\big]^{1/2}, 
\end{align*}
where the last step applies the Poincar\'e inequality for
$\nu_{\bvarphi',i}$ (as implied by its LSI (\ref{ineq:nu_varphi_lsi}) -- see \cite[(9.32)]{villani2009optimal}) to $\var_{\nu_{\bvarphi',i}}[f]$, and the
fact that $\theta_1,\theta_2\in[-M,M]$. To bound $(II)$, using the independence of
$z$ and $\theta$ and applying Stein's lemma,
\begin{align*}
(II) &= \E\Big[f^2(\bvarphi_{-i}, c(\theta) + z) z \;\Big|\;\varphi_i'\Big]
= 2\gamma^2 \cdot \E\Big[f(\bvarphi_{-i}, c(\theta) + z) \partial_i
f(\bvarphi_{-i}, c(\theta) + z)\;\Big|\;\varphi_i'\Big]\\
&\leq C(\tau,\tau')\cdot \big[\E_{\nu_{\bvarphi',i}}(\partial_i f)^2\big]^{1/2}\cdot
\big[\E_{\nu_{\bvarphi',i}}(f^2)\big]^{1/2}.
\end{align*} 
Combining the two bounds and applying Cauchy-Schwarz and the tower property
for $\E_{\nu_{\bvarphi',-i}}$ yields
\[\E_{\nu_{\bvarphi',-i}}
\cov_i(f^2(\bvarphi), \varphi_i \mid \varphi_i') \leq C \cdot
\big[\E_{\nu_{\bvarphi'}}(\partial_i f)^2\big]^{1/2}\cdot
\big[\E_{\nu_{\bvarphi'}}(f^2)\big]^{1/2}.\]
Substituting this into \prettyref{eq:LSI-mu} and \prettyref{eq:Ggrad},
the second term in (\ref{eq:ent_decomposition}) is bounded by
\begin{align}\label{ineq:ent_2}
\ent_{\mu} \E_{\nu_{\bvarphi'}}(f^2) \leq C'\cdot \E_\mu\sum_{i=1}^p
\E_{\nu_{\bvarphi'}}(\partial_i f)^2 = C'\cdot \E_{\nu}\pnorm{\nabla f}{}^2
\end{align} 
for some $C'=C'(\tau,\tau',M)$. Combining (\ref{ineq:ent_1}) and (\ref{ineq:ent_2}) concludes the proof.
\end{proof}

The proof of Theorem \ref{thm:q_lsi}(b) requires the following
additional lemma, whose proof we present after the main proof.

\begin{lemma}\label{lem:exp_tilt_lsi}
Fix any $M > 0$. Uniformly over $h \in \R$, the measure $\mu_h(x) \propto
e^{hx}\bm{1}_{|x|<M}$ satisfies a LSI on $(-M,M)$ with log-Sobolev constant $C = C(M) > 0$.
\end{lemma}

\begin{proof}[Proof of Theorem \ref{thm:q_lsi}(b)]
The proof is similar to Theorem \ref{thm:q_lsi}(a), and we skip some of the
common steps. Pick ${\tau'}^2 \in (M^2,\sigma^2/\|\X\|_\op^2)$
depending only on $(M,\delta)$.
Let $\bvarphi'=\btheta + \z$ with $\z \sim \mathcal{N}(0, {\tau'}^2\Id)$, and
represent
$\y=\X\bvarphi' + \beps'$ with $\beps' \sim \mathcal{N}(0,\bSigma')$
and $\bSigma'=\sigma^2-{\tau'}^2\X\X^\top$.
Then, denoting $\nu=P_g(\btheta \mid \y)$, $\nu_{\bvarphi'}=P_g(\btheta \mid
\bvarphi')$ and $\mu=P_g(\bvarphi' \mid \y)$, we have, analogous to \prettyref{eq:ent_decomposition},
\[\ent_{\nu}(f^2) = \E_{\mu} \ent_{\nu_{\bvarphi'}}(f^2) + \ent_{\mu}
\E_{\nu_{\bvarphi'}}(f^2).\]
Here $\nu_{\bvarphi'}$ is a product measure with marginals
\begin{align*}
\nu_{\bvarphi',j}(\theta_j) \propto e^{-\frac{\varphi_j'}{\tau^2}\theta_j}\cdot
e^{-\frac{\theta_j^2}{2\tau^2}}g(\theta_j).
\end{align*}
By Lemma \ref{lem:exp_tilt_lsi}, the density proportional to
$e^{-\frac{\varphi_j'}{\tau^2}\theta_j}$ on $(-M,M)$ satisfies a LSI
with constant $C = C(M) > 0$ uniformly over $\varphi_j'$. Then by the
Holley-Stroock perturbation principle \cite[Theorem 9.9(ii)]{villani2009optimal}, $\nu_{\bvarphi',j}$ also satisfies a LSI
with constant $C = C'(M,\tau)\kappa(g)^2$.
Applying this and tensorization shows
\[\E_{\mu} \ent_{\nu_{\bvarphi'}}(f^2) \leq C\cdot \E_{\nu} \pnorm{\nabla
f(\btheta)}{}^2.\]

For the second term, we have again that $\mu$ is strongly log-concave
and hence satisfies a LSI with some constant $c(\tau',M)^{-1}<\infty$.
Let $G(\bvarphi') \equiv [\E_{\nu_{\bvarphi'}}(f^2)]^{1/2}$. Then as before,
\begin{align*}
\ent_{\mu} \E_{\nu_{\bvarphi'}}(f^2) = \ent_{\mu}(G^2(\bvarphi')) \leq
c(\tau',M)^{-1} \cdot \E_{\mu} \pnorm{\nabla G(\bvarphi')}{}^2,
\end{align*}
where
\begin{align*}
\partial_{\varphi_i'} G(\bvarphi') = \frac{\partial_{\varphi_i'}
G^2(\bvarphi')}{2G(\bvarphi')} =
\frac{1}{2{\tau'}^2}\frac{\E_{\nu_{\bvarphi',-i}}\cov_{\nu_{\bvarphi',i}}(f^2,\theta_i)}{[\E_{\nu_{\bvarphi'}}(f^2)]^{1/2}}.
\end{align*}
Here, since $\theta_i \in [-M,M]$ is bounded, we have simply
\begin{align*}
\cov_{\nu_{\bvarphi',i}}(f^2,\theta_i) &\leq C(M) \cdot
\var_{\nu_{\bvarphi',i}}(f^2)
\leq C'(M) \cdot \big[\var_{\nu_{\varphi',i}}(f)\big]^{1/2} \cdot
\big[\E_{\nu_{\bvarphi',i}}(f^2)\big]^{1/2}.
\end{align*}
Then, applying again the Poincar\'e inequality for $\nu_{\bvarphi',i}$ gives
\[\ent_{\mu} \nu_{\bvarphi'}(f^2) \leq C\cdot \E_\mu\sum_{i=1}^p
\E_{\nu_{\bvarphi'}}(\partial_i f)^2 =  C \cdot
\E_{\nu}\pnorm{\nabla f}{}^2\]
which completes the proof.
\end{proof}

To prove Lemma \ref{lem:exp_tilt_lsi}, we use the following characterization of
the LSI constant due to \cite[Theorem 5.3]{bobkov1999exponential}. We note that
this result has also been applied in \cite{zhang2011uniform} to prove an analogue of Lemma \ref{lem:exp_tilt_lsi} under spherical priors and general dimensions.

\begin{lemma}[\cite{bobkov1999exponential}]\label{lem:lsi_constant}
Let $\mu$ be a (Borel) probability measure on $\R$ and let $\rho(x)$ be the
density of its absolutely continuous part with respect to Lebesgue measure.
Let $m$ be a median of $\mu$. Then the optimal constant $C_\ast$ in the LSI
\begin{align*}
\ent_\mu(f^2) \leq C_* \E_\mu[{f'}^2]
\end{align*}
satisfies $c_1\max (B_-, B_+) \leq C_\ast \leq c_2\max(B_-, B_+)$ for some universal constants $c_1,c_2$, where
\begin{align*}
B_+ &= \sup_{x > m} \mu\big([x, \infty)\big) \cdot\log \frac{1}{\mu\big([x,
\infty)\big)}  \cdot\int_m^x \frac{1}{\rho(t)}\d t,\\
B_- &= \sup_{x > m} \mu\big((-\infty, x]\big) \cdot\log
\frac{1}{\mu\big((-\infty, x]\big)} \cdot\int_x^m \frac{1}{\rho(t)}\d t.
\end{align*}
\end{lemma}
This result does not require $\mu$ to have full support on all of $\R$, and
$B_+$ or $B_-$ is defined as 0 if $\mu((m,\infty))=0$ or $\mu((-\infty,m))=0$
respectively.

\begin{proof}[Proof of Lemma \ref{lem:exp_tilt_lsi}]
The uniform density on $(-M,M)$ satisfies a LSI with constant $C(M)>0$ by
\cite{ghang2014sharp}. Then by the Holley-Stroock perturbation principle, it
suffices to show the result for $|h|>h_0(M)$ and any fixed and large
constant $h_0(M)>0$. By symmetry, we may assume $h>0$.
For presentational simplicity, we only prove the case $M = 1$, and the case of general $M$ follows similarly.

Using Lemma \ref{lem:lsi_constant}, it suffices to control the
quantities $B_+,B_-$ for $\mu_h$.
Let us write $a(h) \asymp b(h)$ if there exist constants $C,c>0$ for which
$ca(h) \leq b(h) \leq Ca(h)$ for all large $h$.
Let $m \equiv m(h)$ be the median of $\mu_h$, and let
$\rho_h(t)=e^{ht}/(h^{-1}(e^h-e^{-h}))$ be the density of $\mu_h$. 
We claim that $1-m \asymp \frac{1}{h}$.
To see this, note that for any $x \in (-1,1)$,
\begin{align*}
\mu_h([-1,x]) = \int_{-1}^x \frac{e^{ht}}{h^{-1}(e^h - e^{-h})}\d t=\frac{e^{hx} - e^{-h}}{e^h - e^{-h}},
\end{align*}
so for sufficiently large $h$, $\inf\{x: \mu_h([-1,x]) \geq 1/2\} \in [1-C/h,1 - c/h]$. 

We first bound $B_+$. Using the boundedness of $t \mapsto t\log(1/t)$ and
\begin{align*}
\int_{m}^x \frac{1}{\rho_h(t)}\d t &= \frac{1}{h^2}(e^h - e^{-h})(e^{-hm} -
e^{-hx}) \asymp \frac{1}{h^2}(e^{h(1-m)} - e^{h(1-x)}),
\end{align*} 
we have for $1-m \asymp \frac{1}{h}$ and any $x \in (m,1)$ that
\begin{align*}
&\mu_h\big([x,1]\big) \cdot\log \frac{1}{\mu_h\big([x, 1]\big)} \cdot\int_{m}^x
\frac{1}{\rho_h(t)}\d t
\leq \frac{C}{h^2}
\end{align*}
for a constant $C>0$ independent of $x$. Then $B_+ \leq C/h^2$.

Next we bound $B_-$. Using
\begin{align*}
\mu_h\big([-1,x]\big)
=\frac{e^{hx}-e^{-h}}{e^h-e^{-h}} \asymp e^{-h(1-x)}(1-e^{-h(1+x)}),\quad \int_x^m
\frac{1}{\rho_h(t)}\d t \asymp \frac{e^{h(1-x)}}{h^2}(1-e^{h(x-m)}),
\end{align*}
we have for any $x \in (-1, m)$ and a constant $C>0$ that
\begin{align*}
\mu_h\big([-1,x]\big) \cdot\log\frac{1}{\mu_h\big([-1, x])\big)}
\cdot\int_x^{m} \frac{1}{\rho_h(t)}\d t
&\leq \frac{C}{h^2}(1-e^{-h(1+x)}) \log
\frac{e^{h}-e^{-h}}{e^{hx} - e^{-h}}\\
&\leq \frac{C}{h^2}(1-e^{-h(1+x)}) \log \frac{e^{h(1-x)}}{1-e^{-h(1+x)}}.
\end{align*}
Applying again boundedness of $t \mapsto t\log(1/t)$, we have in this case
$B_- \leq C'/h$ for a constant $C'>0$.
In summary, $\max(B_+,B_-)$ is at most a constant,
which in light of Lemma \ref{lem:lsi_constant} completes the proof.
\end{proof}

\subsection{Analysis of the joint gradient flow}\label{subsec:proof_main_algo}

We isolate the following intuitive but somewhat
technical lemma, which verifies the statement
\[\frac{\d}{\d t} F_n(q_t,g_t)
=-p\|\grad_q^{W_2} F_n(q_t,g_t)\|_{q_t}^2
-\alpha\|\grad_g^{\FR} F_n(q_t,g_t)\|_{g_t}^2\]
where $\|\cdot\|_{q_t}^2,\|\cdot\|_{g_t}^2$ are the squared norms under the
Wasserstein-2 metric tensor at $q_t$ and Fisher-Rao metric tensor at $g_t$,
respectively. (Our formal definition of the right side above is stated
in \eqref{eq:Fnderiv} below, and further background for these metric tensors
is discussed in Appendix \ref{sec:flow_background}.)

\begin{lemma}\label{lemma:Fndecays}
Under the conditions of Theorem \ref{thm:algo},
\[\limsup_{t \to 0} F_n(q_t,g_t) \leq F_n(q_0,g_0).\]
Furthermore $t \mapsto F_n(q_t,g_t)$ is differentiable at every $t>0$, and
\begin{equation}\label{eq:Fnderiv}
\frac{\d}{\d t} F_n(q_t,g_t)
=-\frac{1}{p}\int \left\|
\nabla \log\frac{\d q_t}{\d \nu[g_t]}\right\|_2^2 q_t(\bvarphi)\d\bvarphi
- \alpha\int 
\bigg(\bigg[\N_\tau*\frac{\bar q_t}{\N_\tau*g_t}\bigg](\theta)-1\bigg)^2
g_t(\theta)\d\theta.
\end{equation}
\end{lemma}

\noindent The proof of Lemma \ref{lemma:Fndecays} is deferred to Appendix
\ref{sec:Fndecays}.

\begin{proof}[Proof of Theorem \ref{thm:algo}]
Assume $\DKL(h\|g_0)<\infty$, as otherwise there is nothing to prove.
Then by \prettyref{thm:flow_solution}(a) and the same argument leading to
\prettyref{eq:DKLcontinuous}, $\DKL(h\|g_t)$ is finite for all $t$ and
continuous at $t=0$. By the definition \eqref{eq:deltaFn},
for any density $h$ supported on $[-M,M]$, 
\begin{equation}\label{eq:integratedfirstvar}
\int_{-M}^M -\delta \bar F_n[g_t](\theta) \cdot h(\theta)
=\int_\R \frac{\N_\tau*h}{\N_\tau*g_t}(\varphi)\bar\nu[g_t](\varphi)-1
=\frac{1}{p}\sum_{j=1}^p\int_{\R^p}\frac{\N_\tau*h}{\N_\tau*g_t}(\varphi_j)\nu[g_t](\bvarphi)-1.
\end{equation}

Denote
\[f_t(\varphi)=\log \frac{\N_\tau*h}{\N_\tau*g_t}(\varphi),
\qquad \nu_j[g_t](f_t)=\int f_t(\varphi_j)\nu[g_t](\bvarphi).\]
Then
\begin{align}
\int\frac{\N_\tau*h}{\N_\tau*g_t}(\varphi_j)\nu[g_t](\bvarphi) &= \int
e^{f_t(\varphi_j)}\nu[g_t](\bvarphi)\nonumber\\
&=\int e^{f_t(\varphi_j)}q_t(\bvarphi)
+e^{\nu_j[g_t](f_t)}
\int e^{f_t(\varphi_j)-\nu_j[g_t](f_t)}(\nu[g_t](\bvarphi)-q_t(\bvarphi)).
\label{eq:integratedfirstvarterm}
\end{align}
Observe that $f_t(\varphi_j)$ is $2M$-Lipschitz in $\bvarphi$, because
\begin{align*}
\frac{\d }{\d \varphi} f_t(\varphi) = \frac{(\N_\tau \ast h)'}{\N_\tau \ast
h}(\varphi) - \frac{(\N_\tau \ast g_t)'}{\N_\tau \ast g_t}(\varphi) =
\E_h[\theta \mid \theta + \tau Z = \varphi] -  \E_{g_t}[\theta \mid \theta +
\tau Z = \varphi] \in [-2M,2M]
\end{align*}
with $Z \sim \N(0,1)$. 
Then for some $(M,\tau)$-dependent constants $C,C',c>0$,
$f_t(\varphi_j)$ is $C'$-subgaussian under the law $\nu[g_t](\bvarphi)$
by the given LSI assumption (cf.~\cite[Theorem 9.9(iv)]{villani2009optimal}),
so for any $L>0$,
\begin{align}
&\int_{f_t(\varphi_j)-\nu_j[g_t](f_t)>L}
e^{f_t(\varphi_j)-\nu_j[g_t](f_t)}(\nu[g_t](\bvarphi)-q_t(\bvarphi))\nonumber\\
&\hspace{1in}\leq \int_{f_t(\varphi_j)-\nu_j[g_t](f_t)>L}
e^{f_t(\varphi_j)-\nu_j[g_t](f_t)} \nu[g_t](\bvarphi)
\leq Ce^{-cL^2}.\label{eq:nuqcompare1}
\end{align}

Next, applying $\dTV(\nu[g_t],q_t)=\frac{1}{2}\int
|\nu[g_t](\bvarphi)-q_t(\bvarphi)|$
and Pinsker's inequality $\dTV(\nu[g_t],q_t) \leq
\sqrt{\frac{1}{2}\DKL(q_t\|\nu[g_t])}$, we have
\begin{equation}\label{eq:nuqcompare2}
\int_{f_t(\varphi_j)-\nu_j[g_t](f_t) \leq L}
e^{f_t(\varphi_j)-\nu_j[g_t](f_t)}(\nu[g_t](\bvarphi)-q_t(\bvarphi))
\leq 2e^L\dTV(\nu[g_t],q_t)
\leq e^L\sqrt{2\DKL(q_t\|\nu[g_t])}.
\end{equation}
Applying (\ref{eq:Fnderiv}) without
the second term as an upper bound,
and applying again the LSI for $\nu[g_t](\bvarphi)$ with
$f = \sqrt{\d q_t/\d \nu[g_t]}$ in its definition (\ref{def:LSI}),
we have for a constant $c:=c(M,\tau)>0$,
\begin{align}
\frac{\d}{\d t}F_n(q_t,g_t) &\leq
-\frac{1}{p}\int \biggpnorm{\nabla \log\frac{\d q_t}{\d \nu[g_t]}}{2}^2
q_t(\bvarphi)\d\bvarphi
\leq -\frac{c}{p}\DKL(q_t \|\nu[g_t]) \leq 0.\label{ineq:flow_descent}
\end{align}
Applying (\ref{ineq:flow_descent}) to (\ref{eq:nuqcompare2}) and combining with
(\ref{eq:nuqcompare1}), we obtain for $(M,\tau)$-dependent constants $C_0,c_0>0$ that
\begin{equation}\label{eq:nuqcompare}
\int e^{f_t(\varphi_j)-\nu_j[g_t](f_t)}(\nu[g_t](\bvarphi)-q_t(\bvarphi))
\leq C_0\left(e^{-c_0L^2}+e^L\sqrt{-p \cdot \frac{\d}{\d t} F_n(q_t,g_t)}\right):=H_t(L)
\end{equation}
where we define $H_t(L)$ to be this (non-negative) upper bound.

Now by Jensen's inequality,
$e^{\nu_j[g_t](f_t)} \leq \int e^{f_t(\varphi_j)}\nu[g_t](\bvarphi)$.
Applying this and (\ref{eq:nuqcompare}) back to (\ref{eq:integratedfirstvarterm}),
\[\int e^{f_t(\varphi_j)}\nu[g_t](\bvarphi)-1
\leq \int e^{f_t(\varphi_j)}q_t(\bvarphi)-1
+H_t(L)\int e^{f_t(\varphi_j)}\nu[g_t](\bvarphi).\]
Then, averaging over $j=1,\ldots,p$, applying the identity
(\ref{eq:integratedfirstvar}), and using
\[\frac{1}{p}\sum_{j=1}^p e^{f_t(\varphi_j)}q_t(\bvarphi)-1
=\int \frac{\N*h}{\N*g_t}(\varphi)\bar
q_t(\varphi)\d\varphi-1=\frac{\d}{\d t}[-\DKL(h\|g_t)]\]
(where we may differentiate under the integral by the same
argument as in (\ref{eq:dominatedconvergence})), we get
\[\int_{-M}^M -\delta \bar F_n[g_t](\theta) \cdot h(\theta)
\leq \frac{\d}{\d t}[-\DKL(h\|g_t)]+H_t(L)\Big(\int_{-M}^M -\delta \bar F_n[g_t](\theta)
\cdot h(\theta)+1\Big).\]

Now denote
\[E(t)=\int_{-M}^M -\delta \bar F_n[g_t](\theta) \cdot h(\theta),
\quad A(t)={-}\DKL(h\|g_t),
\quad B(t)={-}C_0^2e^{2L}p \cdot F_n(q_t,g_t),\]
and fix $L=L(\eps)$ such that
$C_0e^{-c_0L(\eps)^2}=\eps/4$, i.e.\ $L(\eps)=\sqrt{(-1/c_0)\log[\eps/(4C_0)]}$.
Then the above differential inequality reads
\begin{equation}\label{eq:diffineq}
E(t) \leq A'(t)+(\eps/4+\sqrt{B'(t)})(E(t)+1).
\end{equation}
Note that by the definition \eqref{eq:deltaFn}, we have
$E(t)+1 \geq 0$ for all $t \geq 0$. Furthermore $B'(t) \geq 0$ by
\eqref{ineq:flow_descent}, and
\begin{align*}
A'(t)=\frac{\d }{\d t}\int_{-M}^M h(\theta)\log g_t(\theta)
=\alpha \cdot \int_{-M}^M \Big(\Big[\N_\tau \ast \frac{\bar{q}_t}{\N_\tau \ast
g_t}\Big](\theta) - 1\Big)h(\theta) \geq 
\alpha \cdot \int_{-M}^M {-}h(\theta)=-\alpha.
\end{align*}
Let $\cT=\{t>0:B'(t)<(\eps/4)^2\}$. Then \eqref{eq:diffineq} implies
$E(t) \leq A'(t)+(\eps/2)(E(t)+1)$ for $t \in \cT$.
Rearranging this inequality,
integrating over $\cT \cap [0,T]$ for any $T>0$, and applying the
above lower bound $A'(t) \geq -\alpha$ gives
\begin{equation}\label{eq:intEtbound}
\int_{\cT \cap [0,T]} E(t) \d t
\leq \frac{1}{1-\eps/2}\int_{\cT \cap [0,T]} (A'(t)+\eps/2)\d t
\leq \frac{1}{1-\eps/2}\left(\int_0^T A'(t)\d t
+\alpha |\cT^c \cap [0,T]|+\frac{\eps T}{2}\right)
\end{equation}
where $|\cT^c \cap [0,T]|$ denotes the Lebesgue measure of this set.

We have $\limsup_{t \to 0} F_n(q_t,g_t) \leq F_n(q_0,g_0)$ by Lemma
\ref{lemma:Fndecays}, and $\lim_{t \to 0} \DKL(h\|g_t) = \DKL(h\|g_0)$. Then
\begin{align*}
\int_0^T A'(t)\d t&=\DKL(h\|g_0)-\DKL(h\|g_T) \leq \DKL(h\|g_0),\\
\int_0^T B'(t)\d t&\leq C_0^2e^{2L(\eps)}p\big(F_n(q_0,g_0)-F_n(q_t,g_t)\big)
\leq C_0^2e^{2L(\eps)}p\big(F_n(q_0,g_0)-F_{n,*}\big).
\end{align*}
The condition $B'(t) \geq 0$ and definition of $\cT$ then imply
\[|\cT^c \cap [0,T]| \leq
(4/\eps)^2C_0^2e^{2L(\eps)}p\big(F_n(q_0,g_0)-F_{n,*}\big)
\leq C_1\eps^{-2.01}p\big(F_n(q_0,g_0)-F_{n,*}\big),\]
the last inequality holding for a constant $C_1>0$ and all
$\eps \in (0,\eps_0)$ using
$e^{2L(\eps)} \leq Ce^{c\sqrt{{-}\log \eps}} \leq C'\eps^{-0.01}$.
Then choosing any $T \geq 4C_1\eps^{-2.01}p(F_n(q_0,g_0)-F_{n,*})$,
we must have $|\cT \cap [0,T]| \geq 3T/4$, and
applying these bounds into \eqref{eq:intEtbound} gives
\[\frac{1}{|\cT \cap [0,T]|} \int_{\cT \cap [0,T]} E(t)\d t
\leq \frac{1}{1-\eps/2}
\cdot \frac{4}{3T}
\Big(\DKL(h\|g_0)+C_1\alpha \eps^{-2.01}p(F_n(q_0,g_0)-F_{n,*})
+\frac{\eps T}{2}\Big).\]
The right side is at most $\eps$ for any $\eps \in (0,\eps_0)$
and the choice $T=T(h,\eps)$ defined by \eqref{def:joint_flow_time},
for sufficiently large and small constants $T_0,\eps_0>0$.
This implies that there exists $t \in \cT \cap [0,T(h,\eps)]$ for which
\[E(t)=\int_{-M}^M -\delta \bar F_n[g_t](\theta) \cdot h(\theta) \leq \eps,\]
establishing \eqref{ineq:subgrad}. For this time $t \in \cT$, we have
\[(\eps/4)^2>B'(t)={-}C_0^2e^{2L(\eps)}p \cdot \frac{\d}{\d t}
F_n(q_t,g_t)
\geq {-}C\eps^{-0.01}\DKL(q_t\|\nu[g_t])
\geq {-}C'\eps^{-0.01}W_2(q_t,\nu[g_t])^2,\]
the last two inequalities holding by \eqref{ineq:flow_descent}
and the $T_2$-transportation inequality implied by the LSI for $\nu[g_t]$
\cite{otto2000generalization}. Rearranging shows \eqref{ineq:langevinconverge}
for any $\eps \in (0,\eps_0)$ and a sufficiently small constant $\eps_0>0$.

For (\ref{ineq:subopt_gap_joint}), since $\bar{F}_n(g_t) \leq
F_n(q_t, g_t) \leq F_n(q_0,g_0)$ where the second inequality follows from
(\ref{ineq:flow_descent}), we have that $g_t$
belongs to $\mathcal{K}_n$ for all $t \geq 0$. Let $t \leq
T(g_*,\eps)$ be as in (\ref{ineq:subgrad}) such that $\int_{-M}^M -\delta
\bar{F}_n[g_{t}](\theta)g_*(\theta) \leq \eps$, and define
$H(\lambda)=\bar{F}_n(\lambda g_* + (1-\lambda)g_{t})$. Since also $g_* \in
\mathcal{K}_n$, by the assumed convexity of $\bar F_n$ on $\mathcal{K}_n$
we have $H''(\lambda) \geq 0$ for all $\lambda \in [0,1]$
so $H(1) \geq H(0) + H'(0)$, which is equivalent to
\[\bar{F}_n(g_{t}) - \bar{F}_n(g_*) \leq \int_{-M}^M -\delta
\bar{F}_n[g_{t}](\theta)g_*(\theta)\d\theta \leq \eps.\]
This shows (\ref{ineq:subopt_gap_joint}).
\end{proof}

\begin{proof}[Proof of Corollary \ref{cor:algo}]
The convergence $g_{t_n} \to g_*$ weakly in probability 
and the statement \eqref{eq:qnuW2converge} follow directly from
from Theorem \ref{thm:consistency} and Corollary
\ref{cor:consistencyposterior}.

For the final statement \eqref{eq:posteriormeanconverge}, note that by
the $T_2$-inequality implied by the LSI for $\nu[g_*]$,
we have $W_2(q,\nu[g_*])^2 \leq C\DKL(q\|\nu[g_*])$ for a constant $C>0$
and any density $q \in \cP_*(\R^p)$ \cite{otto2000generalization}. Hence
\eqref{eq:qnuW2converge} implies also that
\begin{equation}
\frac{1}{p}\,W_2(\nu[g_{t_n}],\nu[g_*])^2 \to 0
\label{eq:W2-joint}
\end{equation}
in probability, i.e.\ for any $\delta>0$, with probability 
(over $\y$ defining the posterior laws $\nu[g_{t_n}],\nu[g_*]$) approaching 1
as $n,p \to \infty$,
there exists a coupling of
$\bvarphi \sim \nu[g_{t_n}]$ and $\tilde \bvarphi \sim \nu[g_*]$ such that
\[\frac{1}{p}\|\bvarphi-\tilde \bvarphi\|_2^2 \leq \delta.\]
Write each coordinate of $\E_g[f(\btheta) \mid \y]$ as
\begin{align*}
\E_g[f(\theta_j) \mid \y]
&=\int_{-M}^M f(\theta_j)P_g(\theta_j \mid \y)\d \theta_j\\
&=\int \left(\int_{-M}^M f(\theta_j)P_g(\theta_j \mid
\varphi_j)\d\theta_j\right) P_g(\varphi_j \mid \y) \d \bvarphi
=\E_g[\langle f(\theta_j) \rangle_{\varphi_j} \mid \y],
\end{align*}
where we denote by
$\langle f(\theta) \rangle_\varphi=\int_{-M}^M f(\theta) P_g(\theta \mid
\varphi)\d\theta$ the posterior average of $f(\theta)$ given $\varphi$.
Observe that
\begin{align*}
\frac{\d}{\d \varphi} \langle f(\theta) \rangle_\varphi
&=\frac{\d}{\d \varphi} \int_{-M}^M f(\theta)
\frac{e^{-\frac{(\theta-\varphi)^2}{2\tau^2}}g(\theta)}
{\int_{-M}^M e^{-\frac{(\theta'-\varphi)^2}{2\tau^2}}g(\theta')\d\theta'}
\d \theta\\
&=\left\langle f(\theta) \cdot
\frac{\theta-\varphi}{\tau^2}\right\rangle_\varphi
-\langle f(\theta)\rangle_\varphi \left\langle
\frac{\theta-\varphi}{\tau^2}\right\rangle_\varphi
=\frac{1}{\tau^2}\Big(\langle f(\theta) \cdot \theta\rangle_\varphi
-\langle f(\theta)\rangle_\varphi \langle \theta\rangle_\varphi\Big)
\end{align*}
(a covariance between $f(\theta)$ and $\theta$ under this posterior law).
If $f:\R \to \R$ is bounded, say as $\|f\|_\infty \leq B$, then this is at most
$2MB/\tau^2$, establishing that $\varphi \mapsto \langle
f(\theta) \rangle_\varphi$ is Lipschitz. Then, under the above coupling,
\begin{align*}
\frac{1}{p}\big\|\E_{g_{t_n}}[f(\btheta) \mid \y]
-\E_{g_*}[f(\btheta) \mid \y]\big\|_2^2
&=\frac{1}{p}\sum_{j=1}^p
\Big(\E_{g_{t_n}}[\langle f(\theta_j) \rangle_{\varphi_j} \mid \y]
-\E_{g_*}[\langle f(\theta_j) \rangle_{\varphi_j} \mid \y]\Big)^2\\
&\leq \frac{1}{p}\sum_{j=1}^p \left(\frac{2MB}{\tau^2}\right)^2
\E[|\varphi_j-\tilde \varphi_j|^2]
\leq \left(\frac{2MB}{\tau^2}\right)^2 \delta.
\end{align*}
As $\delta>0$ here is arbitrary, this implies the convergence
\eqref{eq:posteriormeanconverge} in probability.
(The same proof shows that posterior mean consistency only requires 
the posterior marginals are on average consistent, namely, $\frac{1}{p}\sum_{j=1}^p W_2^2(\varphi_j,\tilde{\varphi}_j)=o(1)$, which is implied by the normalized $W_2$-consistency of the full posteriors in \prettyref{eq:W2-joint}.)
\end{proof}

\subsection{Analysis of $T(g_*,\eps)$}\label{appendix:initialgap}

Suppose that Assumption \ref{assump:design} holds, and
$\tau^2<\sigma^2/\|\X\|_\op^2-\delta$. Suppose also that
$g_*$ and $(g_0,q_0)$ satisfy
\begin{equation}\label{eq:initialcond}
\DKL(g_*\|g_0)=O(1),
\qquad \DKL(q_0\|(\N_\tau*g_0)^{\otimes p})=O(p),
\qquad \E_{q_0}[\|\bvarphi\|_2^2]=O(p).
\end{equation}
We verify the bound \eqref{eq:Tbound} for $T(g_*,\eps)$ under these conditions.

Recall that $F_{n,*}=\min_{g \in \calP(M)} \bar F_n(g)$.
Bounding the exponential term in \eqref{def:npmle} from above
by 1, we have $F_{n,*} \geq \frac{n}{2p} \log 2\pi\sigma^2$.
Next, writing the definition \eqref{eq:varrep} as
\[F_n(q,g)={-}\frac{1}{p}\int
\Big[\log P(\y \mid \bvarphi)
-\log \frac{q(\bvarphi)}{P_g(\bvarphi)}\Big]\d q(\bvarphi)
={-}\frac{1}{p}\E_{\bvarphi \sim q}[\log P(\y \mid \bvarphi)]
+\frac{1}{p}\DKL(q\|P_g),\]
we have
\begin{align}
F_n(q_0,g_0)&= \frac{1}{2p}\Expect_{\bvarphi\sim q_0}[(\y-\X\bvarphi)^\top
\bSigma^{-1} (\y-\X\bvarphi)]+\frac{1}{p} \DKL(q_0\|(\N_\tau*g_0)^{\otimes
p})\nonumber\\
&\hspace{1in}+\frac{n}{2p}\log 2\pi + \frac{1}{2p}\log\det(\bSigma).
\label{eq:F0}
\end{align}
Note that 
\[
\Expect_{\bvarphi\sim q_0}[(\y-\X\bvarphi)^\top \bSigma^{-1} (\y-\X\bvarphi)]
\leq
\|\bSigma^{-1}\|_{\op} \Expect_{\bvarphi\sim q_0}[\|\y-\X\bvarphi\|_2^2] \leq 2
\|\bSigma^{-1}\|_{\op} (\|\y\|_2^2+\E_{q_0}[\|\bvarphi\|_2^2]\|\X\|_\op^2),
\]
where $\|\bSigma^{-1}\|_{\op}=(\sigma^2-\tau^2
\|\X\|_{\op}^2)^{-1}<c/\sigma^2$
and $\|\X\|_\op^2\leq C\sigma^2$  by \prettyref{assump:design}.
Furthermore $\|\y\|_2^2 \leq 2\|\X\|_{\op}^2 \|\btheta\|_2^2 + 2\|\beps\|_2^2
\leq C(p+n)\sigma^2$ with probability approaching 1 as $n,p \to \infty$.
Thus, under \eqref{eq:initialcond},
the first term in \prettyref{eq:F0} is at most $C'(1+\frac{n}{p})$
for a constant $C'>0$. The second term is at most $O(1)$ by
\eqref{eq:initialcond}. For the last two terms
we have $n \log 2\pi+\log \det(\bSigma) \leq n\log 2\pi+n \log
\pnorm{\bSigma}{\op} \leq n\log 2\pi \sigma^2$, so the last two terms are at
most $F_{n,*}$. Thus
\[F_n(q_0,g_0)-F_{n,*} \leq C\Big(1+\frac{n}{p}\Big),\]
and combining with the first condition of \eqref{eq:initialcond} shows
\eqref{eq:Tbound}.

\section{Solution properties of gradient flows}\label{sec:existenceuniqueness}

In this appendix, we prove Theorem \ref{thm:flow_solution}.
Throughout this appendix, $n,p,\X,\y$ are fixed, and $C,C',C_0,K$ etc.\ denote
constants that may depend on $n,p,\X,\y$ unless otherwise specified.

By a solution $\{q_t,g_t\}_{t \in [0,T]}$ to
(\ref{eq:qflow}--\ref{eq:gflow}), we mean functions
$t \mapsto g_t(\theta)$ in $C^1([0,T])$ for each $\theta \in [-M,M]$ and
$(t,\bvarphi) \mapsto q_t(\bvarphi)$ in $C^{1,2}([0,T] \times \R^p)$,
such that $q_t,g_t$ are probability densities for all $t \in [0,T]$ and
(\ref{eq:qflow}--\ref{eq:gflow}) hold in the classical sense of pointwise
equality. We call the solution unique if for any other such solution
$\{\tilde q_t,\tilde g_t\}_{t \in [0,T]}$, we have
$q_t(\bvarphi)=\tilde q_t(\bvarphi)$ and $g_t(\theta)=\tilde g_t(\theta)$
for all $t \in [0,T]$, $\theta \in [-M,M]$, and $\bvarphi \in \R^p$.

We equip $\cP(M)$ with the total variation metric
$\dTV(g,\tilde{g})=\sup_{A\subseteq [-M,M]} \big|g(A) - \tilde{g}(A)\big|$,
or equivalently
$\dTV(g,\tilde{g})=\frac{1}{2}\int_{-M}^M |g(\theta)-\tilde{g}(\theta)|
\d\theta$ when $g,\tilde{g} \in \cP_*(M)$ both admit densities. Then $\cP(M)$ is
complete under $\dTV$. For any $T>0$, let
\begin{align*}
\mathcal{L}_g \equiv \mathcal{L}_g(T)=C\big([0,T], \cP(M)\big),
\end{align*}
the space of continuous (with respect to $\dTV$) $\cP(M)$-valued processes $\{g_t\}_{t \in [0,T]}$,
equipped with the metric
$d_{\infty,\mathrm{TV}}(g,\tilde{g})=\sup_{t\in [0,T]} \dTV(g_t,\tilde{g}_t)$.
Then also $\mathcal{L}_g$ is complete under $d_{\infty,\mathrm{TV}}$.

Next, let
\begin{align*}
E_q \equiv E_q(T)=C([0,T], \R^p),
\end{align*}
the space of continuous $\R^p$-valued processes $\{\bvarphi_t\}_{t \in [0,T]}$,
equipped analogously with
the norm $\|\bvarphi\|_{\infty,2}=\sup_{t\in [0,T]} \pnorm{\bvarphi_t}{2}$. Let $\mathcal{L}_q
\equiv \mathcal{L}_q(T)$ be the space of probability distributions on $E_q$
for which $\E[\|\bvarphi\|_{\infty,2}^2]<\infty$,
equipped with the Wasserstein-2 distance induced by this norm,
\begin{align}\label{def:W2_process}
\notag d_W(q,\tilde{q}) &=\inf_{\bvarphi =
\{\bvarphi_t\}_{t\in[0,T]} \sim q,\, \tilde\bvarphi =
\{\tilde\bvarphi_t\}_{t\in[0,T]} \sim \tilde{q}} \big(\E \pnorm{\bvarphi -
\tilde\bvarphi}{\infty,2}^2\big)^{1/2}\\
&= \inf_{\bvarphi = \{\bvarphi_t\}_{t\in[0,T]} \sim q,\, \tilde\bvarphi =
\{\tilde\bvarphi_t\}_{t\in[0,T]} \sim \tilde{q}} \Big(\E \sup_{t \in
[0,T]} \pnorm{\bvarphi_t - \tilde\bvarphi_t}{2}^2\Big)^{1/2}
\end{align}
where the infimum is taken over all couplings of $q,\tilde{q} \in
\mathcal{L}_q$. Then $E_q$ is complete, and hence $\mathcal{L}_q$ is
also complete by \cite[Theorem 6.18]{villani2009optimal}. We will write
$\{q_t\}_{t \in [0,T]}$ for the marginal laws of $q \in \mathcal{L}_q$.

\subsection{Solution properties of the g-flow}\label{subsec:exist_gflow}

The following result gives the solution properties of the Fisher-Rao $g$-flow
for a fixed law $q \in \mathcal{L}_q(T)$.

\begin{lemma}\label{lem:exist_g_fixed_q}
\phantom{}\\
\begin{enumerate}
\item 
Fix any $T > 0$. Let $q\in \mathcal{L}_q(T)$ be
such that there exists some $K<\infty$ for which its marginal laws $\{q_t\}_{t
\in [0,T]}$ satisfy
\begin{equation}
\sup_{t\in [0,T]} \sup_{j \in [p]} \E_{\bvarphi_t\sim q_t} e^{(8M/\tau^2)|\varphi_{t,j}|}<K.
\label{eq:qtMGF}
\end{equation}
Then for any initialization $g_0 \in \cP_*(M)$ 
that is strictly positive on $[-M,M]$,
there exists a unique solution $\{g_t\}_{t\in [0,T]}$
to
\begin{align}\label{def:ODE_g}
\frac{\d}{\d t}g_t(\theta) = g_t(\theta) \Big(\N_\tau \ast
\frac{\bar{q}_t}{\N_\tau \ast g_t}(\theta) - 1\Big)
\end{align}
such that $g_t \in \cP_*(M)$ for all $t \in [0,T]$. Furthermore,
$g_t(\theta)\geq g_0(\theta) e^{-t}$ for all $t \in [0,T]$ and $\theta \in
[-M,M]$.\\

\item For any $K>0$, there exists $C_0>0$ depending only on $(M,\tau,K)$
such that the following holds: Fix any $T \leq 1/C_0$.
For any $q,\tilde{q} \in \mathcal{L}_q(T)$ both satisfying the property
\prettyref{eq:qtMGF}, and for any $g_0 \in \cP_*(M)$ satisfying $g_0(\theta)>0$
for all $\theta \in [-M,M]$, if $g,\tilde{g}$ are the solutions of
(\ref{def:ODE_g}) associated to $q,\tilde{q}$ given by Part (1), then
\begin{align*}
d_{\infty,\mathrm{TV}}(g,\tilde{g}) \leq 2C_0T \cdot d_W(q,\tilde{q}).
\end{align*}
\end{enumerate}
\end{lemma}
\begin{proof}[Proof of Lemma \ref{lem:exist_g_fixed_q}-\textbf{Part (1)}]
Fix $g_0 \in \cP_*(M)$ with $g_0(\theta)>0$ for all $\theta \in [-M,M]$. Define
\[B(g_0)=\{g \in L^1([-M,M]):g(\theta) \in
[g_0(\theta)/2,\,2g_0(\theta)] \text{ for Lebesgue-a.e.~} \theta \in[-M,M]\},\]
equipped with the $L^1$ metric
$\|g-\tilde{g}\|_{L^1} = \int_{-M}^{M} |g(\theta)-\tilde{g}(\theta)|$.
Note that $L^1([-M,M])$ is complete, and if
$g_n \in B(g_0)$ converges in $L^1$ to some
$g_\ast \in L^1([-M,M])$, then also $g_n(\theta)-g_\ast(\theta)
\to 0$ for Lebesgue-a.e.\ $\theta \in [-M,M]$. Then $g_\ast(\theta) \in
[g_0(\theta)/2,\,2g_0(\theta)]$ for Lebesgue-a.e.\ $\theta \in [-M,M]$,
implying that $B(g_0)$ is closed in $L^1([-M,M])$ and hence complete.

For some $t_0>0$ to be specified later, define
$E(g_0)=C([0,t_0],B(g_0))$ equipped with the metric
$d_{\infty,L^1}(g,\tilde{g})=\sup_{t \in [0,t_0]} \|g_t-\tilde{g}_t\|_{L^1}$.
Then $E(g_0)$ is also complete, by the above completeness of $B(g_0)$.
Define a map $\Gamma:E(g_0) \to E(g_0)$ with $\{g_t\}_{t \in [0,t_0]}
\mapsto \{\Gamma g\}_{t \in [0,t_0]}$, by
\begin{align}\label{def:Gamma_map}
(\Gamma g)_t(\theta)=g_0(\theta) \cdot \exp\left(\int_0^t \Big[\N_\tau \ast
\frac{\bar{q}_s}{\N_\tau \ast g_s}(\theta) - 1\Big]\d s\right).
\end{align}
We claim that for small enough $t_0$, this definition of $\Gamma$ indeed
maps $E(g_0)$ to itself, and furthermore is contractive under $d_{\infty,L^1}$.
To see that $\Gamma$ maps $E(g_0)$ to itself,
note that for any $g \in E(g_0)$ and $t \in [0,t_0]$,
\begin{equation}\label{eq:logglower}
\log\,(\Gamma g)_t(\theta) = \log g_0(\theta) + \int_0^t\Big[\N_\tau \ast
\frac{\bar{q}_s}{\N_\tau \ast g_s}(\theta) - 1\Big]\d s \geq \log g_0(\theta) - t.
\end{equation}
Observe that condition \prettyref{eq:qtMGF}
implies also $\int e^{(8M/\tau^2)|\varphi|}\bar q_s(\varphi)<K$.
Applying this,
$\frac{\N_\tau(\varphi-\theta)}{\N_\tau(\varphi-\theta')}
\leq Ce^{(2M/\tau^2)|\varphi|}$ for all $\theta,\theta' \in [-M,M]$ by
Proposition \ref{prop:Nratiobound},
and $\int_{-M}^M g_s \geq \int_{-M}^M g_0/2=1/2$,
\begin{align}
\N_\tau \ast \frac{\bar{q}_s}{\N_\tau \ast g_s}(\theta) &= \int
\frac{\N_\tau(\varphi-\theta)}{\int_{-M}^M \N_\tau(\varphi -
\theta')g_s(\theta')\d\theta'} \bar{q}_s(\varphi) \d \varphi\notag\\
&\leq \int Ce^{(2M/\tau^2)|\varphi|} \bar{q}_s(\varphi)\d \varphi \cdot
\frac{1}{\int_{-M}^M g_s(\theta') \d\theta'}<C'\label{eq:NqNhbound}
\end{align}
for some $C'>0$. Here and below, $C,C'$ etc.\ are constants depending only on
$(M,\tau,K)$. So
\begin{equation}\label{eq:loggupper}
\log\,(\Gamma g)_t(\theta) = \log g_0(\theta) + \int_0^t\Big[\N_\tau \ast
\frac{\bar{q}_s}{\N_\tau \ast g_s}(\theta) - 1\Big]\d s \leq \log g_0(\theta) +
C't. 
\end{equation}
Choosing $t_0<(\log 2)/\max(C',1)$, 
this shows that for all $t \leq t_0$ and $\theta \in [-M,M]$, we have
$(\Gamma g)_t(\theta) \in (g_0(\theta)/2,\,2g_0(\theta))$, so $(\Gamma g)_t \in
B(g_0)$. Furthermore, it is clear from the definition of $\Gamma$
that if $s \to t$, then
$(\Gamma g)_s(\theta) \to (\Gamma g)_t(\theta)$ pointwise over
$\theta \in [-M,M]$. Then also $\int_{-M}^M |(\Gamma g)_s-(\Gamma g)_t| \to 0$
by (\ref{eq:logglower}), (\ref{eq:loggupper}), and the dominated convergence
theorem, so $t \mapsto (\Gamma g)_t$ is continuous in $L^1$. Thus $\Gamma$ is a
well-defined map from $E(g_0)$ to itself.

To see that $\Gamma$ is contractive, observe that for any $g,\tilde{g} \in E(g_0)$ and $t\in
[0,t_0]$,
\begin{align*}
&\int_{-M}^M \Big|(\Gamma g)_t(\theta) - (\Gamma \tilde{g})_t(\theta)\Big| \d\theta\\
&=\int_{-M}^M g_0(\theta)\Big|\exp\Big(\int_0^t \big[\N_\tau \ast
\frac{\bar{q}_s}{\N_\tau \ast g_s}(\theta) - 1\big] \d s\Big) -
\exp\Big(\int_0^t \big[\N_\tau \ast \frac{\bar{q}_s}{\N_\tau \ast \tilde{g}_s}(\theta) - 1\big] \d s\Big)\Big| \d\theta\\
&\leq \sup_{\theta\in[-M,M]}\Big|\exp\Big(\int_0^t \N_\tau \ast
\frac{\bar{q}_s}{\N_\tau \ast g_s}(\theta)\d s\Big) - \exp\Big(\int_0^t\N_\tau
\ast \frac{\bar{q}_s}{\N_\tau \ast \tilde{g}_s}(\theta) \d s\Big)\Big|\\
&\leq t\cdot \sup_{\theta \in [-M,M]}e^{\xi(\theta)} \cdot \sup_{s\in
[0,t]}\Big|\N_\tau \ast \frac{\bar{q}_s}{\N_\tau \ast g_s}(\theta) - \N_\tau
\ast \frac{\bar{q}_s}{\N_\tau \ast \tilde{g}_s}(\theta)\Big|,
\end{align*}
where $\xi(\theta)$ is some value between $\int_0^t \N_\tau \ast
\frac{\bar{q}_s}{\N_\tau \ast g_s}(\theta)\d s$ and $\int_0^t \N_\tau
\ast \frac{\bar{q}_s}{\N_\tau \ast \tilde{g}_s}(\theta)\d s$. Using
(\ref{eq:NqNhbound}), we have $\sup_{\theta\in[-M,M]} e^{\xi(\theta)}<2$
because $t \leq t_0<(\log 2)/C'$.
Moreover, using again Proposition \ref{prop:Nratiobound} and
the lower bounds $\int_{-M}^M g_s(\theta) \geq 1/2$ and
$\int_{-M}^M \tilde g_s(\theta) \geq 1/2$ for all
$s\in [0,t_0]$, we have
\begin{align*}
&\sup_{\theta \in [-M,M]}\sup_{s\in [0,t]}\Big|\N_\tau \ast
\frac{\bar{q}_s}{\N_\tau \ast g_s}(\theta) - \N_\tau \ast
\frac{\bar{q}_s}{\N_\tau \ast \tilde{g}_s}(\theta)\Big|\\
&\leq \sup_{\theta \in [-M,M]}\sup_{s\in [0,t]} \int_{-M}^M\int
\frac{\N_\tau(\varphi-\theta)\N_\tau(\varphi -
\theta')\bar{q}_s(\varphi)}{\int_{-M}^M
\N_\tau(\varphi - \theta')g_s(\theta')\d\theta'\int_{-M}^M \N_\tau(\varphi -
\theta')\tilde{g}_s(\theta')\d\theta'} \cdot |g_s(\theta') -
\tilde{g}_s(\theta')|\d \varphi\d\theta'\\
&\leq \sup_{s\in [0,t]} \int Ce^{(4M/\tau^2)|\varphi|} \bar{q}_s(\varphi) \d \varphi \cdot
\sup_{s\in[0,t]} \frac{1}{\int_{-M}^M g_s(\theta')\d\theta' \cdot
\int_{-M}^M \tilde{g}_s(\theta')\d\theta'} \cdot \sup_{s\in[0,t]} \int_{-M}^M
|g_s(\theta') - \tilde{g}_s(\theta')|\d\theta'\\
&\leq C''\pnorm{g - \tilde{g}}{\infty,L^1}.
\end{align*}
For some $C''>0$. In summary,
\begin{align*}
\pnorm{\Gamma g - \Gamma \tilde{g}}{\infty,L^1}=\sup_{t\in [0,t_0]} \int_{-M}^M
\Big|(\Gamma g)_t(\theta) - (\Gamma \tilde g)_t(\theta)\Big| \d\theta \leq 2C''t_0
\pnorm{g - \tilde{g}}{\infty,L^1}.
\end{align*}
Choosing $t_0 \leq 1/4C''$, this
establishes that $\Gamma: E(g_0) \rightarrow E(g_0)$ is a
contraction.

By the Banach fixed point theorem, there exists a unique fixed
point $g=\{g_t\}_{t \in [0,t_0]} \in E(g_0)$ of $\Gamma$. Choosing the version
of $g$ satisfying (\ref{def:Gamma_map}) with pointwise equality for all $t \in
[0,t_0]$ and $\theta \in [-M,M]$, we must then have
\begin{equation}\label{eq:gfixedcondition}
\log g_t(\theta) = \log g_0(\theta) + \int_0^t \Big[\N_\tau \ast
\frac{\bar{q}_s}{\N_\tau \ast g_s}(\theta) - 1\Big]\d s
\end{equation}
for all $t \in [0,t_0]$ and $\theta \in [-M,M]$. 
The functions $\bvarphi \mapsto p^{-1}\sum_{j=1}^p
\frac{\N_\tau(\varphi_j-\theta)}{\N_\tau*g_s(\varphi_j)} q_s(\bvarphi)$
are uniformly integrable over $s \in [0,t_0]$, by
Proposition \ref{prop:Nratiobound} and the condition (\ref{eq:qtMGF})
for $q_s$. Thus $s \mapsto \N_\tau*\frac{\bar q_s}{\N_\tau*g_s}(\theta)=p^{-1}\sum_{j=1}^p
\int \frac{\N_\tau(\varphi_j-\theta)}{\N_\tau*g_s(\varphi_j)} q_s(\bvarphi) \d\bvarphi$
is continuous in $s$, so $g_t(\theta)$ is $C^1$ in $t$. Differentiating
(\ref{eq:gfixedcondition}) in $t$ shows that $\{g_t\}_{t \in [0,t_0]}$ solves
(\ref{def:ODE_g}). Furthermore, $g_t \in \cP_*(M)$ for all $t \in [0,t_0]$,
because
\begin{align*}
\partial_t \int_{-M}^M g_t = \int_{-M}^M g_t \cdot\partial_t \log g_t =
\int_{-M}^M g_t \cdot \Big[\N_\tau \ast \frac{\bar{q}_t}{\N_\tau \ast g_t} -
1\Big] = 1 - \int_{-M}^M g_t, 
\end{align*}
which has the unique solution $\int_{-M}^M g_t=1$ with initialization
$\int_{-M}^M g_0 = 1$.

Conversely, let $\{\tilde g_t\}_{t \in [0,t_0]}$ be any other
solution of (\ref{def:ODE_g}) where $\tilde g_t \in \cP_*(M)$ for all $t \in
[0,t_0]$. Suppose by contradiction that there exists some
$\theta \in [-M,M]$ for which $\tilde g_t(\theta) \notin
[g_0(\theta)/2,\,2g_0(\theta)]$ for some $t \in [0,t_0]$. Fixing this $\theta$,
let $t'$ be the infimum of all such $t$. Then 
$\tilde g_{t'}(\theta)=g_0(\theta)/2$ or $\tilde g_{t'}=2g_0(\theta)$ by
continuity of $t \mapsto \tilde g_t(\theta)$. On the other hand, we must have
$\log \tilde g_t(\theta) \geq \log g_0(\theta)-t$ and
$\log \tilde g_t(\theta) \leq \log g_0(\theta)+C't$ for all $t \in [0,t']$,
by the same arguments as (\ref{eq:logglower}) and (\ref{eq:loggupper})
where we may use $\int_{-M}^M \tilde{g}_s=1$ in place of the previously used
condition $\int_{-M}^M g_s \geq 1/2$ to derive (\ref{eq:loggupper}).
Since $t' \leq t_0<(\log 2)/\max(C',1)$, integrating these inequalities leads to
the contradiction that
$\tilde g_{t'}(\theta)>g_0(\theta)/2$ and $\tilde g_{t'}(\theta)<2g_0(\theta)$
strictly. Thus, $\tilde g_t(\theta) \in [g_0(\theta)/2,\,2g_0(\theta)]$ for all
$t \in [0,t_0]$ and $\theta \in [-M,M]$. So
$\{\tilde g_t\}_{t \in [0,t_0]}$ is also a fixed point of $\Gamma$ in $E(g_0)$,
implying by uniqueness of this fixed point that $g_t=\tilde g_t$
as elements of $L^1([-M,M])$ for all $t \in [0,t_0]$. This implies that
the right sides of (\ref{eq:gfixedcondition}) are equal for
$g_t$ and $\tilde g_t$, and consequently the left sides are equal pointwise,
i.e.\ $g_t(\theta)=\tilde g_t(\theta)$ pointwise for all $\theta \in [-M,M]$
and $t \in [0,t_0]$. Thus the solution
$\{g_t\}_{t \in [0,t_0]}$ to (\ref{def:ODE_g}) is unique.

This solution satisfies that $g_{t_0}$ is a strictly positive density in
$\cP_*(M)$. Then we may
apply the above argument sequentially to $[t_0,2t_0]$, $[2t_0,3t_0]$ etc.\ with
$g_{t_0}$, $g_{2t_0}$ etc.\ in place of $g_0$. This proves the existence and
uniqueness of the solution $\{g_t\}_{t \in [0,T]}$ to (\ref{def:ODE_g}).
Finally, from \prettyref{eq:gfixedcondition} we get $g_t(\theta)/g_0(\theta)
\geq e^{-t}$ for all $\theta \in [-M,M]$ and $t \in [0,T]$, 
completing the proof of Part (1).
\end{proof}

\begin{proof}[Proof of Lemma \ref{lem:exist_g_fixed_q}-\textbf{Part (2)}]
By definition,
\begin{align*}
g_t(\theta) &= g_0(\theta) + \int_0^t g_s(\theta)\Big[\N_\tau \ast \frac{\bar{q}_s}{\N_\tau \ast g_s}(\theta) - 1\Big] \d s,\\
\tilde{g}_t(\theta) &= g_0(\theta) + \int_0^t \tilde{g}_s(\theta) \Big[\N_\tau \ast \frac{\bar{\tilde{q}}_s}{\N_\tau \ast g_s}(\theta) - 1\Big] \d s. 
\end{align*}
Then for any $t\in[0,T]$, 
\begin{align}\label{ineq:g_bound_diff_q}
\notag\int_{-M}^M |g_t(\theta) - \tilde{g}_t(\theta)|\d\theta &\leq \int_{-M}^M
\Big|\int_0^t \bigg(g_s(\theta)\Big[\N_\tau \ast \frac{\bar{q}_s}{\N_\tau \ast g_s}(\theta) \Big] - \tilde{g}_s(\theta)\Big[\N_\tau \ast \frac{\bar{\tilde{q}}_s}{\N_\tau \ast \tilde{g}_s}(\theta)\Big]\bigg) \d s\Big|\d\theta\\
\notag&\quad + \int_0^t \int_{-M}^M |g_s(\theta) - \tilde{g}_s(\theta)|\d\theta
\d s\\
&\leq (I) + (II) + (III) + 2T \cdot \dTV(g, \tilde{g}),
\end{align}
where
\begin{align*}
(I) &= \int_{-M}^M\int_0^t \big|g_s(\theta) - \tilde{g}_s(\theta)\big| \cdot \N_\tau \ast \frac{\bar{q}_s}{\N_\tau \ast g_s}(\theta) \d s\d\theta,\\
(II) &=  \int_{-M}^M\int_0^t \tilde{g}_s(\theta) \cdot \Big|\N_\tau \ast \frac{\bar{q}_s - \bar{\tilde{q}}_s}{\N_\tau \ast g_s}(\theta)\Big| \d s\d\theta,\\
(III) &= \int_{-M}^M\int_0^t \tilde{g}_s(\theta) \cdot \Big|\N_\tau \ast \Big(\frac{\bar{\tilde{q}}_s}{\N_\tau \ast g_s}  -  \frac{\bar{\tilde{q}}_s}{\N_\tau \ast \tilde{g}_s}\Big)(\theta)\Big|\d s\d\theta.
\end{align*}
Let $C,C'$ etc.\ denote constants depending only on $(M,\tau,K)$ and changing
from instance to instance.
Applying Proposition \ref{prop:Nratiobound}, we have $\N_\tau \ast
\frac{\bar{q}_s}{\N_\tau \ast g_s}(\theta) \leq C\int
e^{(2M/\tau^2)|\varphi|}\bar{q}_s(\varphi) \d \varphi$ for any $\theta\in[-M,M]$, hence
\begin{align*}
(I) \leq Ct \sup_{s \in [0,t]} \Big(\int_{-M}^M\big|g_s(\theta) -
\tilde{g}_s(\theta)\big|\d\theta \cdot \int e^{(2M/\tau^2)|\varphi|}\bar{q}_s(\varphi) \d
\varphi\Big) \leq C'T \cdot d_{\infty,\mathrm{TV}}(g,\tilde{g}).
\end{align*}
Similarly,
\begin{align*}
(III) &\leq \int_{-M}^M \int_0^t \tilde{g}_s(\theta) \cdot
\N_\tau(\theta - \varphi)
\bar{\tilde{q}}_s(\varphi)
\cdot \frac{\int_{-M}^M \N_\tau(\varphi - \theta')|g_s(\theta') -
\tilde{g}_s(\theta')|\d\theta'}{\N_\tau * g_s(\varphi) \cdot
\N_\tau*\tilde{g}_s(\varphi)}\d s\d\theta\\
&\leq Ct\sup_{s \in [0,t]} \Big(\int_{-M}^M\big|g_s(\theta') -
\tilde{g}_s(\theta')\big|\d\theta' \cdot \int
e^{(4M/\tau^2)|\varphi|}\bar{\tilde{q}}_s(\varphi) \d \varphi\Big) \leq C'T \cdot
d_{\infty,\mathrm{TV}}(g,\tilde{g}).
\end{align*}
To bound $(II)$, let $\bvarphi = \{\bvarphi_t\}_{t\in [0,T]}$ and
$\tilde\bvarphi =
\{\tilde\bvarphi_t\}_{t\in [0,T]}$ be two processes on the same probability space
such that $\bvarphi \sim q$ and $\tilde\bvarphi \sim \tilde{q}$. Then
\begin{align*}
(II) &= \int_{-M}^M\int_0^t \tilde{g}_s(\theta) \cdot \Big|\int
\frac{\N_\tau(\theta - \varphi)}{\N_\tau \ast
g_s(\varphi)}\big(\bar{q}_s(\varphi) - \bar{\tilde{q}}_s(\varphi)\big) \d
\varphi\Big| \d s\d\theta\\
&= \int_{-M}^M\int_0^t \tilde{g}_s(\theta) \cdot \big|\frac{1}{p}\sum_{j=1}^p
\E\mathsf{f}_s(\varphi_{s,j};\theta) - \E\mathsf{f}_s(\tilde{\varphi}_{s,j};\theta)\big|\d s\d\theta,
\end{align*}
where $\mathsf{f}_s(\varphi;\theta) =
\frac{\N_\tau(\varphi-\theta)}{\N_\tau \ast g_s(\varphi)}$. For any
$\theta\in[-M,M]$, applying Proposition \ref{prop:Nratiobound},
\begin{align*}
\big|\mathsf{f}'_s(\varphi;\theta)\big| = \frac{1}{\tau^2}\Big|\frac{\N_\tau(\varphi -
\theta)\cdot \theta}{\N_\tau \ast g_s(\varphi)} - \frac{\N_\tau(\varphi-\theta)
\cdot \int \N_\tau(\varphi-\theta')\cdot\theta'g_s(\theta')\d\theta'}{\big(\N_\tau
\ast g_s(\varphi)\big)^2}\Big| \leq Ce^{(4M/\tau^2)|\varphi|}, 
\end{align*}
so it holds that
\begin{align*}
\big|\E\mathsf{f}_s(\varphi_{s,j};\theta) -
\E\mathsf{f}_s(\tilde{\varphi}_{s,j};\theta)\big| &\leq C\,\E
\big(e^{(4M/\tau^2)|\varphi_{s,j}|} +
e^{(4M/\tau^2)|\tilde{\varphi}_{s,j}|}\big) \cdot |\varphi_{s,j} -
\tilde{\varphi}_{s,j}|\\
&\leq C\big(\E e^{(8M/\tau^2)|\varphi_{s,j}|} + e^{(8M/\tau^2)|\tilde{\varphi}_{s,j}|}\big)^{1/2} \cdot
\big(\E (\varphi_{s,j} - \tilde{\varphi}_{s,j})^2\big)^{1/2}\\
& \leq C'\big(\E (\varphi_{s,j} - \tilde{\varphi}_{s,j})^2\big)^{1/2},
\end{align*}
using the condition \prettyref{eq:qtMGF} for $\{q_t\}$ and
$\{\tilde{q}_t\}$. Therefore,
\begin{align*}
(II) \leq C\int_0^t \frac{1}{p}\sum_{j=1}^p \big[\E (\varphi_{s,j} -
\tilde{\varphi}_{s,j})^2\big]^{1/2}\d s
\leq C'T \cdot \big(\E \sup_{s\in[0,T]} \pnorm{\bvarphi_{s} - \tilde{\bvarphi}_{s}}{2}^2\big)^{1/2}. 
\end{align*}
Taking the infimum over couplings of $q,\tilde q \in \mathcal{L}_q$ yields
$(II) \leq C'T \cdot d_W(q,\tilde q)$.
Plugging the bounds for $(I), (II), (III)$ into (\ref{ineq:g_bound_diff_q})
yields, for some $C_0>0$ (not depending on $T$),
\begin{align*}
d_{\infty,\mathrm{TV}}(g,\tilde{g}) \leq C_0T \cdot (d_{\infty,\mathrm{TV}}(g,\tilde{g}) + d_{W}(q,\tilde{q}))
\end{align*}
Then choosing $T \leq 1/(2C_0)$ and rearranging concludes the proof.
\end{proof}

An immediate corollary is the existence and uniqueness of the univariate g-flow
$\{g_t\}_{t \geq 0}$ in the setting of Theorem \ref{thm:g_flow}.

\begin{corollary}\label{cor:exist_gflow}
Fix any $T>0$.
Let $\bar q \in \cP(\R)$ satisfy $\int e^{\lambda \varphi}\bar
q(\varphi)\d\varphi<\infty$ for all $\lambda \in \R$. Then, given any initial
density $g_0 \in \cP_*(M)$ that is strictly positive on $[-M,M]$,
there exists a unique solution $\{g_t\}_{t \in [0,T]}$ to (\ref{def:g_flow})
where $g_t \in \cP_*(M)$ for all $t \in [0,T]$. Furthermore,
$g_t(\theta)\geq g_0(\theta) e^{-t}$ for all $\theta \in [-M,M]$ and $t \in
[0,T]$.
\end{corollary}
\begin{proof}
Let $q \in \mathcal{L}_q$ be the law of a constant process
$\bvarphi_t=\bvarphi_0$ for all $t \in [0,T]$, where coordinates of $\bvarphi_0$ are
i.i.d.\ with distribution $\bar q$. Then the condition (\ref{eq:qtMGF}) for $q$
is satisfied by the given moment generating function condition for $\bar q$, and
$\bar q_t=\bar q$ for all $t \in [0,T]$, so the result follows from
Lemma \ref{lem:exist_g_fixed_q}-(1).
\end{proof}

\subsection{Solution properties of the q-flow}\label{subsec:exist_qflow}

Recall the negative log-posterior-density $U_g(\cdot)$ of $\bvarphi$
from (\ref{Ug_smooth}).
The following result gives the solution properties of the Wasserstein $q$-flow
for any fixed $g=\{g_t\}_{t \in [0,T]} \in \mathcal{L}_g(T)$.

\begin{lemma}\label{lem:exist_q_fixed_g}
\phantom{}\\
\begin{enumerate}
\item Fix any $T > 0$, and let $g = \{g_t\}_{t\in [0,T]} \in \mathcal{L}_g(T)$. 
Then for any initialization $q_0 \in \cP(\R^p)$,
there is a strong solution $\bvarphi=\{\bvarphi_t\}_{t \in [0,T]}$ to the Langevin diffusion
\begin{align}\label{def:fix_g_sde}
\d \bvarphi_t = -\nabla U_{g_t}(\bvarphi_t) \d t + \sqrt{2}\,\d \B_t,
\end{align}
which is pathwise unique given the initialization $\bvarphi_0 \sim q_0$.

The marginal laws
$\{q_t\}_{t \in [0,T]}$ of $\{\bvarphi_t\}_{t \in [0,T]}$
form the unique weak solution valued in $\cP(\R^p)$ to the Fokker-Planck PDE
\begin{equation}\label{eq:fokkerplanckappendix}
\partial_t q_t(\bvarphi) = \nabla \cdot [q_t(\bvarphi) \nabla
U_{g_t}(\bvarphi)] + \Delta q_t(\bvarphi),
\end{equation}
in the sense that $q_t \to q_0$ weakly as $t \to 0$ and
\begin{equation}\label{eq:weaksolution}
0=\int_0^t
\int \Big(\partial_s \zeta_s(\bvarphi)-\nabla U_{g_s}(\bvarphi)^\top
\nabla \zeta_s(\bvarphi)+\Delta \zeta_s(\bvarphi)\Big)q_s(\d\bvarphi)\d s
\end{equation}
for any smooth function $(t,\varphi) \mapsto \zeta_t(\bvarphi)$ with compact
support in $(0,T) \times \R^p$.\\

\item Let $q_0 \in \cP(\R^p)$ satisfy
$\E_{\bvarphi_0 \sim q_0} e^{\blambda^\top \bvarphi_0}<\infty$
for all $\blambda \in \R^p$. Then, for each $\blambda \in \R^p$,
there exists some $C(q_0,\blambda)<\infty$ (depending also on
$n,p,\X,\y,\sigma,\tau,M$) such that for any $T>0$ and $g \in \mathcal{L}_g(T)$,
the law $q \in \mathcal{L}_q(T)$ of the solution $\{\bvarphi_t\}_{t \in [0,T]}$
of Part (1) satisfies
\begin{align}\label{ineq:mgf_bound}
\sup_{t \in [0,T]} \E_{\bvarphi_t \sim q_t} e^{\blambda^\top
\bvarphi_t}<C(q_0,\blambda).
\end{align}
In particular, there exists some $K(q_0)<\infty$ such that for any $T>0$
and $g \in \mathcal{L}_g(T)$,
\begin{align}\label{ineq:mgf_bound_component}
\sup_{t \in [0,T]} \max_{j\in[p]} \E_{\bvarphi_t \sim q_t}
e^{(8M/\tau^2)|\varphi_{t,j}|} \leq K(q_0).
\end{align}

\item For any $K>0$, there exist $T_0,C_0>0$ (depending on $K$ and also
$(n,p,\X,\y,\sigma,\tau,M)$)
such that the following holds: Let $q_0 \in \cP(\R^p)$
be any initialization as in Part (2), such that 
(\ref{ineq:mgf_bound_component}) holds with $K(q_0)=K$.
Then for any $T \leq T_0$ and $g,\tilde{g} \in \mathcal{L}_g(T)$, if
$q,\tilde{q} \in \mathcal{L}_q(T)$ are the associated solutions of
(\ref{def:fix_g_sde}) given by Part (1), then
\begin{align*}
d_W(q,\tilde{q}) \leq C_0T \cdot d_{\infty,\mathrm{TV}}(g,\tilde{g}).
\end{align*}
\end{enumerate}
\end{lemma}
\begin{proof}[Proof of Lemma \ref{lem:exist_q_fixed_g}-\textbf{Part (1)}]
We show that $\nabla U_{g}: \R^p \rightarrow \R^p$ is Lipschitz uniformly
over $g \in \cP(M)$. Observe that by Tweedie's formula,
\begin{align}
\nabla U_{g}(\bvarphi)&=-\X^\top \bSigma^{-1}(\y - \X \bvarphi) -
\left(\frac{[\N_\tau \ast g]'(\varphi_j)}{[\N_\tau \ast
g](\varphi_j)}\right)_{j=1}^p\label{eq:Uratiog}\\
&=\big(\X^\top \bSigma^{-1}\X+\tau^{-2}\Id\big)\bvarphi-\X^\top \bSigma^{-1}\y
-\tau^{-2}\theta_g(\bvarphi)\label{eq:Upostmean}
\end{align}
where $\theta_g(\varphi)=\E_g[\theta \mid \varphi]$ is the posterior mean in the
scalar model
\[\varphi=\theta+z, \qquad \theta \sim g, \qquad z \sim \N(0,\tau^2),\]
and $\theta_g(\bvarphi)$ is its entrywise application to
$\bvarphi=(\varphi_1,\ldots,\varphi_p)$. We have
$|\theta_g'(\varphi)|=|\tau^{-2}\Var[\theta \mid \varphi]| \leq M^2/\tau^2$.
Then $\pnorm{\nabla U_{g}(\bvarphi) - \nabla U_{g}(\tilde\bvarphi)}{2} \leq
C \pnorm{\bvarphi-\tilde\bvarphi}{2}$ for some $C>0$ (depending on
$\X,\y,\sigma,\tau,M$), uniformly over $g \in \cP(M)$, as claimed.

From the weak continuity of $t \mapsto g_t$ and the form
(\ref{eq:Uratiog}) for $U_g$, we see that $(t,\bvarphi)
\mapsto U_{g_t}(\bvarphi)$ is jointly continuous.
Then the existence of a pathwise unique strong
solution $\bvarphi=\{\bvarphi_t\}_{t \in [0,T]}$
to the Langevin diffusion (\ref{def:fix_g_sde}) follows from
\cite[Theorem 8.3]{legall2016brownian}. By Ito's formula, for any smooth bounded
function $\zeta_t(\bvarphi)$, we have
\begin{equation}\label{eq:ito}
\E[\zeta_t(\bvarphi_t)-\zeta_0(\bvarphi_0)]=\E\left[\int_0^t
\Big(\partial_s \zeta_s(\bvarphi_s)
-\nabla U_{g_s}(\bvarphi_s)^\top \nabla \zeta_t(\bvarphi_s)
+\Delta \zeta_s(\bvarphi_s) \Big)\d s\right]
\end{equation}
and the statement (\ref{eq:weaksolution}) follows from
applying the boundary condition $\zeta_t(\bvarphi_t)=\zeta_0(\bvarphi_0)=0$.
This is the unique weak solution to
(\ref{eq:fokkerplanckappendix}) for which every $q_t \in \cP(\R^p)$,
by \cite[Theorem 4.9]{bogachev2011uniqueness}.
\end{proof}

\begin{proof}[Proof of Lemma \ref{lem:exist_q_fixed_g}-\textbf{Part (2)}]
Let $\{\bvarphi_t\}_{t \in [0,T]}$ be a solution of the Langevin diffusion
in Part (1), with initialization $\bvarphi_0 \sim q_0$.
Let $\bGamma = \X^\top \bSigma^{-1}\X + \tau^{-2}\Id \in \R^{p\times p}$, and
let $\u \in
\R^p$ be an eigenvector of $\bGamma$ (having arbitrary $\ell_2$ norm) with eigenvalue
$\mu>0$.
Fix some $B > 0$ to be specified later. Then
\begin{align*}
\E e^{\u^\top \bvarphi_t} = \E e^{\u^\top \bvarphi_t}\bm{1}_{\u^\top \bvarphi_t
\leq B} + \E e^{\u^\top \bvarphi_t}\bm{1}_{\u^\top \bvarphi_t > B} \leq e^{B} +
\E e^{\u^\top \bvarphi_t}\bm{1}_{\u^\top \bvarphi_t > B}. 
\end{align*}

To bound the second term, let $h$ be a twice
differentiable function such that: (i) $h(u) = 0$ for $u\leq B -1$; (ii)
$h(u) = 1$ for $u \geq B$; (iii) $h(u),h'(u) \geq 0$ for $u\in\R$; (iv)
$\sup_{u\in \R} \max(h'(u),|h''(u)|)<C$ for some $C > 0$.
Then with $f(\bvarphi)=e^{\u^\top \bvarphi}h(\u^\top \bvarphi)$,
\begin{align*}
\E e^{\u^\top \bvarphi_t}\bm{1}_{\u^\top \bvarphi_t > B} \leq \E e^{\u^\top
\bvarphi_t}h(\u^\top \bvarphi_t) = \E f(\bvarphi_t). 
\end{align*}
Using Ito's formula, we have
\begin{align}\label{eq:mgf_ito}
\notag \frac{\d}{\d t} \E f(\bvarphi_t) &= \E \iprod{\nabla f(\bvarphi_t)}{\d \bvarphi_t} + \E
\Tr(\nabla^2 f(\bvarphi_t))\\
&= \E \iprod{\nabla f(\bvarphi_t)}{-\nabla U_{g_t}(\bvarphi_t)}+ \E
\Tr(\nabla^2 f(\bvarphi_t)).
\end{align}
For the first term, note from (\ref{eq:Upostmean}) that $-\nabla
U_{g_t}(\bvarphi) = -\bGamma \bvarphi+ \s$ where
$\s=\X^\top \bSigma^{-1}\y +\tau^{-2}\theta_{g_t}(\bvarphi)$
Since $g_t$ is supported on $[-M,M]$, there exists some $C>0$
(depending on $\X,\y,\sigma,\tau,M$)
such that $\pnorm{\s}{}<C$. Hence using
\begin{align*}
\nabla f(\bvarphi) = \u e^{\u^\top \bvarphi}h(\u^\top \bvarphi) + \u e^{\u^\top
\bvarphi} h'(\u^\top \bvarphi) = \u \cdot \big(f(\bvarphi) + e^{\u^\top
\bvarphi} h'(\u^\top \bvarphi)\big),
\end{align*} 
and applying $\bGamma \u=\mu \cdot \u$, $f(\cdot) \geq 0$, and $h'(\cdot) \geq 0$,
we have
\begin{align*}
\iprod{\nabla f(\bvarphi)}{-\nabla U_{g_t}(\bvarphi)} &= -\big(f(\bvarphi) +
e^{\u^\top \bvarphi} h'(\u^\top \bvarphi)\big) \cdot \u^\top\bGamma \bvarphi +
\big(f(\bvarphi) + e^{\u^\top \bvarphi} h'(\u^\top \bvarphi)\big) \u^\top \s\\
&\leq -(\mu \cdot \u^\top \bvarphi - \pnorm{\u}{}\pnorm{\s}{}) f(\bvarphi) - \mu
e^{\u^\top \bvarphi} h'(\u^\top \bvarphi) \cdot \u^\top \bvarphi +
\pnorm{\u}{}\pnorm{\s}{} \cdot e^{\u^\top \bvarphi} h'(\u^\top \bvarphi).
\end{align*}
Choose $B=B(\pnorm{\u}{},\mu) > 1$ large enough such that
$\mu \cdot \u^\top \bvarphi-\|\u\|\,\|\s\|>B\mu/2$ whenever $\bvarphi \in \R^p$ satisfies
$\u^\top \bvarphi > B - 1$. Then, the first term is bounded above by
$-B\mu/2 \cdot f(\bvarphi)$ (which holds also when $\u^\top \bvarphi \leq B-1$, in which
case $f(\bvarphi) = 0$). The second term $-\mu e^{\u^\top \bvarphi} h'(\u^\top
\bvarphi) \cdot \u^\top \bvarphi$ is non-positive, because $h'(\u^\top
\bvarphi) \geq 0$ and equals 0 unless $\u^\top \bvarphi>B-1>0$.
For the third term, since also $h'(\u^\top \bvarphi) = 0$ unless $\u^\top
\bvarphi \leq B$, it is
bounded by some $C = C(\pnorm{\u}{},\mu) > 0$. In summary, 
\begin{align*}
\E \iprod{\nabla f(\bvarphi_t)}{-\nabla U_{g_t}(\bvarphi_t)} \leq
-\frac{B\mu}{2}\,\E f(\bvarphi_t) + C.
\end{align*}
On the other hand, since
\begin{align*}
\nabla^2 f(\bvarphi) = \u\u^\top\Big(e^{\u^\top \bvarphi}h(\u^\top \bvarphi) +
2e^{\u^\top \bvarphi} h'(\u^\top \bvarphi) + e^{\u^\top \bvarphi}h''(\u^\top
\bvarphi)\Big), 
\end{align*}
we have
\begin{align*}
\tr\big(\nabla^2 f(\bvarphi)\big) = \pnorm{\u}{}^2 \cdot f(\bvarphi) +
\pnorm{\u}{}^2\Big(2e^{\u^\top \bvarphi} h'(\u^\top \bvarphi) + e^{\u^\top
\bvarphi}h''(\u^\top \bvarphi)\Big)
\end{align*}
Using $h'(\u^\top \bvarphi) = h''(\u^\top \bvarphi) = 0$ unless $\u^\top
\bvarphi \leq B$, and choosing
$B=B(\|\u\|,\mu)$ large enough so that $B\mu\geq 4\pnorm{\u}{}^2$, for some
$C=C(\|\u\|,\mu) > 0$ we have
\begin{align*}
\E \tr(\nabla^2 f(\bvarphi_t)) \leq \frac{B\mu}{4}\E f(\bvarphi_t) + C. 
\end{align*}
Plugging the above bounds into (\ref{eq:mgf_ito}), we have
\begin{align*}
\frac{\d}{\d t} \E f(\bvarphi_t) \leq -\frac{B\mu}{4}\E f(\bvarphi_t) + C'. 
\end{align*}
This implies that $\E f(\bvarphi_t) \leq x(t)$ where
$x(t)=(\E[f(\bvarphi_0)]+4C'/B\mu)e^{-tB\mu/4}-4C'/B\mu$ is the solution of
$x'(t)=-(B\mu/4)x(t)+C'$ with initialization $x(0)=\E f(\bvarphi_0)$.
If $\E e^{\u^\top \bvarphi_0}<A$, then also
$\E f(\bvarphi_0) \leq \E e^{\u^\top \bvarphi_0}<A$
and $\E e^{\u^\top \bvarphi_t} \leq e^B+\E f(\bvarphi_t)$, so this shows
that there exists some $C=C(A,\|\u\|,\mu)>0$ for which
$\E e^{\u^\top \bvarphi_t}<C$ and every $t \in [0,T]$.

Finally, for any $\blambda \in \R^p$, we may write $\blambda=\sum_{j=1}^p \u_j$
where $\u_1,\ldots,\u_p$ are eigenvectors of $\bGamma$. Under the given
finite moment-generating-function
condition for $q_0$, there exists some $A=A(q_0,\blambda)>0$ for which
$\E_{\bvarphi_0 \sim q_0} e^{p\u_j^\top\bvarphi_0}<A$ for every $j=1,\ldots,p$.
Then for any $t \in [0,T]$, applying the above with $\u=p\u_j$ yields
\begin{align*}
\E e^{\blambda^\top \bvarphi_t}=\E \prod_{j=1}^p e^{\u_j^\top \bvarphi_t} \leq
\prod_{j=1}^p \big(\E e^{p\u_j^\top \bvarphi_t} \big)^{1/p}<C(q_0,\blambda)
\end{align*}
for some $C(q_0,\blambda)<\infty$. This shows (\ref{ineq:mgf_bound}),
and specializing to $\blambda \in \{\pm (8M/\tau^2)\e_j\}_{j \in [p]}$
for the standard basis vectors $\e_j$ shows (\ref{ineq:mgf_bound_component}).
\end{proof}

\begin{proof}[Proof of Lemma \ref{lem:exist_q_fixed_g}-\textbf{Part (3)}]
On a probability space containing $\bvarphi_0 \sim q_0$ and a standard
Brownian motion $\{\B_t\}_{t \in [0,T]}$, define
$\bvarphi=\{\bvarphi_t\}_{t \in [0,T]}$ and
$\tilde\bvarphi=\{\tilde{\bvarphi}_t\}_{t \in [0,T]}$ as the (pathwise unique)
solutions to
\begin{align*}
\bvarphi_t &= \bvarphi_0 + \int_0^t -\nabla U_{g_s}(\bvarphi_s)\d s + \sqrt{2}
\B_t, \\
\tilde{\bvarphi}_t &= \bvarphi_0 + \int_0^t -\nabla
U_{\tilde{g}_s}(\tilde{\bvarphi}_s)\d s + \sqrt{2} \B_t.
\end{align*}
Thus $\bvarphi \sim q$ and $\tilde \bvarphi \sim \tilde q$.
For any $t\in[0,T]$, 
\begin{align*}
\sup_{t\in[0,T]}\pnorm{\bvarphi_t - \tilde{\bvarphi}_t}{}^2 &= \sup_{t\in[0,T]}\left\|\int_0^t
\big(\nabla U_{g_s}(\bvarphi_s) - \nabla U_{\tilde{g}_s}(\tilde{\bvarphi}_s)\big) \d
s\right\|^2\\
&\leq T \cdot \sup_{t\in[0,T]} \int_0^t \biggpnorm{\nabla U_{g_s}(\bvarphi_s) -
\nabla U_{\tilde{g}_s}(\tilde{\bvarphi}_s)}{}^2\d s\\
&\leq 2T^2 \sup_{t\in[0,T]} \biggpnorm{\nabla U_{\tilde{g}_t}(\bvarphi_t) -
\nabla U_{\tilde{g}_t}(\tilde{\bvarphi}_t)}{}^2 + 2T \int_0^T
\biggpnorm{\nabla U_{g_t}(\bvarphi_t) - \nabla
U_{\tilde{g}_t}(\bvarphi_t)}{}^2\d t.
\end{align*}

As shown in the proof of Part (1),
$\nabla U_g: \R^p \rightarrow \R^p$ is Lipschitz uniformly over $g \in \cP(M)$,
so the first term is at most $CT^2\sup_{t\in[0,T]}\pnorm{\bvarphi_t -
\tilde{\bvarphi}_t}{}^2 \leq \frac{1}{2} \sup_{t\in[0,T]}\pnorm{\bvarphi_t -
\tilde{\bvarphi}_t}{}^2$ for some $C,T_0>0$ and any $T \leq T_0$.
For the second term, using (\ref{eq:Upostmean}),
\begin{align*}
\pnorm{\nabla U_{g_t}(\bvarphi_t) - \nabla U_{\tilde{g}_t}(\bvarphi_t)}{}^2 \leq
C\sum_{j=1}^p \big(\theta_{g_t}(\varphi_{t,j})-\theta_{\tilde{g}_t}(\varphi_{t,j})\big)^2.
\end{align*}
Applying Proposition \ref{prop:Nratiobound}, we have
\begin{align*}
\big|\theta_{g}(\varphi) - \theta_{\tilde{g}}(\varphi)\big|
&=\Big|\frac{\int_{-M}^M \theta \cdot \N_\tau(\varphi - \theta) g(\theta)\d
\theta}{\int_{-M}^M
\N_\tau(\varphi - \theta)g(\theta)\d \theta} - \frac{\int_{-M}^M \theta\cdot \N_\tau(\varphi -
\theta) \tilde{g}(\theta)\d \theta}{\int_{-M}^M \N_\tau(\varphi -
\theta)\tilde{g}(\theta)\d \theta}\Big|\\
&\leq \frac{\int_{-M}^M \theta\cdot \N_\tau(\varphi - \theta) |g(\theta) - \tilde{g}(\theta)|\d
\theta}{\int_{-M}^M \N_\tau(\varphi - \theta)g(\theta)\d \theta}\\
&\quad + \int_{-M}^M \theta\cdot \N_\tau(\varphi - \theta) \tilde{g}(\theta)\d \theta \cdot
\frac{\int_{-M}^M \N_\tau(\varphi - \theta)|g(\theta) - \tilde{g}(\theta)|\d
\theta}{\int_{-M}^M
\N_\tau(\varphi - \theta)g(\theta)\d \theta \cdot \int_{-M}^M \N_\tau(\varphi -
\theta)\tilde{g}(\theta)\d \theta}\\
&\leq Ce^{(4M/\tau^2)|\varphi|} \cdot \int_{-M}^M |g(\theta) -
\tilde{g}(\theta)| \d \theta.
\end{align*}
In summary, there exist $C',T_0>0$ such that for any $T \leq T_0$,
\begin{align*}
\sup_{t\in[0,T]} \pnorm{\bvarphi_t - \tilde{\bvarphi}_t}{}^2  \leq C'T
\int_0^T \sum_{j=1}^p
e^{(4M/\tau^2)|\varphi_{t,j}|}\d t \cdot d_{\infty,\mathrm{TV}}(g,\tilde{g})^2.
\end{align*}
Taking expectation on both sides and
using the bound (\ref{ineq:mgf_bound_component}) with $K(q_0)=K$ gives
$d_W^2(q,\tilde q) \leq C_0^2T^2 d_{\infty,\mathrm{TV}}(g,\tilde{g})^2$,
for some $C_0>0$ depending on $K$. This completes the proof.
\end{proof}

We now provide a bootstrapping lemma that will allow us to inductively establish the
smoothness of the preceding solutions $\{q_t,g_t\}$ in both space and time.

\begin{lemma}\label{lemma:smoothness}
In the setting of Lemma \ref{lem:exist_q_fixed_g}, suppose in addition that
$g_t \in \cP_*(M)$ admits a density for each $t \in [0,T]$, where
$t \mapsto g_t(\theta)$ belongs to $C^r((0,T))$ for some $r \geq 0$ and every
$\theta \in [-M,M]$,
and $\sup_{t \in (0,T)}\int_{-M}^M |\frac{\d^a}{\d t^a}
g_t(\theta)|\d\theta<\infty$ for all $a=0,\ldots,r$. Then:
\begin{enumerate}
\item The map
$t \mapsto f_t(\theta):=\N_\tau*\frac{\bar q_t}{\N_\tau*g_t}(\theta)$ belongs
also to $C^r((0,T))$, and for all $a=0,\ldots,r$,
\begin{equation}\label{eq:dtguniformbound}
\sup_{t \in (0,T)} \sup_{\theta \in [-M,M]}
\left|\frac{\d^a}{\d t^a} f_t(\theta)\right|<\infty.
\end{equation}
\item If $q_0 \in \cP_*(\R^p)$ has a continuous bounded density, 
then the marginal laws $\{q_t\}_{t \in [0,T]}$ of $\{\bvarphi_t\}_{t \in [0,T]}$
admit continuous densities bounded uniformly over $t \in [0,T]$. Furthermore if $r \geq 1$, then
$(t,\bvarphi) \mapsto q_t(\bvarphi)$ belongs to $C^{r-1,\infty}((0,T) \times
\R^p)$.
\end{enumerate}
\end{lemma}
\begin{proof}
For Part (1), we claim by induction on $a \in \{0,1,\ldots,r\}$
that $f_t(\theta)$ belongs to
$C^a((0,T))$, the bound (\ref{eq:dtguniformbound}) holds for $a$, and
$\frac{\d^a}{\d t^a} f_t(\theta)=\E_{\bvarphi \sim
q_t}[F_{t,\theta}^a(\bvarphi)]$ for some $F_{t,\theta}^a(\bvarphi)$
that is a polynomial of the arguments
\begin{equation}\label{eq:polyterms}
\varphi_k, \qquad
\frac{\int_{-M}^M \frac{\d^i}{\d \varphi_k^i} \N_\tau(\varphi_k-\theta)
\frac{\d^j}{\d t^j} g_t(\theta)\d\theta}{\N_\tau * g_t(\varphi_k)}
\quad \text{ and } \quad
\frac{\frac{\d^i}{\d \varphi_k^i} \N_\tau(\varphi_k-\theta)}{\N_\tau * g_t(\varphi_k)}
\end{equation}
for $k \in \{1,\ldots,p\}$, $i \geq 0$, and $j \in \{0,\ldots,a\}$.

For the base case $a=0$, observe that
\[f_t(\theta)=\N_\tau*\frac{\bar q_t}{\N_\tau*g_t}(\theta)
=\E_{\bvarphi \sim q_t}[F_{t,\theta}^0(\bvarphi)],
\qquad F_{t,\theta}^0(\bvarphi):=\frac{1}{p}\sum_{k=1}^p
\frac{\N_\tau(\varphi_k-\theta)}{\N_\tau*g_t(\varphi_k)}.\]
By the bound $\sup_{t \in [0,T]} \sup_{\theta \in [-M,M]}
|\N_\tau(\varphi_k-\theta)/\N_\tau*g_t(\varphi_k)|
\leq Ce^{(2M/\tau^2)|\varphi|}$ from Proposition \ref{prop:Nratiobound} and
Lemma \ref{lem:exist_q_fixed_g}-(2),
we see that $\sup_{t \in [0,T]} \sup_{\theta \in [-M,M]}
|f_t(\theta)|<\infty$, and the maps
$\bvarphi \mapsto F_{t,\theta}^0(\bvarphi)q_t(\bvarphi)$ for $t \in [0,T]$ are
uniformly integrable, hence $f_t(\theta)$ is continuous in $t$. So the inductive
claims hold for $a=0$. Assuming they hold for $a-1$ where $1 \leq a \leq r$, we
have that
$F_{t,\theta}^{a-1}(\bvarphi)$ belongs to $C^{1,\infty}((0,T) \times \R^p)$.
Then by Ito's formula,
\[\frac{\d^a}{\d t^a} f_t(\theta)
=\frac{\d}{\d t} \E[F_{t,\theta}^{a-1}(\bvarphi)]
=\E_{\bvarphi \sim q_t}\left[\left(\frac{\d}{\d t} F_{t,\theta}^{a-1}
-\nabla U_{g_t}^\top \nabla F_{t,\theta}^{a-1}+\Delta F_{t,\theta}^{a-1}\right)
(\bvarphi)\right].\]
From the form (\ref{eq:Uratiog}), we have that each coordinate of
$\nabla U_{g_t}(\bvarphi)$ is a polynomial of terms of the form
(\ref{eq:polyterms}), and it is easily verified
that differentiating any term (\ref{eq:polyterms}) in either $t$ or $\varphi_k$
yields again a polynomial of terms (\ref{eq:polyterms}), where differentiating
in $t$ increases the maximum order $j$ of the time derivative of $g_t(\theta)$
by at most 1. Thus the above may be written as
$\frac{\d^a}{\d t^a} f_t(\theta)=\E_{\bvarphi \sim
q_t}[F_{t,\theta}^a(\bvarphi)]$ where $F_{t,\theta}^a(\bvarphi)$ is a
polynomial of the terms (\ref{eq:polyterms}) with $j \leq a$.
By Proposition \ref{prop:Nratiobound}, for each fixed $i \geq 0$ we have
$\sup_{t \in [0,T]}\sup_{\theta \in [-M,M]}|\frac{\d^i}{\d \varphi_k^i} \N_\tau(\varphi_k-\theta)/\N*g_t(\varphi_k)|
\leq Ce^{C|\varphi_k|}$ for some $C>0$.
Then, applying the given condition $\sup_{t \in (0,T)}
\int_{-M}^M |\frac{\d^a}{\d t^a}g_t(\theta)|\d\theta<\infty$ and
Lemma \ref{lem:exist_q_fixed_g}-(2),
we get $\sup_{t \in (0,T)} \sup_{\theta \in [-M,M]}
|\frac{\d^a}{\d t^a} f_t(\theta)|<\infty$. Furthermore, 
$\bvarphi \mapsto F_{t,\theta}^a(\bvarphi)q_t(\bvarphi)$ is uniformly
integrable over $t \in (0,T)$,
so $\frac{\d^a}{\d t^a} f_t(\theta)$ is continuous
in $t$. This completes the induction, showing the inductive claim
for all $a=0,\ldots,r$, and hence establishes Part (1).

For Part (2), let us write $q_t(\bvarphi) \in H^{j,k} \equiv H^{j,k}((0,T)
\times \R^p)$ if all components of the weak derivatives $D_t^a D_\bvarphi^b
q_t(\bvarphi)$ for $a \leq j$ and $b \leq k$ exist and belong to $L^r((0,T)
\times \R^p)$ (defined with respect to Lebesgue measure) for every $r<\infty$.
We write $q_t(\bvarphi) \in H_\text{loc}^{j,k}$ if for any compact subset $S
\subset (0,T) \times \R^p$, these derivatives belong to $L^r(S)$ for every
$r<\infty$. By Holder's inequality and the Leibniz rule for weak
differentiation, $H_\text{loc}^{j,k}$ is closed under sums and products.
Lemma \ref{lem:exist_q_fixed_g}-(2) implies that for any fixed
$k \geq 0$, some $C,C'>0$, and all $t \in [0,T]$,
\begin{equation}\label{eq:phiboundedmoments}
\E_{\bvarphi \sim q_t}[\|\bvarphi\|_2^k]
\leq C\sum_{j=1}^p \E_{\bvarphi \sim q_t}[|\varphi_j|^k] \leq C'.
\end{equation}
Then the expression (\ref{eq:Upostmean}) for $\nabla U_{g_t}$ 
implies, for any fixed $k>0$, that
$\E_{\bvarphi \sim q_t}[\|\nabla U_{g_t}(\bvarphi)\|^k]<C$ uniformly over $t \in [0,T]$.
Then \cite[Corollaries 6.4.3, 7.3.8]{bogachev2022fokker} show that
$\{q_t\}_{t \in [0,T]}$ admit continuous densities with respect to Lebesgue
measure, uniformly bounded over $t \in [0,T]$, which further satisfy
\begin{equation}\label{eq:qtbounded}
q_t(\bvarphi) \in H_\text{loc}^{0,1}.
\end{equation}
To improve this smoothness to $H_{\text{loc}}^{r,\infty}$, we
apply a similar argument as in \cite[Theorem 5.1]{jordan1998variational}:
Fix $(t,\bvarphi) \in (0,T) \times \R^p$ and apply 
(\ref{eq:ito}) with
$\zeta_s(\x)=\omega(s)\eta(\x)\N_{2(\delta+t-s)}(\bvarphi-\x)$,
where $\delta>0$, $\eta:\R^p \to [0,1]$ is a smooth compactly supported
function on $\R^p$, and $\omega:\R \to [0,1]$ is a smooth compactly supported
function satisfying $\omega(s)=0$ for all $s \leq 0$ and $\omega(t)=1$.
Then, applying $\partial_s \N_{2(\delta+t-s)}(\bvarphi-\x)
+\Delta \N_{2(\delta+t-s)}(\bvarphi-\x)=0$ (the heat equation) and the chain
rule, (\ref{eq:ito}) yields
\begin{align*}
&\int \eta(\x)q_t(\x)\N_{2\delta}(\bvarphi-\x)\d\x\\
&=\int_{-\infty}^t \int \Big[\omega'(s)\eta(\x)-\omega(s)\nabla U_{g_s}(\x)^\top
\nabla \eta(\x)+\omega(s)\Delta \eta(\x)\Big]\N_{2(\delta+t-s)}(\bvarphi-\x)
q_s(\x)\d\x\d s\\
&\hspace{1in}+\int_{-\infty}^t \int \Big[\omega(s)\eta(\x)\nabla
U_{g_s}(\x)-\omega(s)\nabla \eta(\x)\Big]^\top
\nabla \N_{2(\delta+t-s)}(\bvarphi-\x)q_s(\x)\d\x\d s.
\end{align*}
Applying the weak differentiability in $\x$ of $q_s(\x)$
from (\ref{eq:qtbounded})
to integrate by parts in the second term, and taking the limit $\delta \to 0$,
this yields
\begin{align}
\eta(\bvarphi)q_t(\bvarphi)
&=\int_{-\infty}^t \Big[\Big(q_s[\omega'(s)\eta-\omega(s)\nabla U_{g_s}^\top \nabla
\eta+\omega(s)\Delta \eta]\nonumber\\
&\hspace{1in}-\nabla \cdot (q_s[\omega(s)\eta\nabla U_{g_s}-\omega(s)\nabla \eta])
\Big)*\N_{t-s}\Big](\bvarphi)\d s.\label{eq:bootstrap}
\end{align}
The estimate of
\cite[Chapter 4, (3.1)]{ladyzhenskaia1968linear} shows the implication
\begin{equation}\label{eq:boosting}
f_t(\bvarphi) \in H^{j,k} \text{ for all } 2j+k \leq 2m \Rightarrow
\int_{-\infty}^t [f_s*\N_{t-s}](\bvarphi)\d s \in H^{j,k}
\text{ for all } 2j+k \leq 2m+2.
\end{equation}
From (\ref{eq:Uratiog}) and the given condition
$\sup_{t \in (0,T)}
\int_{-M}^M |\frac{\d^a}{\d t^a} g_t(\theta)|\d\theta<\infty$
for each $a \leq r$, it is direct to check
that each coordinate of $(t,\bvarphi) \mapsto \nabla U_{g_t}(\bvarphi)$ belongs to
$H_\text{loc}^{r,\infty}$. Then (\ref{eq:boosting}) yields the implications
$q_t(\bvarphi) \in H_\text{loc}^{0,1} \Rightarrow q_t(\bvarphi) \in
H_\text{loc}^{0,2} \Rightarrow q_t(\bvarphi) \in
H_\text{loc}^{0,3} \Rightarrow \ldots$ where the initial condition comes from
(\ref{eq:qtbounded}), hence $q_t(\bvarphi) \in H_\text{loc}^{0,\infty}$.
Then (\ref{eq:boosting}) further yields the implications
$q_t(\bvarphi) \in H_\text{loc}^{0,\infty} \Rightarrow
q_t(\bvarphi) \in H_\text{loc}^{1,\infty} \Rightarrow
\ldots \Rightarrow q_t(\bvarphi) \in H_\text{loc}^{r,\infty}$. Finally, by the
Sobolev embedding theorem (c.f.~\cite[Theorem 4.2 Part II]{adams2003sobolev}) we have
$H_\text{loc}^{r,\infty} \subset C^{r-1,\infty}((0,T) \times \R^p)$.
\end{proof}

\subsection{Completing the proof}\label{subsec:joint_flow_existence}

We are now ready to prove Theorem \ref{thm:flow_solution}. 

\begin{proof}[Proof of Theorem \ref{thm:flow_solution}]
We start by showing the existence and uniqueness of a solution of
(\ref{eq:qflow}--\ref{eq:gflow}) on $[0,t_0]$ for some small $t_0$ to be
specified later.
For the given initialization $q_0 \in \cP_*(\R^p)$, let $K>0$ be such that
(\ref{ineq:mgf_bound_component}) holds with $K(q_0)=K$, and define
\begin{align}\label{def:Lq_subset}
\mathcal{L}_q'(t_0)=\{q \in \mathcal{L}_q(t_0): \max_{t\in
[0,t_0]}\max_{j\in[p]} \E_{\bvarphi_t \sim q_t}e^{(8M/\tau^2)|\varphi_{t,j}|}
\leq K\}.
\end{align}
Observe that $\mathcal{L}'_q(t_0)$ is closed in $\mathcal{L}_q(t_0)$. Indeed,
since $d_W$ metrizes weak convergence, for any sequence $q_n \in
\mathcal{L}'_q(t_0)$ converging to $q \in \mathcal{L}_q(t_0)$, it holds that
$q_n \to q$ weakly, and thus
$q_{n,t} \to q_t$ weakly for all $t\in [0,t_0]$. Then by Fatou's lemma,
$\E_{\bvarphi_t \sim q_{t}}e^{(8M/\tau^2)|\varphi_{t,j}|} \leq \liminf_n
\E_{\bvarphi_t \sim q_{n,t}}e^{(8M/\tau^2)|\varphi_{t,j}|} \leq K(q_0)$.
This shows that $\mathcal{L}'_q(t_0)$ is closed, and hence also complete.

Define a map $\Gamma:\mathcal{L}'_q(t_0) \times \mathcal{L}_g(t_0)
\to \mathcal{L}'_q(t_0) \times \mathcal{L}_g(t_0)$ where
$\Gamma(q,g)=(\Gamma_q(g),\Gamma_g(q))$, as follows:
\begin{itemize}
\item $\Gamma_q(g)$ is the law of the unique strong solution to
(\ref{def:fix_g_sde}) defined by $g$, with initialization $q_0$,
as guaranteed by Lemma \ref{lem:exist_q_fixed_g}-(1). Here,
Lemma \ref{lem:exist_q_fixed_g}-(2) ensures that
$\Gamma_q(g) \in \mathcal{L}_q'(t_0)$.
\item $\Gamma_g(q)$ is the
unique solution to (\ref{def:ODE_g}) defined by $q$, with initialization $g_0$,
as guaranteed by Lemma \ref{lem:exist_g_fixed_q}-(1).
\end{itemize}
Equip $\mathcal{L}'_q(t_0) \times \mathcal{L}_g(t_0)$ with the metric
\begin{align*}
d_\ast((q,g), (\tilde q,\tilde g))=d_W(q,\tilde q) +
d_{\infty,\mathrm{TV}}(g,\tilde g). 
\end{align*}
Since both $\mathcal{L}'_q$ and $\mathcal{L}_g$ are complete metric spaces,
$\mathcal{L}'_q \times \mathcal{L}_g$ is also complete.
Moreover, by Lemmas \ref{lem:exist_g_fixed_q}-(2) and
\ref{lem:exist_q_fixed_g}-(3), for any $q,\tilde q \in \mathcal{L}'_q(t_0)$ and
$g,\tilde g \in \mathcal{L}_g(t_0)$, there exist $C_0,T_0>0$ depending on
$K$ for which
\[d_\ast(\Gamma(q,g), \Gamma(\tilde q,\tilde g))
\leq C_0t_0 \cdot d_\ast\big((q,g),(\tilde q,\tilde g)\big)\]
for any $t_0 \leq T_0$. Then, choosing $t_0$ to satisfy also
$t_0 \leq 1/2C_0$, we obtain that $\Gamma$ is a contraction in $d_\ast$.
By the Banach fixed point theorem, there
is a unique fixed point $(q^\ast, g^\ast)
 \in \mathcal{L}_q'(t_0) \times \mathcal{L}_g(t_0)$ of $\Gamma$.

By Lemma \ref{lem:exist_g_fixed_q}-(1), $g^\ast$ is a (classical) solution to
(\ref{eq:gflow}) defined by $q^\ast$ and the initial condition $g_0$, and
by Lemma \ref{lem:exist_q_fixed_g}-(1), $q^\ast$ is a weak solution to
(\ref{eq:fokkerplanckappendix}) defined by $g^\ast$ with the initial condition
$q_0$, in the sense (\ref{eq:weaksolution}).
Suppose inductively that
$t \mapsto g_t^\ast(\theta)$ belongs to $C^r((0,t_0))$
for each fixed $\theta \in [-M,M]$ and that $\sup_{t \in (0,t_0)}
\int_{-M}^M |\frac{\d^a}{\d t^a} g_t^\ast(\theta)|\d\theta<\infty$
for all $a=0,\ldots,r$.  Then Lemma \ref{lemma:smoothness}-(1) shows
that $t\mapsto f_t(\theta)=\N_\tau * \frac{\bar
q_t^\ast}{\N_\tau*g_t^\ast}(\theta)$ belongs also to $C^r((0,t_0))$
and $\sup_{t \in (0,t_0)}\sup_{\theta \in [-M,M]}
|\frac{\d^a}{\d t^a} f_t(\theta)|<\infty$ for each $a=0,\ldots,r$. Then
$\frac{\d^{r+1}}{\d t^{r+1}} g_t(\theta)
=\frac{\d^r}{\d t^r} [g_t(\theta)(f_t(\theta)-1)]$ is continuous in $t$,
and satisfies
\begin{align*}
\sup_{t \in (0,t_0)}
\int_{-M}^M \left|\frac{\d^{r+1}}{\d t^{r+1}} g_t(\theta)\right|\d\theta
&=\sup_{t \in (0,t_0)} \int_{-M}^M \left|\frac{\d^r}{\d t^r}
[g_t(\theta)(f_t(\theta)-1)]\right|\d\theta \\
&\leq \sup_{t \in (0,t_0)} \sum_{a=0}^r \binom{r}{a} \sup_{\theta \in [-M,M]}
\left|\frac{\d^{r-a}}{\d t^{r-a}}(f_t(\theta)-1)\right|
\int_{-M}^M \left|\frac{\d^a}{\d t^a} g_t(\theta)\right|\d\theta<\infty.
\end{align*}
Thus the inductive claim holds for $r+1$. Then
$t \mapsto g_t^\ast(\theta)$ belongs to $C^\infty((0,t_0))$ for each fixed
$\theta$, and $\sup_{t \in (0,t_0)} \int_{-M}^M |\frac{\d^r}{\d t^r}
g_t^\ast(\theta)|\d\theta<\infty$ for each fixed $r \geq 0$.
Then Lemma \ref{lemma:smoothness}-(2) implies that $(t,\bvarphi) \mapsto
q_t^\ast(\bvarphi)$ belongs to $C^\infty((0,t_0) \times \R^p)$.
Applying this smoothness of $q_t^\ast(\bvarphi)$ and
integrating (\ref{eq:weaksolution}) by parts, we obtain that $q^\ast$ satisfies
(\ref{eq:qflow}) pointwise at each $(t,\bvarphi) \in (0,t_0) \times \R^p$.
Furthermore, the uniform boundedness and continuity of $q_t^\ast(\bvarphi)$ and
weak convergence $q_t^\ast \to q_0^\ast$ imply that $q_t^\ast(\bvarphi)
\to q_0^\ast(\bvarphi)$ pointwise over $\bvarphi \in \R^p$, so $\{q_t^\ast\}_{t
\in [0,t_0]}$ is a classical solution to (\ref{eq:qflow}).

Conversely, for any solution $(q,g)$ to (\ref{eq:gflow}--\ref{eq:qflow}),
Lemma \ref{lem:exist_q_fixed_g}-(2) applied
with this $g$ shows that we must have $q \in \mathcal{L}_q'(t_0)$. Then $(q,g)
\in \mathcal{L}_q'(t_0) \times \mathcal{L}_g(t_0)$ is also a fixed point of
$\Gamma$, so by the uniqueness guarantees in Lemmas
\ref{lem:exist_g_fixed_q}-(1) and \ref{lem:exist_q_fixed_g}-(1),
$q=q^\ast$ and $g=g^\ast$ as equalities of elements in
$\mathcal{L}_q'(t_0)$ and $\mathcal{L}_g(t_0)$. Since $(q,g)$ are continuous,
this implies $g_t(\theta)=g^\ast_t(\theta)$ and
$q_t(\bvarphi)=q^\ast_t(\bvarphi)$ pointwise,
so the solution $(q^\ast, g^\ast)$ to (\ref{eq:qflow}--\ref{eq:gflow}) is
unique.

We now extend this solution to $[0,T]$, by applying this argument sequentially.
Let $\{g_t,q_t\}_{t \in [0,t_0]}$ be the above solution over $[0,t_0]$.
For any $T>0$ and extension of $\{g_t\}_{t \in [0,t_0]}$ to
$\{g_t\}_{t \in [0,t_0+T]} \in \mathcal{L}_g(t_0+T)$,
Lemma \ref{lem:exist_q_fixed_g} shows that
there is a unique solution to the diffusion
(\ref{def:fix_g_sde}) over $[0,t_0+T]$
with initialization $q_0$, which satisfies (\ref{ineq:mgf_bound_component})
for all $t \in [0,t_0+T]$ with $K(q_0)=K$ for the above value $K>0$ defining
(\ref{def:Lq_subset}).
The portion of this solution on $[t_0,t_0+T]$ must coincide with the unique
solution to (\ref{eq:fokkerplanckappendix}) over $[0,T]$
that is initialized at $q_{t_0}$ and defined by $\{g_t\}_{t \in [t_0,t_0+T]}$.
So for the initialization $q_{t_0}$, we must have that
 (\ref{ineq:mgf_bound_component}) holds with $K(q_{t_0})=K$ and the same value
$K$ as above. Then we may apply the above argument to $[t_0,2t_0]$
with $(q_{t_0},g_{t_0})$ in place of $(q_0,g_0)$, and then sequentially to
$[2t_0,3t_0]$, etc. This proves the existence and uniqueness of the solution
$\{(q_t,g_t)\}_{t \in [0,T]}$ to (\ref{eq:qflow}--\ref{eq:gflow}), as claimed.
The properties $g_t(\theta)>0$ and $\sup_{t \in [0,T]} \E_{\bvarphi_t \sim q_t}
e^{\blambda^\top \bvarphi_t}<\infty$ follow from the properties of the solutions
guaranteed by Lemmas \ref{lem:exist_g_fixed_q}-(1) and
\ref{lem:exist_q_fixed_g}-(2), and this concludes the proof.
\end{proof}

\subsection{Proof of Lemma \ref{lemma:Fndecays}}\label{sec:Fndecays}

\begin{proof}[Proof of Lemma \ref{lemma:Fndecays}]
We have
\begin{equation}\label{eq:Fnqgform}
F_n(q,g)=\frac{1}{p}\int\Big[U_g(\bvarphi)+\log
q(\bvarphi)\Big]q(\bvarphi)\d\bvarphi
+\frac{1}{2p}\log \det (2\pi\bSigma).
\end{equation}
To show $\limsup_{t \to 0} F_n(q_t,g_t) \leq F_n(q_0,g_0)$, recall from
(\ref{eq:qtbounded}) that for any fixed $T>0$, $q_t(\bvarphi)$ is uniformly
bounded over $t \in [0,T]$. Since $U_{g_t}(\bvarphi)$ is also uniformly bounded
over $\|\bvarphi\|_2 \leq R$, for any $R>0$,
\[\lim_{t \to 0} \int \1\{\|\bvarphi\|_\infty \leq R\}
U_{g_t}(\bvarphi)q_t(\bvarphi)\d\bvarphi
=\int \1\{\|\bvarphi\|_\infty \leq R\}
U_{g_0}(\bvarphi)q_0(\bvarphi)\d\bvarphi\]
by the bounded convergence theorem. Applying
$U_{g_t}(\bvarphi)<C(\|\bvarphi\|_2^2+1)$ and Cauchy-Schwarz,
\[\int \1\{\|\bvarphi\|_\infty>R\}
|U_{g_t}(\bvarphi)|q_t(\bvarphi)\d\bvarphi
\leq \P_{q_t}[\|\bvarphi_t\|_\infty>R]^{1/2}
(\E_{q_t}[\|\bvarphi_t\|^4]+1)^{1/2}.\]
By Lemma \ref{lem:exist_q_fixed_g}-(2), $\E_{q_t}[\|\bvarphi_t\|^4]<C$
and $\P_{q_t}[|\varphi_{t,j}|>R] \leq e^{-R}\E[e^{|\varphi_{t,j}|}]<Ce^{-R}$
for a constant $C>0$, uniformly over $j=1,\ldots,p$ and $t \in [0,T]$.
Thus the above display converges to 0 uniformly over $t \in [0,T]$
as $R \to \infty$, so
\[\lim_{t \to 0} \int U_{g_t}(\bvarphi)q_t(\bvarphi)\d\bvarphi
=\int U_{g_0}(\bvarphi)q_0(\bvarphi)\d\bvarphi.\]
Applying the uniform upper bound $\log q_t(\bvarphi) \leq C_0$ and Fatou's
lemma, we have also
\[\int [C_0-\log q_0(\bvarphi)]q_0(\bvarphi)\d\bvarphi
\leq \liminf_{t \to 0} \int [C_0-\log q_t(\bvarphi)]q_t(\bvarphi)\d\bvarphi,\]
hence $\limsup_{t \to 0} \int (\log q_t)q_t \leq \int (\log q_0)q_0$.
Combining these statements gives $\limsup_{t \to 0} F_n(q_t,g_t) \leq
F_n(q_0,g_0)$.

To show the form (\ref{eq:Fnderiv}) for the derivative of $F_n$,
we will appeal to the following properties of $\{q_t\}_{t \in [0,T]}$: For any
fixed $0<t<T$ and compact subset $K \subset \R^p$,
\begin{align}
\sup_{s \in [t,T]} \sup_{\bvarphi \in \R^p} \frac{|\log q_s(\bvarphi)|}
{\|\bvarphi\|^2+1}&<\infty,\label{eq:logqbound}\\
\sup_{s \in [t,T]} \sup_{\bvarphi \in K} \|\nabla
q_s(\bvarphi)\|_2&<\infty,\label{eq:sobolev1norm}\\
\sup_{s \in [t,T]} \sup_{\bvarphi \in K}
\|\nabla^2 q_s(\bvarphi)\|_{\mathrm{F}}
&<\infty,\label{eq:sobolev2norm}\\
\int_t^T \int_{\R^p}
\|\nabla \log q_s(\bvarphi)\|_2^2 q_s(\bvarphi)
\d\bvarphi\d s&<\infty \label{eq:entropybound}.
\end{align}
The upper bound for $\log q_s(\bvarphi)$ in 
(\ref{eq:logqbound}) follows from (\ref{eq:qtbounded}), and the lower bound
follows from \cite[Example 8.2.3]{bogachev2022fokker}, where we may take
$\beta=1$ in light of the form (\ref{eq:Upostmean}) for the drift
$\nabla U_{g_t}(\bvarphi)$.
The next two statements (\ref{eq:sobolev1norm}--\ref{eq:sobolev2norm}) follow
from the $C^\infty$ smoothness of $(t,\bvarphi) \mapsto q_t(\bvarphi)$.
By Lemma \ref{lem:exist_q_fixed_g}-(2), we must have
$\E_{q_s}[\|\bvarphi_s\|_2^2]<C$ for some $C>0$ uniformly over $s \in
[0,T]$. Then (\ref{eq:logqbound}) implies $\int |\log
q_t(\bvarphi)|q_t(\bvarphi)\d\bvarphi<\infty$, (\ref{eq:Upostmean})  implies
\begin{equation}\label{eq:Usquarebound}
\int \|\nabla U_{g_s}(\bvarphi)\|_2^2 q_s(\bvarphi)\d\bvarphi<C
\end{equation}
for some $C>0$ uniformly over $s \in [0,T]$, and
the last statement (\ref{eq:entropybound}) then follows from
\cite[Theorem 7.4.1]{bogachev2022fokker}.

Let $\boldeta(\bvarphi)=\prod_{j=1}^p \eta(\varphi_j)$ where $\eta$ is
a smooth, compactly supported bump function on $\R$. Define a truncated version of $F_n$ by
\[F_n^\eta(q_t,g_t)=\frac{1}{p}\int \boldeta(\bvarphi)[U_{g_t}(\bvarphi)+\log q_t(\bvarphi)]
q_t(\bvarphi)\d\bvarphi+\frac{n}{2p}\log (2\pi)+\frac{1}{2p}\log \det \bSigma.\]
Then for any $t>0$, we may decompose
\[\frac{\d}{\d t} F_n^\eta(q_t,g_t)=\mathrm{I}^\eta(t)+\mathrm{II}^\eta(t).\]
Here, we set
\begin{align*}
\mathrm{I}^\eta(t)&=\frac{1}{p}
\int \boldeta(\bvarphi) \cdot \frac{\d}{\d t}U_{g_t}(\bvarphi)
\cdot q_t(\bvarphi)\d\bvarphi\\
&=\int \eta(\varphi) \cdot \frac{\d}{\d
t}\Big({-}\log[\N_\tau*g_t](\varphi)\Big) \cdot \bar q_t(\varphi)\d\varphi
={-}\alpha\int 
\bigg(\bigg[\N_\tau*\frac{\eta \bar q_t}{\N_\tau*g_t}\bigg](\theta)-1\bigg)^2
g_t(\theta)\d\theta,
\end{align*}
where these calculations and argument for differentiating under the integral
are the same as in (\ref{eq:Fbarnonincreasing}). We also set
\begin{align*}
\mathrm{II}^\eta(t)&=\frac{1}{p}\int \boldeta(\bvarphi)[U_{g_t}(\bvarphi)+\log
q_t(\bvarphi)+1] \cdot \frac{\d}{\d t}q_t(\bvarphi)\d\bvarphi\\
&=\frac{1}{p}\int \boldeta(\bvarphi)[U_{g_t}(\bvarphi)+\log
q_t(\bvarphi)+1] \cdot \big(
\nabla \cdot [q_t(\bvarphi)\nabla U_{g_t}(\bvarphi)]+\Delta q_t(\bvarphi)\big)
\d\bvarphi,
\end{align*}
where we may differentiate under the integral in the first line using the
dominated convergence theorem and boundedness of
$U_{g_t},\nabla U_{g_t},\log q_t,\|\nabla q_t\|,|\Delta q_t|$ in compact
subsets of $(0,T) \times \R^p$ as follows from
(\ref{eq:logqbound}--\ref{eq:sobolev2norm}). Applying integration by parts in
$\bvarphi$, we then have
$\mathrm{II}^\eta(t)=\mathrm{III}^\eta(t)+\mathrm{IV}^\eta(t)$ where
\begin{align*}
\mathrm{III}^\eta(t)&={-}\frac{1}{p}\int \boldeta(\bvarphi)
\Big\|\nabla\Big(U_{g_t}(\bvarphi)+\log
q_t(\bvarphi)\Big)\Big\|_2^2q_t(\bvarphi)\d\bvarphi,\\
\mathrm{IV}^\eta(t)&={-}\frac{1}{p}\int 
[U_{g_t}(\bvarphi)+\log q_t(\bvarphi)+1] \cdot \nabla \boldeta(\bvarphi)^\top
\nabla [U_{g_t}(\bvarphi)+\log q_t(\bvarphi)] q_t(\bvarphi)\d\bvarphi.
\end{align*}

Take $\eta \equiv \eta_R$ such that $\eta_R(\varphi)=1$ for $|\varphi| \leq R$,
$\eta_R(\varphi)=0$ for $|\varphi| \geq R+1$, and $\eta_R'(\varphi)$ is bounded
by a constant. Then, recalling from the arguments preceding
(\ref{eq:dominatedconvergence}) that $\N_\tau*\frac{\bar q_t}{\N_\tau*g}$ is
uniformly bounded, a direct application of the dominated convergence theorem
yields
\[\lim_{R \to \infty} \int_t^T \mathrm{I}^{\eta_R}(s)\d s
={-}\alpha \int_t^T \int_{-M}^M
\bigg(\bigg[\N_\tau*\frac{\bar q_s}{\N_\tau*g_s}\bigg](\theta)-1\bigg)^2
g_s(\theta)\d s.\]
Applying the dominated convergence theorem using
(\ref{eq:Usquarebound}) and (\ref{eq:entropybound}), we also have
\[\lim_{R \to \infty} \int_t^T \mathrm{III}^{\eta_R}(s)\d s
={-}\frac{1}{p}\int_t^T \int \Big\|\nabla\Big(U_{g_s}(\bvarphi)+\log
q_s(\bvarphi)\Big)\Big\|_2^2q_s(\bvarphi)\d\bvarphi\d s.\]
Finally, applying again
(\ref{eq:Usquarebound}) and (\ref{eq:entropybound}), the bounds
$|U_{g_s}(\bvarphi)|,|\log q_s(\bvarphi)| \leq C(\|\bvarphi\|_2^2+1)$ uniformly
over $s \in [t,T]$, and Holder's inequality, for some $C,C'>0$ we have
\begin{align*}
\int_t^T |\mathrm{IV}^{\eta_R}(s)|\d s
&\leq C\left[\int_t^T \int \1\{\nabla \boldeta_R(\bvarphi) \neq 0\}
q_s(\bvarphi)\d\bvarphi\d s\right]^{1/4}
\cdot \left[\int_t^T \int 
(\|\bvarphi\|_2^8+1)q_s(\bvarphi)\d\bvarphi\d s\right]^{1/4}\\
&\hspace{1in}
\cdot \left[\int_t^T \int 
\Big\|\nabla \Big(U_{g_s}(\bvarphi)+\log q_s(\bvarphi)\Big)\Big\|_2^2
q_s(\bvarphi)\d\bvarphi\d s\right]^{1/2}\\
&\leq C'\sup_{s \in [t,T]} \P_{q_s}[\nabla \boldeta_R(\bvarphi_s) \neq
0]^{1/4} \cdot \sup_{s \in [t,T]} (\E_{q_s}[\|\bvarphi_s\|_2^8]+1)^{1/4}.
\end{align*}
By Lemma \ref{lem:exist_q_fixed_g}-(2), $\E_{q_s}[\|\bvarphi_s\|^8]<C$
and $\P_{q_s}[|\varphi_{s,j}|>R] \leq e^{-R}\E[e^{|\varphi_{s,j}|}]<Ce^{-R}$
for a constant $C>0$, uniformly over $j=1,\ldots,p$ and $s \in [0,T]$. Thus
\[\lim_{R \to \infty} \int_t^T |\mathrm{IV}^{\eta_R}(s)|\d s=0\]

Combining these statements,
\begin{align*}
&F_n(q_T,g_T)-F_n(q_t,g_t)\\
&=\lim_{R \to \infty} F_n^{\eta_R}(q_T,g_T)-F_n^{\eta_R}(q_t,g_t)\\
&=\lim_{R \to \infty} \int_t^T \Big(\mathrm{I}^{\eta_R}(s)+
\mathrm{III}^{\eta_R}(s)+\mathrm{IV}^{\eta_R}(s)\Big)\d s\\
&=\int_t^T \left({-}\alpha \int 
\bigg(\bigg[\N_\tau*\frac{\bar q_s}{\N_\tau*g_s}\bigg](\theta)-1\bigg)^2
g_s(\theta)\d s
{-}\frac{1}{p}\int \Big\|\nabla\Big(U_{g_s}(\bvarphi)+\log
q_s(\bvarphi)\Big)\Big\|_2^2q_s(\bvarphi)\d\bvarphi\right)\d s
\end{align*}
where the first equality follows from applying the dominated convergence
theorem to $\lim_{R \to \infty} F_n^{\eta_R}$. Differentiating both sides in $T$
shows that $T \mapsto F_n(q_T,g_T)$ is differentiable at any $T>0$,
with derivative given by (\ref{eq:Fnderiv}).
\end{proof}

\section{Derivations of gradient flows}

In this section, we discuss the derivations of the Wasserstein-2 and Fisher-Rao gradient flow equations (\ref{eq:qflow}--\ref{eq:gflow})
used in this paper.

This discussion is provided at a more heuristic level than the rest of our work, in order to motivate the
forms of the PDEs (\ref{eq:qflow}--\ref{eq:gflow}).
We will not aim to establish a rigorous connection between 
(\ref{eq:qflow}--\ref{eq:gflow}) and notions of gradient flows on abstract
metric spaces as developed in \cite{ambrosio2005gradient}; however, we
emphasize that formalizing such a connection is not needed for our theoretical
and methodological development.
All theorems and proofs of our work are fully rigorous in that
they pertain to the PDE system (\ref{eq:qflow}--\ref{eq:gflow})
as the definitions of the gradient
flows, and are to be understood as results about this PDE system. This PDE
system then directly motivates the form of the proposed EBflow method as a
discretized and empirical approximation of its solution.

\subsection{General background}\label{sec:flow_background}

Let $\cP$ be the space of probability density functions 
with finite second moment on a Euclidean space $\mathcal{X} \subset \R^d$. (In
our examples, $\mathcal{X}=\R^p$ or $\mathcal{X}=[-M,M] \subset \R$.)
For a smooth curve $\{\rho_t\}_{t \in [0,1]}$ in $\cP$, we identify its tangent
vector $\dot \rho_t=\partial_t \rho_t(\cdot)$ at $\rho_t$
as a function $\dot \rho_t:\mathcal{X} \to \R$. We equip $\cP$
with the formal structure of a Riemannian manifold, associating to each
$\rho \in \cP$ a tangent space $\tangent_\rho \cP$ of such tangent vectors,
and a metric tensor
$g_\rho(\cdot, \cdot): \tangent_\rho \cP \times \tangent_\rho \cP \rightarrow
\R$. This induces a metric, for any $\mu,\nu \in \cP$, 
\begin{align}\label{def:rieman_dist}
\mathsf{d}(\mu,\nu) = \inf\Big\{\int_0^1 \pnorm{\dot{\rho}_t}{\rho_t}:
\rho_0=\mu,\rho_1=\nu\Big\}
\end{align}
where the infimum is over all smooth curves $\{\rho_t\}_{t \in [0,1]}$ in $\cP$
from $\mu$ to $\nu$, and
$\pnorm{\zeta}{\rho}=\sqrt{g_{\rho}(\zeta,\zeta)}$ for any $\zeta \in \tangent_{\rho} \cP$.
Given a functional $F:\cP \rightarrow \R$, the gradient of $F$ at $\rho \in
\cP$, denoted by $\grad F[\rho]$, is the unique element in $\tangent_\rho \cP$
such that, for all curves $\{\rho_t\}_{t\in [0,1]}$ with $\rho_0 = \rho$
and $\dot \rho_0=\zeta \in \tangent_\rho \cP$,
\begin{align}\label{def:rieman_grad}
g_{\rho}\big(\grad F[\rho], \zeta\big)
=\frac{\d}{\d t}\Big|_{t=0} F(\rho_t)
=\int \delta F[\rho](x)\zeta(x)\d x
\end{align}
where $\delta F[\rho]:\R^d \to \R$ is the first variation of $F$.
We may then define the gradient flow
$\frac{\d}{\d t} \rho_t={-}\grad F[\rho_t]$.

We now specialize to the Wasserstein-2 and Fisher-Rao geometries, following
the presentation in \cite{yan2023learning}.\\

\noindent \emph{Wasserstein-2 geometry.} For the formal Riemannian structure
underlying the Wasserstein-2 geometry on $\cP_*(\R^p)$, the tangent space at
$\rho \in \cP_*(\R^p)$ is
\begin{align*}
\tangent_\rho^{W_2} \cP_*(\R^p) = \Big\{\zeta:\R^p \to \R \;\Big|\; \zeta = -\nabla \cdot (\rho \nabla u) \text{ for some $u$ satisfying}  \int \pnorm{\nabla u}{}^2\d \rho < \infty \Big\},
\end{align*}
equipped with the metric tensor
$g^{W_2}_\rho(\zeta_1,\zeta_2) = \int \iprod{\nabla u_1}{\nabla u_2}\rho(x)\d
x$.
To compute the Wasserstein-2 gradient, 
let $\{\rho_t\}_{t\in[0,1]}$ be a curve with $\rho_0=\rho$ and
$\dot\rho_0=\zeta=-\nabla \cdot (\rho \nabla v)$. Then (\ref{def:rieman_grad})
reads
\begin{align*}
g_{\rho}^{W_2}\big(\grad^{W_2} F[\rho], \zeta\big) &= \int \delta F[\rho] \cdot \zeta\,\d x
= \int {-}\delta F[\rho]\Big(\nabla \cdot (\rho \nabla v)\Big) \d x
= \int \iprod{\nabla \delta F[\rho]}{\nabla v} \rho\,\d x,
\end{align*}
the last equality applying integration by parts. Comparing with the definition
of $g_\rho^{W_2}$, this identifies the Wasserstein-2 gradient as
\begin{align}\label{def:W2_grad}
\grad^{W_2} F[\rho] = -\nabla \cdot (\rho \nabla \delta F[\rho]). 
\end{align}

\noindent\emph{Fisher-Rao geometry.} For the formal Riemannian structure
underlying the Fisher-Rao geometry on $\cP_*(M)$, the tangent space at $\rho \in
\cP_*(M)$ is
\begin{align*}
\tangent_\rho^{\FR} \cP = \Big\{\zeta:[-M,M] \to \R \;\Big|\;
\zeta=\rho\Big(\alpha - \int_{-M}^M \alpha \d
\rho\Big) \text{ for some $\alpha$ satisfying} \int_{-M}^M \alpha^2 \d\rho < \infty\Big\}, 
\end{align*}
equipped with the metric tensor
$g^{\FR}_\rho(\zeta_1,\zeta_2) = \int \frac{\zeta_1 \cdot \zeta_2}{\rho^2} \d\rho
= \int (\alpha_1-\int \alpha_1\d\rho)(\alpha_2-\int \alpha_2 \d\rho)\d\rho$.
To compute the Fisher-Rao gradient,
let $\{\rho_t\}_{t\in[0,1]}$ be a curve with $\rho_0=\rho$ and $\dot \rho_0 =
\zeta = \rho(\alpha - \int \alpha\,\d\rho)$. Then (\ref{def:rieman_grad}) reads
\begin{align*}
g^{\FR}_\rho(\grad^{\FR}F[\rho], \zeta) = 
\int_{-M}^M \delta F[\rho] \cdot \zeta\,\d x
&= \int_{-M}^M \delta F[\rho] \cdot \Big(\alpha - \int_{-M}^M \alpha \d\rho\Big)\d\rho\\
&=\int_{-M}^M \Big(\delta F[\rho] - \int_{-M}^M \delta F[\rho]\d\rho\Big)\cdot
\Big(\alpha - \int_{-M}^M \alpha \d\rho\Big)\d\rho.
\end{align*}
Comparing with the definition of $g_\rho^{\FR}$,
this identifies the Fisher-Rao gradient as
\begin{align}\label{def:FR_grad}
\grad^{\FR} F[\rho] = \rho\Big(\delta F[\rho] - \int \delta F[\rho] \d\rho\Big).
\end{align}
We note that the second term is a centering term for the first
variation $\delta
F[\rho]$ (uniquely defined only up to an additive constant) to ensure that
the condition $\int_{-M}^M \d\rho_t=1$ is conserved along the flow.

\subsection{Derivatives of $\bar{F}_n(g)$ and $F_n(q,g)$}\label{sec:Fn_grad}

We explain the calculations of the forms of the first variations of
$\bar F_n(g)$ and $F_n(q,g)$. First, from the form
\[\bar F_n(g)={-}\frac{1}{p}\log \int_{[-M,M]^p}
\frac{1}{(2\pi\sigma^2)^{n/2}}
\exp\Big({-}\frac{\|\y-\X\btheta\|^2}{2\sigma^2}\Big)
\prod_{j=1}^p g(\theta_j)\d\btheta,\]
for any density $g \in \cP_*(M)$ that is strictly positive on $[-M,M]$ and
for any bounded function $\chi:[-M,M] \to \R$ with $\int_{-M}^M
\chi(\theta)\d\theta=0$, we have
\begin{align*}
\frac{\d}{\d \eps}\Big|_{\eps=0} \bar F_n(g+\eps \chi)
&={-}\frac{1}{p} \sum_{j=1}^p
\frac{\int_{[-M,M]^p} \exp(-\frac{\|\y-\X\btheta\|^2}{2\sigma^2})
\frac{\chi(\theta_j)}{g(\theta_j)}\prod_{i=1}^p g(\theta_i)\d\btheta}
{\int_{[-M,M]^p} \exp(-\frac{\|\y-\X\btheta\|^2}{2\sigma^2})
\prod_{i=1}^p g(\theta_i)\d\btheta}\\
&=-\frac{1}{p}\sum_{j=1}^p \int_{[-M,M]^p}
\frac{\chi(\theta_j)}{g(\theta_j)} P_g(\btheta \mid \y)\d \btheta
=\int_{-M}^M {-}\frac{\chi(\theta)}{g(\theta)}\bar\mu[g](\theta)\d\theta
=\int_{-M}^M {-}\frac{\bar\mu[g](\theta)}{g(\theta)}\chi(\theta)\d\theta,
\end{align*}
where $\bar \mu[g]$ is the average marginal density of
$\mu[g](\btheta)=P_g(\btheta \mid \y)$. Alternatively, from the equivalent form
\[\bar F_n(g)={-}\frac{1}{p}\log \int
\frac{1}{\det(2\pi\bSigma)^{1/2}}
\exp\Big({-}\frac{(\y-\X\bvarphi)^\top \bSigma^{-1}(\y-\X\bvarphi)}{2}\Big)
\prod_{j=1}^p \N_\tau*g(\varphi_j)\d\bvarphi,\]
we have analogously
\begin{align*}
\frac{\d}{\d \eps}\Big|_{\eps=0} \bar F_n(g+\eps \chi)
&={-}\frac{1}{p}\sum_{j=1}^p \int \frac{[\N_\tau*\chi](\varphi_j)}
{[\N_\tau*g](\varphi_j)} P_g(\bvarphi \mid \y) \d\bvarphi\\
&=\int {-}\frac{[\N_\tau*\chi](\varphi)}{[\N_\tau*g](\varphi)}
\bar\nu[g](\varphi)\d\varphi
=\int_{-M}^M
{-}\left[\N_\tau*\frac{\bar\nu[g]}{\N_\tau*g}\right](\theta)\chi(\theta)\d\theta,
\end{align*}
where the last equality uses the identity
$\int f(\varphi)[\N_\tau*\chi](\varphi)\d\varphi
=\int_{-M}^M [\N_\tau*f](\theta)\chi(\theta)\d\theta$. Thus we identify two
equivalent forms for the first variation of $\bar F_n(g)$,
\begin{align}\label{eq:first_variation_equiv}
\delta \bar F_n[g](\theta)={-}\frac{\bar\nu[g](\theta)}{g(\theta)}+1
={-}\N_\tau*\frac{\bar\nu[g]}{\N_\tau*g}(\theta)+1,
\end{align}
where the constant $+1$ is chosen to center $\delta \bar F_n[g]$ to have mean 0
under $g$. The second form shows (\ref{eq:deltaFn}).

Applying analogous calculations for the functional $F_n(q,g)$ given explicitly
by \eqref{eq:Fnqgform}, we have
\begin{align*}
\frac{\d}{\d \eps}\Big|_{\eps=0} F_n(q, g + \eps\chi) &=
-\frac{1}{p}\sum_{j=1}^p\int \frac{[\N_\tau \ast \chi](\varphi_j)}{[\N_\tau \ast
g](\varphi_j)}q(\bvarphi)\d\bvarphi
=\int_{-M}^M {-}\left[\N_\tau \ast \frac{\bar{q}}{\N_\tau \ast g}\right](\theta)
\chi(\theta) \d\theta,\\
\frac{\d}{\d \eps}\Big|_{\eps=0} F_n(q+\eps\chi,g)&=\frac{1}{p} \int
\big[U_g(\bvarphi)+\log q(\bvarphi)+1\big] \chi(\bvarphi) \d\bvarphi,
\end{align*}
where $U_g(\bvarphi)$ is the negative log-posterior-density
(\ref{Ug_smooth}). Then, applying $\int_{-M}^M \chi(\theta)\d\theta=0$
in the first equation and $\int \chi(\bvarphi)\d\bvarphi=0$ in the second,
the first variations of $F_n(q,g)$ are
\[\delta_g F_n[q,g](\theta) = -\N_\tau \ast\frac{\bar{q}}{\N_\tau \ast
g}(\theta) + 1,
\qquad
\delta_q F_n[q,g](\bvarphi) = \frac{1}{p}\big[U_g(\bvarphi)+\log
q(\bvarphi)\big],\]
where again we have centered this first quantity
$\delta_g F_n[q,g](\theta)$ to have mean 0 under $g$.
This leads to the forms (\ref{eq:qflow}--\ref{eq:gflow}) for the gradient flows
(\ref{eq:joint_flow_intro}), in light of the general definitions
(\ref{def:W2_grad}) and (\ref{def:FR_grad}) for the Wasserstein-2 and
Fisher-Rao gradients.

\subsection{Non-convexity}
\label{app:nonconvex}

We provide a concrete example to show that even for the identity design
$\X=\Id$, the functional $\bar{F}_n(g)$ is not geodesically convex under either
the Wasserstein-2 geometry or the Fisher-Rao geometry, even though it is convex
in the sense (\ref{eq:barFconvex}) in the linear geometry.

For the Wasserstein-2 geometry, consider
$g_0=g_\ast=\delta_0$, $g_1=\frac{1}{2}\delta_{-1}+\frac{1}{2}\delta_1$,
and $y_1,\ldots,y_n \overset{iid}{\sim} \N(0,1)*g_\ast$.
The constant-speed Wasserstein-2 geodesic connecting $g_0$ and $g_1$ is
$g_t=\frac{1}{2}\delta_{-t}+\frac{1}{2}\delta_t$ for $t\in[0,1]$. Explicit
computation shows
\begin{align*}
\frac{\d^2}{\d t^2}\bar{F}_n(g_t) = 1 - \frac{1}{n}\sum_{i=1}^n
\frac{y_i^2}{\cosh^2(ty_i)},
\end{align*}
yielding that $\Prob_{g_\ast}(\frac{\d^2}{\d^2 t}|_{t=0} \bar{F}_n(g_t) \geq 0)
= \Prob_{g_\ast}(n^{-1}\sum_{i=1}^n Y_i^2 \leq 1)\approx 1/2$ for large $n$.

For the Fisher-Rao geometry, let $g_*=\operatorname{Bernoulli}(1/2)$,
$g_0=\operatorname{Bernoulli}(1/4)$,
$g_1=\operatorname{Bernoulli}(3/4)$, and $y_1,\ldots,y_n \overset{iid}{\sim}
\N(0,1)*g_\ast$.
Using that the geodesic distance in (\ref{def:rieman_dist}) is
$d_{\FR}(\text{Bernoulli}(\theta_1), \text{Bernoulli}(\theta_2)) =
2|\arcsin(\theta_1) - \arcsin(\theta_2)|$
for $\theta_1,\theta_2 \in (0,1)$ \citep{miyamoto2023closed},
we have $d_{\FR}(g_0,g_1) = \pi/3$. The
constant-speed Fisher-Rao geodesic connecting $g_0$ and $g_1$ is
$g_t=\operatorname{Bernoulli}(p_t)$ with $p_t = \sin^2(\frac{\pi}{6}(t+1))$,
from equating $d_{\FR}(g_0,g_t) = t \cdot \pi/3$. Explicit computation
shows that
\begin{align*}
\E_{g_\ast}[\bar F_n(g_t)]
=-\frac{1}{2}\int \Big(\N(x) + \N(x-1)\Big) \cdot
\log\Big[\cos^2\Big(\frac{\pi}{6}(t+1)\Big)\N(x) +
\sin^2\Big(\frac{\pi}{6}(t+1)\Big)\N(x-1)\Big] \d x
\end{align*}
is not convex in $t$, where $\N(x)$ is the standard $\N(0,1)$ density at $x$.

\section{Details of algorithms}\label{sec:other_algo}

\subsection{Smoothness regularization of EBflow}\label{subsec:ebflow_smoothness}

As discussed around (\ref{eq:smoothobjective}), one may adapt EBFlow 
to perform estimation in a space of smooth densities by
imposing additional smoothness regularization for $\hat g$:
\begin{align*}
\bar F_{n,\lambda}(g)&=-\frac{1}{p}\log P_g(\y)
+\frac{\lambda}{2}\int_{-M}^M g''(\theta)^2 d\theta\notag\\
&=\min_{q \in \cP_*(\R^p)} F_n(q,g)+
\frac{\lambda}{2}\int_{-M}^M g''(\theta)^2 d\theta.
\end{align*}
This does not affect the gradient equation \eqref{eq:qflow} for $q_t(\cdot)$,
and leads to the addition of a regularization term in the gradient equation
\eqref{eq:gflow} for $g_t(\cdot)$. We note that the addition of this
smoothness regularizer may ensure sufficient regularity of the prior and
posterior in $\btheta$ along the gradient flow,
to circumvent some of the need to perform the posterior sampling
in the smoothed regression variable $\bvarphi$.
However, even in this context, we advocate
retaining the variable reparametrization using $\bvarphi$, as it
serves to decouple the (statistical) issue of choosing a smoothness penalty
$\lambda$ from the (computational) issue of
ensuring sufficient regularity for the Langevin diffusion process.

To incorporate the additional smoothing-spline regularization in
\eqref{eq:smoothobjective}, we may adjust the weight update \eqref{eq:wupdate} as follows:
Denoting by $\Delta=b_{k+1}-b_k$ the support spacing in
\eqref{eq:gdiscretization}, we may approximate
\[\frac{\lambda}{2}\int_{-M}^M g''(\theta)^2 d\theta
\approx \frac{\lambda\Delta}{2}
\sum_{k=1}^K \left[D\left(\frac{w}{\Delta}\right)\right]_k^2\]
where $D$ is a differential quadrature matrix approximating the second
derivative, e.g.\ $D \in \R^{(K-2) \times K}$ having entries
$D[i,i+1]=-2/\Delta^2$,
$D[i,i]=D[i,i+2]=1/\Delta^2$, and remaining entries 0. Then 
(\ref{eq:wupdate}) may be replaced by
\begin{equation}\label{eq:smoothwupdate}
\begin{aligned}
w_{t+1,k}&=w_{t,k}+\eta_t^w w_{t,k}
\Bigg[\frac{1}{p}\sum_{j=1}^p \frac{\N_\tau(\varphi_{t+1,j}-b_k)}
{\sum_{i=1}^K w_{t,i}\,\N_\tau(\varphi_{t+1,j}-b_i)}-1\\
&\hspace{1in}-\lambda\left(\frac{D^\top Dw_t}{\Delta}\right)_k
+\lambda\sum_{i=1}^K w_{t,i}\left(\frac{D^\top Dw_t}{\Delta}\right)_i\Bigg].
\end{aligned}
\end{equation}
The last two terms correspond to a discrete approximation of the first
variation of the spline regularizer, which is again centered to
ensure the mass conservation property $\sum_k w_{t,k}=1$ in
every iteration.

EBflow alternately applies the ULA iteration
(\ref{eq:phiupdatereparam}) and the $w$ iteration (\ref{eq:smoothwupdate}) with
spline smoothing, or 
(\ref{eq:phiupdateprecond}) and (\ref{eq:smoothwupdate}) in the preconditioned
setting. We initialize EBflow also from $\bvarphi=0$ and
$g=\operatorname{Uniform}(\{b_1,\ldots,b_K\})$, and again run the first 200
sampling iterations at large step size $\eta_t^\varphi=1.0$ without updating
prior weights, before transitioning to the log-linear step size decay schedule
(\ref{eq:loglinearstep}).

\subsection{Details of other algorithms}
\smallskip
\ \newline

\noindent\textbf{CAVI.}
We consider the Gibbs variational representation
\[\bar F_n(g)=\min_q \tilde{F}_n(q,g):=\min_q \frac{1}{p}\int_{[-M,M]^p}
\bigg[\frac{\pnorm{\y - \X\btheta}{}^2}{2\sigma^2} - \sum_{j=1}^p \log g(\theta_j)+\log q(\btheta)\bigg] q(\btheta) \d\btheta\]
in the parametrization by $\btheta$. Mean-field variational inference
optimizes $\tilde F_n(q,g)$ over product distributions $q(\btheta) = \prod_{j=1}^p
q_j(\theta_j)$. It is easy to check that, for fixed $g$ and
$q_{-j}=(q_1,\ldots,q_{j-1},q_{j+1},\ldots,q_p)$, the minimizer over $q_j$ has
the explicit form
\begin{align}\label{def:VI_solution}
q_j(\theta_j) \propto \exp\Big({-}\frac{\|\x_j\|^2}{2\sigma^2}\theta_j^2 + \frac{1}{\sigma^2}(\y - \X_{-j}\E_{q_{-j}}[\btheta_{-j}])^\top \x_j \cdot \theta_j\Big)g(\theta_j)
\end{align}
where $\X_{-j} \in \R^{n \times (p-1)}$ is $\X$ without the $j$-th column, and
$\btheta_{-j} \in\R^{p-1}$ is $\btheta$ without the $j$-th coordinate. Over
a discrete support $b_1,\ldots,b_K \in [-M,M]$, we represent the prior
$g$ by probability weights $w=(w_1,\ldots,w_K)$ and each $q_j$
by weights $w_j=(w_{j,1},\ldots, w_{j,K})$. Then $w_j$ is updated as
\begin{align}
\label{def:VI_solution3}
w_{j,k}=\frac{w_k\exp\Big(-\frac{\|\x_j\|^2}{2\sigma^2}b_k^2+\frac{1}{\sigma^2}(\y
- \X_{-j}\E_{q_{-j}}[\btheta_{-j}])^\top \x_j \cdot b_k\Big)}{\sum_{i=1}^K w_i
\exp\Big(-\frac{\|\x_j\|^2}{2\sigma^2}b_i^2 + \frac{1}{\sigma^2}(\y -
\X_{-j}\E_{q_{-j}}[\btheta_{-j}])^\top \x_j \cdot b_i\Big)} \quad \mbox{for all }  k=1,\ldots, K,
\end{align}
where $\E_{q_{-j}}[\btheta_{-j}]=(\sum_{k=1}^K b_k w_{j',k})_{j' \neq j}$.

Fixing $q=\prod_{j=1}^p q_j$, we update $g$ to minimize
\begin{equation}\label{eq:gobjectiveGibbs}
\tilde F_n(q,g)+\frac{\lambda}{2}\int_{-M}^M g''(\theta)^2\d\theta=\int_{-M}^M [{-}\log g(\theta)]\bar
q(\theta)\d\theta+\frac{\lambda}{2}\int_{-M}^M g''(\theta)^2\d\theta+\text{constant}
\end{equation}
where $\bar q=p^{-1}\sum_j q_j$. This is implemented by the weight update
\[w=\argmin_{w\geq 0, \|w\|_1=1} {-}\sum_{k=1}^K\frac{\sum_{j=1}^p w_{j,k}}{p}
\log w_k+\frac{\lambda\Delta}{2}\sum_{k=1}^K\left[\frac{Dw}{\Delta}\right]_k^2\]
where $D$ is the differential quadrature matrix described in Section
\ref{sec:discretize}.
When $\lambda=0$, this takes the closed form $w_k=\frac{1}{p}\sum_{j=1}^p
w_{j,k}$; when $\lambda>0$, we solve this convex program using CVXR \citep{cvxr}.
Each iteration of CAVI iteratively updates $(q_1,\ldots,q_p,g)$ once in this order.

We initialize CAVI from
$q_1=\ldots=q_p=g=\operatorname{Uniform}(\{b_1,\ldots,b_K\})$.\\

\noindent\textbf{Gibbs-MCEM.} Under the prior $g(\theta)$, the conditional
posterior density of $\theta_j$ given $\btheta_{-j}$ is
\begin{align}\label{def:Gibbs_solution}
P_g(\theta_j \mid \y, \btheta_{-j}) \propto
\exp\Big({-}\frac{\|\x_j\|^2}{2\sigma^2}\theta_j^2 + \frac{1}{\sigma^2}(\y -
\X_{-j}\btheta_{-j})^\top \x_j \cdot \theta_j\Big)g(\theta_j).
\end{align}
Over a discrete support $b_1,\ldots,b_K \in [-M,M]$, this is represented by
the weight vector $w_j=(w_{j,1},\ldots, w_{j,K})$ with weights
\begin{equation}\label{def:Gibbs_solution3}
w_{j,k}=\frac{w_k\exp\Big(-\frac{\|\x_j\|^2}{2\sigma^2}b_k^2 +
\frac{1}{\sigma^2}(\y - \X_{-j}\btheta_{-j})^\top \x_j \cdot
b_k\Big)}{\sum_{i=1}^K w_i\exp\Big(-\frac{\|\x_j\|^2}{2\sigma^2}b_i^2 +
\frac{1}{\sigma^2}(\y - \X_{-j}\btheta_{-j})^\top \x_j \cdot b_i\Big)} \quad
\mbox{for all }  k=1,\ldots, K
\end{equation}
where $w=(w_1,\ldots,w_k)$ are weights for the prior.

Gibbs-MCEM iteratively resamples the coordinates $\theta_1,\ldots,\theta_p$ of
$\btheta$ in sequence
from the conditional law (\ref{def:Gibbs_solution3}) over the support points
$b_1,\ldots,b_K$, fixing the priors weights $w=(w_1,\ldots,w_K)$. Having drawn
$T_\text{iter}$ samples $\theta_j^{(1)},\ldots,\theta_j^{(T_\text{iter})}$
of each $j$-th coordinate, it then approximates $\bar q$ in
(\ref{eq:gobjectiveGibbs}) by the empirical distribution of these
samples, yielding the update for prior weights
\[w=\argmin_{w\geq 0, \|w\|_1=1} {-}\frac{\sum_{j=1}^p\sum_{t=1}^{T_\text{iter}}
\sum_{k=1}^K 1\{\theta_j^{(t)}=b_k\}}{pT_\text{iter}}\log
w_k+\frac{\lambda\Delta}{2}\sum_{k=1}^K\left[\frac{Dw}{\Delta}\right]_k^2.\]

We initialize Gibbs-MCEM from $\btheta=0$ and
$g=\operatorname{Uniform}(\{b_1,\ldots,b_K\})$, and run the first 200 sampling
iterations without updating prior weights.\\

\noindent\textbf{Langevin-MCEM.}
 Langevin-MCEM applies the ULA iterations
(\ref{eq:phiupdatereparam}) to sample from $P_g(\bvarphi \mid \y)$,
fixing the prior weights $w=w_t$ between every $T_\text{iter}$ iterations.
After every $T_\text{iter}$ iterations, denoting by
$\bvarphi^{(1)},\ldots,\bvarphi^{(T_\text{iter})}$ the sampled iterates,
we sample $S=10000$ coordinates $\varphi_1,\ldots,\varphi_S$
uniformly at random from the $pT_\text{iter}$ total coordinates of
$\bvarphi^{(1)},\ldots,\bvarphi^{(T_\text{iter})}$ (for reasons of computational
efficiency), and approximate the exact M-step
update (\ref{eq:EM}) in EM by the prior weight update
\[w=\argmin_{w\geq 0, \|w\|_1=1}
{-}\frac{1}{S}\sum_{s=1}^S \log \left(\sum_{k=1}^K
w_k\N_\tau(b_k-\varphi_s)\right)
+\frac{\lambda\Delta}{2}\sum_{k=1}^K\left[\frac{Dw}{\Delta}\right]_k^2.\]
In other words, this is a discretized NPMLE with sample size $S$ and a
smoothness penalty. This convex program is solved using CVXR and the CSC solver.

We initialize Langevin-MCEM from $\bvarphi=0$ and
$q=\operatorname{Uniform}(\{b_1,\ldots,b_K\})$, and run the first 200 sampling
iterations at large step size $\eta_t^\varphi=1.0$ without updating prior
weights, before transitioning to the fixed step size $\eta_t^\varphi \in
\{0.1,1.0\}$ with prior weight updates.\\

\pagebreak

\section{Supplementary figures and tables}

\begin{table}[h!]
\centering 
{\scriptsize
\begin{tabular}{c|c|c|cc|cc|cc}
Design & $n/p$ & Prior & \multicolumn{2}{|c|}{$\tau^2 \text{ scale}=0.8$} & \multicolumn{2}{|c|}{$\tau^2 \text{ scale}=0.5$} & \multicolumn{2}{|c}{$\tau^2 \text{ scale}=0.2$} \\
\hline
&&& TV & $\bar{F}_n$ & TV & $\bar{F}_n$ & TV & $\bar{F}_n$ \\
\hline
\multirow{4}{*}{identity} & \multirow{4}{*}{1.0} & gaussian & 0.041 & 1.771 & 0.039 & 1.772 & 0.037 & 1.772 \\
& & scale & 0.165 & 0.917 & 0.162 & 0.917 & 0.160 & 0.917 \\
& & skew & 0.085 & 1.547 & 0.078 & 1.547 & 0.074 & 1.548 \\
& & bimodal & 0.083 & 1.897 & 0.096 & 1.898 & 0.102 & 1.899 \\
\hline
&&& TV & MSE & TV & MSE & TV & MSE \\
\hline
\multirow{4}{*}{iid} & \multirow{4}{*}{2.0} & gaussian & 0.056 & 0.421 & 0.051 & 0.420 & 0.052 & 0.421 \\
& & scale & 0.151 & 0.137 & 0.149 & 0.137 & 0.151 & 0.138 \\
& & skew & 0.069 & 0.115 & 0.069 & 0.115 & 0.070 & 0.115 \\
& & bimodal & 0.079 & 0.106 & 0.074 & 0.105 & 0.077 & 0.106 \\
\hline
\multirow{4}{*}{iid} & \multirow{4}{*}{0.5} & gaussian & 0.039 & 0.799 & 0.040 & 0.797 & 0.042 & 0.799 \\
& & scale & 0.278 & 0.624 & 0.279 & 0.628 & 0.281 & 0.631 \\
& & skew & 0.203 & 0.326 & 0.194 & 0.324 & 0.179 & 0.325 \\
& & bimodal & 0.362 & 0.651 & 0.347 & 0.644 & 0.335 & 0.650 \\
\hline
\multirow{4}{*}{block02corr0.9} & \multirow{4}{*}{2.0} & gaussian & 0.048 & 0.283 & 0.043 & 0.282 & 0.048 & 0.284 \\
& & scale & 0.201 & 0.114 & 0.202 & 0.111 & 0.205 & 0.114 \\
& & skew & 0.113 & 0.070 & 0.111 & 0.069 & 0.110 & 0.070 \\
& & bimodal & 0.112 & 0.102 & 0.104 & 0.099 & 0.115 & 0.102 \\
\hline
\multirow{4}{*}{block02corr0.9} & \multirow{4}{*}{0.5} & gaussian & 0.051 & 0.682 & 0.054 & 0.680 & 0.057 & 0.683 \\
& & scale & 0.261 & 0.447 & 0.262 & 0.434 & 0.261 & 0.440 \\
& & skew & 0.207 & 0.201 & 0.195 & 0.201 & 0.173 & 0.201 \\
& & bimodal & 0.359 & 0.459 & 0.343 & 0.455 & 0.332 & 0.459 \\
\hline
\multirow{4}{*}{block10corr0.5} & \multirow{4}{*}{2.0} & gaussian & 0.058 & 0.328 & 0.055 & 0.325 & 0.058 & 0.328 \\
& & scale & 0.188 & 0.126 & 0.185 & 0.125 & 0.196 & 0.129 \\
& & skew & 0.172 & 0.054 & 0.163 & 0.054 & 0.156 & 0.054 \\
& & bimodal & 0.144 & 0.112 & 0.138 & 0.109 & 0.152 & 0.113 \\
\hline
\multirow{4}{*}{block10corr0.5} & \multirow{4}{*}{0.5} & gaussian & 0.054 & 0.591 & 0.054 & 0.589 & 0.058 & 0.592 \\
& & scale & 0.282 & 0.463 & 0.288 & 0.448 & 0.285 & 0.464 \\
& & skew & 0.209 & 0.101 & 0.197 & 0.101 & 0.171 & 0.101 \\
& & bimodal & 0.350 & 0.448 & 0.341 & 0.440 & 0.333 & 0.449 \\
\hline
\multirow{4}{*}{genome} & \multirow{4}{*}{2.0} & gaussian & 0.060 & 0.228 & 0.060 & 0.225 & 0.060 & 0.228 \\
& & scale & 0.233 & 0.101 & 0.235 & 0.096 & 0.255 & 0.106 \\
& & skew & 0.179 & 0.056 & 0.166 & 0.055 & 0.174 & 0.056 \\
& & bimodal & 0.296 & 0.092 & 0.273 & 0.089 & 0.293 & 0.091 \\
\hline
\multirow{4}{*}{genome} & \multirow{4}{*}{0.5} & gaussian & 0.067 & 0.483 & 0.060 & 0.480 & 0.071 & 0.484 \\
& & scale & 0.306 & 0.313 & 0.294 & 0.304 & 0.310 & 0.322 \\
& & skew & 0.196 & 0.115 & 0.178 & 0.113 & 0.169 & 0.114 \\
& & bimodal & 0.384 & 0.325 & 0.360 & 0.319 & 0.371 & 0.325

\end{tabular}}
\caption{Comparison of different settings for the variable reparametrization variance $\tau^2=c_\tau \cdot
\sigma^2/\|\X\|_\op^2$. The setup is
$(n,p)=(500,1000)$, standard Gaussian
prior $g_\ast=\N(0,1)$, i.i.d.\ Gaussian design $x_{ij} \overset{iid}{\sim}
\N(0,\frac{1}{n})$, and
noise variance $\sigma^2$ such that $\beps$ explains
50\% of the variance of $\y$. EBflow is performed with $T=10{,}000$
iterations, smoothing-spline regularization $\lambda=0.003$,
and the log-linear step size schedule \eqref{eq:loglinearstep}. All TV errors,
log-likelihoods, and prediction MSEs are averaged over 10 trials with
independently generated $(\btheta,\beps)$.}
\label{tab:tausq}
\end{table}

\begin{table}[h!]
\centering 
{\scriptsize
\begin{tabular}{c|c|c|cc|cc|cc|cc|cc|cc}
Design & $n/p$ & Prior & \multicolumn{2}{|c|}{$\lambda=0.01$} & \multicolumn{2}{|c|}{$\lambda=0.003$} & \multicolumn{2}{|c|}{$\lambda=0.001$} & \multicolumn{2}{|c|}{$\lambda=1e-4$} & \multicolumn{2}{|c|}{$\lambda=1e-5$} & \multicolumn{2}{|c}{$\lambda=0$} \\
\hline
&&& TV & $\bar{F}_n$ & TV & $\bar{F}_n$ & TV & $\bar{F}_n$ & TV & $\bar{F}_n$ & TV & $\bar{F}_n$ & TV & $\bar{F}_n$ \\
\hline
\multirow{4}{*}{identity} & \multirow{4}{*}{1.0} & gaussian & 0.059 & 1.772 & 0.039 & 1.772 & 0.042 & 1.771 & 0.044 & 1.771 & 0.044 & 1.771 & 0.045 & 1.771 \\
& & scale & 0.268 & 0.953 & 0.249 & 0.934 & 0.231 & 0.926 & 0.194 & 0.919 & 0.162 & 0.917 & 0.145 & 0.917 \\
& & skew & 0.116 & 1.551 & 0.096 & 1.549 & 0.086 & 1.548 & 0.078 & 1.548 & 0.078 & 1.547 & 0.078 & 1.547 \\
& & bimodal & 0.184 & 1.906 & 0.120 & 1.900 & 0.103 & 1.899 & 0.096 & 1.898 & 0.096 & 1.898 & 0.096 & 1.898 \\
\hline
&&& TV & MSE & TV & MSE & TV & MSE & TV & MSE & TV & MSE & TV & MSE \\
\hline
\multirow{4}{*}{iid} & \multirow{4}{*}{2.0} & gaussian & 0.056 & 0.419 & 0.051 & 0.420 & 0.057 & 0.420 & 0.062 & 0.421 & 0.063 & 0.421 & 0.063 & 0.421 \\
& & scale & 0.264 & 0.160 & 0.242 & 0.152 & 0.223 & 0.145 & 0.184 & 0.138 & 0.149 & 0.137 & 0.128 & 0.136 \\
& & skew & 0.129 & 0.118 & 0.099 & 0.116 & 0.085 & 0.115 & 0.070 & 0.115 & 0.069 & 0.115 & 0.069 & 0.115 \\
& & bimodal & 0.130 & 0.110 & 0.080 & 0.106 & 0.070 & 0.105 & 0.074 & 0.105 & 0.076 & 0.105 & 0.076 & 0.105 \\
\hline
\multirow{4}{*}{iid} & \multirow{4}{*}{0.5} & gaussian & 0.162 & 0.820 & 0.040 & 0.797 & 0.038 & 0.797 & 0.038 & 0.797 & 0.038 & 0.797 & 0.038 & 0.797 \\
& & scale & 0.377 & 0.705 & 0.323 & 0.643 & 0.306 & 0.635 & 0.284 & 0.625 & 0.279 & 0.628 & 0.279 & 0.628 \\
& & skew & 0.294 & 0.334 & 0.204 & 0.324 & 0.202 & 0.324 & 0.195 & 0.324 & 0.194 & 0.324 & 0.194 & 0.324 \\
& & bimodal & 0.374 & 0.654 & 0.356 & 0.644 & 0.352 & 0.644 & 0.347 & 0.644 & 0.348 & 0.644 & 0.348 & 0.644 \\
\hline
\multirow{4}{*}{block02corr0.9} & \multirow{4}{*}{2.0} & gaussian & 0.077 & 0.284 & 0.043 & 0.282 & 0.048 & 0.282 & 0.051 & 0.282 & 0.052 & 0.282 & 0.052 & 0.282 \\
& & scale & 0.297 & 0.118 & 0.277 & 0.115 & 0.264 & 0.114 & 0.229 & 0.113 & 0.202 & 0.111 & 0.192 & 0.111 \\
& & skew & 0.197 & 0.071 & 0.134 & 0.070 & 0.124 & 0.070 & 0.112 & 0.069 & 0.111 & 0.069 & 0.111 & 0.069 \\
& & bimodal & 0.289 & 0.105 & 0.130 & 0.100 & 0.110 & 0.099 & 0.104 & 0.099 & 0.105 & 0.099 & 0.105 & 0.099 \\
\hline
\multirow{4}{*}{block02corr0.9} & \multirow{4}{*}{0.5} & gaussian & 0.150 & 0.713 & 0.054 & 0.680 & 0.053 & 0.680 & 0.053 & 0.680 & 0.053 & 0.680 & 0.054 & 0.680 \\
& & scale & 0.417 & 0.543 & 0.321 & 0.447 & 0.300 & 0.438 & 0.270 & 0.435 & 0.262 & 0.434 & 0.260 & 0.434 \\
& & skew & 0.285 & 0.206 & 0.206 & 0.201 & 0.204 & 0.201 & 0.197 & 0.201 & 0.195 & 0.201 & 0.196 & 0.201 \\
& & bimodal & 0.448 & 0.464 & 0.350 & 0.455 & 0.346 & 0.455 & 0.343 & 0.455 & 0.344 & 0.455 & 0.344 & 0.455 \\
\hline
\multirow{4}{*}{block10corr0.5} & \multirow{4}{*}{2.0} & gaussian & 0.071 & 0.326 & 0.055 & 0.325 & 0.062 & 0.326 & 0.068 & 0.326 & 0.070 & 0.326 & 0.070 & 0.326 \\
& & scale & 0.279 & 0.140 & 0.267 & 0.135 & 0.252 & 0.129 & 0.211 & 0.125 & 0.185 & 0.125 & 0.175 & 0.125 \\
& & skew & 0.209 & 0.055 & 0.178 & 0.054 & 0.172 & 0.054 & 0.164 & 0.054 & 0.163 & 0.054 & 0.163 & 0.054 \\
& & bimodal & 0.236 & 0.117 & 0.144 & 0.110 & 0.136 & 0.109 & 0.138 & 0.109 & 0.140 & 0.109 & 0.141 & 0.109 \\
\hline
\multirow{4}{*}{block10corr0.5} & \multirow{4}{*}{0.5} & gaussian & 0.270 & 0.647 & 0.054 & 0.589 & 0.053 & 0.589 & 0.054 & 0.589 & 0.054 & 0.589 & 0.054 & 0.589 \\
& & scale & 0.487 & 0.820 & 0.326 & 0.456 & 0.310 & 0.448 & 0.291 & 0.448 & 0.288 & 0.448 & 0.287 & 0.448 \\
& & skew & 0.253 & 0.103 & 0.208 & 0.101 & 0.205 & 0.101 & 0.199 & 0.101 & 0.197 & 0.101 & 0.197 & 0.101 \\
& & bimodal & 0.332 & 0.456 & 0.345 & 0.440 & 0.343 & 0.440 & 0.341 & 0.440 & 0.342 & 0.440 & 0.342 & 0.440 \\
\hline
\multirow{4}{*}{genome} & \multirow{4}{*}{2.0} & gaussian & 0.152 & 0.241 & 0.060 & 0.225 & 0.068 & 0.225 & 0.075 & 0.225 & 0.078 & 0.226 & 0.079 & 0.226 \\
& & scale & 0.334 & 0.112 & 0.293 & 0.103 & 0.273 & 0.100 & 0.242 & 0.098 & 0.235 & 0.096 & 0.229 & 0.098 \\
& & skew & 0.236 & 0.058 & 0.170 & 0.055 & 0.167 & 0.055 & 0.164 & 0.055 & 0.166 & 0.055 & 0.166 & 0.055 \\
& & bimodal & 0.386 & 0.092 & 0.275 & 0.088 & 0.274 & 0.088 & 0.273 & 0.089 & 0.273 & 0.088 & 0.273 & 0.088 \\
\hline
\multirow{4}{*}{genome} & \multirow{4}{*}{0.5} & gaussian & 0.316 & 0.562 & 0.060 & 0.480 & 0.060 & 0.480 & 0.061 & 0.480 & 0.061 & 0.480 & 0.061 & 0.480 \\
& & scale & 0.532 & 0.718 & 0.328 & 0.310 & 0.313 & 0.305 & 0.298 & 0.307 & 0.294 & 0.304 & 0.297 & 0.307 \\
& & skew & 0.301 & 0.123 & 0.191 & 0.113 & 0.186 & 0.113 & 0.179 & 0.113 & 0.178 & 0.113 & 0.178 & 0.113 \\
& & bimodal & 0.387 & 0.330 & 0.359 & 0.318 & 0.360 & 0.319 & 0.360 & 0.319 & 0.361 & 0.319 & 0.361 & 0.319

\end{tabular}}
\caption{Comparison of different settings for the smoothing spline penalty
$\lambda$.
The simulation setup is the same as in Table \ref{tab:tausq}, and
EBflow is performed with $c_\tau=0.5$.}
\label{tab:lambda}
\end{table}

\begin{figure}[h!]
\includegraphics[width=1\textwidth]{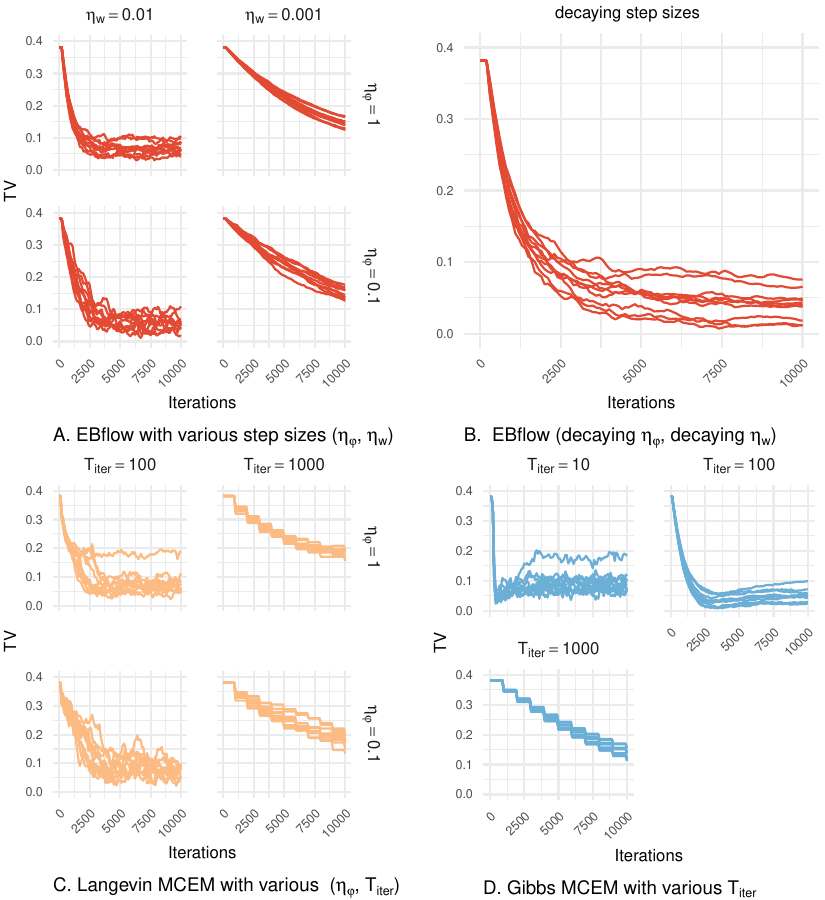}%
\caption{TV distance $\dTV(g_t,g_\ast)$ across 10000 iterations for estimating
a standard Gaussian prior $g_\ast=\N(0,1)$ via optimization of a smoothed
log-likelihood objective, using various parameter tunings for EBflow,
Langevin-MCEM, and Gibbs-MCEM (each with $10$ independent runs).}\label{fig:tuning}
\end{figure}

\begin{figure}[h!]
\centering
\includegraphics[width=0.5\textwidth]{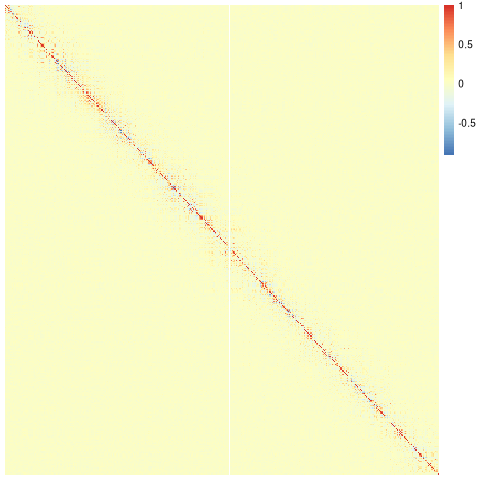}
\caption{Pairwise variable correlations in the simulated genotype
design}\label{fig:LD}
\end{figure}

\begin{figure}[h!]
\centering
\includegraphics[width=1\textwidth]{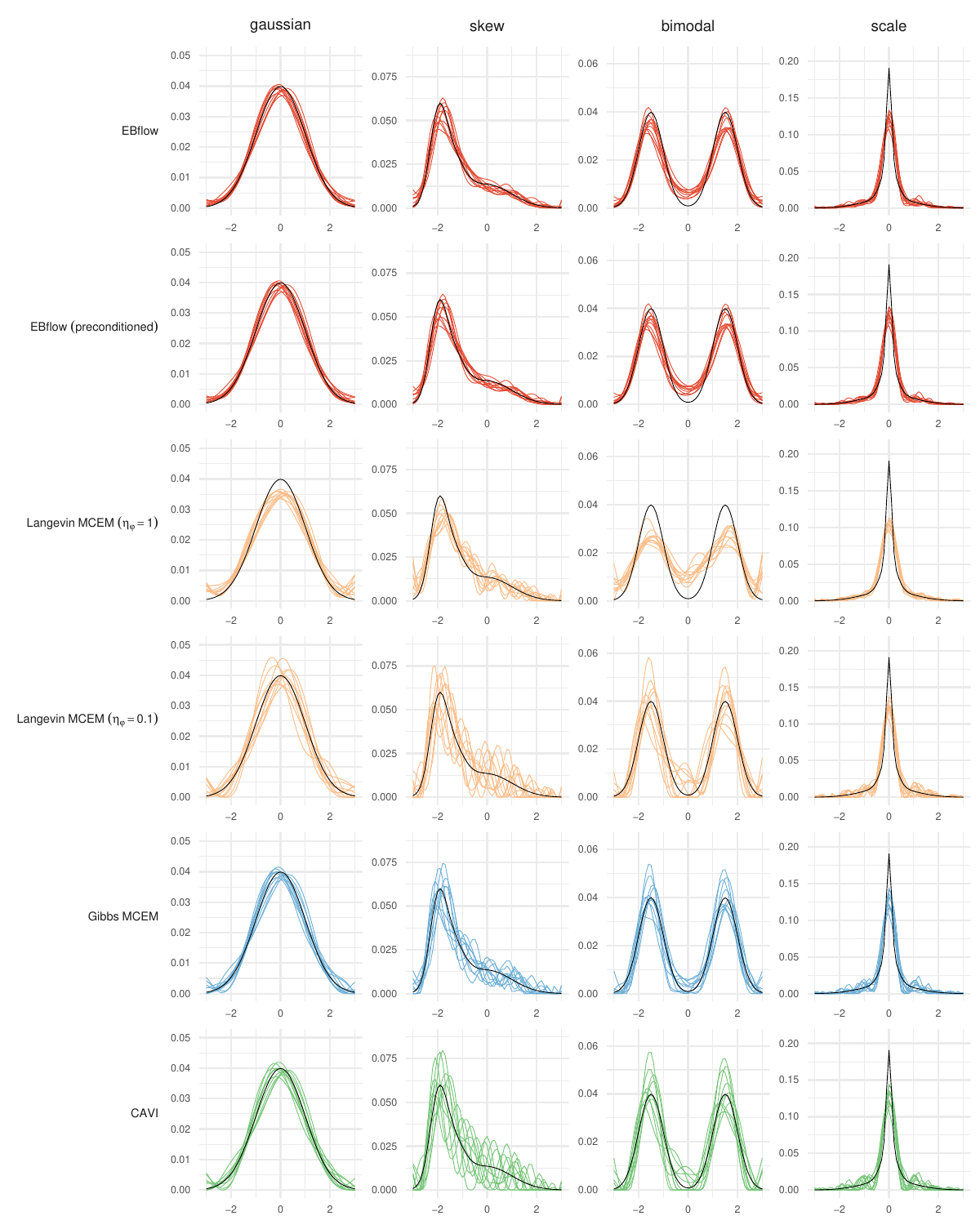}%
\caption{Estimated prior densities for identity design, across 10
independent trials.}\label{fig:identity}
\end{figure}

\begin{figure}[h!]
\centering
\includegraphics[width=1\textwidth]{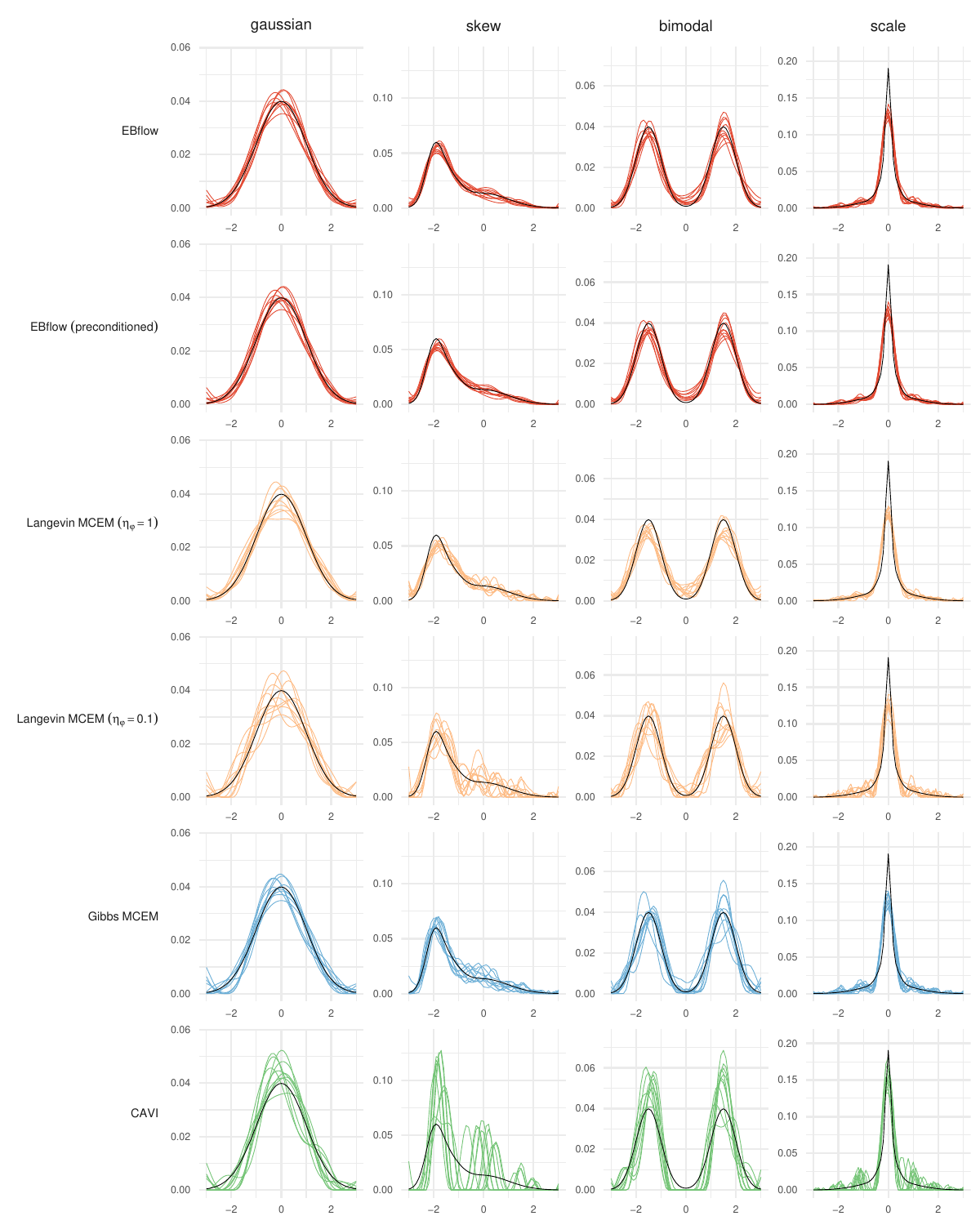}%
\caption{Estimated prior densities for iid design, $n/p=2.0$, across 10
independent trials.}\label{fig:iid}
\end{figure}

\begin{figure}[h!]
\centering
\includegraphics[width=1\textwidth]{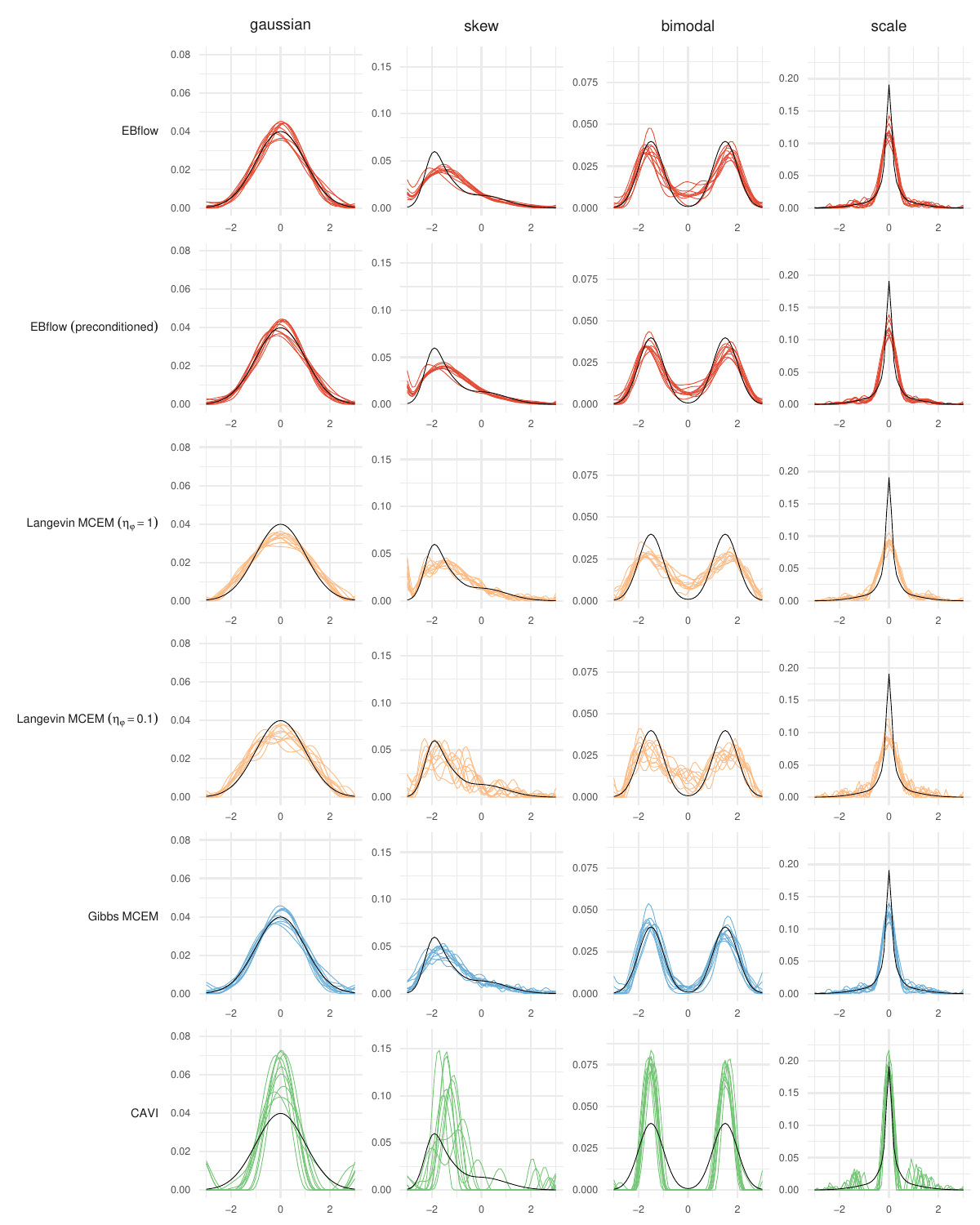}%
\caption{Estimated prior densities for block10corr0.5 design,
$n/p=2.0$, across 10 independent trials.}\label{fig:block10corr0.5}
\end{figure}

\begin{figure}[h!]
\centering
\includegraphics[width=1\textwidth]{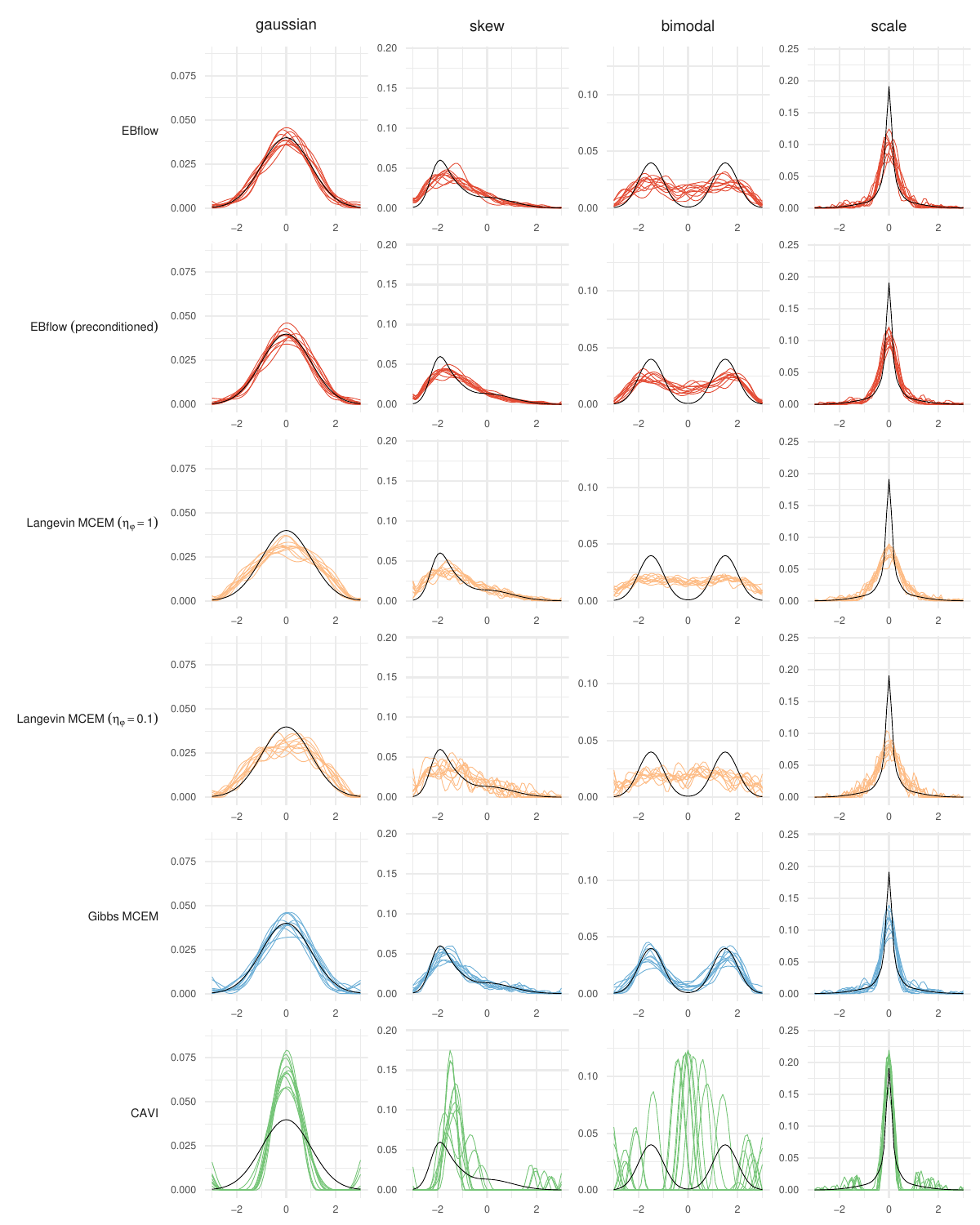}%
\caption{Estimated prior densities for genotype design, $n/p=2.0$, across 10 independent trials.}\label{fig:genotype}
\end{figure}

\begin{table}[h!]
\centering
{\scriptsize
\begin{tabular}{l|l|l|ccccc}
Design & $n/p$ & Method & TV & $\bar{F}_n$ & Time (s) \\
\hline
\multirow{7}{*}{identity} & \multirow{7}{*}{1.0} & EBflow & 0.039 & 1.772 & 3.08\\
&& EBflow.preconditioned & 0.039 & 1.772 & 4.62\\
&& Langevin.eta1.0 & 0.072 & 1.775 & 19.50\\
&& Langevin.eta0.1 & 0.075 & 1.771 & 19.49\\
&& Gibbs & 0.048 & 1.771 & 6.32\\
&& CAVI & 0.052 & 1.771 & 0.64\\
&& oracle & NA & 1.773 & NA\\
\hline
Design & $n/p$ & Method & TV & Prediction MSE & Time (s) \\
\hline
\multirow{7}{*}{iid} & \multirow{7}{*}{2.0} & EBflow & 0.051 & 0.420 & 3.13\\
&& EBflow.preconditioned & 0.046 & 0.419 & 4.34\\
&& Langevin.eta1.0 & 0.055 & 0.419 & 20.48\\
&& Langevin.eta0.1 & 0.087 & 0.423 & 20.66\\
&& Gibbs & 0.066 & 0.422 & 8.01\\
&& CAVI & 0.105 & 0.435 & 0.53\\
&& oracle & NA & 0.414 & NA\\
\hline
\multirow{7}{*}{iid} & \multirow{7}{*}{0.5} & EBflow & 0.040 & 0.797 & 5.45\\
&& EBflow.preconditioned & 0.045 & 0.795 & 7.78\\
&& Langevin.eta1.0 & 0.080 & 0.793 & 39.20\\
&& Langevin.eta0.1 & 0.073 & 0.793 & 76.14\\
&& Gibbs & 0.051 & 0.793 & 10.62\\
&& CAVI & 0.163 & 0.812 & 0.72\\
&& oracle & NA & 0.787 & NA\\
\hline
\multirow{7}{*}{block02corr0.9} & \multirow{7}{*}{2.0} & EBflow & 0.043 & 0.282 & 4.25\\
&& EBflow.preconditioned & 0.038 & 0.282 & 6.22\\
&& Langevin.eta1.0 & 0.092 & 0.284 & 25.76\\
&& Langevin.eta0.1 & 0.093 & 0.285 & 25.98\\
&& Gibbs & 0.058 & 0.281 & 11.21\\
&& CAVI & 0.153 & 0.302 & 0.55\\
&& oracle & NA & 0.279 & NA\\
\hline
\multirow{7}{*}{block02corr0.9} & \multirow{7}{*}{0.5} & EBflow & 0.054 & 0.680 & 5.87\\
&& EBflow.preconditioned & 0.056 & 0.680 & 8.51\\
&& Langevin.eta1.0 & 0.146 & 0.700 & 39.69\\
&& Langevin.eta0.1 & 0.116 & 0.685 & 58.07\\
&& Gibbs & 0.058 & 0.675 & 12.32\\
&& CAVI & 0.152 & 0.697 & 0.78\\
&& oracle & NA & 0.670 & NA\\
\hline
\multirow{7}{*}{block10corr0.5} & \multirow{7}{*}{2.0} & EBflow & 0.055 & 0.325 & 3.29\\
&& EBflow.preconditioned & 0.049 & 0.323 & 4.56\\
&& Langevin.eta1.0 & 0.087 & 0.325 & 15.23\\
&& Langevin.eta0.1 & 0.098 & 0.330 & 15.26\\
&& Gibbs & 0.058 & 0.323 & 9.82\\
&& CAVI & 0.255 & 0.404 & 0.60\\
&& oracle & NA & 0.320 & NA\\
\hline
\multirow{7}{*}{block10corr0.5} & \multirow{7}{*}{0.5} & EBflow & 0.054 & 0.589 & 6.42\\
&& EBflow.preconditioned & 0.055 & 0.587 & 9.01\\
&& Langevin.eta1.0 & 0.174 & 0.601 & 19.24\\
&& Langevin.eta0.1 & 0.168 & 0.605 & 19.10\\
&& Gibbs & 0.058 & 0.587 & 12.82\\
&& CAVI & 0.187 & 0.617 & 0.82\\
&& oracle & NA & 0.584 & NA\\
\hline
\multirow{7}{*}{genome} & \multirow{7}{*}{2.0} & EBflow & 0.060 & 0.225 & 3.40\\
&& EBflow.preconditioned & 0.055 & 0.222 & 6.25\\
&& Langevin.eta1.0 & 0.121 & 0.225 & 20.81\\
&& Langevin.eta0.1 & 0.132 & 0.230 & 20.87\\
&& Gibbs & 0.074 & 0.221 & 19.85\\
&& CAVI & 0.287 & 0.268 & 0.83\\
&& oracle & NA & 0.219 & NA\\
\hline
\multirow{7}{*}{genome} & \multirow{7}{*}{0.5} & EBflow & 0.060 & 0.480 & 6.29\\
&& EBflow.preconditioned & 0.053 & 0.477 & 9.85\\
&& Langevin.eta1.0 & 0.188 & 0.499 & 21.06\\
&& Langevin.eta0.1 & 0.193 & 0.506 & 20.71\\
&& Gibbs & 0.053 & 0.478 & 18.36\\
&& CAVI & 0.201 & 0.519 & 0.79\\
&& oracle & NA & 0.471 & NA\\

\end{tabular}}
\caption{Comparison of methods: Gaussian prior. TV errors, negative
log-likelihoods $\bar F_n$, and
prediction MSEs are as discussed in the main text, and averaged over 10
independent trials. Time (s) is the median compute time in seconds over the 10
trials needed to achieve TV error $<0.3$, or Inf if fewer than half of the
trials attained this TV error.}\label{tab:gaussian}

\end{table}

\begin{table}[h!]
\centering
{\scriptsize
\begin{tabular}{l|l|l|ccccc}
Design & $n/p$ & Method & TV & $\bar{F}_n$ & Time (s) \\
\hline
\multirow{7}{*}{identity} & \multirow{7}{*}{1.0} & EBflow & 0.078 & 1.547 & 3.05\\
&& EBflow.preconditioned & 0.078 & 1.547 & 4.71\\
&& Langevin.eta1.0 & 0.123 & 1.551 & 12.55\\
&& Langevin.eta0.1 & 0.191 & 1.547 & 12.55\\
&& Gibbs & 0.132 & 1.546 & 6.18\\
&& CAVI & 0.202 & 1.546 & 0.52\\
&& oracle & NA & 1.551 & NA\\
\hline
Design & $n/p$ & Method & TV & Prediction MSE & Time (s) \\
\hline
\multirow{7}{*}{iid} & \multirow{7}{*}{2.0} & EBflow & 0.069 & 0.115 & 3.28\\
&& EBflow.preconditioned & 0.070 & 0.115 & 4.77\\
&& Langevin.eta1.0 & 0.104 & 0.116 & 23.94\\
&& Langevin.eta0.1 & 0.171 & 0.116 & 22.24\\
&& Gibbs & 0.111 & 0.115 & 8.34\\
&& CAVI & 0.412 & 0.136 & 0.59\\
&& oracle & NA & 0.114 & NA\\
\hline
\multirow{7}{*}{iid} & \multirow{7}{*}{0.5} & EBflow & 0.194 & 0.324 & 5.51\\
&& EBflow.preconditioned & 0.192 & 0.322 & 7.40\\
&& Langevin.eta1.0 & 0.216 & 0.320 & 27.18\\
&& Langevin.eta0.1 & 0.213 & 0.322 & 56.57\\
&& Gibbs & 0.189 & 0.319 & 9.06\\
&& CAVI & 0.568 & 0.472 & 1.15\\
&& oracle & NA & 0.321 & NA\\
\hline
\multirow{7}{*}{block02corr0.9} & \multirow{7}{*}{2.0} & EBflow & 0.111 & 0.069 & 3.32\\
&& EBflow.preconditioned & 0.113 & 0.069 & 4.58\\
&& Langevin.eta1.0 & 0.129 & 0.069 & 17.13\\
&& Langevin.eta0.1 & 0.176 & 0.070 & 17.03\\
&& Gibbs & 0.120 & 0.069 & 9.88\\
&& CAVI & 0.497 & 0.099 & 0.99\\
&& oracle & NA & 0.069 & NA\\
\hline
\multirow{7}{*}{block02corr0.9} & \multirow{7}{*}{0.5} & EBflow & 0.195 & 0.201 & 4.32\\
&& EBflow.preconditioned & 0.194 & 0.201 & 6.54\\
&& Langevin.eta1.0 & 0.220 & 0.202 & 16.22\\
&& Langevin.eta0.1 & 0.262 & 0.202 & 16.11\\
&& Gibbs & 0.188 & 0.202 & 5.15\\
&& CAVI & 0.323 & 0.216 & 1.00\\
&& oracle & NA & 0.199 & NA\\
\hline
\multirow{7}{*}{block10corr0.5} & \multirow{7}{*}{2.0} & EBflow & 0.163 & 0.054 & 3.42\\
&& EBflow.preconditioned & 0.161 & 0.054 & 4.70\\
&& Langevin.eta1.0 & 0.178 & 0.054 & 13.85\\
&& Langevin.eta0.1 & 0.204 & 0.054 & 13.69\\
&& Gibbs & 0.142 & 0.053 & 9.25\\
&& CAVI & 0.499 & 0.074 & 1.29\\
&& oracle & NA & 0.053 & NA\\
\hline
\multirow{7}{*}{block10corr0.5} & \multirow{7}{*}{0.5} & EBflow & 0.197 & 0.101 & 4.51\\
&& EBflow.preconditioned & 0.199 & 0.100 & 6.74\\
&& Langevin.eta1.0 & 0.222 & 0.100 & 12.02\\
&& Langevin.eta0.1 & 0.222 & 0.100 & 11.97\\
&& Gibbs & 0.207 & 0.100 & 5.63\\
&& CAVI & 0.336 & 0.106 & 1.11\\
&& oracle & NA & 0.100 & NA\\
\hline
\multirow{7}{*}{genome} & \multirow{7}{*}{2.0} & EBflow & 0.166 & 0.055 & 4.45\\
&& EBflow.preconditioned & 0.150 & 0.054 & 7.21\\
&& Langevin.eta1.0 & 0.175 & 0.054 & 43.82\\
&& Langevin.eta0.1 & 0.246 & 0.056 & 46.15\\
&& Gibbs & 0.127 & 0.054 & 11.95\\
&& CAVI & 0.526 & 0.084 & 1.84\\
&& oracle & NA & 0.054 & NA\\
\hline
\multirow{7}{*}{genome} & \multirow{7}{*}{0.5} & EBflow & 0.178 & 0.113 & 6.00\\
&& EBflow.preconditioned & 0.171 & 0.113 & 7.85\\
&& Langevin.eta1.0 & 0.193 & 0.112 & 45.92\\
&& Langevin.eta0.1 & 0.254 & 0.114 & 311.69\\
&& Gibbs & 0.206 & 0.112 & 10.57\\
&& CAVI & 0.371 & 0.127 & 1.37\\
&& oracle & NA & 0.113 & NA\\

\end{tabular}}
\caption{Comparison of methods: Skew prior.}\label{tab:skew}
\end{table}

\begin{table}[h!]
\centering
{\scriptsize
\begin{tabular}{l|l|l|ccccc}
Design & $n/p$ & Method & TV & $\bar{F}_n$ & Time (s) \\
\hline
\multirow{7}{*}{identity} & \multirow{7}{*}{1.0} & EBflow & 0.162 & 0.917 & 4.34\\
&& EBflow.preconditioned & 0.162 & 0.917 & 6.39\\
&& Langevin.eta1.0 & 0.189 & 0.923 & 20.20\\
&& Langevin.eta0.1 & 0.196 & 0.918 & 18.75\\
&& Gibbs & 0.191 & 0.916 & 6.44\\
&& CAVI & 0.208 & 0.916 & 0.53\\
&& oracle & NA & 0.924 & NA\\
\hline
Design & $n/p$ & Method & TV & Prediction MSE & Time (s) \\
\hline
\multirow{7}{*}{iid} & \multirow{7}{*}{2.0} & EBflow & 0.149 & 0.137 & 4.47\\
&& EBflow.preconditioned & 0.146 & 0.136 & 6.05\\
&& Langevin.eta1.0 & 0.159 & 0.138 & 23.38\\
&& Langevin.eta0.1 & 0.204 & 0.140 & 23.31\\
&& Gibbs & 0.183 & 0.137 & 9.94\\
&& CAVI & 0.248 & 0.170 & 0.71\\
&& oracle & NA & 0.134 & NA\\
\hline
\multirow{7}{*}{iid} & \multirow{7}{*}{0.5} & EBflow & 0.279 & 0.628 & 24.02\\
&& EBflow.preconditioned & 0.277 & 0.618 & 34.45\\
&& Langevin.eta1.0 & 0.290 & 0.614 & 379.48\\
&& Langevin.eta0.1 & 0.305 & 0.641 & 1150.62\\
&& Gibbs & 0.264 & 0.605 & 39.17\\
&& CAVI & 0.316 & 0.869 & 0.93\\
&& oracle & NA & 0.590 & NA\\
\hline
\multirow{7}{*}{block02corr0.9} & \multirow{7}{*}{2.0} & EBflow & 0.202 & 0.111 & 6.74\\
&& EBflow.preconditioned & 0.197 & 0.108 & 9.81\\
&& Langevin.eta1.0 & 0.230 & 0.109 & 25.97\\
&& Langevin.eta0.1 & 0.268 & 0.113 & 28.03\\
&& Gibbs & 0.203 & 0.106 & 18.27\\
&& CAVI & 0.274 & 0.147 & 0.62\\
&& oracle & NA & 0.107 & NA\\
\hline
\multirow{7}{*}{block02corr0.9} & \multirow{7}{*}{0.5} & EBflow & 0.262 & 0.434 & 22.35\\
&& EBflow.preconditioned & 0.268 & 0.432 & 41.60\\
&& Langevin.eta1.0 & 0.297 & 0.427 & 575.01\\
&& Langevin.eta0.1 & 0.327 & 0.448 & Inf\\
&& Gibbs & 0.261 & 0.424 & 43.89\\
&& CAVI & 0.276 & 0.547 & 1.03\\
&& oracle & NA & 0.431 & NA\\
\hline
\multirow{7}{*}{block10corr0.5} & \multirow{7}{*}{2.0} & EBflow & 0.185 & 0.125 & 4.92\\
&& EBflow.preconditioned & 0.182 & 0.123 & 6.95\\
&& Langevin.eta1.0 & 0.234 & 0.126 & 22.31\\
&& Langevin.eta0.1 & 0.254 & 0.138 & 22.42\\
&& Gibbs & 0.183 & 0.121 & 15.12\\
&& CAVI & 0.289 & 0.178 & 0.64\\
&& oracle & NA & 0.122 & NA\\
\hline
\multirow{7}{*}{block10corr0.5} & \multirow{7}{*}{0.5} & EBflow & 0.288 & 0.448 & 33.00\\
&& EBflow.preconditioned & 0.288 & 0.437 & 56.40\\
&& Langevin.eta1.0 & 0.332 & 0.432 & Inf\\
&& Langevin.eta0.1 & 0.337 & 0.466 & Inf\\
&& Gibbs & 0.268 & 0.423 & 50.41\\
&& CAVI & 0.276 & 0.612 & 1.19\\
&& oracle & NA & 0.437 & NA\\
\hline
\multirow{7}{*}{genome} & \multirow{7}{*}{2.0} & EBflow & 0.235 & 0.096 & 4.69\\
&& EBflow.preconditioned & 0.212 & 0.093 & 8.50\\
&& Langevin.eta1.0 & 0.276 & 0.092 & 20.54\\
&& Langevin.eta0.1 & 0.288 & 0.099 & 21.59\\
&& Gibbs & 0.210 & 0.087 & 26.80\\
&& CAVI & 0.288 & 0.132 & 0.80\\
&& oracle & NA & 0.093 & NA\\
\hline
\multirow{7}{*}{genome} & \multirow{7}{*}{0.5} & EBflow & 0.294 & 0.304 & 33.32\\
&& EBflow.preconditioned & 0.289 & 0.296 & 57.14\\
&& Langevin.eta1.0 & 0.361 & 0.305 & 24.67\\
&& Langevin.eta0.1 & 0.381 & 0.356 & 25.63\\
&& Gibbs & 0.266 & 0.298 & 60.51\\
&& CAVI & 0.273 & 0.423 & 1.13\\
&& oracle & NA & 0.302 & NA\\

\end{tabular}}
\caption{Comparison of methods: Scale-mixture prior.}\label{tab:scale}
\end{table}

\begin{table}[h!]
\centering
{\scriptsize
\begin{tabular}{l|l|l|ccccc}
Design & $n/p$ & Method & TV & $\bar{F}_n$ & Time (s) \\
\hline
\multirow{7}{*}{identity} & \multirow{7}{*}{1.0} & EBflow & 0.096 & 1.898 & 6.41\\
&& EBflow.preconditioned & 0.096 & 1.898 & 9.28\\
&& Langevin.eta1.0 & 0.202 & 1.909 & 17.01\\
&& Langevin.eta0.1 & 0.141 & 1.896 & 16.32\\
&& Gibbs & 0.101 & 1.896 & 13.44\\
&& CAVI & 0.131 & 1.895 & 1.16\\
&& oracle & NA & 1.899 & NA\\
\hline
Design & $n/p$ & Method & TV & Prediction MSE & Time (s) \\
\hline
\multirow{7}{*}{iid} & \multirow{7}{*}{2.0} & EBflow & 0.074 & 0.105 & 4.36\\
&& EBflow.preconditioned & 0.073 & 0.105 & 5.95\\
&& Langevin.eta1.0 & 0.100 & 0.106 & 45.09\\
&& Langevin.eta0.1 & 0.116 & 0.104 & 67.79\\
&& Gibbs & 0.104 & 0.103 & 14.82\\
&& CAVI & 0.187 & 0.122 & 1.01\\
&& oracle & NA & 0.103 & NA\\
\hline
\multirow{7}{*}{iid} & \multirow{7}{*}{0.5} & EBflow & 0.347 & 0.644 & Inf\\
&& EBflow.preconditioned & 0.345 & 0.639 & Inf\\
&& Langevin.eta1.0 & 0.353 & 0.634 & Inf\\
&& Langevin.eta0.1 & 0.372 & 0.644 & 1126.43\\
&& Gibbs & 0.353 & 0.634 & Inf\\
&& CAVI & 0.717 & 1.027 & Inf\\
&& oracle & NA & 0.655 & NA\\
\hline
\multirow{7}{*}{block02corr0.9} & \multirow{7}{*}{2.0} & EBflow & 0.104 & 0.099 & 10.20\\
&& EBflow.preconditioned & 0.104 & 0.097 & 12.43\\
&& Langevin.eta1.0 & 0.209 & 0.098 & 223.92\\
&& Langevin.eta0.1 & 0.190 & 0.100 & 394.13\\
&& Gibbs & 0.077 & 0.092 & 36.39\\
&& CAVI & 0.564 & 0.166 & Inf\\
&& oracle & NA & 0.103 & NA\\
\hline
\multirow{7}{*}{block02corr0.9} & \multirow{7}{*}{0.5} & EBflow & 0.343 & 0.455 & Inf\\
&& EBflow.preconditioned & 0.336 & 0.452 & Inf\\
&& Langevin.eta1.0 & 0.345 & 0.448 & Inf\\
&& Langevin.eta0.1 & 0.400 & 0.458 & Inf\\
&& Gibbs & 0.347 & 0.448 & Inf\\
&& CAVI & 0.508 & 0.512 & Inf\\
&& oracle & NA & 0.475 & NA\\
\hline
\multirow{7}{*}{block10corr0.5} & \multirow{7}{*}{2.0} & EBflow & 0.138 & 0.109 & 9.92\\
&& EBflow.preconditioned & 0.126 & 0.107 & 11.84\\
&& Langevin.eta1.0 & 0.201 & 0.108 & 114.75\\
&& Langevin.eta0.1 & 0.217 & 0.112 & 485.17\\
&& Gibbs & 0.101 & 0.103 & 28.15\\
&& CAVI & 0.330 & 0.145 & 1.99\\
&& oracle & NA & 0.112 & NA\\
\hline
\multirow{7}{*}{block10corr0.5} & \multirow{7}{*}{0.5} & EBflow & 0.341 & 0.440 & Inf\\
&& EBflow.preconditioned & 0.330 & 0.433 & Inf\\
&& Langevin.eta1.0 & 0.355 & 0.428 & Inf\\
&& Langevin.eta0.1 & 0.377 & 0.446 & Inf\\
&& Gibbs & 0.355 & 0.428 & Inf\\
&& CAVI & 0.687 & 0.553 & Inf\\
&& oracle & NA & 0.480 & NA\\
\hline
\multirow{7}{*}{genome} & \multirow{7}{*}{2.0} & EBflow & 0.273 & 0.089 & 59.10\\
&& EBflow.preconditioned & 0.246 & 0.084 & 37.56\\
&& Langevin.eta1.0 & 0.325 & 0.082 & 614.34\\
&& Langevin.eta0.1 & 0.344 & 0.089 & Inf\\
&& Gibbs & 0.138 & 0.079 & 90.30\\
&& CAVI & 0.809 & 0.162 & Inf\\
&& oracle & NA & 0.099 & NA\\
\hline
\multirow{7}{*}{genome} & \multirow{7}{*}{0.5} & EBflow & 0.360 & 0.319 & Inf\\
&& EBflow.preconditioned & 0.336 & 0.311 & Inf\\
&& Langevin.eta1.0 & 0.361 & 0.305 & Inf\\
&& Langevin.eta0.1 & 0.354 & 0.319 & Inf\\
&& Gibbs & 0.371 & 0.304 & Inf\\
&& CAVI & 0.787 & 0.423 & Inf\\
&& oracle & NA & 0.343 & NA\\

\end{tabular}}
\caption{Comparison of methods: Bimodal prior.}\label{tab:bimodal}
\end{table}

\FloatBarrier

\bibliographystyle{amsalpha}

{\small
\bibliography{langevin_EB} }

@book {nemirovsky1983problem,
    AUTHOR = {Nemirovsky, A. S. and Yudin, D. B.},
     TITLE = {Problem complexity and method efficiency in optimization},
    SERIES = {Wiley-Interscience Series in Discrete Mathematics},
      NOTE = {Translated from the Russian and with a preface by E. R. Dawson,
              A Wiley-Interscience Publication},
 PUBLISHER = {John Wiley \& Sons, Inc., New York},
      YEAR = {1983},
     PAGES = {xv+388},
      ISBN = {0-471-10345-4},
   MRCLASS = {90C25 (68C25)},
  MRNUMBER = {702836},
}

@inproceedings{kuntz2023particle,
  title={Particle algorithms for maximum likelihood training of latent variable models},
  author={Kuntz, Juan and Lim, Jen Ning and Johansen, Adam M},
  booktitle={International Conference on Artificial Intelligence and Statistics},
  pages={5134--5180},
  year={2023},
  organization={PMLR}
}

@article{akyildiz2023interacting,
  title={Interacting Particle {L}angevin Algorithm for Maximum Marginal Likelihood Estimation},
  author={Akyildiz, {\"O} Deniz and Crucinio, Francesca Romana and Girolami, Mark and Johnston, Tim and Sabanis, Sotirios},
  journal={arXiv preprint arXiv:2303.13429},
  year={2023}
}

@book {boucheron2013concentration,
    AUTHOR = {Boucheron, St\'{e}phane and Lugosi, G\'{a}bor and Massart,
              Pascal},
     TITLE = {Concentration inequalities: A nonasymptotic theory of independence},
 PUBLISHER = {Oxford University Press, Oxford},
      YEAR = {2013},
     PAGES = {x+481},
      ISBN = {978-0-19-953525-5},
   MRCLASS = {60E15 (60D05 60G15 60G50 60G70 62G20)},
  MRNUMBER = {3185193},
MRREVIEWER = {Sreenivasan\ Ravi},
       DOI = {10.1093/acprof:oso/9780199535255.001.0001},
       URL = {https://doi.org/10.1093/acprof:oso/9780199535255.001.0001},
}

@book {legall2016brownian,
    AUTHOR = {Le Gall, Jean-Fran\c{c}ois},
     TITLE = {Brownian motion, martingales, and stochastic calculus},
    SERIES = {Graduate Texts in Mathematics},
    VOLUME = {274},
   EDITION = {French},
 PUBLISHER = {Springer, [Cham]},
      YEAR = {2016},
     PAGES = {xiii+273},
      ISBN = {978-3-319-31088-6; 978-3-319-31089-3},
   MRCLASS = {60H05 (60G07 60G44 60H10 60J25 60J55 60J65)},
  MRNUMBER = {3497465},
       DOI = {10.1007/978-3-319-31089-3},
       URL = {https://doi.org/10.1007/978-3-319-31089-3},
}

@article {wang2007fast,
    AUTHOR = {Wang, Yong},
     TITLE = {On fast computation of the non-parametric maximum likelihood
              estimate of a mixing distribution},
   JOURNAL = {J. R. Stat. Soc. Ser. B Stat. Methodol.},
  FJOURNAL = {Journal of the Royal Statistical Society. Series B.
              Statistical Methodology},
    VOLUME = {69},
      YEAR = {2007},
    NUMBER = {2},
     PAGES = {185--198},
      ISSN = {1369-7412,1467-9868},
   MRCLASS = {62G05 (62G07)},
  MRNUMBER = {2325271},
       DOI = {10.1111/j.1467-9868.2007.00583.x},
       URL = {https://doi.org/10.1111/j.1467-9868.2007.00583.x},
}

@article {liu2007partially,
    AUTHOR = {Liu, Lei and Zhu, Yu},
     TITLE = {Partially projected gradient algorithms for computing
              nonparametric maximum likelihood estimates of mixing
              distributions},
   JOURNAL = {J. Statist. Plann. Inference},
  FJOURNAL = {Journal of Statistical Planning and Inference},
    VOLUME = {137},
      YEAR = {2007},
    NUMBER = {7},
     PAGES = {2509--2522},
      ISSN = {0378-3758,1873-1171},
   MRCLASS = {62G05},
  MRNUMBER = {2325452},
       DOI = {10.1016/j.jspi.2006.10.003},
       URL = {https://doi.org/10.1016/j.jspi.2006.10.003},
}

@article{lesperance1992algorithm,
  title={An algorithm for computing the nonparametric {MLE} of a mixing distribution},
  author={Lesperance, Mary L and Kalbfleisch, John D},
  journal={Journal of the American Statistical Association},
  volume={87},
  number={417},
  pages={120--126},
  year={1992},
  publisher={Taylor \& Francis}
}

@article {laird1978nonparametric,
    AUTHOR = {Laird, Nan},
     TITLE = {Nonparametric maximum likelihood estimation of a mixed
              distribution},
   JOURNAL = {J. Amer. Statist. Assoc.},
  FJOURNAL = {Journal of the American Statistical Association},
    VOLUME = {73},
      YEAR = {1978},
    NUMBER = {364},
     PAGES = {805--811},
      ISSN = {0162-1459,1537-274X},
   MRCLASS = {62G05},
  MRNUMBER = {521328},
MRREVIEWER = {Ramesh\ M.\ Korwar},
       URL =
              {http://links.jstor.org/sici?sici=0162-1459(197812)73:364<805:NMLEOA>2.0.CO;2-J&origin=MSN},
}

@article{zhang2022efficient,
  title={On efficient and scalable computation of the nonparametric maximum likelihood estimator in mixture models},
  author={Zhang, Yangjing and Cui, Ying and Sen, Bodhisattva and Toh, Kim-Chuan},
  journal={arXiv preprint arXiv:2208.07514},
  year={2022}
}

@article {koenker2014convex,
    AUTHOR = {Koenker, Roger and Mizera, Ivan},
     TITLE = {Convex optimization, shape constraints, compound decisions,
              and empirical {B}ayes rules},
   JOURNAL = {J. Amer. Statist. Assoc.},
  FJOURNAL = {Journal of the American Statistical Association},
    VOLUME = {109},
      YEAR = {2014},
    NUMBER = {506},
     PAGES = {674--685},
      ISSN = {0162-1459,1537-274X},
   MRCLASS = {62C12 (62C25 62G05 62G07)},
  MRNUMBER = {3223742},
MRREVIEWER = {Jingjing\ Wu},
       DOI = {10.1080/01621459.2013.869224},
       URL = {https://doi.org/10.1080/01621459.2013.869224},
}

@article {groeneboom2008support,
    AUTHOR = {Groeneboom, Piet and Jongbloed, Geurt and Wellner, Jon A.},
     TITLE = {The support reduction algorithm for computing non-parametric
              function estimates in mixture models},
   JOURNAL = {Scand. J. Statist.},
  FJOURNAL = {Scandinavian Journal of Statistics. Theory and Applications},
    VOLUME = {35},
      YEAR = {2008},
    NUMBER = {3},
     PAGES = {385--399},
      ISSN = {0303-6898,1467-9469},
   MRCLASS = {62G08 (62F10 62G05 65C60 82C80)},
  MRNUMBER = {2446726},
MRREVIEWER = {Carlos\ Martins-Filho},
       DOI = {10.1111/j.1467-9469.2007.00588.x},
       URL = {https://doi.org/10.1111/j.1467-9469.2007.00588.x},
}

@incollection {bohning1995review,
    AUTHOR = {B{\"{o}}hning, Dankmar},
     TITLE = {A review of reliable maximum likelihood algorithms for
              semiparametric mixture models},
      NOTE = {Statistical modelling (Leuven, 1993)},
   JOURNAL = {J. Statist. Plann. Inference},
  FJOURNAL = {Journal of Statistical Planning and Inference},
    VOLUME = {47},
      YEAR = {1995},
    NUMBER = {1-2},
     PAGES = {5--28},
      ISSN = {0378-3758,1873-1171},
   MRCLASS = {62G05},
  MRNUMBER = {1360956},
       DOI = {10.1016/0378-3758(94)00119-G},
       URL = {https://doi.org/10.1016/0378-3758(94)00119-G},
}

@article {saha2020nonparametric,
    AUTHOR = {Saha, Sujayam and Guntuboyina, Adityanand},
     TITLE = {On the nonparametric maximum likelihood estimator for
              {G}aussian location mixture densities with application to
              {G}aussian denoising},
   JOURNAL = {Ann. Statist.},
  FJOURNAL = {The Annals of Statistics},
    VOLUME = {48},
      YEAR = {2020},
    NUMBER = {2},
     PAGES = {738--762},
      ISSN = {0090-5364,2168-8966},
   MRCLASS = {62G07 (62C10 62C12)},
  MRNUMBER = {4102674},
       DOI = {10.1214/19-AOS1817},
       URL = {https://doi.org/10.1214/19-AOS1817},
}

@article{polyanskiy2021sharp,
  title={Sharp regret bounds for empirical Bayes and compound decision problems},
  author={Polyanskiy, Yury and Wu, Yihong},
  journal={arXiv preprint arXiv:2109.03943},
  year={2021}
}

@article {ghosal2001entropies,
    AUTHOR = {Ghosal, Subhashis and van der Vaart, Aad W.},
     TITLE = {Entropies and rates of convergence for maximum likelihood and
              {B}ayes estimation for mixtures of normal densities},
   JOURNAL = {Ann. Statist.},
  FJOURNAL = {The Annals of Statistics},
    VOLUME = {29},
      YEAR = {2001},
    NUMBER = {5},
     PAGES = {1233--1263},
      ISSN = {0090-5364,2168-8966},
   MRCLASS = {62G07 (62G20)},
  MRNUMBER = {1873329},
       DOI = {10.1214/aos/1013203453},
       URL = {https://doi.org/10.1214/aos/1013203453},
}

@article {genovese2000rates,
    AUTHOR = {Genovese, Christopher R. and Wasserman, Larry},
     TITLE = {Rates of convergence for the {G}aussian mixture sieve},
   JOURNAL = {Ann. Statist.},
  FJOURNAL = {The Annals of Statistics},
    VOLUME = {28},
      YEAR = {2000},
    NUMBER = {4},
     PAGES = {1105--1127},
      ISSN = {0090-5364,2168-8966},
   MRCLASS = {62G07 (62G20)},
  MRNUMBER = {1810921},
MRREVIEWER = {Tobias\ Ryd\'{e}n},
       DOI = {10.1214/aos/1015956709},
       URL = {https://doi.org/10.1214/aos/1015956709},
}

@article {kakutani1948equivalence,
    AUTHOR = {Kakutani, Shizuo},
     TITLE = {On equivalence of infinite product measures},
   JOURNAL = {Ann. of Math. (2)},
  FJOURNAL = {Annals of Mathematics. Second Series},
    VOLUME = {49},
      YEAR = {1948},
     PAGES = {214--224},
      ISSN = {0003-486X},
   MRCLASS = {27.2X},
  MRNUMBER = {23331},
MRREVIEWER = {B.\ Jessen},
       DOI = {10.2307/1969123},
       URL = {https://doi.org/10.2307/1969123},
}

@article {hellinger1909,
    AUTHOR = {Hellinger, E.},
     TITLE = {Neue {B}egr\"{u}ndung der {T}heorie quadratischer {F}ormen von
              unendlichvielen {V}er\"{a}nderlichen},
   JOURNAL = {J. Reine Angew. Math.},
  FJOURNAL = {Journal f\"{u}r die Reine und Angewandte Mathematik. [Crelle's
              Journal]},
    VOLUME = {136},
      YEAR = {1909},
     PAGES = {210--271},
      ISSN = {0075-4102,1435-5345},
   MRCLASS = {99-04},
  MRNUMBER = {1580780},
       DOI = {10.1515/crll.1909.136.210},
       URL = {https://doi.org/10.1515/crll.1909.136.210},
}

@article{CHIZAT20183090,
title = {Unbalanced optimal transport: Dynamic and {K}antorovich formulations},
journal = {Journal of Functional Analysis},
volume = {274},
number = {11},
pages = {3090-3123},
year = {2018},
issn = {0022-1236},
doi = {https://doi.org/10.1016/j.jfa.2018.03.008},
url = {https://www.sciencedirect.com/science/article/pii/S0022123618301058},
author = {Lénaïc Chizat and Gabriel Peyré and Bernhard Schmitzer and François-Xavier Vialard}
}

@article {liero2016optimal,
    AUTHOR = {Liero, Matthias and Mielke, Alexander and Savar\'{e},
              Giuseppe},
     TITLE = {Optimal transport in competition with reaction: the
              {H}ellinger-{K}antorovich distance and geodesic curves},
   JOURNAL = {SIAM J. Math. Anal.},
  FJOURNAL = {SIAM Journal on Mathematical Analysis},
    VOLUME = {48},
      YEAR = {2016},
    NUMBER = {4},
     PAGES = {2869--2911},
      ISSN = {0036-1410,1095-7154},
   MRCLASS = {28A33 (49K21 58E30)},
  MRNUMBER = {3542003},
MRREVIEWER = {Ilaria\ Fragal\`a},
       DOI = {10.1137/15M1041420},
       URL = {https://doi.org/10.1137/15M1041420},
}

@article {kondratyev2016new,
    AUTHOR = {Kondratyev, Stanislav and Monsaingeon, L\'{e}onard and
              Vorotnikov, Dmitry},
     TITLE = {A new optimal transport distance on the space of finite
              {R}adon measures},
   JOURNAL = {Adv. Differential Equations},
  FJOURNAL = {Advances in Differential Equations},
    VOLUME = {21},
      YEAR = {2016},
    NUMBER = {11-12},
     PAGES = {1117--1164},
      ISSN = {1079-9389},
   MRCLASS = {49Q20 (28A33 35L60 35Q92 58B20)},
  MRNUMBER = {3556762},
MRREVIEWER = {Giandomenico\ Orlandi},
       URL = {http://projecteuclid.org/euclid.ade/1476369298},
}

@article {liero2018optimal,
    AUTHOR = {Liero, Matthias and Mielke, Alexander and Savar\'{e},
              Giuseppe},
     TITLE = {Optimal entropy-transport problems and a new
              {H}ellinger-{K}antorovich distance between positive measures},
   JOURNAL = {Invent. Math.},
  FJOURNAL = {Inventiones Mathematicae},
    VOLUME = {211},
      YEAR = {2018},
    NUMBER = {3},
     PAGES = {969--1117},
      ISSN = {0020-9910,1432-1297},
   MRCLASS = {49Q20 (28A33 28C15)},
  MRNUMBER = {3763404},
MRREVIEWER = {Davide\ Vittone},
       DOI = {10.1007/s00222-017-0759-8},
       URL = {https://doi.org/10.1007/s00222-017-0759-8},
}

@article {chizat2018interpolating,
    AUTHOR = {Chizat, L\'{e}na\"{\i}c and Peyr\'{e}, Gabriel and Schmitzer,
              Bernhard and Vialard, Fran\c{c}ois-Xavier},
     TITLE = {An interpolating distance between optimal transport and
              {F}isher-{R}ao metrics},
   JOURNAL = {Found. Comput. Math.},
  FJOURNAL = {Foundations of Computational Mathematics. The Journal of the
              Society for the Foundations of Computational Mathematics},
    VOLUME = {18},
      YEAR = {2018},
    NUMBER = {1},
     PAGES = {1--44},
      ISSN = {1615-3375,1615-3383},
   MRCLASS = {49Q20},
  MRNUMBER = {3749413},
MRREVIEWER = {Luca\ Granieri},
       DOI = {10.1007/s10208-016-9331-y},
       URL = {https://doi.org/10.1007/s10208-016-9331-y},
}

@inproceedings{wibisono2018sampling,
  title={Sampling as optimization in the space of measures: The {L}angevin dynamics as a composite optimization problem},
  author={Wibisono, Andre},
  booktitle={Conference on Learning Theory},
  pages={2093--3027},
  year={2018},
  organization={PMLR}
}

@article{vempala2019rapid,
  title={Rapid convergence of the unadjusted {L}angevin algorithm: Isoperimetry suffices},
  author={Vempala, Santosh and Wibisono, Andre},
  journal={Advances in neural information processing systems},
  volume={32},
  year={2019}
}

@article {durmus2017nonasymptotic,
    AUTHOR = {Durmus, Alain and Moulines, \'{E}ric},
     TITLE = {Nonasymptotic convergence analysis for the unadjusted
              {L}angevin algorithm},
   JOURNAL = {Ann. Appl. Probab.},
  FJOURNAL = {The Annals of Applied Probability},
    VOLUME = {27},
      YEAR = {2017},
    NUMBER = {3},
     PAGES = {1551--1587},
      ISSN = {1050-5164,2168-8737},
   MRCLASS = {65C05 (60F05 60J05 65C40 93E35)},
  MRNUMBER = {3678479},
       DOI = {10.1214/16-AAP1238},
       URL = {https://doi.org/10.1214/16-AAP1238},
}

@article {dalalyan2017theoretical,
    AUTHOR = {Dalalyan, Arnak S.},
     TITLE = {Theoretical guarantees for approximate sampling from smooth
              and log-concave densities},
   JOURNAL = {J. R. Stat. Soc. Ser. B. Stat. Methodol.},
  FJOURNAL = {Journal of the Royal Statistical Society. Series B.
              Statistical Methodology},
    VOLUME = {79},
      YEAR = {2017},
    NUMBER = {3},
     PAGES = {651--676},
      ISSN = {1369-7412,1467-9868},
   MRCLASS = {60J22 (60B10 60J05 60J70)},
  MRNUMBER = {3641401},
       DOI = {10.1111/rssb.12183},
       URL = {https://doi.org/10.1111/rssb.12183},
}

@article {otto2001geometry,
    AUTHOR = {Otto, Felix},
     TITLE = {The geometry of dissipative evolution equations: the porous
              medium equation},
   JOURNAL = {Comm. Partial Differential Equations},
  FJOURNAL = {Communications in Partial Differential Equations},
    VOLUME = {26},
      YEAR = {2001},
    NUMBER = {1-2},
     PAGES = {101--174},
      ISSN = {0360-5302,1532-4133},
   MRCLASS = {35K57 (35K65 76S05)},
  MRNUMBER = {1842429},
MRREVIEWER = {Antonio\ Fasano},
       DOI = {10.1081/PDE-100002243},
       URL = {https://doi.org/10.1081/PDE-100002243},
}

@article{hotelling1930spaces,
  title={Spaces of statistical parameters},
  author={Hotelling, Harold},
  journal={Bull. Amer. Math. Soc},
  volume={36},
  pages={191},
  year={1930}
}

@article {rao1945information,
    AUTHOR = {Radhakrishna Rao, C.},
     TITLE = {Information and the accuracy attainable in the estimation of
              statistical parameters},
   JOURNAL = {Bull. Calcutta Math. Soc.},
  FJOURNAL = {Bulletin of the Calcutta Mathematical Society},
    VOLUME = {37},
      YEAR = {1945},
     PAGES = {81--91},
      ISSN = {0008-0659},
   MRCLASS = {62.0X},
  MRNUMBER = {15748},
MRREVIEWER = {J.\ W.\ Tukey},
}

@article {vardi1993image,
    AUTHOR = {Vardi, Y. and Lee, D.},
     TITLE = {From image deblurring to optimal investments: maximum
              likelihood solutions for positive linear inverse problems},
      NOTE = {With discussion},
   JOURNAL = {J. Roy. Statist. Soc. Ser. B},
  FJOURNAL = {Journal of the Royal Statistical Society. Series B.
              Methodological},
    VOLUME = {55},
      YEAR = {1993},
    NUMBER = {3},
     PAGES = {569--612},
      ISSN = {0035-9246},
   MRCLASS = {62M20 (45B05 60G60 65U05)},
  MRNUMBER = {1223930},
MRREVIEWER = {V.\ V.\ Anh},
       URL =
              {http://links.jstor.org/sici?sici=0035-9246(1993)55:3<569:FIDTOI>2.0.CO;2-N&origin=MSN},
}

@article {cover1984algorithm,
    AUTHOR = {Cover, Thomas M.},
     TITLE = {An algorithm for maximizing expected log investment return},
   JOURNAL = {IEEE Trans. Inform. Theory},
  FJOURNAL = {Institute of Electrical and Electronics Engineers.
              Transactions on Information Theory},
    VOLUME = {30},
      YEAR = {1984},
    NUMBER = {2},
     PAGES = {369--373},
      ISSN = {0018-9448,1557-9654},
   MRCLASS = {90A09},
  MRNUMBER = {754868},
       DOI = {10.1109/TIT.1984.1056869},
       URL = {https://doi.org/10.1109/TIT.1984.1056869},
}

@article {dempster1977maximum,
    AUTHOR = {Dempster, A. P. and Laird, N. M. and Rubin, D. B.},
     TITLE = {Maximum likelihood from incomplete data via the {EM}
              algorithm},
      NOTE = {With discussion},
   JOURNAL = {J. Roy. Statist. Soc. Ser. B},
  FJOURNAL = {Journal of the Royal Statistical Society. Series B.
              Methodological},
    VOLUME = {39},
      YEAR = {1977},
    NUMBER = {1},
     PAGES = {1--38},
      ISSN = {0035-9246},
   MRCLASS = {62F10},
  MRNUMBER = {501537},
MRREVIEWER = {Rolf\ Sundberg},
       URL =
              {http://links.jstor.org/sici?sici=0035-9246(1977)39:1<1:MLFIDV>2.0.CO;2-Z&origin=MSN},
}

@article{robbins1950generalization,
  title={A generalization of the method of maximum likelihood-estimating a mixing distribution},
  author={Robbins, Herbert},
  booktitle={Annals of Mathematical Statistics},
  volume={21},
  number={2},
  pages={314--315},
  year={1950}
}

@book {carlin1996bayes,
    AUTHOR = {Carlin, Bradley P. and Louis, Thomas A.},
     TITLE = {Bayes and empirical {B}ayes methods for data analysis},
    SERIES = {Monographs on Statistics and Applied Probability},
    VOLUME = {69},
 PUBLISHER = {Chapman \& Hall, London},
      YEAR = {1996},
     PAGES = {xvi+399},
      ISBN = {0-412-05611-9},
   MRCLASS = {62-01 (62-07 62C12 62F15)},
  MRNUMBER = {1427749},
}

@book {efron2010large,
    AUTHOR = {Efron, Bradley},
     TITLE = {Large-scale inference},
    SERIES = {Institute of Mathematical Statistics (IMS) Monographs},
    VOLUME = {1},
      NOTE = {Empirical Bayes methods for estimation, testing, and
              prediction},
 PUBLISHER = {Cambridge University Press, Cambridge},
      YEAR = {2010},
     PAGES = {xii+263},
      ISBN = {978-0-521-19249-1},
   MRCLASS = {62-02 (62-07 62C12 62H15 62J15)},
  MRNUMBER = {2724758},
MRREVIEWER = {I.\ N.\ Volodin},
       DOI = {10.1017/CBO9780511761362},
       URL = {https://doi.org/10.1017/CBO9780511761362},
}

@book {maritz1989empirical,
    AUTHOR = {Maritz, J. S. and Lwin, T.},
     TITLE = {Empirical {B}ayes methods},
    SERIES = {Monographs on Statistics and Applied Probability},
    VOLUME = {35},
   EDITION = {2nd},
 PUBLISHER = {Chapman \& Hall, London},
      YEAR = {1989},
     PAGES = {xii+284},
      ISBN = {0-412-27760-3},
   MRCLASS = {62C12 (62-01)},
  MRNUMBER = {1019835},
MRREVIEWER = {Lionel\ Weiss},
}

@misc{efron2021empirical,
  title={Empirical bayes: Concepts and methods},
  author={Efron, Bradley},
  year={2021}
}

@incollection {zhang2003compound,
    AUTHOR = {Zhang, Cun-Hui},
     TITLE = {Compound decision theory and empirical {B}ayes methods},
      NOTE = {Dedicated to the memory of Herbert E. Robbins},
   JOURNAL = {Ann. Statist.},
  FJOURNAL = {The Annals of Statistics},
    VOLUME = {31},
      YEAR = {2003},
    NUMBER = {2},
     PAGES = {379--390},
      ISSN = {0090-5364,2168-8966},
   MRCLASS = {62C12 (01A70)},
  MRNUMBER = {1983534},
MRREVIEWER = {Goetz\ Trenkler},
       DOI = {10.1214/aos/1051027872},
       URL = {https://doi.org/10.1214/aos/1051027872},
}

@article {casella1985introduction,
    AUTHOR = {Casella, George},
     TITLE = {An introduction to empirical {B}ayes data analysis},
   JOURNAL = {Amer. Statist.},
  FJOURNAL = {The American Statistician},
    VOLUME = {39},
      YEAR = {1985},
    NUMBER = {2},
     PAGES = {83--87},
      ISSN = {0003-1305,1537-2731},
   MRCLASS = {62C12},
  MRNUMBER = {789118},
       DOI = {10.2307/2682801},
       URL = {https://doi.org/10.2307/2682801},
}

@inproceedings{Robbins51,
  title={Asymptotically Subminimax Solutions of Compound Statistical Decision Problems},
  author={Robbins, Herbert},
  booktitle={Proceedings of the Second Berkeley Symposium on Mathematical Statistics and Probability},
  year={1951},
  organization={The Regents of the University of California}
}

@inproceedings{Robbins56,
  title={An empirical {B}ayes Approach to Statistics},
  author={Robbins, Herbert},
  booktitle={Proceedings of the Third Berkeley Symposium on Mathematical Statistics and Probability, Volume 1: Contributions to the Theory of Statistics},
  year={1956},
  organization={The Regents of the University of California}
}

@article {chen2017consistency,
    AUTHOR = {Chen, Jiahua},
     TITLE = {Consistency of the {MLE} under mixture models},
   JOURNAL = {Statist. Sci.},
  FJOURNAL = {Statistical Science. A Review Journal of the Institute of
              Mathematical Statistics},
    VOLUME = {32},
      YEAR = {2017},
    NUMBER = {1},
     PAGES = {47--63},
      ISSN = {0883-4237,2168-8745},
   MRCLASS = {62G05 (62G20)},
  MRNUMBER = {3634306},
MRREVIEWER = {Kaushik\ Ghosh},
       DOI = {10.1214/16-STS578},
       URL = {https://doi.org/10.1214/16-STS578},
}

@article {pfanzagl1988consistency,
    AUTHOR = {Pfanzagl, J.},
     TITLE = {Consistency of maximum likelihood estimators for certain
              nonparametric families, in particular: mixtures},
   JOURNAL = {J. Statist. Plann. Inference},
  FJOURNAL = {Journal of Statistical Planning and Inference},
    VOLUME = {19},
      YEAR = {1988},
    NUMBER = {2},
     PAGES = {137--158},
      ISSN = {0378-3758,1873-1171},
   MRCLASS = {62G05 (62F12)},
  MRNUMBER = {944202},
MRREVIEWER = {Ram\ Shanmugam},
       DOI = {10.1016/0378-3758(88)90069-9},
       URL = {https://doi.org/10.1016/0378-3758(88)90069-9},
}

@article {lambert1984asymptotic,
    AUTHOR = {Lambert, Diane and Tierney, Luke},
     TITLE = {Asymptotic properties of maximum likelihood estimates in the
              mixed {P}oisson model},
   JOURNAL = {Ann. Statist.},
  FJOURNAL = {The Annals of Statistics},
    VOLUME = {12},
      YEAR = {1984},
    NUMBER = {4},
     PAGES = {1388--1399},
      ISSN = {0090-5364,2168-8966},
   MRCLASS = {62E20 (62G05)},
  MRNUMBER = {765931},
MRREVIEWER = {P.\ K.\ Sen},
       DOI = {10.1214/aos/1176346799},
       URL = {https://doi.org/10.1214/aos/1176346799},
}

@article{lindsay198geometry2,
  title={The geometry of mixture likelihoods, part {II}: the exponential family},
  author={Lindsay, Bruce G},
  journal={The Annals of Statistics},
  volume={11},
  number={3},
  pages={783--792},
  year={1983},
  publisher={Institute of Mathematical Statistics}
}

@article {jewell1982mixtures,
    AUTHOR = {Jewell, Nicholas P.},
     TITLE = {Mixtures of exponential distributions},
   JOURNAL = {Ann. Statist.},
  FJOURNAL = {The Annals of Statistics},
    VOLUME = {10},
      YEAR = {1982},
    NUMBER = {2},
     PAGES = {479--484},
      ISSN = {0090-5364,2168-8966},
   MRCLASS = {62G05 (62F99 62N05)},
  MRNUMBER = {653523},
MRREVIEWER = {John\ Kalbfleisch},
       URL =
              {http://links.jstor.org/sici?sici=0090-5364(198206)10:2<479:MOED>2.0.CO;2-U&origin=MSN},
}

@article {kiefer1956consistency,
    AUTHOR = {Kiefer, J. and Wolfowitz, J.},
     TITLE = {Consistency of the maximum likelihood estimator in the
              presence of infinitely many incidental parameters},
   JOURNAL = {Ann. Math. Statist.},
  FJOURNAL = {Annals of Mathematical Statistics},
    VOLUME = {27},
      YEAR = {1956},
     PAGES = {887--906},
      ISSN = {0003-4851},
   MRCLASS = {62.0X},
  MRNUMBER = {86464},
MRREVIEWER = {A.\ Dvoretzky},
       DOI = {10.1214/aoms/1177728066},
       URL = {https://doi.org/10.1214/aoms/1177728066},
}

@article{polyanskiy2020self,
  title={Self-regularizing property of nonparametric maximum likelihood estimator in mixture models},
  author={Polyanskiy, Yury and Wu, Yihong},
  journal={arXiv preprint arXiv:2008.08244},
  year={2020}
}

@inproceedings{lindsay1995mixture,
  title={Mixture models: {T}heory, geometry, and applications},
  author={Lindsay, Bruce G},
  year={1995},
  organization={Ims}
}

@article {lindsay1983ageometry1,
    AUTHOR = {Lindsay, Bruce G.},
     TITLE = {The geometry of mixture likelihoods: a general theory},
   JOURNAL = {Ann. Statist.},
  FJOURNAL = {The Annals of Statistics},
    VOLUME = {11},
      YEAR = {1983},
    NUMBER = {1},
     PAGES = {86--94},
      ISSN = {0090-5364,2168-8966},
   MRCLASS = {62A10 (52A40 62G05)},
  MRNUMBER = {684866},
MRREVIEWER = {Rashid\ Ahmad},
       DOI = {10.1214/aos/1176346059},
       URL = {https://doi.org/10.1214/aos/1176346059},
}

@article{miyamoto2023closed,
  title={On Closed-Form expressions for the {F}isher-{R}ao Distance},
  author={Miyamoto, Henrique K and Meneghetti, F{\'a}bio CC and Costa, Sueli IR},
  journal={arXiv preprint arXiv:2304.14885},
  year={2023}
}

@article {mukherjee2022variational,
    AUTHOR = {Mukherjee, Sumit and Sen, Subhabrata},
     TITLE = {Variational inference in high-dimensional linear regression},
   JOURNAL = {J. Mach. Learn. Res.},
  FJOURNAL = {Journal of Machine Learning Research (JMLR)},
    VOLUME = {23},
      YEAR = {2022},
     PAGES = {Paper No. [304], 56},
      ISSN = {1532-4435,1533-7928},
   MRCLASS = {62J05 (60F10 62F15)},
  MRNUMBER = {4577743},
MRREVIEWER = {Marvin\ H. J. Gruber},
}

@article {carbonetto2012scalable,
    AUTHOR = {Carbonetto, Peter and Stephens, Matthew},
     TITLE = {Scalable variational inference for {B}ayesian variable
              selection in regression, and its accuracy in genetic
              association studies},
   JOURNAL = {Bayesian Anal.},
  FJOURNAL = {Bayesian Analysis},
    VOLUME = {7},
      YEAR = {2012},
    NUMBER = {1},
     PAGES = {73--107},
      ISSN = {1936-0975,1931-6690},
   MRCLASS = {62F07 (60J22 62F15 62P10)},
  MRNUMBER = {2896713},
       DOI = {10.1214/12-BA703},
       URL = {https://doi.org/10.1214/12-BA703},
}

@article{kim2022flexible,
  title={A flexible empirical Bayes approach to multiple linear regression and connections with penalized regression},
  author={Kim, Youngseok and Wang, Wei and Carbonetto, Peter and Stephens, Matthew},
  journal={arXiv preprint arXiv:2208.10910},
  year={2022}
}

@article {yuan2005efficient,
    AUTHOR = {Yuan, Ming and Lin, Yi},
     TITLE = {Efficient empirical {B}ayes variable selection and estimation
              in linear models},
   JOURNAL = {J. Amer. Statist. Assoc.},
  FJOURNAL = {Journal of the American Statistical Association},
    VOLUME = {100},
      YEAR = {2005},
    NUMBER = {472},
     PAGES = {1215--1225},
      ISSN = {0162-1459,1537-274X},
   MRCLASS = {62C12 (62F10 62J05)},
  MRNUMBER = {2236436},
       DOI = {10.1198/016214505000000367},
       URL = {https://doi.org/10.1198/016214505000000367},
}

@article {george2000calibration,
    AUTHOR = {George, Edward I. and Foster, Dean P.},
     TITLE = {Calibration and empirical {B}ayes variable selection},
   JOURNAL = {Biometrika},
  FJOURNAL = {Biometrika},
    VOLUME = {87},
      YEAR = {2000},
    NUMBER = {4},
     PAGES = {731--747},
      ISSN = {0006-3444,1464-3510},
   MRCLASS = {62C12 (62F15 62G15 62J07)},
  MRNUMBER = {1813972},
       DOI = {10.1093/biomet/87.4.731},
       URL = {https://doi.org/10.1093/biomet/87.4.731},
}

@article {nebebe1986bayes,
    AUTHOR = {Nebebe, Fassil and Stroud, T. W. F.},
     TITLE = {Bayes and empirical {B}ayes shrinkage estimation of regression
              coefficients},
   JOURNAL = {Canad. J. Statist.},
  FJOURNAL = {The Canadian Journal of Statistics. La Revue Canadienne de
              Statistique},
    VOLUME = {14},
      YEAR = {1986},
    NUMBER = {4},
     PAGES = {267--280},
      ISSN = {0319-5724,1708-945X},
   MRCLASS = {62C12 (62J07)},
  MRNUMBER = {876752},
       DOI = {10.2307/3315184},
       URL = {https://doi.org/10.2307/3315184},
}

@article{mukherjee2023mean,
  title={A Mean Field Approach to Empirical Bayes Estimation in High-dimensional Linear Regression},
  author={Mukherjee, Sumit and Sen, Bodhisattva and Sen, Subhabrata},
  journal={arXiv preprint arXiv:2309.16843},
  year={2023}
}

@article{nitanda2017stochastic,
  title={Stochastic particle gradient descent for infinite ensembles},
  author={Nitanda, Atsushi and Suzuki, Taiji},
  journal={arXiv preprint arXiv:1712.05438},
  year={2017}
}

@article{sirignano2020mean,
  title={Mean field analysis of neural networks: A law of large numbers},
  author={Sirignano, Justin and Spiliopoulos, Konstantinos},
  journal={SIAM Journal on Applied Mathematics},
  volume={80},
  number={2},
  pages={725--752},
  year={2020},
  publisher={SIAM}
}

@article{rotskoff2018neural,
  title={Neural networks as interacting particle systems: {A}symptotic convexity of the loss landscape and universal scaling of the approximation error},
  author={Rotskoff, Grant M and Vanden-Eijnden, Eric},
  journal={Stat},
  volume={1050},
  pages={22},
  year={2018}
}

@article{yao2022mean,
  title={Mean field variational inference via Wasserstein gradient flow},
  author={Yao, Rentian and Yang, Yun},
  journal={arXiv preprint arXiv:2207.08074},
  year={2022}
}

@article {mei2018mean,
    AUTHOR = {Mei, Song and Montanari, Andrea and Nguyen, Phan-Minh},
     TITLE = {A mean field view of the landscape of two-layer neural
              networks},
   JOURNAL = {Proc. Natl. Acad. Sci. USA},
  FJOURNAL = {Proceedings of the National Academy of Sciences of the United
              States of America},
    VOLUME = {115},
      YEAR = {2018},
    NUMBER = {33},
     PAGES = {E7665--E7671},
      ISSN = {0027-8424,1091-6490},
   MRCLASS = {92B20 (62M45)},
  MRNUMBER = {3845070},
       DOI = {10.1073/pnas.1806579115},
       URL = {https://doi.org/10.1073/pnas.1806579115},
}

@article{chizat2018global,
  title={On the global convergence of gradient descent for over-parameterized models using optimal transport},
  author={Chizat, Lenaic and Bach, Francis},
  journal={Advances in neural information processing systems},
  volume={31},
  year={2018}
}

@article {santambrogio2017euclidean,
    AUTHOR = {Santambrogio, Filippo},
     TITLE = {\{{E}uclidean, metric, and {W}asserstein\} gradient flows: an
              overview},
   JOURNAL = {Bull. Math. Sci.},
  FJOURNAL = {Bulletin of Mathematical Sciences},
    VOLUME = {7},
      YEAR = {2017},
    NUMBER = {1},
     PAGES = {87--154},
      ISSN = {1664-3607,1664-3615},
   MRCLASS = {34G25 (35Q84 49J45 49M29 49Q20)},
  MRNUMBER = {3625852},
MRREVIEWER = {Beno\^{i}t\ Kloeckner},
       DOI = {10.1007/s13373-017-0101-1},
       URL = {https://doi.org/10.1007/s13373-017-0101-1},
}

@book {villani2009optimal,
    AUTHOR = {Villani, C\'{e}dric},
     TITLE = {Optimal transport: Old and new},
    SERIES = {Grundlehren der mathematischen Wissenschaften [Fundamental
              Principles of Mathematical Sciences]},
    VOLUME = {338},
 PUBLISHER = {Springer-Verlag, Berlin},
      YEAR = {2009},
     PAGES = {xxii+973},
      ISBN = {978-3-540-71049-3},
   MRCLASS = {49-02 (28A75 37J50 49Q20 53C23 58E30)},
  MRNUMBER = {2459454},
MRREVIEWER = {Dario\ Cordero-Erausquin},
       DOI = {10.1007/978-3-540-71050-9},
       URL = {https://doi.org/10.1007/978-3-540-71050-9},
}

@book {ambrosio2008gradient,
    AUTHOR = {Ambrosio, Luigi and Gigli, Nicola and Savar\'{e}, Giuseppe},
     TITLE = {Gradient flows in metric spaces and in the space of
              probability measures},
    SERIES = {Lectures in Mathematics ETH Z\"{u}rich},
   EDITION = {2nd},
 PUBLISHER = {Birkh\"{a}user Verlag, Basel},
      YEAR = {2008},
     PAGES = {x+334},
      ISBN = {978-3-7643-8721-1},
   MRCLASS = {49-02 (28A33 35K55 35K90 49Q20 60B05)},
  MRNUMBER = {2401600},
MRREVIEWER = {Pietro\ Celada},
}

@incollection {bakry1985diffusions,
    AUTHOR = {Bakry, D. and \'{E}mery, Michel},
     TITLE = {Diffusions hypercontractives},
 BOOKTITLE = {S\'{e}minaire de probabilit\'{e}s, {XIX}, 1983/84},
    SERIES = {Lecture Notes in Math.},
    VOLUME = {1123},
     PAGES = {177--206},
 PUBLISHER = {Springer, Berlin},
      YEAR = {1985},
      ISBN = {3-540-15230-X},
   MRCLASS = {60J60 (58C40 58G32)},
  MRNUMBER = {889476},
MRREVIEWER = {Jacques\ Vauthier},
       DOI = {10.1007/BFb0075847},
       URL = {https://doi.org/10.1007/BFb0075847},
}

@article {ghang2014sharp,
    AUTHOR = {Ghang, Whan and Martin, Zane and Waruhiu, Steven},
     TITLE = {The sharp log-{S}obolev inequality on a compact interval},
   JOURNAL = {Involve},
  FJOURNAL = {Involve. A Journal of Mathematics},
    VOLUME = {7},
      YEAR = {2014},
    NUMBER = {2},
     PAGES = {181--186},
      ISSN = {1944-4176,1944-4184},
   MRCLASS = {26Dxx},
  MRNUMBER = {3133718},
       DOI = {10.2140/involve.2014.7.181},
       URL = {https://doi.org/10.2140/involve.2014.7.181},
}

@article {zhang2011uniform,
    AUTHOR = {Zhang, Zhengliang and Qian, Bin and Ma, Yutao},
     TITLE = {Uniform logarithmic {S}obolev inequality for {B}oltzmann
              measures with exterior magnetic field over spheres},
   JOURNAL = {Acta Appl. Math.},
  FJOURNAL = {Acta Applicandae Mathematicae},
    VOLUME = {116},
      YEAR = {2011},
    NUMBER = {3},
     PAGES = {305--315},
      ISSN = {0167-8019,1572-9036},
   MRCLASS = {60E15 (26D15 39B62)},
  MRNUMBER = {2854732},
       DOI = {10.1007/s10440-011-9644-4},
       URL = {https://doi.org/10.1007/s10440-011-9644-4},
}

@article {bobkov1999exponential,
    AUTHOR = {Bobkov, S. G. and G\"{o}tze, F.},
     TITLE = {Exponential integrability and transportation cost related to
              logarithmic {S}obolev inequalities},
   JOURNAL = {J. Funct. Anal.},
  FJOURNAL = {Journal of Functional Analysis},
    VOLUME = {163},
      YEAR = {1999},
    NUMBER = {1},
     PAGES = {1--28},
      ISSN = {0022-1236,1096-0783},
   MRCLASS = {46E35 (60E15)},
  MRNUMBER = {1682772},
MRREVIEWER = {Ferenc\ Weisz},
       DOI = {10.1006/jfan.1998.3326},
       URL = {https://doi.org/10.1006/jfan.1998.3326},
}

@article{lambert2022variational,
  title={Variational inference via {W}asserstein gradient flows},
  author={Lambert, Marc and Chewi, Sinho and Bach, Francis and Bonnabel, Silv{\`e}re and Rigollet, Philippe},
  journal={Advances in Neural Information Processing Systems},
  volume={35},
  pages={14434--14447},
  year={2022}
}

@article{lu2023birth,
  title={Birth--death dynamics for sampling: global convergence, approximations and their asymptotics},
  author={Lu, Yulong and Slep{$\v{c}$}ev, Dejan and Wang, Lihan},
  journal={Nonlinearity},
  volume={36},
  number={11},
  pages={5731},
  year={2023},
  publisher={IOP Publishing}
}

@article{domingo2023explicit,
  title={An Explicit Expansion of the {K}ullback-{L}eibler Divergence along its {F}isher-{R}ao Gradient Flow},
  author={Domingo-Enrich, Carles and Pooladian, Aram-Alexandre},
  journal={arXiv preprint arXiv:2302.12229},
  year={2023}
}

@article{lu2019accelerating,
  title={Accelerating {L}angevin sampling with birth-death},
  author={Lu, Yulong and Lu, Jianfeng and Nolen, James},
  journal={arXiv preprint arXiv:1905.09863},
  year={2019}
}

@article {jordan1998variational,
    AUTHOR = {Jordan, Richard and Kinderlehrer, David and Otto, Felix},
     TITLE = {The variational formulation of the {F}okker-{P}lanck equation},
   JOURNAL = {SIAM J. Math. Anal.},
  FJOURNAL = {SIAM Journal on Mathematical Analysis},
    VOLUME = {29},
      YEAR = {1998},
    NUMBER = {1},
     PAGES = {1--17},
      ISSN = {0036-1410,1095-7154},
   MRCLASS = {35Q99 (35A15 49J99 60J60 82C31)},
  MRNUMBER = {1617171},
MRREVIEWER = {Thierry\ Goudon},
       DOI = {10.1137/S0036141096303359},
       URL = {https://doi.org/10.1137/S0036141096303359},
}

@article {zhang2009generalized,
    AUTHOR = {Zhang, Cun-Hui},
     TITLE = {Generalized maximum likelihood estimation of normal mixture
              densities},
   JOURNAL = {Statist. Sinica},
  FJOURNAL = {Statistica Sinica},
    VOLUME = {19},
      YEAR = {2009},
    NUMBER = {3},
     PAGES = {1297--1318},
      ISSN = {1017-0405,1996-8507},
   MRCLASS = {62G07 (60F10)},
  MRNUMBER = {2536157},
MRREVIEWER = {Alireza\ Nematollahi},
}

@article {jiang2009general,
    AUTHOR = {Jiang, Wenhua and Zhang, Cun-Hui},
     TITLE = {General maximum likelihood empirical {B}ayes estimation of
              normal means},
   JOURNAL = {Ann. Statist.},
  FJOURNAL = {The Annals of Statistics},
    VOLUME = {37},
      YEAR = {2009},
    NUMBER = {4},
     PAGES = {1647--1684},
      ISSN = {0090-5364,2168-8966},
   MRCLASS = {62C12 (62C25 62G05)},
  MRNUMBER = {2533467},
MRREVIEWER = {Tonglin\ Zhang},
       DOI = {10.1214/08-AOS638},
       URL = {https://doi.org/10.1214/08-AOS638},
}

@article{yan2023learning,
  title={Learning {G}aussian mixtures using the {W}asserstein-{F}isher-{R}ao gradient flow},
  author={Yan, Yuling and Wang, Kaizheng and Rigollet, Philippe},
  journal={arXiv preprint arXiv:2301.01766},
  year={2023}
}

@article {bauerschmidt2019simple,
    AUTHOR = {Bauerschmidt, Roland and Bodineau, Thierry},
     TITLE = {A very simple proof of the {LSI} for high temperature spin
              systems},
   JOURNAL = {J. Funct. Anal.},
  FJOURNAL = {Journal of Functional Analysis},
    VOLUME = {276},
      YEAR = {2019},
    NUMBER = {8},
     PAGES = {2582--2588},
      ISSN = {0022-1236,1096-0783},
   MRCLASS = {82D30},
  MRNUMBER = {3926125},
MRREVIEWER = {Ji\ Oon\ Lee},
       DOI = {10.1016/j.jfa.2019.01.007},
       URL = {https://doi.org/10.1016/j.jfa.2019.01.007},
}

@inproceedings{goldfeld2020gaussian,
  title={Gaussian-smoothed optimal transport: Metric structure and statistical efficiency},
  author={Goldfeld, Ziv and Greenewald, Kristjan},
  booktitle={International Conference on Artificial Intelligence and Statistics},
  pages={3327--3337},
  year={2020},
  organization={PMLR}
}

@inproceedings{rio2009upper,
  title={Upper bounds for minimal distances in the central limit theorem},
  author={Rio, Emmanuel},
  booktitle={Annales de l'IHP Probabilit{\'e}s et statistiques},
  volume={45},
  number={3},
  pages={802--817},
  year={2009}
}

@article{bardet2018functional,
  title={Functional inequalities for {G}aussian convolutions of compactly supported measures: explicit bounds and dimension dependence},
  author={Bardet, Jean-Baptiste and Gozlan, Natha{\"e}l and Malrieu, Florent and Zitt, Pierre-Andr{\'e}},
  year={2018}
}

@article{rudelson2013hanson,
  title={Hanson-{W}right inequality and sub-gaussian concentration},
  author={Rudelson, Mark and Vershynin, Roman},
  year={2013}
}

@article{zhang2014confidence,
  title={Confidence intervals for low dimensional parameters in high dimensional linear models},
  author={Zhang, Cun-Hui and Zhang, Stephanie S},
  journal={Journal of the Royal Statistical Society: Series B: Statistical Methodology},
  pages={217--242},
  year={2014},
  publisher={JSTOR}
}

@article{wang2021nonparametric,
  title={A nonparametric empirical {B}ayes approach to large-scale multivariate regression},
  author={Wang, Yihe and Zhao, Sihai Dave},
  journal={Computational Statistics \& Data Analysis},
  volume={156},
  pages={107130},
  year={2021},
  publisher={Elsevier}
}

@article{jiang2021nonparametric,
  title={A Nonparametric Maximum Likelihood Approach to Mixture of Regression},
  author={Jiang, Hansheng and Guntuboyina, Adityanand},
  journal={arXiv preprint arXiv:2108.09816},
  year={2021}
}

@article{zhang2018estimation,
  title={Estimation of complex effect-size distributions using summary-level statistics from genome-wide association studies across 32 complex traits},
  author={Zhang, Yan and Qi, Guanghao and Park, Ju-Hyun and Chatterjee, Nilanjan},
  journal={Nature Genetics},
  volume={50},
  number={9},
  pages={1318--1326},
  year={2018},
  publisher={Nature Publishing Group US New York}
}

@article{oconnor2021distribution,
  title={The distribution of common-variant effect sizes},
  author={O'Connor, Luke J},
  journal={Nature Genetics},
  volume={53},
  number={8},
  pages={1243--1249},
  year={2021},
  publisher={Nature Publishing Group US New York}
}

@article{spence2022flexible,
  title={A flexible modeling and inference framework for estimating variant effect sizes from {GWAS} summary statistics},
  author={Spence, Jeffrey P and Sinnott-Armstrong, Nasa and Assimes, Themistocles L and Pritchard, Jonathan K},
  journal={BioRxiv},
  pages={2022--04},
  year={2022},
  publisher={Cold Spring Harbor Laboratory}
}

@article{morgante2023flexible,
  title={A flexible empirical {B}ayes approach to multivariate multiple regression, and its improved accuracy in predicting multi-tissue gene expression from genotypes},
  author={Morgante, Fabio and Carbonetto, Peter and Wang, Gao and Zou, Yuxin and Sarkar, Abhishek and Stephens, Matthew},
  journal={PLoS Genetics},
  volume={19},
  number={7},
  pages={e1010539},
  year={2023},
  publisher={Public Library of Science San Francisco, CA USA}
}

@article{levine2001implementations,
  title={Implementations of the {M}onte {C}arlo {EM} algorithm},
  author={Levine, Richard A and Casella, George},
  journal={Journal of Computational and Graphical Statistics},
  volume={10},
  number={3},
  pages={422--439},
  year={2001},
  publisher={Taylor \& Francis}
}

@article{wei1990monte,
  title={A {M}onte {C}arlo implementation of the {EM} algorithm and the poor man's data augmentation algorithms},
  author={Wei, Greg CG and Tanner, Martin A},
  journal={Journal of the American statistical Association},
  volume={85},
  number={411},
  pages={699--704},
  year={1990},
  publisher={Taylor \& Francis}
}

@article{booth1999maximizing,
  title={Maximizing generalized linear mixed model likelihoods with an automated {M}onte {C}arlo {EM} algorithm},
  author={Booth, James G and Hobert, James P},
  journal={Journal of the Royal Statistical Society Series B: Statistical Methodology},
  volume={61},
  number={1},
  pages={265--285},
  year={1999},
  publisher={Oxford University Press}
}

@article{neath2013convergence,
  title={On convergence properties of the {M}onte {C}arlo {EM} algorithm},
  author={Neath, Ronald C},
  journal={Advances in modern statistical theory and applications: a Festschrift in Honor of Morris L. Eaton},
  volume={10},
  pages={43--63},
  year={2013},
  publisher={Institute of Mathematical Statistics}
}

@article{chan1995monte,
  title={{M}onte {C}arlo {EM} estimation for time series models involving counts},
  author={Chan, KS and Ledolter, Johannes},
  journal={Journal of the American Statistical Association},
  volume={90},
  number={429},
  pages={242--252},
  year={1995},
  publisher={Taylor \& Francis}
}

@article{fort2003convergence,
  title={Convergence of the {M}onte {C}arlo expectation maximization for curved exponential families},
  author={Fort, Gersende and Moulines, Eric},
  journal={The Annals of Statistics},
  volume={31},
  number={4},
  pages={1220--1259},
  year={2003},
  publisher={Institute of Mathematical Statistics}
}

@article{efron2010correlated,
  title={Correlated z-values and the accuracy of large-scale statistical estimates},
  author={Efron, Bradley},
  journal={Journal of the American Statistical Association},
  volume={105},
  number={491},
  pages={1042--1055},
  year={2010},
  publisher={Taylor \& Francis}
}

@article{sun2018solving,
  title={Solving the empirical bayes normal means problem with correlated noise},
  author={Sun, Lei and Stephens, Matthew},
  journal={arXiv preprint arXiv:1812.07488},
  year={2018}
}

@article{zhou2021fast,
  title={A fast and robust {B}ayesian nonparametric method for prediction of complex traits using summary statistics},
  author={Zhou, Geyu and Zhao, Hongyu},
  journal={PLoS genetics},
  volume={17},
  number={7},
  pages={e1009697},
  year={2021},
  publisher={Public Library of Science}
}

@incollection{neal1998view,
  title={A view of the {EM} algorithm that justifies incremental, sparse, and other variants},
  author={Neal, Radford M and Hinton, Geoffrey E},
  booktitle={Learning in graphical models},
  pages={355--368},
  year={1998},
  publisher={Springer}
}

@article{roberts1996exponential,
  title={Exponential convergence of {L}angevin distributions and their discrete approximations},
  author={Roberts, Gareth O and Tweedie, Richard L},
  journal={Bernoulli},
  pages={341--363},
  year={1996},
  publisher={JSTOR}
}

@article{schwartzman2010comment,
  title={Comment on ``{C}orrelated z-values and the accuracy of large-scale statistical estimates'' by {B}radley {E}fron},
  author={Schwartzman, Armin},
  journal={Journal of the American Statistical Association},
  volume={105},
  number={491},
  pages={1059},
  year={2010},
  publisher={NIH Public Access}
}

@article{ahn2021efficient,
  title={Efficient constrained sampling via the mirror-{L}angevin algorithm},
  author={Ahn, Kwangjun and Chewi, Sinho},
  journal={Advances in Neural Information Processing Systems},
  volume={34},
  pages={28405--28418},
  year={2021}
}

@article{durmus2019analysis,
  title={Analysis of {L}angevin {M}onte {C}arlo via convex optimization},
  author={Durmus, Alain and Majewski, Szymon and Miasojedow, B{\l}a{\.z}ej},
  journal={The Journal of Machine Learning Research},
  volume={20},
  number={1},
  pages={2666--2711},
  year={2019},
  publisher={JMLR. org}
}

@article{pereyra2016proximal,
  title={Proximal {M}arkov {C}hain {M}onte {C}arlo algorithms},
  author={Pereyra, Marcelo},
  journal={Statistics and Computing},
  volume={26},
  pages={745--760},
  year={2016},
  publisher={Springer}
}

@article{ma2021there,
  title={Is there an analog of {N}esterov acceleration for gradient-based {MCMC}?},
  author={Ma, Yi-An and Chatterji, Niladri S and Cheng, Xiang and Flammarion, Nicolas and Bartlett, Peter L and Jordan, Michael I},
  year={2021}
}

@book{vershynin2018high,
  title={High-dimensional probability: An introduction with applications in data science},
  author={Vershynin, Roman},
  volume={47},
  year={2018},
  publisher={Cambridge university press}
}

@article{dalalyan2012sparse,
  title={Sparse regression learning by aggregation and {L}angevin {M}onte-{C}arlo},
  author={Dalalyan, Arnak S and Tsybakov, Alexandre B},
  journal={Journal of Computer and System Sciences},
  volume={78},
  number={5},
  pages={1423--1443},
  year={2012},
  publisher={Elsevier}
}

@article{nickl2022polynomial,
  title={On polynomial-time computation of high-dimensional posterior measures by {L}angevin-type algorithms},
  author={Nickl, Richard and Wang, Sven},
  journal={Journal of the European Mathematical Society},
  year={2022}
}

@inproceedings{balasubramanian2022towards,
  title={Towards a theory of non-log-concave sampling: first-order stationarity guarantees for {L}angevin {M}onte {C}arlo},
  author={Balasubramanian, Krishna and Chewi, Sinho and Erdogdu, Murat A and Salim, Adil and Zhang, Shunshi},
  booktitle={Conference on Learning Theory},
  pages={2896--2923},
  year={2022},
  organization={PMLR}
}

@article{celentano2023mean,
  title={Mean-field variational inference with the {TAP} free energy: {G}eometric and statistical properties in linear models},
  author={Celentano, Michael and Fan, Zhou and Lin, Licong and Mei, Song},
  journal={arXiv preprint arXiv:2311.08442},
  year={2023}
}

@article{girolami2011riemann,
  title={Riemann manifold {L}angevin and {H}amiltonian {M}onte {C}arlo methods},
  author={Girolami, Mark and Calderhead, Ben},
  journal={Journal of the Royal Statistical Society Series B: Statistical Methodology},
  volume={73},
  number={2},
  pages={123--214},
  year={2011},
  publisher={Oxford University Press}
}

@inproceedings{welling2011bayesian,
  title={Bayesian learning via stochastic gradient {L}angevin dynamics},
  author={Welling, Max and Teh, Yee W},
  booktitle={Proceedings of the 28th international conference on machine learning (ICML-11)},
  pages={681--688},
  year={2011}
}

@inproceedings{cheng2018underdamped,
  title={Underdamped {L}angevin {MCMC}: {A} non-asymptotic analysis},
  author={Cheng, Xiang and Chatterji, Niladri S and Bartlett, Peter L and Jordan, Michael I},
  booktitle={Conference on learning theory},
  pages={300--323},
  year={2018},
  organization={PMLR}
}

@article{neal2011mcmc,
  title={{MCMC} using {H}amiltonian dynamics},
  author={Neal, Radford M},
  journal={Handbook of Markov Chain Monte Carlo},
  volume={2},
  number={11},
  pages={2},
  year={2011},
  publisher={Chapman and Hall/CRC}
}

@article{ma2015complete,
  title={A complete recipe for stochastic gradient {MCMC}},
  author={Ma, Yi-An and Chen, Tianqi and Fox, Emily},
  journal={Advances in neural information processing systems},
  volume={28},
  year={2015}
}

@article{ledoux2006concentration,
     author = {Ledoux, Michel},
     title = {Concentration of measure and logarithmic {Sobolev} inequalities},
     journal = {S\'eminaire de probabilit\'es de Strasbourg},
     pages = {120--216},
     publisher = {Springer - Lecture Notes in Mathematics},
     volume = {33},
     year = {1999},
     zbl = {0957.60016},
     mrnumber = {1767995},
     language = {en},
     url = {http://www.numdam.org/item/SPS_1999__33__120_0/}
}

@inproceedings{zimmermann2016elementary,
  title={Elementary proof of logarithmic {S}obolev inequalities for Gaussian convolutions on $\mathbb{R}$},
  author={Zimmermann, David},
  booktitle={Annales math{\'e}matiques Blaise Pascal},
  volume={23},
  number={1},
  pages={129--140},
  year={2016}
}

@article{cvxr,
  title = {{CVXR}: An {R} Package for Disciplined Convex Optimization},
  author = {Anqi Fu and Balasubramanian Narasimhan and Stephen Boyd},
  journal = {Journal of Statistical Software},
  year = {2020},
  volume = {94},
  number = {14},
  pages = {1--34},
  doi = {10.18637/jss.v094.i14},
}

@article{otto2000generalization,
  title={Generalization of an inequality by {T}alagrand and links with the logarithmic {S}obolev inequality},
  author={Otto, Felix and Villani, C{\'e}dric},
  journal={Journal of Functional Analysis},
  volume={173},
  number={2},
  pages={361--400},
  year={2000},
  publisher={Elsevier}
}

@article{bogachev2011uniqueness,
  title={On uniqueness problems related to the {F}okker--{P}lanck--{K}olmogorov equation for measures},
  author={Bogachev, VI and R{\"o}ckner, M and Shaposhnikov, SV},
  journal={J. Math. Sci.(New York)},
  volume={179},
  number={1},
  pages={7--47},
  year={2011}
}

@book{bogachev2022fokker,
  title={Fokker--{P}lanck--{K}olmogorov {E}quations},
  author={Bogachev, Vladimir I and Krylov, Nicolai V and R{\"o}ckner, Michael and Shaposhnikov, Stanislav V},
  volume={207},
  year={2022},
  publisher={American Mathematical Society}
}

@book{ladyzhenskaia1968linear,
  title={Linear and quasi-linear equations of parabolic type},
  author={Ladyzhenskaia, Olga Aleksandrovna and Solonnikov, Vsevolod Alekseevich and Ural'tseva, Nina N},
  volume={23},
  year={1968},
  publisher={American Mathematical Soc.}
}

@book{adams2003sobolev,
  title={Sobolev spaces},
  author={Adams, Robert A and Fournier, John JF},
  year={2003},
	edition={2nd},
  publisher={Elsevier}
}

@article{ge2019polygenic,
  title={Polygenic prediction via {B}ayesian regression and continuous shrinkage priors},
  author={Ge, Tian and Chen, Chia-Yen and Ni, Yang and Feng, Yen-Chen Anne and Smoller, Jordan W},
  journal={Nature communications},
  volume={10},
  number={1},
  pages={1776},
  year={2019},
  publisher={Nature Publishing Group UK London}
}

@article{zhou2013polygenic,
  title={Polygenic modeling with Bayesian sparse linear mixed models},
  author={Zhou, Xiang and Carbonetto, Peter and Stephens, Matthew},
  journal={PLoS genetics},
  volume={9},
  number={2},
  pages={e1003264},
  year={2013},
  publisher={Public Library of Science San Francisco, USA}
}

@article{van2014asymptotically,
  title={On asymptotically optimal confidence regions and tests for high-dimensional models},
  author={van de Geer, Sara and B{\"u}hlmann, Peter and Ritov, Ya'acov and Dezeure, Ruben},
  journal={The Annals of Statistics},
  pages={1166--1202},
  year={2014},
  publisher={JSTOR}
}

@article{lloyd2019improved,
  title={Improved polygenic prediction by {B}ayesian multiple regression on summary statistics},
  author={Lloyd-Jones, Luke R and Zeng, Jian and Sidorenko, Julia and Yengo, Lo{\"\i}c and Moser, Gerhard and Kemper, Kathryn E and Wang, Huanwei and Zheng, Zhili and Magi, Reedik and Esko, T{\~o}nu and others},
  journal={Nature communications},
  volume={10},
  number={1},
  pages={5086},
  year={2019},
  publisher={Nature Publishing Group UK London}
}

@article{diebolt1993asymptotic,
  title={Asymptotic properties of a stochastic {EM} algorithm for estimating mixing proportions},
  author={Diebolt, Jean and Celeux, Gilles},
  journal={Stochastic Models},
  volume={9},
  number={4},
  pages={599--613},
  year={1993},
  publisher={Taylor \& Francis}
}

@article{delyon1999convergence,
  title={Convergence of a stochastic approximation version of the {EM} algorithm},
  author={Delyon, Bernard and Lavielle, Marc and Moulines, Eric},
  journal={Annals of statistics},
  pages={94--128},
  year={1999},
  publisher={JSTOR}
}

@article{caprio2025error,
  title={Error bounds for particle gradient descent, and extensions of the log-{S}obolev and {T}alagrand inequalities},
  author={Caprio, Rocco and Kuntz, Juan and Power, Samuel and Johansen, Adam M},
  journal={Journal of Machine Learning Research},
  volume={26},
  number={103},
  pages={1--38},
  year={2025}
}

@inproceedings{marion2024implicit,
  title={Implicit Diffusion: {E}fficient Optimization through Stochastic Sampling},
  author={Marion, Pierre and Korba, Anna and Bartlett, Peter and Blondel, Mathieu and De Bortoli, Valentin and Doucet, Arnaud and Llinares-L{\'o}pez, Felipe and Paquette, Courtney and Berthet, Quentin},
  booktitle={ICML 2024 Workshop on Differentiable Almost Everything: Differentiable Relaxations, Algorithms, Operators, and Simulators}
}

@article{montanari2025provably,
  title={Provably efficient posterior sampling for sparse linear regression via measure decomposition},
  author={Montanari, Andrea and Wu, Yuchen},
  journal={Journal of the American Statistical Association},
  number={just-accepted},
  pages={1--22},
  year={2025},
  publisher={Taylor \& Francis}
}

@article{lee2026parametric,
  title={Parametric Mean-Field empirical {B}ayes in high-dimensional linear regression},
  author={Lee, Seunghyun and Deb, Nabarun},
  journal={arXiv preprint arXiv:2601.16842},
  year={2026}
}

@inproceedings{lim2024momentum,
  title={Momentum particle maximum likelihood},
  author={Lim, JN and Kuntz, J and Power, S and Johansen, Adam M},
  booktitle={Proceedings of 41st International Conference on Machine Learning (ICML)},
  volume={235},
  pages={29816--29871},
  year={2024}
}

@article{fan2025dynamicalI,
  title={Dynamical mean-field analysis of adaptive {L}angevin diffusions: {P}ropagation-of-chaos and convergence of the linear response},
  author={Fan, Zhou and Ko, Justin and Loureiro, Bruno and Lu, Yue M and Shen, Yandi},
  journal={arXiv preprint arXiv:2504.15556},
  year={2025}
}

@article{fan2025dynamicalII,
  title={Dynamical mean-field analysis of adaptive {L}angevin diffusions: {R}eplica-symmetric fixed point and empirical {B}ayes},
  author={Fan, Zhou and Ko, Justin and Loureiro, Bruno and Lu, Yue M and Shen, Yandi},
  journal={arXiv preprint arXiv:2504.15558},
  year={2025}
}

@article{rao1972estimation,
  title={Estimation of variance and covariance components in linear models},
  author={Rao, C Radhakrishna},
  journal={Journal of the American Statistical Association},
  volume={67},
  number={337},
  pages={112--115},
  year={1972},
  publisher={Taylor \& Francis}
}

@article{goodd1971nonparametric,
  title={Nonparametric roughness penalties for probability densities},
  author={Goodd, IJ and Gaskins, Ray A},
  journal={Biometrika},
  volume={58},
  number={2},
  pages={255--277},
  year={1971},
  publisher={Oxford University Press}
}

@article{madrid2018deconvolution,
  title={A deconvolution path for mixtures},
  author={Madrid-Padilla, Oscar-Hernan and Polson, Nicholas G and Scott, James},
  journal={Electronic Journal of Statistics},
  volume={12},
  pages={1717--1751},
  year={2018}
}

@article{adrion2020community,
  title={A community-maintained standard library of population genetic models},
  author={Adrion, Jeffrey R and Cole, Christopher B and Dukler, Noah and Galloway, Jared G and Gladstein, Ariella L and Gower, Graham and Kyriazis, Christopher C and Ragsdale, Aaron P and Tsambos, Georgia and Baumdicker, Franz and others},
  journal={elife},
  volume={9},
  pages={e54967},
  year={2020},
  publisher={eLife Sciences Publications, Ltd}
}

@article{pazokitoroudi2020efficient,
  title={Efficient variance components analysis across millions of genomes},
  author={Pazokitoroudi, Ali and Wu, Yue and Burch, Kathryn S and Hou, Kangcheng and Zhou, Aaron and Pasaniuc, Bogdan and Sankararaman, Sriram},
  journal={Nature communications},
  volume={11},
  number={1},
  pages={4020},
  year={2020},
  publisher={Nature Publishing Group UK London}
}

@book{ambrosio2005gradient,
  title={Gradient flows: in metric spaces and in the space of probability measures},
  author={Ambrosio, Luigi and Gigli, Nicola and Savar{\'e}, Giuseppe},
  year={2005},
  publisher={Springer}
}

@article{singh1979empirical,
  title={Empirical Bayes estimation in Lebesgue-exponential families with rates near the best possible rate},
  author={Singh, RS},
  journal={The Annals of Statistics},
  volume={7},
  number={4},
  pages={890--902},
  year={1979},
  publisher={Institute of Mathematical Statistics}
}

@article{greenshtein2009asymptotic,
  title={Asymptotic efficiency of simple decisions for the compound decision problem},
  author={Greenshtein, Eitan and Ritov, Ya'acov},
  journal={Lecture Notes-Monograph Series},
  pages={266--275},
  year={2009},
  publisher={JSTOR}
}

@article{Dicker2014Variance,
  author    = {Dicker, Lee H.},
  title     = {Variance estimation in high-dimensional linear models},
  journal   = {Biometrika},
  volume    = {101},
  number    = {2},
  pages     = {269--284},
  year      = {2014}
}

\end{document}